\documentclass[draft,reqno]{amsart}
\usepackage[utf8]{inputenc}
\usepackage{amssymb}
\usepackage{euscript}
\usepackage{extarrows}
\usepackage{mathrsfs}
\usepackage{a4wide}
\usepackage{amsmath}
\usepackage{setspace}
\usepackage{yhmath}
\usepackage{scalerel}

\makeatletter
\@namedef{subjclassname@2010}{%
	\textup{2010} Mathematics Subject Classification}
\makeatother

\newcommand\reallywidehat[1]{\arraycolsep=0pt\relax%
	\begin{array}{c}
		\stretchto{
			\scaleto{
				\scalerel*[\widthof{\ensuremath{#1}}]{\kern-.5pt\bigwedge\kern-.5pt}
				{\rule[-\textheight/2]{1ex}{\textheight}} %WIDTH-LIMITED BIG WEDGE
			}{\textheight} % 
		}{0.5ex}\\           % THIS SQUEEZES THE WEDGE TO 0.5ex HEIGHT
		#1\\                 % THIS STACKS THE WEDGE ATOP THE ARGUMENT
		\rule{-1ex}{0ex}
	\end{array}
}

\newtheorem{thm}{Theorem}%[section]
\newtheorem{cor}[thm]{Corollary}
\newtheorem{lem}[thm]{Lemma}
\newtheorem{pro}[thm]{Proposition}

\theoremstyle{remark}
\newtheorem{rem}[thm]{Remark}

\theoremstyle{definition}
\newtheorem{exa}[thm]{Example}
\newtheorem{dfn}[thm]{Definition}

\DeclareMathOperator{\D}{d\hspace{-0.25ex}}

\newcommand*{\cbb}{\mathbb C}
\newcommand*{\cfw}{C_{\phi,w}}

\newcommand*{\efw}{\mathsf{E}_{\phi,w}}

\newcommand*{\dz}[1]{{\EuScript D}(#1)}

\newcommand*{\hfw}{{\mathsf h}_{\phi,w}}

\newcommand*{\smalloplus}{\raise0pt\hbox{$\scriptscriptstyle \oplus$}}

\begin{document}
	\setstretch{1.1}
	\title[Mean transforms of composition operator pairs]{Mean transforms of unbounded weighted composition operator pairs}
	
	\author[J.B. Zhou]{
		Jing-Bin Zhou}
	\address{Jing-Bin Zhou, School of Mathematics,
		Shanghai University of Finance and Economics,
		777 Gouding Road, Shanghai 200433, P. R. China}
	\email{zhoujingbin@stu.sufe.edu.cn}
	
	\author[S. Yang]{Shihai Yang}
	\address{Shihai Yang, School of Mathematics,
		Shanghai University of Finance and Economics,
		777 Gouding Road, Shanghai 200433, P. R. China}
	\email{yang.shihai@mail.shufe.edu.cn}
	
	%\date{\today}
	\maketitle
	
\begin{abstract}
	In this paper, we first characterize the polar decomposition of unbounded weighted composition operator pairs \(\textbf{C}_{\phi,\omega}\) in an \(L^2\)-space. Based on this characterization, we introduce the \(\lambda\)-spherical mean transform \(\mathcal{M}_\lambda(\textbf{C}_{\phi,\omega})\) for $\lambda\in[0,1]$. We then investigate the dense definiteness of \(\mathcal{M}_\lambda(\textbf{C}_{\phi,\omega})\). As an application, we provide an example of a \(p\)-hyponormal operator whose Aluthge transform is densely defined, while its $\lambda$-mean transform has a trivial domain. Furthermore, we establish the relationship between the dense definiteness of $\textbf{C}_{\phi,\omega}$ and $\mathcal{M}_{\lambda}(\textbf{C}_{\phi,\omega})$, based on the notion of powers for operator pairs in the sense of M{\"u}ller and Soltysiak. We also give a characterization of spherically quasinormal weighted composition operator pairs via the $\lambda$-spherical mean transform, revealing some properties that differ from the single operator case. Finally, we characterize a class of spherically \(p\)-hyponormal weighted composition operators on discrete measure spaces. As a corollary, we present corresponding results on the spherical \(p\)-hyponormality of unbounded \(2\)-variable weighted shifts and theirs $\lambda$-spherical mean transforms. 
\end{abstract}

\section{Introduction}
Let $\mathcal{L}(\mathcal{H},\mathcal{K})$ be the set of all linear operators and $\mathcal{B}(\mathcal{H},\mathcal{K})$ the algebra of bounded linear operators between the complex Hilbert spaces $\mathcal{H}$ and $\mathcal{K}$. If $\mathcal{H}=\mathcal{K}$, the notation $\mathcal{B}(\mathcal{H},\mathcal{K})$ is abbreviated to $\mathcal{B}(\mathcal{H})$. For an operator $T\in\mathcal{L}(\mathcal{H},\mathcal{K})$, we denote its null space, adjoint, and norm (in case they exists) by $\mathcal{N}(T)$, $T^\ast$ and $\|T\|$, respectively. Let $T^\frac{1}{2}$ denote the positive square root of the positive selfadjoint operator $T$. It is well known that every densely defined closed operator $T\in\mathcal{L}(\mathcal{H},\mathcal{K})$ has the unique polar decomposition $T=U|T|$, where $|T|=(T^\ast T)^{1/2}$ and $U$ is a partial isometry with $\mathcal{N}(U)=\mathcal{N}(T)$. The Aluthge transform of $T\in\mathcal{B}(\mathcal{H})$, defined as $\Delta(T)=|T|^{\frac{1}{2}}U|T|^{\frac{1}{2}}$, was introduced by Aluthge \cite{alu-ieot-1990} in his study of $p$-hyponormal operators. The generalization by Ch\={o} and Tanahashi \cite{cho-smj-2002}, known as the {\em $\lambda$-Aluthge transform}, is given for $T\in\mathcal{B}(\mathcal{H})$ and $\lambda\in[0,1]$, as follows:\begin{equation}\label{lambdaaluthge}
	\Delta_{\lambda}(T)=|T|^\lambda U|T|^{1-\lambda}.
\end{equation}Many authors have investigated the Aluthge transform in different contexts (see \cite{cho-smj-2002, curto-crm-2019, sjj-mn-2023, curto-ieot-2018, zhou-glma-2023, Zid-Fil-2022, jabb-fil-2017}). One reason the Aluthge transform is of interest may lie in the fact that it preserves many properties of the original operator $T$. For instance, Jung et al. proved that $T$ has a nontrivial invariant subspace if and only if $\Delta(T)$ does \cite{jung-ieot-2000}, and also showed that $T$ and $\Delta(T)$ share the same spectrum \cite{jung-ieot-2001}. Moreover, Kim and Ko proved that $T$ has property ($\beta$) if and only if $\Delta$ has property ($\beta$) \cite{Kim-GM-2005}. While most research on the Aluthge transform concerns bounded operators, the literature on the unbounded case remains limited to \cite{trep-jmaa-2015} and \cite{benhida-mn-2020}. A natural question that arises for the Aluthge transform of unbounded operators is whether it preserves the dense definiteness and closedness of the original operator. Trepkowski in \cite{trep-jmaa-2015} constructed a densely defined weighted composition operator whose Aluthge transform has a trival domain; that is, the domain of the Aluthge transform is $\{0\}$. Benhida et al. \cite{benhida-mn-2020} established that the Aluthge transform of a closed weighted composition operator is not necessarily closed. On the other hand,  Lee et al. in \cite{lee-jmaa-2014} considered the unilateral weighted shift $W_{\alpha}\equiv(\alpha_{0},\alpha_{1},\alpha_{3},\cdots)$, obtaining $$\Delta(W_{\alpha})=|W_{\alpha}|^\frac{1}{2}U_{\alpha}|W_{\alpha}|^\frac{1}{2}=\text{shift}(\sqrt{\alpha_{0}\alpha_{1}},\sqrt{\alpha_{1}\alpha_{2}},\sqrt{\alpha_{2}\alpha_{3}},\cdots)$$
and
$$\mathcal{M}_{\frac{1}{2}}(W_{\alpha})=\frac{U_{\alpha}|W_{\alpha}|+|W_{\alpha}|U_{\alpha}}{2}=\text{shift}(\frac{\alpha_{0}+\alpha_{1}}{2},\frac{\alpha_{1}+\alpha_{2}}{2},\frac{\alpha_{2}+\alpha_{3}}{2},\cdots),$$
where $W_{\alpha}=U_{\alpha}|W_{\alpha}|$ is the polar decomposition of $W_{\alpha}$ and $\alpha\equiv\{\alpha_{n}\}_{n=0}^\infty$ is a positive bounded sequence. Thus, Lee et al. introduced the {\em mean transform} $\mathcal{M}_{\frac{1}{2}}(T)$ of $T$ as $\mathcal{M}_{\frac{1}{2}}(T)=\frac{U|T|+|T|U}{2}$, where $T=U|T|$ is the polar decomposition of $T\in\mathcal{B}(\mathcal{H})$. Later, this transform is generalized in \cite{zamani-jmaa-2021} by Zamani as 
\begin{equation}\label{zamanimean}
\mathcal{M}_{\lambda}(T)=\lambda U|T|+(1-\lambda)|T|U,\quad T\in\mathcal{B}(\mathcal{H})
\end{equation}
 for $\lambda\in [0,1]$. We refer the reader to \cite{lee-jmaa-2014,svic-jmaa-2024,ben-bjma-2020,cha-pro-2019} for more facts and details on generalized mean transform.

To extend the mean transform for bounded operators to the case of unbounded operator pairs, we first recall the $\lambda$-spherical Aluthge transform for operator tuples defined by Curto et al. \cite{curto-crm-2019}. Consider the set of all pairs of operators $\mathcal{L}(\mathcal{H})\times\mathcal{L}(\mathcal{H})=\{(T_{1},T_{2}):T_{i}\in\mathcal{L}(\mathcal{H}),i=1,2\}.$ For $\textbf{T}=(T_{1},T_{2})\in\mathcal{L}(\mathcal{H})\times\mathcal{L}(\mathcal{H}),$ it may be identified with the column operator, still denoted by $\textbf{T},$ that is,
\begin{equation*}
	\textbf{T}=\left(
	\begin{array}{c}
		T_{1} \\
		T_{2}
	\end{array}
	\right):\mathcal{H}\to\mathcal{H}\oplus\mathcal{H},x\mapsto\left(
	\begin{array}{c}
		T_{1}x \\
		T_{2}x
	\end{array}
	\right),
\end{equation*}
hence $\textbf{T}\in\mathcal{L}(\mathcal{H},\mathcal{H}\oplus\mathcal{H})$. Let
\begin{equation*}
	\left(
	\begin{array}{c}
		T_{1} \\
		T_{2}
	\end{array}
	\right)=U|\textbf{T}|=\left(
	\begin{array}{c}
		U_{1} \\
		U_{2}
	\end{array}
	\right)|\textbf{T}|
\end{equation*}
be the polar decomposition of $\textbf{T}\in\mathcal{B}(\mathcal{H},\mathcal{H}\oplus\mathcal{H}),$ where the positive operator $|\textbf{T}|$ is defined by $|\textbf{T}|=\sqrt{\textbf{T}^\ast\textbf{T}}=\sqrt{T_{1}^\ast T_{1}+T_{2}^\ast T_{2}}$ and $U=\left(
\begin{array}{c}
	U_{1} \\
	U_{2}
\end{array}
\right)$ is a spherical isometry from $\mathcal{H}$ into $\mathcal{H}\oplus\mathcal{H},$ then for any $\lambda\in[0,1],$ the pair of operators
\begin{equation}\label{curtoaluthge}
	\Delta_{\lambda}(\textbf{T})=(|\textbf{T}|^\lambda U_{1}|\textbf{T}|^{1-\lambda},|\textbf{T}|^\lambda U_{2}|\textbf{T}|^{1-\lambda})
\end{equation}
is called the {\em $\lambda$-spherical Aluthge transform} of the pair $\textbf{T}=(T_{1},T_{2})$. It should be noted that for closed densely defined operators $T_{1}$ and $T_{2}$, the column operator $\textbf{T}=(T_{1},T_{2})$ may not admit the polar decomposition (\cite[Question 21.2.16]{mortad-2022-book}). When the unbounded column operator $\textbf{T}=(T_{1},T_{2})$ has the polar decomposition $\textbf{T}=\left(
\begin{array}{c}
	U_{1} \\
	U_{2}
\end{array}
\right)|\textbf{T}|$, the forms of Equations \eqref{lambdaaluthge}, \eqref{curtoaluthge} and \eqref{zamanimean} lead to the definition of its $\lambda$-spherical mean transform$\colon$ $$\mathcal{M}_{\lambda}(\textbf{T})=(\lambda U_{1}|\textbf{T}|+(1-\lambda)|\textbf{T}|U_{1},\lambda U_{2}|\textbf{T}|+(1-\lambda)|\textbf{T}|U_{2})$$
for $\lambda\in[0,1]$. In particular, if $\textbf{T}$ is a densely defined weighted composition operator pair $\textbf{C}_{\phi,\omega}$ in an $L^2$-space, we prove that the column operator $\textbf{C}_{\phi,\omega}$ always admits a polar decomposition (Proposition \ref{basic}). Motivated by \cite{trep-jmaa-2015} and \cite{benhida-mn-2020}, we focus on the study of the mean transform for weighted composition operator pairs in the $L^2$ setting.

This paper is organized as follows. In Section 3, we show that a weighted composition operator pair $\textbf{C}_{\phi,\omega}=(C_{\phi_{1},\omega_{1}},C_{\phi_{2},\omega_{2}})$ is densely defined and closed, which guarantees the existence of its polar decomposition. We then provide an explicit description of this decomposition.
 
  In Section 4, we introduce the $\lambda$-spherical mean transform of $\textbf{C}_{\phi,\omega}$ for $\lambda\in[0,1]$, extending the $\lambda$-mean transform defined by Zamani \cite{zamani-jmaa-2021} for single bounded operators. For $\lambda\in(0,1)$, we prove that the $\lambda$-spherical mean transform $\mathcal{M}_{\lambda}(\textbf{C}_{\phi,\omega})$ of $(C_{\phi_{1},\omega_{1}},C_{\phi_{2},\omega_{2}})$ coincides with the composition operator pair $(C_{\phi_{1},\omega_{\lambda}^1},C_{\phi_{2},\omega_{\lambda}^2})$ (see Theorem \ref{meanbasic} (iv)). At the endpoints $\lambda=0,1$, it generally equals the restriction of that pair $(C_{\phi_{1},\omega_{\lambda}^1},C_{\phi_{2},\omega_{\lambda}^2})$ to the domain $\EuScript{D}(\mathcal{M}_{\lambda}(\textbf{C}_{\phi,\omega}))$, and we provide conditions for exact equality. For simplicity, we subsequently focus on the case $\lambda\in(0,1)$. Unlike the Aluthge transform for composition operators, the $\lambda$-spherical mean transform of $\textbf{C}_{\phi,\omega}$ is always closed. We characterize the dense definiteness of $\mathcal{M}_{\lambda}(\textbf{C}_{\phi,\omega})$ and, as an application, present a $p$-hyponormal operator $C_{\phi,\omega}$ whose Aluthge transform is densely defined, while its $\lambda$ mean transform $\mathcal{M}_{\lambda}(C_{\phi,\omega})$ has trivial domain. Furthermore, although the $\lambda$-spherical mean transform of $\textbf{C}_{\phi,\omega}$ does not preserve dense definiteness in general, it maps a weighted composition operator pair to another. This naturally raises the question of whether the $\lambda$-spherical mean transform is bijective on the set of bounded composition operator pairs. Example \ref{17} provides a negative answer. To conclude this section, we establish the relationship between the dense definiteness of $\textbf{C}_{\phi,\omega}$, its toral $\lambda$-spherical mean transform $\widehat{\mathcal{M}_{\lambda}(\textbf{C}_{\phi,\omega})}$, and its $\lambda$-spherical mean transform $\mathcal{M}_{\lambda}(\textbf{C}_{\phi,\omega})$, based on the notion of powers for operator pairs in the sense of M{\"u}ller and Soltysiak \cite{muller-studia-1992}. Some examples on $l^2$ (the space of square summable sequences) illustrate our results. 
  
  The final section of this paper is devoted to the study of generalized normality for composition operator pairs and their $\lambda$-spherical mean transforms, particularly spherical quasinormality and spherical $p$-hyponormality. We provide several characterizations of spherical quasinormality for such operator pairs and observe a key distinction from the single unbounded quasinormal operator case: for the pair $\textbf{C}_{\phi,\omega}$ with the polar decomposition $\textbf{C}_{\phi,\omega}=\left(
  \begin{array}{c}
  	U_{1} \\
  	U_{2}
  \end{array}
  \right)|\textbf{C}_{\phi,\omega}|$, the components are not necessarily commutative, i.e., $U_{i}|\textbf{C}_{\phi,\omega}|\ne|\textbf{C}_{\phi,\omega}|U_{i}$ for $i=1,2$. Furthermore, in order to investigate spherical $p$-hyponormality of $\textbf{C}_{\phi,\omega}=(C_{\phi_{1},\omega_{1}},C_{\phi_{2},\omega_{2}})$ defined via functional calculus for self-adjoint operators, we prove that the associated operator $\sum\limits_{i=1}^2C_{\phi_{i},\omega_{i}}C_{\phi_{i},\omega_{i}}^\ast$ is always densely defined, thus guaranteeing the existence of its adjoint. While $\sum\limits_{i=1}^2C_{\phi_{i},\omega_{i}}C_{\phi_{i},\omega_{i}}^\ast$ is automatically self-adjoint in the bounded case, Example \ref{notselfadjoint} shows that $\sum\limits_{i=1}^2C_{\phi_{i},\omega_{i}}C_{\phi_{i},\omega_{i}}^\ast$ may fail to be self-adjoint in the unbounded setting. Finally, we extend the recent work of Kim et al. \cite[Theorem 2.3]{Yoon-glma-2022} on spherical $p$-hyponormality for $2$-variable weighted shifts. We characterize spherical $p$-hyponormality for a class of weighted composition operator pairs over discrete measure spaces. As a corollary, we derive an analogous result for unbounded $2$-variable weighted shifts. Moreover, we establish sufficient conditions under which the $\lambda$-spherical mean transform preserves spherical $p$-hyponormality, thereby generalizing a result of Stanković \cite[Theorem 3.6]{svic-jmaa-2024}.

\section{Preliminaries}	
Let $\mathbb{R}$, $\mathbb{C}$, $\mathbb{N}$, and $\mathbb{Z}_{+}$ denote the sets of real numbers, complex numbers, positive integers, and nonnegative integers, respectively. The closure of a subset $A$ in a metric space is denoted by $\overline{A}$. For a function $f:X\to\mathbb{C}$ or $f:X\to[0,\infty]$ and a relation $\square\in\{=,>,\ne\}$, we write $\{f\,\square\,0\}$ for $\{x\in X:f(x)\,\square\,0\}$. The notation $\{f\,\square\,\infty\}$ is defined analogously. The characteristic function of a subset $\varOmega \subseteq X$ is denoted by $\chi_{\varOmega}$. Fix $k\in\mathbb{N}$, and let $\textbf{T}=(T_{1},T_{2},\cdots,T_{k})$ be a $k$-tuple of linear operators in a Hilbert space $(\mathcal{H},\|\cdot\|)$. Consider the Hilbert space $\mathcal{H}^{k+1}$, defined as the direct sum of $k+1$ copies of $\mathcal{H}$. Following \cite{js-jmsj-2003}, we define the {\em $\left(joint\right)$ domain} $\EuScript{D}(\textbf{T})$, the {\em $\left(joint\right)$ graph} $\mathcal{G}(\textbf{T})$ and the {\em $\left(joint\right)$ graph norm} $\|\cdot\|_{\textbf{T}}$ of the operator tuple $\textbf{T}$ by
\begin{equation*}
	\EuScript{D}(\textbf{T})=\bigcap\limits_{i=1}^k\EuScript{D}(T_{i}),
\end{equation*}
\begin{equation*}
	\mathcal{G}(\textbf{T})=\big\{(f,T_{1}f,T_{2}f,\cdots,T_{k}f)\in \mathcal{H}^{k+1}:f\in\EuScript{D}(\textbf{T})\big\},
\end{equation*}
and
\begin{equation*}
	\|f\|_{\textbf{T}}^2=\|f\|^2+\sum\limits_{i=1}^k\|T_{i}f\|^{2},
\end{equation*}
respectively. %The pair $\textbf{C}_{\phi,\omega}$ is said to be densely defined if the set $\EuScript{D}(\textbf{C}_{\phi,\omega})=\EuScript{D}(C_{\phi_{1},\omega_{1}})\bigcap\EuScript{D}(C_{\phi_{1},\omega_{1}})$ is dense in $L^{2}(\mu).$ 
The operator tuple $\textbf{T}$ is said to be closed if $\mathcal{G}(\textbf{T}$ is closed in $\mathcal{H}^{k+1}$. For another operator tuple $\textbf{S}=(S_{1},S_{2},\cdots,S_{k})$, we write $\textbf{T}\subseteq \textbf{S}$ if $\mathcal{G}(\textbf{T})\subseteq\mathcal{G}(\textbf{S})$, which is equivalent to $\EuScript{D}(\textbf{T})\subseteq\EuScript{D}(\textbf{S})$ and $T_{i}f=S_{i}f$ for $i=1,2,\cdots,k$ and $f\in\EuScript{D}(\textbf{T})$.

Let $(X,\mathscr{A},\mu)$ be a $\sigma$-finite measure space. Denote by $L^{2}(\mu)$ the Hilbert space of all square $\mu$-integrable $\mathscr{A}$-measurable complex functions on $X$, with inner product $\langle f,g\rangle=\int_{X}f\bar{g}\D\mu,$ $f,g\in L^{2}(\mu).$ A mapping $\phi:X\to X$ is called an {\em $\mathscr{A}$-measurable transformation} of $X$ if $\varphi^{-1}(\mathscr{A})\subset\mathscr{A}$, where $\varphi^{-1}(\mathscr{A})=\{\phi^{-1}(\varDelta)\colon\varDelta\in\mathscr{A}\}.$ Given a complex $\mathscr{A}$-measurable function $\omega$ on $X$ and an $\mathscr{A}$-measurable transformation $\phi$ of $X,$ define the $weighted$ $ composition$ $operator$ $C_{\phi,\omega}$ on $\EuScript{D}(C_{\phi,\omega})=\{f\in L^2(\mu)\colon\omega\cdot(f\circ\phi)\in L^2(\mu)\}$ as\begin{equation*}
	C_{\phi,\omega}f=\omega\cdot( f\circ\phi),\ f\in\EuScript{D}(C_{\phi,\omega}).
\end{equation*}
If we take \(\phi(x)=x\), then \(C_{\phi,\omega} = M_{\omega}\), the well-known multiplication operator. Define the measure $\mu^\omega$, $\mu_{\omega}\colon\mathscr{A}\to[0,+\infty]$ by
\begin{equation}\label{muomega1}
	\mu^\omega(\Delta)=\mu(\{\omega\ne0\}\bigcap\varDelta)
\end{equation}
and
\begin{equation}\label{muomega}
	\mu_{\omega}(\varDelta)=\int_{\varDelta}|\omega|^2\D\mu,\ \varDelta\in\mathscr{A}.
\end{equation}
One can verify that the measure $\mu^\omega$ and $\mu_{\omega}$ are $\sigma$-finite and  mutuallly absolutely continuous. It should be mentioned that $C_{\phi,\omega}$ is well-defined if and only if $\mu_{\omega}\circ\phi^{-1}\ll\mu$, i.e., $\mu_{\omega}\circ\phi^{-1}$ is absolutely continuous with respect to $\mu$. %The pair $\textbf{C}_{\phi,\omega}$ is said to be densely defined if the set $\EuScript{D}(\textbf{C}_{\phi,\omega})=\EuScript{D}(C_{\phi_{1},\omega_{1}})\bigcap\EuScript{D}(C_{\phi_{1},\omega_{1}})$ is dense in $L^{2}(\mu).$ 
On the other hand, by the Radon-Nikodym theorem, there exists a unique (up to a set of $\mu$-measure zero) $\mathscr{A}$-measurable function $\mathsf{h}_{\phi,\omega}:X\to[0,\infty]$ such that
\begin{equation}\label{rd}
	\mu_{\omega}\circ\phi^{-1}(\varDelta)=\int_{\phi^{-1}(\varDelta)}|\omega|^2\D\mu=\int_{\varDelta}\mathsf{h}_{\phi,\omega}\D\mu.
\end{equation}
%The function $$\mathsf{h}$_{\phi,\omega}$ is called the Radon-Nikodym derivative of $\mu_{\omega}\circ\phi^{-1}$ with respect to $\mu$.
%In fact, the Radon-Nikodym derivative
The function $\mathsf{h}_{\phi,\omega}$, called the Radon-Nikodym derivative, plays a major role in the study on unbounded weighted composition operator $C_{\phi,\omega}$. Combining \eqref{rd} with the change of variable formula \cite[Chapter 6]{makar-spr-2013} yields\begin{align}\label{intcom}
	\int_{X}f\circ\phi\D\mu_{\omega}=\int_{X}f\mathsf{h}_{\phi,\omega}\D\mu
\end{align}
for every $\mathscr{A}$-measurable function $f$ such that either $\int_{X}|f|\circ\phi\D\mu_{\omega}<\infty$ or $f\ge0$. Those properties of $C_{\phi,\omega}$ used in our discussion are summarized below.
\begin{thm}\label{combasic}
	Suppose $C_{\phi,\omega}$  is well-defined. Then
	\begin{enumerate}
		\item[(i)]
		$\EuScript{D}(C_{\phi,\omega})=L^2((1+\mathsf{h}_{\phi,\omega})\D\mu)$,
		\item[(iii)]
		\text{$C_{\phi,\omega}$ is closed,} 
		\item[(iii)]
		\text{$C_{\phi,\omega}$ is densely defined if and only if $\mathsf{h}_{\phi,\omega} < \infty$ a.e.\ $[\mu]$,} 
		\item[(iv)]
		\text{$C_{\phi,\omega}$ is bounded on $L^2(\mu)$ if and only if $\mathsf{h}_{\phi,\omega}$ is essentially bounded with respect to $\mu$, and} \text{in this case, $\|C_{\phi,\omega}\|^2=\|\mathsf{h}_{\phi,\omega}\|_{\infty}$, where $\| h_{\phi,\omega}\|_{\infty}$ denotes the essential supremum of $\mathsf{h}_{\phi,\omega}$,}
		\item[(vi)]
		$\mathcal{N}(C_{\phi,\omega})=\chi_{\{\mathsf{h}_{\phi,\omega}=0\}}L^2(\mu)$.
	\end{enumerate}
\end{thm}
\begin{thm}\label{adjoint}
	Suppose $C_{\phi,\omega}$ is well-defined and densely defined. Then the following equalities hold$:$
	\begin{align*}
		\dz{\cfw^*}&=\big\{f\in L^2(\mu)\colon \hfw\cdot\efw(f_w)\circ\phi^{-1}\in L^2(\mu)\big\},\\
		\cfw^*f&=\hfw\cdot\efw(f_w)\circ\phi^{-1},\quad f\in\dz{\cfw^*},
	\end{align*}
	where $f_w=\chi_{\{w\neq 0\}}\frac{f}{w}$.
\end{thm}
\textbf{Throughout the remainder of this paper, we always assume that every weighted composition operator $C_{\phi,\omega}$ in $L^2(\mu)$ is well-defined; that is, $\mu_{\omega}\circ\phi^{-1}\ll\mu$.}

A further useful concept is the conditional expectation operator associated with the measure space $(X, \phi^{-1}(\mathscr{A}), \mu_{\omega})$. Assume that $C_{\phi,\omega}$ is densely defined, then the measure $\mu_{\omega}|_{\phi^{-1}(\mathscr{A})}$ is $\sigma$-finite \cite[Proposition 10]{b-j-j-sW}. An application of the Radon--Nikodym theorem implies the existence of a unique $\phi^{-1}(\mathscr{A})$-measurable function $E_{\phi,\omega}(f)$ satisfying
 $$\int_{\phi^{-1}(\Delta)}f\D\mu_{\omega}=\int_{\phi^{-1}(\Delta)}\mathsf{E}_{\phi,\omega}(f)\D\mu_{\omega},\quad\forall\ \Delta\in\mathscr{A}$$
whenever $f\ge0$ a.e. $[\mu_{\omega}]$ or $f\in L^2(\mu_{\omega})$. The transformation
$$E_{\phi,\omega}:L^{2}(X,\mathscr{A},\mu_{\omega})\bigcup\{f:f\ge0\ \text{a.e.}\ [\mu_{\omega}]\}\to L^2(X, \phi^{-1}(\mathscr{A}), \mu_{\omega}),\quad\quad f\mapsto\mathsf{E}_{\phi,\omega}(f)$$is called  the {\em conditional expectation operator}. Now, let $f,g\ge0$ a.e. $[\mu_{\omega}]$ or $f,g \colon X \to \cbb$ be $\mathscr{A}$-measurable functions such that $f \in L^2(\mu_w)$ and $g\circ \phi\in L^2(\mu_w)$. Then the following identities hold:\begin{align}\label{exp}
	\int_X g \circ \phi \cdot f \D \mu_w = \int_X g \circ \phi \cdot \mathsf{E}_{\phi,\omega}(f) \D \mu_w,
\end{align}
\begin{align}\label{exp+}
	\efw(g \circ \phi \cdot f) = g \circ \phi \cdot \efw(f).
\end{align}
An important formal expression, denoted by $\mathsf{E}_{\phi,\omega}(f)\circ\phi^{-1}$, is defined as the unique (up to a set of $\mu$-measure zero) $\mathscr{A}$-measurable function on $X$ such that\begin{align}\label{Epro}
  \text{$\mathsf{E}_{\phi,\omega}(f)\circ \phi^{-1}=0$ a.e. $[\mu]$ on $\{\mathsf{h}_{\phi,\omega}=0\}$,}
\end{align}
and\begin{align} \label{invexp}
	(\efw(f) \circ \phi^{-1})\circ \phi = \efw(f) \quad \text{a.e.\ $[\mu_w]$}.
\end{align}
For proofs and more facts about $C_{\phi,\omega}$ and $E_{\phi,\omega}$, we refer to \cite{b-j-j-sW}. 

For the reader's convenience, we state some preliminary facts about the space 
$L^2(Y,2^Y,\nu)$, where $Y$ is countable, $2^Y$ is its power set, and $\nu$ is the counting measure. In this setting, $\nu(\Delta)=0$ if and only if $\Delta=\emptyset$. Consequently, $\nu_{\omega}\circ\phi^{-1}\ll\nu$ for any measurable transform $\phi$ and function $\omega$, which guarantees that every weighted composition operator $C_{\phi,\omega}$ in this space is well-defined. The explicit formulas for computing $\mathsf{h}_{\phi,\omega}$, $\mathsf{E}_{\phi,\omega}(f)$ and $\mathsf{E}_{\phi,\omega}(f)\circ\phi^{-1}$ are given in \cite[Proposition 79 and 80]{b-j-j-sW} by:\begin{equation}\label{h}
	\mathsf{h}_{\phi,\omega}(y)=\sum\limits_{x\in\phi^{-1}(y)}|\omega(x)|^2,
\end{equation}
\begin{equation}\label{Ef}
	\mathsf{E}_{\phi,\omega}(f)(y)=\begin{cases}
		\frac{\sum\limits_{x\in\phi^{-1}(\phi(y))}f(x)|\omega(x)|^2}{\sum\limits_{x\in\phi^{-1}(\phi(y))}|\omega(x)|^2}&\mathsf{h}_{\phi,\omega}(\phi(y))\ne0\\
		0&\mathsf{h}_{\phi,\omega}(\phi(y))=0
	\end{cases}
\end{equation}
and
\begin{equation}\label{Efphi}
(\mathsf{E}_{\phi,\omega}(f)\circ\phi^{-1})(y)=\begin{cases}
	\frac{\sum\limits_{x\in\phi^{-1}(y)}f(x)|\omega(x)|^2}{\sum\limits_{x\in\phi^{-1}(y)}|\omega(x)|^2}&\mathsf{h}_{\phi,\omega}(y)\ne0\\
	0&\mathsf{h}_{\phi,\omega}(y)=0
\end{cases}.
\end{equation}

%(in $L^2$-spaces).We list here some basic properties of $C_{\phi,\omega}$, which are needed in the following section:
 %In \cite{b-j-j-sW}, Budzy\'{n}ski et al. characterized some properties of $C_{\phi,\omega}$ in terms of $\mathsf{h}_{\phi,\omega}$ as follows:
%\begin{itemize}
%\item[(i)]
%The domain $\dz{\cfw}$ of %$C_{\phi,\omega}$ is $L^2((1+\hfw)\D \mu)$.
%\item[(ii)] 
%\text{$\cfw$ is densely defined if and %only if $\hfw < \infty$ a.e.\ $[\mu]$.}
%\item[(iii)] \text{$\cfw$ is bounded on %$L^2(\mu)$ if and only if $\hfw$ is %essentially bounded with respect to $\mu$.}
%\end{itemize}
%For proofs and more details about the operator $C_{\phi,\omega},$ we refer the reader to\cite{b-j-j-sW}. %It is seen that the Radon-Nikodym derivative $h_{\phi,\omega}$ plays a major role in the study on unbounded weighted composition operator $C_{\phi,\omega}$ (in $L^2$-spaces).

\section{Polar decomposition of $\textbf{C}_{\phi,\omega}$}
In this section, we first show that the pair $\textbf{C}_{\phi,\omega}=(C_{\phi_{1},\omega_{1}},C_{\phi_{2},\omega_{2}})$ as an operator from $L^{2}(\mu)$ into $L^{2}(\mu)\oplus L^{2}(\mu)$ is densely defined and closed. Then, based on the polar decomposition of a densely defined closed operator, we give an explicit description of the polar decomposition of the pair $\textbf{C}_{\phi,\omega}.$ We denote by $\textbf{C}_{\phi,\omega}^\ast$ the adjoint of the column operator $\textbf{C}_{\phi,\omega}\in\mathcal{L}\big(L^{2}(\mu),L^{2}(\mu)\oplus L^{2}(\mu)\big)$, provided it exists.

\begin{pro}\label{basic}
	Let $C_{\phi_{1},\omega_{1}}$ and $C_{\phi_{2},\omega_{2}}$ be densely defined operators. Then the pair of operators $\textbf{C}_{\phi,\omega}=(C_{\phi_{1},\omega_{1}},C_{\phi_{2},\omega_{2}})$ satisfies the following properties$:$
	\begin{enumerate}
		\item[(i)] $\EuScript{D}(\textbf{C}_{\phi,\omega})=L^2\Big(\big(1+\mathsf{h}_{\phi_{1},\omega_{1}}+\mathsf{h}_{\phi_{2},\omega_{2}})\D\mu\Big)$,
		\item[(ii)] 
		$\overline{\EuScript{D}(\textbf{C}_{\phi,\omega})}=L^{2}(\mu)$,
		\item[(iii)] 
		$\mathcal{G}(\textbf{C}_{\phi,\omega})$ is closed,
		\item[(iv)] $\textbf{C}_{\phi,\omega}^\ast=\overline{(C_{\phi_{1},\omega_{1}}^\ast,C_{\phi_{2},\omega_{2}}^\ast)}$.
		\item[(v)]
		$\EuScript{D}(\textbf{C}_{\phi,\omega}^\ast)=\Bigg\{\left(
		\begin{array}{c}
			f_{1} \\
			f_{2}
		\end{array}
		\right)\in L^2(\mu)\oplus L^2(\mu)\colon\sum\limits_{i=1}^{2}\mathsf{h}_{\phi_{i},\omega_{i}}\cdot\mathsf{E}_{\phi_{i},\omega_{i}}(f_{i_{\omega_{i}}})\circ\phi_{i}^{-1}\in L^{2}(\mu)\Bigg\}$.
	\end{enumerate}
\end{pro}
\begin{proof}
(i)  If $f\in\EuScript{D}(\textbf{C}_{\phi,\omega})$, then $f\in\EuScript{D}(C_{\phi_{1},\omega_{1}})\bigcap\EuScript{D}(C_{\phi_{2},\omega_{2}}).$
By Theorem \ref{combasic} (i), we have
\begin{align*}
	\int_{X}|f|^2(1+\mathsf{h}_{\phi_{i},\omega_{i}})\D\mu<\infty,\ i=1,2,
\end{align*}
which implies $\int_{X}|f|^{2}\mathsf{h}_{\phi_{i},\omega_{i}}\D\mu<\infty,\ i=1,2$. Hence $\int_{X}|f|^2(1+\mathsf{h}_{\phi_{1},\omega_{1}}+\mathsf{h}_{\phi_{2},\omega_{2}})\D\mu<\infty,$ i.e., $f\in L^2\Big(\big(1+\mathsf{h}_{\phi_{1},\omega_{1}}+\mathsf{h}_{\phi_{2},\omega_{2}})\D\mu\Big).$ The reverse inclusion follows immediately from Theorem \ref{combasic} (i) and the definition of $\EuScript{D}(\textbf{C}_{\phi,\omega})$.

(ii) Suppose $f\ge0$ in $L^{2}(\mu)$, then there exists an increasing sequence $\{s_{n}\}$ of nonnegative simple functions that converges pointwise to $f$. Let $a_{n_{1}},a_{n_{2}},\cdots,a_{n_{k}}$ be the distinct values of $s_{n}$ and set $\varDelta_{n_{k}}=\{x:s_{n}(x)=a_{n_{k}}\}$. We can write $s_{n}=\sum\limits_{k}a_{n_{k}}\chi_{\varDelta_{n_{k}}}$, where $\chi_{\varDelta_{n_{k}}}$ is the characteristic function of $\varDelta_{n_{k}}.$ Since $0\le s_{n}\le f$, we have\begin{equation*}
	\chi_{\varDelta_{n_{k}}}\in L^{2}(\mu).
\end{equation*} 
 Then, let $m$ be a positive integer and take $\varDelta_{n_{k}}^{m}=\{x\in\varDelta_{n_{k}}:1+\mathsf{h}_{\phi_{1},\omega_{1}}+\mathsf{h}_{\phi_{2},\omega_{2}}\le m\}$. It is easy to see that the characteristic function 
 \begin{equation*}
 \chi_{\varDelta_{n_{k}}^m}\in L^2\Big(\big(1+\mathsf{h}_{\phi_{1},\omega_{1}}+\mathsf{h}_{\phi_{2},\omega_{2}})\D\mu\Big)=\EuScript{D}(\textbf{C}_{\phi,\omega}).
 \end{equation*}
 By Theorem \ref{combasic} (iii), we obtain
$\lim\limits_{m\to\infty}\chi_{\varDelta_{n_{k}}^m}=\chi_{\varDelta_{n_{k}}}$ a.e.\ $[\mu]$. Hence, the Dominated Convergence Theorem shows that
\begin{equation*}
	\lim\limits_{m\to\infty}\int_{X}|\chi_{\varDelta_{n_{k}}^m}-\chi_{\varDelta_{n_{k}}}|^2\D\mu=0.
\end{equation*} Thus, by a routine approximation argument via simple functions, one can verify that $\EuScript{D}(\textbf{C}_{\phi,\omega})$ is dense in $L^{2}(\mu).$ This completes the proof.

(iii) Note that $\mathcal{G}(\textbf{C}_{\phi,\omega
})$ is closed if and only if $\EuScript{D}(\textbf{C}_{\phi,\omega})$ is a Banach space with respect to the graph norm of $\textbf{C}_{\phi,\omega}$. Let $\{f_{n}\}$ be a Cauchy sequence in $(\EuScript{D}(\textbf{C}_{\phi,\omega}),\|\cdot\|_{\textbf{C}_{\phi,\omega}})$, then $\{f_{n}\}$, $\{C_{\phi_{1},\omega_{1}}f_{n}\}$ and $\{C_{\phi_{2},\omega_{2}}f_{n}\}$ are all Cauchy sequences in $L^{2}(\mu).$ Suppose that $f$ is the limit of $\{f_{n}\}$ in $L^{2}(\mu)$. Since $C_{\phi_{i},\omega_{i}}$ ($i=1,2$) are closed, then $f\in\EuScript{D}(C_{\phi_{i},\omega_{i}})$ and $\lim\limits_{n\to\infty}C_{\phi_{i},\omega_{i}}f_{n}=C_{\phi_{i},\omega_{i}}f,$ which implies $f\in\EuScript{D}(\textbf{C}_{\phi,\omega})$ and $f_{n}\xrightarrow{\|\cdot\|_{\textbf{C}_{\phi,\omega}}}f$ as $n\to\infty$, that is, $(\EuScript{D}(\textbf{C}_{\phi,\omega}),\|\cdot\|_{\textbf{C}_{\phi,\omega}})$ is a Banach space.

(iv) Consider the column operator $\textbf{C}_{\phi,\omega}\in\mathcal{L}\big(L^{2}(\mu),L^{2}(\mu)\oplus L^{2}(\mu)\big).$ Since $C_{\phi_{i},\omega_{i}}$ ($i=1,2$) are closed, then $C_{\phi_{i},\omega_{i}}^\ast$ ($i=1,2$) are densely defined. It follows that the row operator $(C_{\phi_{1},\omega_{1}}^\ast,C_{\phi_{2},\omega_{2}}^\ast)\in\mathcal{L}\big(L^{2}(\mu)\oplus L^{2}(\mu),L^{2}(\mu)\big)$ is also densely defined. Thus, by \cite[Proposition 3.1]{M-pams-2008}, we have
\begin{equation*}
(C_{\phi_{1},\omega_{1}}^\ast,C_{\phi_{2},\omega_{2}}^\ast)^\ast=\left(
\begin{array}{c}
	C_{\phi_{1},\omega_{1}} \\
	C_{\phi_{2},\omega_{2}}
\end{array}
\right),
\end{equation*}
then
\begin{equation*}
\overline{(C_{\phi_{1},\omega_{1}}^\ast,C_{\phi_{2},\omega_{2}}^\ast)}=(C_{\phi_{1},\omega_{1}}^\ast,C_{\phi_{2},\omega_{2}}^\ast)^{\ast\ast}=\left(
\begin{array}{c}
	C_{\phi_{1},\omega_{1}} \\
	C_{\phi_{2},\omega_{2}}
\end{array}
\right)^\ast.
\end{equation*}

(v) Suppose $g\in\EuScript{D}(\textbf{C}_{\phi,\omega})$ and $\left(
\begin{array}{c}
	f_{1} \\
	f_{2}
\end{array}
\right)\in L^{2}(\mu)\oplus L^{2}(\mu)$. Set $f_{i_{\omega_{i}}}=\chi_{\{\omega_{i}\ne0\}}\cdot\frac{f_{i}}{\omega_{i}}$, ($i=1,2$). Then
\begin{equation}\label{innerp}
\begin{aligned}
&\Bigg\langle\left(
\begin{array}{c}
	C_{\phi_{1},\omega_{1}} \\
	C_{\phi_{2},\omega_{2}}
\end{array}
\right)g,\left(
\begin{array}{c}
	f_{1} \\
	f_{2}
\end{array}
\right)\Bigg\rangle=\langle C_{\phi_{1},\omega_{1}}g,f_{1}\rangle+\langle C_{\phi_{2},\omega_{2}}g,f_{2}\rangle\\&=\int_{X}\omega_{1}\cdot g\circ\phi_{1}\cdot\bar{f_{1}}\D\mu+\int_{X}\omega_{2}\cdot g\circ\phi_{2}\cdot\bar{f_{2}}\D\mu\\&=\int_{X}|\omega_{1}|^2\cdot g\circ\phi_{1}\cdot\overline{ f_{1_{\omega_{1}}}}\D\mu+\int_{X}|\omega_{2}|^2\cdot g\circ\phi_{2}\cdot\overline{ f_{2_{\omega_{2}}}}\D\mu\\&\stackrel{\eqref{muomega}}{=}\int_{X}g\circ\phi_{1}\cdot\overline{ f_{1_{\omega_{1}}}}\D\mu_{\omega_{1}}+\int_{X}g\circ\phi_{2}\cdot\overline{ f_{2_{\omega_{2}}}}\D\mu_{\omega_{2}}\\&\stackrel{\eqref{exp}}{=}\int_{X}g\circ\phi_{1}\cdot\overline{\mathsf{E}_{\phi_{1},\omega_{1}}(f_{1_{\omega_{1}}})}\D\mu_{\omega_{1}}+\int_{X}g\circ\phi_{2}\cdot\overline{\mathsf{E}_{\phi_{2},\omega_{2}}(f_{2_{\omega_{2}}})}\D\mu_{\omega_{2}}\\&\xlongequal{\eqref{rd},\eqref{invexp}}\int_{X}g\cdot h_{\phi_{1},\omega_{1}}\cdot\overline{\mathsf{E}_{\phi_{1},\omega_{1}}(f_{1_{\omega_{1}}})\circ\phi_{1}^{-1}}\D\mu+\int_{X}g\cdot h_{\phi_{2},\omega_{2}}\cdot\overline{\mathsf{E}_{\phi_{2},\omega_{2}}(f_{2_{\omega_{2}}})\circ\phi_{2}^{-1}}\D\mu\\&=\int_{X}g\cdot\sum\limits_{i=1}^{2}\mathsf{h}_{\phi_{i},\omega_{i}}\cdot\overline{\mathsf{E}_{\phi_{i},\omega_{i}}(f_{i_{\omega_{i}}})\circ\phi_{i}^{-1}}\D\mu.
\end{aligned}
\end{equation}
Let $\EuScript{D}=\Bigg\{\left(
\begin{array}{c}
	f_{1} \\
	f_{2}
\end{array}
\right)\in L^2(\mu)\oplus L^2(\mu)\colon\sum\limits_{i=1}^{2}\mathsf{h}_{\phi_{i},\omega_{i}}\cdot\mathsf{E}_{\phi_{i},\omega_{i}}(f_{i_{\omega_{i}}})\circ\phi_{i}^{-1}\in L^{2}(\mu)\Bigg\}$. By \eqref{innerp}, we have $\EuScript{D}\subset\EuScript{D}(\textbf{C}_{\phi,\omega}^\ast).$ On the other hand, by \cite[Lemma 9]{b-j-j-sW}, there exists an increasing sequence $\{X_{n}\}\subset\mathscr{A}$ (resp. $\{Y_{m}\}\subset\mathscr{A}$) such that $\mu(X_{n})<\infty$ (resp. $\mu(Y_{m})<\infty$), $h_{\phi_{1},\omega_{1}}\le n$ (resp. $h_{\phi_{2},\omega_{2}}\le m$) a.e. $[\mu]$ on $X_{n}$ (resp. $Y_{m}$) for $n\in\mathbb{N}$ (resp. $m\in\mathbb{N}$) and $\bigcup\limits_{n=1}^\infty X_{n}=X$ (resp. $\bigcup\limits_{m=1}^\infty Y_{m}=X$). Let $\varDelta\in\mathscr{A}$ and $Z_{nm}=X_{n}\bigcap Y_{m}$. By (i), we obtain $g=\chi_{\varDelta\bigcap Z_{nm}}\in\EuScript{D}(\textbf{C}_{\phi,\omega})$. If $\left(
\begin{array}{c}
	f_{1} \\
	f_{2}
\end{array}
\right)\in\EuScript{D}(\textbf{C}_{\phi,\omega}^\ast)$, then
\begin{equation*}
\Bigg\langle\textbf{C}_{\phi,\omega}g,\left(
\begin{array}{c}
	f_{1} \\
	f_{2}
\end{array}
\right)\Bigg\rangle=\Bigg\langle g,\textbf{C}_{\phi,\omega}^\ast\left(
\begin{array}{c}
	f_{1} \\
	f_{2}
\end{array}
\right)\Bigg\rangle.
\end{equation*}
It follows from the above equality and \eqref{innerp} that\begin{equation}\label{inteq}
	\int_{X}\chi_{\Delta\bigcap Z_{nm}}\cdot\sum\limits_{i=1}^{2}\mathsf{h}_{\phi_{i},\omega_{i}}\cdot\overline{\mathsf{E}_{\phi_{i},\omega_{i}}(f_{i_{\omega_{i}}})\circ\phi_{i}^{-1}}\D\mu=\int_{X}\chi_{\Delta\bigcap Z_{nm}}\cdot\overline{\textbf{C}_{\phi,\omega}^\ast\left(
	\begin{array}{c}
		f_{1} \\
		f_{2}
	\end{array}
	\right)}\D\mu<\infty.
\end{equation}
By \eqref{inteq} and \cite[Lemma 2]{b-j-j-sW}, we have $\sum\limits_{i=1}^{2}\mathsf{h}_{\phi_{i},\omega_{i}}\cdot\overline{\mathsf{E}_{\phi_{i},\omega_{i}}(f_{i_{\omega_{i}}})\circ\phi_{i}^{-1}}=\overline{\textbf{C}_{\phi,\omega}^\ast\left(
	\begin{array}{c}
		f_{1} \\
		f_{2}
	\end{array}
	\right)}\ \text{a.e. $[\mu]$}$
on $Z_{nm}$ for every $n,m\in\mathbb{N}.$ Hence,
\begin{equation*}
	\sum\limits_{i=1}^{2}\mathsf{h}_{\phi_{i},\omega_{i}}\cdot\mathsf{E}_{\phi_{i},\omega_{i}}(f_{i_{\omega_{i}}})\circ\phi_{i}^{-1}=\textbf{C}_{\phi,\omega}^\ast\left(
	\begin{array}{c}
		f_{1} \\
		f_{2}
	\end{array}
	\right)\ \text{a.e. $[\mu]$}
\end{equation*}  on $X=\bigcup\limits_{n,m}Z_{nm}$, which yields $\sum\limits_{i=1}^{2}\mathsf{h}_{\phi_{i},\omega_{i}}\cdot\mathsf{E}_{\phi_{i},\omega_{i}}(f_{i_{\omega_{i}}})\circ\phi_{i}^{-1}\in L^{2}(\mu)$, that is, $\EuScript{D}(\textbf{C}_{\phi,\omega}^\ast)\subset\EuScript{D}.$ The proof is completed.
\end{proof}
\begin{rem}
The equality in the assertion (iv) of this proposition yields that \begin{equation*}
	(C_{\phi_{1},\omega_{1}}^\ast,C_{\phi_{2},\omega_{2}}^\ast)\subseteq\overline{(C_{\phi_{1},\omega_{1}}^\ast,C_{\phi_{2},\omega_{2}}^\ast)}=\left(
	\begin{array}{c}
		C_{\phi_{1},\omega_{1}} \\
		C_{\phi_{2},\omega_{2}}
	\end{array}
	\right)^\ast.
\end{equation*}
It follows immediately from the above inclusions that
\begin{equation}\label{asteq}
(C_{\phi_{1},\omega_{1}}^\ast,C_{\phi_{2},\omega_{2}}^\ast)=\left(
\begin{array}{c}
	C_{\phi_{1},\omega_{1}} \\
	C_{\phi_{2},\omega_{2}}
\end{array}
\right)^\ast
\end{equation} if and only if the row operator $(C_{\phi_{1},\omega_{1}}^\ast,C_{\phi_{2},\omega_{2}}^\ast)$ is closed if and only if $\EuScript{D}(C_{\phi_{1},\omega_{1}}^\ast)\oplus\EuScript{D}(C_{\phi_{2},\omega_{2}}^\ast)=\EuScript{D}(\textbf{C}_{\phi,\omega}^\ast)$. In particular, for $C_{\phi_{i},\omega_{i}}\in\mathcal{B}(L^2(\mu))$, $i=1,2$, the fact $\EuScript{D}(C_{\phi_{i},\omega_{i}}^\ast)=L^{2}(\mu)$ implies that \eqref{asteq} holds. 
\end{rem}
In general, \eqref{asteq} fails even if $C_{\phi_{1},\omega_{1}}=C_{\phi_{2},\omega_{2}},$ which implies the row operator $(C_{\phi_{1},\omega_{1}}^\ast,C_{\phi_{2},\omega_{2}}^\ast)$, formed by closed operators $C_{\phi_{1},\omega_{1}}^\ast$ and $C_{\phi_{2},\omega_{2}}^\ast$, is not necessarily closed. For the reader's convenience, we give the following example.
\begin{exa}\label{ex1}
Let $X=\mathbb{N}$, $\mathscr{A}=2^X$ and $\mu$ be the counting measure on $X$. The $\mathscr{A}$-measurable transform $\phi$  and function $\omega$ are defined by $\phi(n)=\omega(n)=n$, $n\in\mathbb{N}$. Since $\mu(\varDelta)=0$ if and only if $\varDelta=\emptyset$, then $\mu_{\omega}\circ\phi^{-1}\ll\mu$. Hence, $C_{\phi,\omega}$ is well-defined. It follows from $\phi^{-1}(\mathscr{A})=\mathscr{A}$ that $E_{\phi,\omega}$ is the identity operator on $L^{2}(\mu_{\omega}).$ By \eqref{h}, we have
\begin{equation*}
\mathsf{h}_{\phi,\omega}(n)=\sum\limits_{x\in\phi^{-1}(n)}|\omega(x)|^2=|\omega(\phi^{-1}(n))|^2=|\omega(n)|^2=n^2,
\end{equation*}
which implies $C_{\phi,\omega}$ is densely defined. On the other hand, we have
\begin{equation*}
\mathsf{h}_{\phi,\omega}\cdot\mathsf{E}_{\phi,\omega}(f_{i_{\omega}})\circ\phi^{-1}=(\phi(n))^2\cdot\frac{f_{i}(\phi(n))}{\omega(\phi(n))}=\phi(n)f_{i}(\phi(n))=nf_{i}(n),\ (i=1,2),
\end{equation*}
where $f_{1}(n)=n^{-\frac{3}{2}}$ and $f_{2}(n)=-n^{-\frac{3}{2}}$ in $L^{2}(\mu).$ A simple calculation shows that \begin{equation}\label{domainnoteq1}
\int_{X}\Big|\sum\limits_{i=1}^{2}\mathsf{h}_{\phi,\omega}\cdot\mathsf{E}_{\phi,\omega}(f_{i_{\omega}})\circ\phi^{-1}\Big|^2\D\mu=0
\end{equation}
and\begin{equation}\label{domainnoteq2}
	\int_{X}\Big|\mathsf{h}_{\phi,\omega}\cdot\mathsf{E}_{\phi,\omega}(f_{i_{\omega}})\circ\phi^{-1}\Big|^2\D\mu=\sum\limits_{n=1}^\infty(nf_{i}(n))^2=\sum\limits_{n=1}^\infty n^{-1}>\infty,\ (i=1,2).
\end{equation}
Then, consider the adjoint $\textbf{C}_{\phi,\omega}^\ast$ of the column operator $\textbf{C}_{\phi,\omega}=\left(
\begin{array}{c}
	C_{\phi,\omega} \\
	C_{\phi,\omega}
\end{array}
\right)\in\mathcal{L}\big(L^{2}(\mu),L^{2}(\mu)\oplus L^{2}(\mu)\big)$. By Theorem \ref{adjoint}, \eqref{domainnoteq1}, \eqref{domainnoteq2} and Proposition \ref{basic} (v), we conclude that $\left(
\begin{array}{c}
	f_{1} \\
	f_{2}
\end{array}
\right)\in\EuScript{D}(\textbf{C}_{\phi,\omega}^\ast)$ but $f_{i}\notin\EuScript{D}(C_{\phi,\omega}^\ast)$, $(i=1,2)$, that is, $\EuScript{D}(\textbf{C}_{\phi,\omega}^\ast)\ne\EuScript{D}(C_{\phi,\omega}^\ast)\oplus\EuScript{D}(C_{\phi,\omega}^\ast).$
\end{exa} 

Since the graph of $\textbf{C}_{\phi,\omega} \in \mathcal{L}(L^{2}(\mu), L^{2}(\mu)\oplus L^{2}(\mu))$ is isometrically isomorphic to its joint graph, we conclude from parts (ii) and (iii) of Proposition \ref{basic} that $\textbf{C}_{\phi,\omega}$ is densely defined and closed, and therefore admits a polar decomposition. To present the polar decomposition of $\textbf{C}_{\phi,\omega}=(C_{\phi_{1},\omega_{1}},C_{\phi_{2},\omega_{2}})$, we state a lemma regarding the sum and product of multiplication operators (a special case of composition operators).

\begin{lem}\label{mulsp}
Suppose $\mathsf{M}_{\mathsf{h}_{1}}$ and $\mathsf{M}_{\mathsf{h}_{2}}$ are multiplication operators in $L^2(X,\mathscr{A},\mu)$, where $\mathsf{h}_{i}$, $i=1,2$ is an $\mathscr{A}$-measurable function satisfying $0\le\mathsf{h}_{i}<\infty$, a.e. $[\mu]$, then we have$:$
	\begin{enumerate}
	\item[(i)]
	$\mathsf{M}_{\mathsf{h}_{1}}+\mathsf{M}_{\mathsf{h}_{2}}=\mathsf{M}_{\mathsf{h}_{1}+\mathsf{h}_{2}}$,
	\item[(ii)] 
	$\mathsf{M}_{\mathsf{h}_{1}}\cdot\mathsf{M}_{\mathsf{h}_{2}}=\mathsf{M}_{\mathsf{h}_{1}\cdot\mathsf{h}_{2}}$ if and only if $\mathsf{h}_{1}^2+\frac{1}{\mathsf{h}_{2}^2}\cdot\chi_{\{\mathsf{h}_{2}\ne0\}}\ge c$ a.e. $[\mu]$ for some $c\in(0,\infty)$. 
\end{enumerate}
\end{lem}
\begin{proof}
(i) By the definition of the multiplication operator in $L^2(\mu)$, we immediately obtain
\begin{equation*}
\mathsf{M}_{\mathsf{h}_{1}}+\mathsf{M}_{\mathsf{h}_{2}}\supseteq\mathsf{M}_{\mathsf{h}_{1}+\mathsf{h}_{2}}.
\end{equation*}
It suffices to prove that $\EuScript{D}(\mathsf{M}_{\mathsf{h}_{1}}+\mathsf{M}_{\mathsf{h}_{2}})\subseteq\EuScript{D}(\mathsf{M}_{\mathsf{h}_{1}+\mathsf{h}_{2}})$. Let $f\in\EuScript{D}(\mathsf{M}_{\mathsf{h}_{1}}+\mathsf{M}_{\mathsf{h}_{2}})$, i.e., $\int_{X}|\mathsf{h}_{i}f|^2\D\mu<\infty$, $i=1,2$. By the Schwarz inequality, we have 
\begin{equation*}
	\int_{X}\mathsf{h}_{1}\mathsf{h}_{2}|f|^2\D\mu\le\left(\int_{X}|\mathsf{h}_{1}f|^2\D\mu\right)^{1/2}\left(\int_{X}|\mathsf{h}_{2}f|^2\D\mu\right)^{1/2}<\infty
\end{equation*}
Thus, it follows that\begin{equation*}
	\int_{X}|\mathsf{M}_{\mathsf{h}_{1}+\mathsf{h}_{2}}f|^2\D\mu=\int_{X}|\mathsf{h}_{1}f|^2+2\mathsf{h}_{1}\mathsf{h}_{2}|f|^2+|\mathsf{h}_{2}|^2\D\mu<\infty,
\end{equation*}
that is, $f\in\EuScript{D}(\mathsf{M}_{\mathsf{h}_{1}+\mathsf{h}_{2}})$. Therefore, we get the desired result.

(ii) It can be easily verified that
\begin{equation}\label{muld1}
\EuScript{D}(\mathsf{M}_{\mathsf{h}_{1}\cdot\mathsf{h}_{2}})=\Big\{f\in L^2(\mu):\int_{X}|\mathsf{M}_{\mathsf{h}_{1}\cdot\mathsf{h}_{2}}f|^2\D\mu<\infty\Big\}=L^2\Big((1+(\mathsf{h}_{1}\cdot\mathsf{h}_{2})^2)\D\mu\Big)
\end{equation}
and\begin{equation}\label{muld2}
	\EuScript{D}(\mathsf{M}_{\mathsf{h}_{1}}\cdot\mathsf{M}_{\mathsf{h}_{2}})=\EuScript{D}(\mathsf{M}_{\mathsf{h}_{1}\cdot\mathsf{h}_{2}})\bigcap\EuScript{D}(\mathsf{M}_{\mathsf{h}_{2}})=L^2\Big((1+\mathsf{h}_{2}^2+(\mathsf{h}_{1}\cdot\mathsf{h}_{2})^2)\D\mu\Big).
\end{equation}
If $\mathsf{h}_{1}^2+\frac{1}{\mathsf{h}_{2}^2}\cdot\chi_{\{\mathsf{h}_{2}\ne0\}}\ge c$ a.e. $[\mu]$ for some $c\in(0,\infty)$, then \begin{align*}
\text{$\frac{1+\mathsf{h}_{2}^2+(\mathsf{h}_{1}\cdot\mathsf{h}_{2})^2}{1+(\mathsf{h}_{1}\cdot\mathsf{h}_{2})^2}=1+\frac{\mathsf{h}_{2}^2}{1+(\mathsf{h}_{1}\cdot\mathsf{h}_{2})^2}=1+\frac{1}{\mathsf{h}_{1}^2+\frac{1}{\mathsf{h}_{2}^2}\cdot\chi_{\{\mathsf{h}_{2}\ne0\}}}\le1+\frac{1}{c}$\quad a.e. $[\mu]$},
\end{align*}
which implies \begin{equation*}
	\int_{X}|f|^2(1+\mathsf{h}_{2}^2+(\mathsf{h}_{1}\cdot\mathsf{h}_{2})^2)\D\mu\le\int_{X}(1+\frac{1}{c})|f|^2(1+(\mathsf{h}_{1}\cdot\mathsf{h}_{2})^2)\D\mu<\infty
\end{equation*}
for every $f\in L^2\Big((1+(\mathsf{h}_{1}\cdot\mathsf{h}_{2})^2)\D\mu\Big)$, i.e.,
\begin{equation}\label{domul}
	\EuScript{D}(\mathsf{M}_{\mathsf{h}_{1}\cdot\mathsf{h}_{2}})\subseteq	\EuScript{D}(\mathsf{M}_{\mathsf{h}_{1}}\cdot\mathsf{M}_{\mathsf{h}_{2}}).
\end{equation}
 On the other hand, the definition of the multiplication operator in $L^2(\mu)$ readily yields\begin{equation}\label{muleq}
	\mathsf{M}_{\mathsf{h}_{1}}\cdot\mathsf{M}_{\mathsf{h}_{2}}\subseteq\mathsf{M}_{\mathsf{h}_{1}\cdot\mathsf{h}_{2}}
\end{equation}
By \eqref{domul} and \eqref{muleq}, we have $\mathsf{M}_{\mathsf{h}_{1}}\cdot\mathsf{M}_{\mathsf{h}_{2}}=\mathsf{M}_{\mathsf{h}_{1}\cdot\mathsf{h}_{2}}$.

Conversely, suppose the part "only if" fails, then there exists a sequence $\{\varDelta_{n}\}\subset\mathscr{A}$ such that\begin{equation}\label{mu0}
	\mu(\varDelta_{n})\ne0,
\end{equation} 
where
\begin{equation}\label{deltacou}
\varDelta_{n}=\{x:\mathsf{h}_{1}^2(x)+\frac{1}{\mathsf{h}_{2}^2(x)}\cdot\chi_{\{\mathsf{h}_{2}\ne0\}}(x)<\frac{1}{n}\}
\end{equation}for $n\in\mathbb{N}$. Moreover, since $\mu$ is $\sigma$-finite, then there exists a sequence $\{X_{m}\}\subset\mathscr{A}$ such that $\mu(X_{m})<\infty$ and $\bigcup\limits_{m=1}^\infty X_{m}=X$. We claim that for every $n\in\mathbb{N}$ there exists some $k_{n}\in\mathbb{N}$ such that\begin{equation*}
	\mu(\varDelta_{n}\bigcap Y_{k_{n}})\ne0,
\end{equation*}where $Y_{k_{n}}=\bigcup\limits_{m=1}^{k_{n}} X_{m}$. Otherwise,we have $\mu(\varDelta_{n})=\mu\Bigg(\bigcup\limits_{k=1}^\infty(\varDelta_{n}\bigcap Y_{k})\Bigg)\le0$, which yields $\mu(\varDelta_{n})=0$ for $n\in\mathbb{N}$. This contradicts \eqref{mu0}. Further, setting $Z_{n}=\varDelta_{n}\bigcap Y_{k_{n}}$ and considering the function
\begin{equation*}
	f=\Bigg(\frac{1}{1+(\mathsf{h}_{1}\cdot\mathsf{h}_{2})^2}\sum\limits_{n=1}^\infty\frac{\chi_{Z_{n}}}{n^2\mu(Z_{n})}\Bigg)^\frac{1}{2},
\end{equation*} 
we obtain\begin{align*}
\int_{X}|f|^2(1+(\mathsf{h}_{1}\cdot\mathsf{h}_{2})^2)\D\mu=\int_{X}\sum\limits_{n=1}^\infty\frac{\chi_{Z_{n}}}{n^2\mu(Z_{n})}\D\mu=\sum\limits_{n=1}^\infty\frac{1}{n^2}<\infty
\end{align*}
and
\begin{align*}
	\int_{X}|f|^2(1+\mathsf{h}_{2}^2+(\mathsf{h}_{1}\cdot\mathsf{h}_{2})^2)\D\mu&=\int_{X}\frac{1+\mathsf{h}_{2}^2+(\mathsf{h}_{1}\cdot\mathsf{h}_{2})^2}{1+(\mathsf{h}_{1}\cdot\mathsf{h}_{2})^2}\sum\limits_{n=1}^\infty\frac{\chi_{Z_{n}}}{n^2\mu(Z_{n})}\D\mu\\&=\sum\limits_{n=1}^\infty\int_{X}\frac{1+\mathsf{h}_{2}^2+(\mathsf{h}_{1}\cdot\mathsf{h}_{2})^2}{1+(\mathsf{h}_{1}\cdot\mathsf{h}_{2})^2}\frac{\chi_{Z_{n}}}{n^2\mu(Z_{n})}\D\mu\\&\stackrel{\eqref{deltacou}}{>}\sum\limits_{n=1}^\infty\int_{Z_{n}}(1+n)\frac{1}{n^2\mu(Z_{n})}\D\mu\\&=\sum\limits_{n=1}^\infty\frac{1+n}{n^2}>\infty.
\end{align*}
So, combining \eqref{muld1} and \eqref{muld2} with the above integral inequalities, we have $	\EuScript{D}(\mathsf{M}_{\mathsf{h}_{1}\cdot\mathsf{h}_{2}})\not\subseteq	\EuScript{D}(\mathsf{M}_{\mathsf{h}_{1}}\cdot\mathsf{M}_{\mathsf{h}_{2}})$. 
However, the fact $\mathsf{M}_{\mathsf{h}_{1}}\cdot\mathsf{M}_{\mathsf{h}_{2}}=\mathsf{M}_{\mathsf{h}_{1}\cdot\mathsf{h}_{2}}$ implies that \eqref{domul} holds. This leads to a contradiction. The proof is completed.
\end{proof}
\begin{rem}\label{rkex1}
(i) If $\mathsf{h}_{2}$ is bounded in the above lemma, then $\mathsf{h}_{1}^2+\frac{1}{\mathsf{h}_{2}^2}\cdot\chi_{\{\mathsf{h}_{2}\ne0\}}\ge c$ a.e. $[\mu]$ obviously holds for some $c\in(0,\infty)$. In this case, $\mathsf{M}_{\mathsf{h}_{1}}\cdot\mathsf{M}_{\mathsf{h}_{2}}=\mathsf{M}_{\mathsf{h}_{1}\cdot\mathsf{h}_{2}}$ follows immediately from Lemma \ref{mulsp} (ii). We will use this in the poof of Theorem \ref{polar} (ii). 

(ii) We present a simple example to demonstrate that there indeed exist cases in which no constant $c\in(0,\infty)$ satisfies $\mathsf{h}_{1}^2+\frac{1}{\mathsf{h}_{2}^2}\cdot\chi_{\{\mathsf{h}_{2}\ne0\}}\ge c$ a.e. $[\mu]$ and $\mathsf{M}_{\mathsf{h}_{1}}\cdot\mathsf{M}_{\mathsf{h}_{2}}\ne\mathsf{M}_{\mathsf{h}_{1}\cdot\mathsf{h}_{2}}$, even if $h_{i}$ ($i=1,2$) are defined as in \eqref{rd}. Let $X=\mathbb{N}$, $\mathscr{A}=2^X$ and $\mu$ be the counting measure on $X$. Consider the composition operators $C_{\phi,\omega_{1}}$ and $C_{\phi,\omega_{2}}$, where $\phi(n)=n$, $\omega_{1}(n)=\frac{1}{n}$ and $\omega_{2}(n)=n$. By applying the approach detailed in Example \ref{ex1}, we can get
$\mathsf{h}_{\phi,\omega_{1}}(n)=\frac{1}{n^2}$
and
$\mathsf{h}_{\phi,\omega_{2}}(n)=n^2$. It follows immediately that $\mathsf{h}_{1}^2+\frac{1}{\mathsf{h}_{2}^2}\cdot\chi_{\{\mathsf{h}_{2}\ne0\}}\ge c$ a.e. $[\mu]$ fails for all $c\in(0,\infty)$. Further, let $f(n)=\frac{1}{n}\in L^2(\mu)$. We have
\begin{equation*}
\int_{X}|\mathsf{M}_{\mathsf{h}_{\phi,\omega_{1}}\cdot\mathsf{h}_{\phi,\omega_{2}}}f|^2\D\mu=\int_{X}|f|^2\D\mu<\infty	
\end{equation*}
and
\begin{equation*}
	\int_{X}|\mathsf{M}_{\mathsf{h}_{\phi,\omega_{2}}}f|^2\D\mu=\int_{X}n^2\D\mu>\infty,	
\end{equation*}
which implies \begin{equation*}
		\EuScript{D}(\mathsf{M}_{\mathsf{h}_{\phi,\omega_{1}}}\cdot\mathsf{M}_{\mathsf{h}_{\phi,\omega_{2}}})\ne\EuScript{D}(\mathsf{M}_{\mathsf{h}_{\phi,\omega_{1}}\cdot\mathsf{h}_{\phi,\omega_{2}}}).
\end{equation*}
Therefore, we get the desired result.
\end{rem}

%\begin{thm}
	%Suppose $C_{\phi_{1},\omega_{1}}$ and $C_{\phi_{2},\omega_{2}}$ are densely defined. Then  $\left(
	%\begin{array}{c}
		%C_{\phi_{1},\omega_{1}} \\
		%C_{\phi_{2},\omega_{2}}
	%\end{array}
	%\right)\in\mathcal{B}\big(L^{2}(\mu)\oplus L^{2}(\mu),L^{2}(\mu)\big)$ is a partial isometry if and only if $h_{\phi_{1},\omega_{1}}+h_{\phi_{2},\omega_{2}}=1$, a.e. $[\mu]$.
%\end{thm}
%\begin{proof}
%Recall that an operatr $U\in\mathcal{B}(\mathcal{H},\mathcal{K})$ is a partial isometry if and only if $U^\ast U$ is a projection. Suppose $\left(
%\begin{array}{c}
	%C_{\phi_{1},\omega_{1}} \\
	%C_{\phi_{2},\omega_{2}}
%\end{array}
%\right)\in\mathcal{B}\big(L^{2}(\mu)\oplus L^{2}(\mu),L^{2}(\mu)\big)$ is a partial isometry. By \eqref{asteq} and (i) of Theorem \ref{polar}, we have \begin{align*}
	%\left(
	%\begin{array}{c}
		%C_{\phi_{1},\omega_{1}} \\
		%C_{\phi_{2},\omega_{2}}
	%\end{array}
	%\right)^\ast\left(
	%\begin{array}{c}
		%C_{\phi_{1},\omega_{1}} \\
		%C_{\phi_{2},\omega_{2}}
	%\end{array}
	%\right)=\sum\limits_{i=1}^2C_{\phi_{i},\omega_{i}}^\ast C_{\phi_{i},\omega_{i}}=\sum\limits_{i=1}^2|C_{\phi_{i},\omega_{i}}|^2=\sum\limits_{i=1}^2\mathsf{M}_{\mathsf{h}_{\phi_{i},\omega_{i}}}=\mathsf{M}_{\mathsf{h}_{\phi_{1},\omega_{1}}+\mathsf{h}_{\phi_{2},\omega_{2}}}.
%\end{align*}
%Hence, $\mathsf{M}_{\mathsf{h}_{\phi_{1},\omega_{1}}+\mathsf{h}_{\phi_{2},\omega_{2}}}^2=\mathsf{M}_{\mathsf{h}_{\phi_{1},\omega_{1}}+\mathsf{h}_{\phi_{2},\omega_{2}}}$
%\end{proof}

Next, we are in a position to characterize the polar decomposition of $\textbf{C}_{\phi,\omega}=(C_{\phi_{1},\omega_{1}},C_{\phi_{2},\omega_{2}})$ in terms of the Radon-Nikodym derivatives $\mathsf{h}_{\phi_{1},\omega_{1}}$ and $\mathsf{h}_{\phi_{2},\omega_{2}}$.
\begin{thm}\label{polar}
Suppose $C_{\phi_{1},\omega_{1}}$ and $C_{\phi_{2},\omega_{2}}$ are densely defined. Let $\textbf{C}_{\phi,\omega}=\left(
\begin{array}{c}
	U_{1} \\
	U_{2}
\end{array}
\right)|\textbf{C}_{\phi,\omega}|$ be the polar decomposition of the pair $\textbf{C}_{\phi,\omega}=(C_{\phi_{1},\omega_{1}},C_{\phi_{2},\omega_{2}})$. For $i=1,2$, then the following properties hold$:$
\begin{enumerate}
	\item[(i)] 
	$|\textbf{C}_{\phi,\omega}|=\mathsf{M}_{\sqrt{\mathsf{h}_{\phi_{1},\omega_{1}}+\mathsf{h}_{\phi_{2},\omega_{2}}}}$. In particular, $|C_{\phi_{i},\omega_{i}}|=\Bigg|\left(
	\begin{array}{c}
		C_{\phi_{i},\omega_{i}} \\
		0
	\end{array}
	\right)\Bigg|=\mathsf{M}_{\sqrt{\mathsf{h}_{\phi_{i},\omega_{i}}}}$,
	\item[(ii)] $U_{i}=C_{\phi_{i},\widetilde{\omega_{i}}}$, where $\widetilde{\omega_{i}}$ is a complex $\mathscr{A}$-measurable function defined by
	\begin{equation*}
		\widetilde{\omega_{i}}=\frac{\omega_{i}\cdot\chi_{\{\omega_{i}\ne 0\}}}{\sqrt{(\mathsf{h}_{\phi_{1},\omega_{1}}+\mathsf{h}_{\phi_{2},\omega_{2}})\circ\phi_{i}}},\ \text{a.e. $[\mu]$}, 
	\end{equation*}
and $\|U_{i}\|=\Big\|\frac{\mathsf{h}_{\phi_{i},\omega_{i}}}{\mathsf{h}_{\phi_{1},\omega_{1}}+\mathsf{h}_{\phi_{2},\omega_{2}}}\cdot\chi_{\{\mathsf{h}_{\phi_{1},\omega_{1}}+\mathsf{h}_{\phi_{2},\omega_{2}}>0\}}\Big\|_{\infty}^{\frac{1}{2}}$. 
\end{enumerate}
\end{thm}
\begin{proof}
(i) Since $\sqrt{\mathsf{h}_{\phi_{1},\omega_{1}}+\mathsf{h}_{\phi_{2},\omega_{2}}}$ is a $\mu$-a.e. nonnegative real-valued function, then the multiplication operator $\mathsf{M}_{\sqrt{\mathsf{h}_{\phi_{1},\omega_{1}}+\mathsf{h}_{\phi_{2},\omega_{2}}}}$ is positive and selfadjoint. By Proposition \ref{basic} (i), we see that \begin{equation*}
\EuScript{D}(|\textbf{C}_{\phi,\omega}|)=\EuScript{D}(\textbf{C}_{\phi,\omega})=\EuScript{D}(\mathsf{M}_{\sqrt{\mathsf{h}_{\phi_{1},\omega_{1}}+\mathsf{h}_{\phi_{2},\omega_{2}}}}).
\end{equation*}
Moreover, applying \eqref{intcom}, we have 
\begin{equation}\label{mulcomeq}
\begin{aligned}
	\||\textbf{C}_{\phi,\omega}|f\|^2&=\|\textbf{C}_{\phi,\omega}f\|^2=\|C_{\phi_{1},\omega_{1}}f\|^2+\|C_{\phi_{2},\omega_{2}}f\|^2\\&=\int_{X}|\omega_{1}|^2\cdot |f|^2\circ\phi_{1}\D\mu+\int_{X}|\omega_{2}|^2\cdot |f|^2\circ\phi_{2}\D\mu\\&=\int_{X}|f|^2\mathsf{h}_{\phi_{1},\omega_{1}}\D\mu+\int_{X}|f|^2\mathsf{h}_{\phi_{2},\omega_{2}}\D\mu\\&=\int_{X}|f|^2(\mathsf{h}_{\phi_{1},\omega_{1}}+\mathsf{h}_{\phi_{2},\omega_{2}})\D\mu\\&=\|\mathsf{M}_{\sqrt{\mathsf{h}_{\phi_{1},\omega_{1}}+\mathsf{h}_{\phi_{2},\omega_{2}}}}f\|^2
\end{aligned}
\end{equation}
for $f\in\EuScript{D}(|\textbf{C}_{\phi,\omega}|)$. Thus, $|\textbf{C}_{\phi,\omega}|=\mathsf{M}_{\sqrt{\mathsf{h}_{\phi_{1},\omega_{1}}+\mathsf{h}_{\phi_{2},\omega_{2}}}}$ follows from \cite[Lemma 4]{b-j-j-sW}. Further, by the closedness of $C_{\phi_{i},\omega_{i}}^\ast$ (resp. $C_{\phi_{i},\omega_{i}}$), we immediately deduce that the row operator
$(C_{\phi_{i},\omega_{i}}^\ast,C_{x,0}^\ast)=(C_{\phi_{i},\omega_{i}}^\ast,0)$ (resp. $\left(
\begin{array}{c}
	C_{\phi_{i},\omega_{i}} \\
	C_{x,0}
\end{array}
\right)=\left(
\begin{array}{c}
	C_{\phi_{i},\omega_{i}} \\
	0
\end{array}
\right)$) is closed for $i=1,2$, where $C_{x,0}$ is a weighted composition operator $C_{\phi,\omega}$ with $\phi(x)=x$ and $\omega(x)=0$ on $X$.
 Using \eqref{asteq} and the previous result, it follows that\begin{align*}
	|C_{\phi_{i},\omega_{i}}|=(C_{\phi_{i},\omega_{i}}^\ast C_{\phi_{i},\omega_{i}})^{\frac{1}{2}}=\Bigg(\left(
	\begin{array}{c}
		C_{\phi_{i},\omega_{i}} \\
		0
	\end{array}
	\right)^\ast\left(
	\begin{array}{c}
		C_{\phi_{i},\omega_{i}} \\
		0
	\end{array}
	\right)\Bigg)^{\frac{1}{2}}=\Bigg|\left(
	\begin{array}{c}
		C_{\phi_{i},\omega_{i}} \\
		0
	\end{array}
	\right)\Bigg|=\mathsf{M}_{\sqrt{\mathsf{h}_{\phi_{i},\omega_{i}}}}.
\end{align*}

(ii) Given $i\in\{1,2\}$. We first show $\widetilde{\omega_{i}}:X\to\mathbb{C}$ is well-defined by proving
\begin{equation}\label{mueq0}
	\mu\big(\{\omega_{i}\ne0\}\bigcap\{(\mathsf{h}_{\phi_{1},\omega_{1}}+\mathsf{h}_{\phi_{2},\omega_{2}})\circ\phi_{i}=0\}\big)=0.
\end{equation} 
Since
\begin{equation*}
	\{\omega_{i}\ne0\}\bigcap\{(\mathsf{h}_{\phi_{1},\omega_{1}}+\mathsf{h}_{\phi_{2},\omega_{2}})\circ\phi_{i}=0\}=\bigcup\limits_{n=1}^{\infty}\Big(\Big\{|\omega_{i}|\ge\frac{1}{n}\Big\}\bigcap\{\mathsf{h}_{\phi_{1},\omega_{1}}\circ\phi_{i}=0\}\bigcap\{\mathsf{h}_{\phi_{2},\omega_{2}}\circ\phi_{i}=0\}\Big),
\end{equation*}
then it suffices to prove that\begin{equation*}
\mu\Big(\Big\{|\omega_{i}|\ge\frac{1}{n}\Big\}\bigcap\{\mathsf{h}_{\phi_{1},\omega_{1}}\circ\phi_{i}=0\}\bigcap\{\mathsf{h}_{\phi_{2},\omega_{2}}\circ\phi_{i}=0\}\Big)=0
\end{equation*}
for every $n\in\mathbb{N}$.
Suppose the above equality fails, then there exists $n_{0}\in\mathbb{N}$ such that $\mu\Big(\Big\{|\omega_{i}|\ge\frac{1}{n_{0}}\Big\}\bigcap\{\mathsf{h}_{\phi_{i},\omega_{i}}\circ\phi_{i}=0\}\Big)\ne0$. Thus, we have\begin{align*}
	\mu_{\omega_{i}}(\mathsf{h}_{\phi_{i},\omega_{i}}\circ\phi_{i}=0)&=\int_{X}\chi_{\{\mathsf{h}_{\phi_{i},\omega_{i}}\circ\phi_{i}=0\}}\cdot|\omega_{i}|^2\D\mu\\&\ge\frac{1}{n_{0}^2}\int_{\big\{|\omega_{i}|\ge\frac{1}{n_{0}}\big\}}\chi_{\{\mathsf{h}_{\phi_{i},\omega_{i}}\circ\phi_{i}=0\}}\D\mu\\&=\frac{1}{n_{0}^2}\mu\Big(\Big\{|\omega_{i}|\ge\frac{1}{n_{0}}\Big\}\bigcap\{\mathsf{h}_{\phi_{i},\omega_{i}}\circ\phi_{i}=0\}\Big)\\&\ne0.
\end{align*}
This contradicts Lemma 6 in \cite{b-j-j-sW}. Consequently, $\widetilde{\omega_{i}}$ is well-defined. On the other hand, for $\varDelta\in\mathscr{A}$, it follows from the change of variable formula \cite[Chapter 6]{makar-spr-2013} and \eqref{rd} that
\begin{align*}
\mu_{\widetilde{\omega_{i}}}\circ\phi_{i}^{-1}(\varDelta)&=\int_{\phi_{i}^{-1}(\varDelta)}|\widetilde{\omega_{i}}|^2\D\mu=\int_{\phi_{i}^{-1}(\varDelta)}\frac{|\chi_{\{\omega_{i}\ne0\}}\cdot\omega_{i}|^2}{(\mathsf{h}_{\phi_{1},\omega_{1}}+\mathsf{h}_{\phi_{2},\omega_{2}})\circ\phi_{i}}\D\mu\\&=\int_{\phi_{i}^{-1}(\varDelta\bigcap\{\sum\limits_{i=1}^2\mathsf{h}_{\phi_{i},\omega_{i}\ne0}\})}\frac{|\chi_{\{\omega_{i}\ne0\}}\cdot\omega_{i}|^2}{(\mathsf{h}_{\phi_{1},\omega_{1}}+\mathsf{h}_{\phi_{2},\omega_{2}})\circ\phi_{i}}\D\mu\\&=\int_{\varDelta\bigcap\{\sum\limits_{i=1}^2\mathsf{h}_{\phi_{i},\omega_{i}\ne0}\}}\frac{1}{\mathsf{h}_{\phi_{1},\omega_{1}}+\mathsf{h}_{\phi_{2},\omega_{2}}}\D (\mu_{\omega_{i}}\circ\phi_{i}^{-1})\\&=\int_{\varDelta\bigcap\{\sum\limits_{i=1}^2\mathsf{h}_{\phi_{i},\omega_{i}\ne0}\}}\frac{\mathsf{h}_{\phi_{i},\omega_{i}}}{\mathsf{h}_{\phi_{1},\omega_{1}}+\mathsf{h}_{\phi_{2},\omega_{2}}}\D\mu\\&=\int_{\varDelta}\frac{\mathsf{h}_{\phi_{i},\omega_{i}}}{\mathsf{h}_{\phi_{1},\omega_{1}}+\mathsf{h}_{\phi_{2},\omega_{2}}}\cdot\chi_{\{\mathsf{h}_{\phi_{1},\omega_{1}}+\mathsf{h}_{\phi_{2},\omega_{2}}>0\}}\D\mu,
\end{align*}
which yields $\mu_{\widetilde{\omega_{i}}}\circ\phi_{i}^{-1}\ll\mu$ and
\begin{align}\label{polarrn}
	\text{$\mathsf{h}_{\phi_{i},\widetilde{\omega_{i}}}=\frac{\mathsf{h}_{\phi_{i},\omega_{i}}}{\mathsf{h}_{\phi_{1},\omega_{1}}+\mathsf{h}_{\phi_{2},\omega_{2}}}\cdot\chi_{\{\mathsf{h}_{\phi_{1},\omega_{1}}+\mathsf{h}_{\phi_{2},\omega_{2}}>0\}}$\quad a.e. $[\mu]$.}
\end{align} Hence, $C_{\phi_{i},\widetilde{\omega_{i}}}$ is well-defined. Moreover, by the above equality and $\|\mathsf{h}_{\phi_{i},\widetilde{\omega_{i}}}\|_{\infty}\le1$, we see that $C_{\phi_{i},\widetilde{\omega_{i}}}$ is bounded. Using (i) of this theorem, a direct computation establishes that $\textbf{C}_{\phi,\omega}=\left(
\begin{array}{c}
	C_{\phi_{1},\widetilde{\omega_{1}}} \\
	C_{\phi_{2},\widetilde{\omega_{2}}}
\end{array}
\right)|\textbf{C}_{\phi,\omega}|$. Finally, if we can prove that $\textbf{C}_{\phi,\widetilde{\omega}}\triangleq\left(
\begin{array}{c}
	C_{\phi_{1},\widetilde{\omega_{1}}} \\
	C_{\phi_{2},\widetilde{\omega_{2}}}
\end{array}
\right)$ is a partial isometry %(i.e., $\textbf{C}_{\phi,\widetilde{\omega}}^\ast\textbf{C}_{\phi,\widetilde{\omega}}$ is a projection)
with
\begin{equation*}
	\mathcal{N}\Bigg(\left(
	\begin{array}{c}
		C_{\phi_{1},\widetilde{\omega_{1}}} \\
		C_{\phi_{2},\widetilde{\omega_{2}}}
	\end{array}
	\right)\Bigg)=\mathcal{N}\Bigg(\left(
	\begin{array}{c}
		C_{\phi_{1},\omega_{1}} \\
		C_{\phi_{2},\omega_{2}}
	\end{array}
	\right)\Bigg),
\end{equation*}then by the uniqueness of the polar decomposition, the proof is complete. Now, recall that an operator $U\in\mathcal{B}(\mathcal{H},\mathcal{K})$ is a partial isometry if and only if $U^\ast U$ is a projection. By \eqref{asteq}, (i) of this theorem and Lemma \ref{mulsp}, we have\begin{align*}
\textbf{C}_{\phi,\widetilde{\omega}}^\ast\textbf{C}_{\phi,\widetilde{\omega}}=\sum\limits_{i=1}^2C_{\phi_{i},\widetilde{\omega_{i}}}^\ast C_{\phi_{i},\widetilde{\omega_{i}}}=\sum\limits_{i=1}^2|C_{\phi_{i},\widetilde{\omega_{i}}}|^2=\sum\limits_{i=1}^2\mathsf{M}_{\mathsf{h}_{\phi_{i},\widetilde{\omega_{i}}}}=\mathsf{M}_{\mathsf{h}_{\phi_{1},\widetilde{\omega_{1}}}+\mathsf{h}_{\phi_{2},\widetilde{\omega_{2}}}}.
\end{align*}
It follows from \eqref{polarrn} that\begin{align*}
\text{$(\mathsf{h}_{\phi_{1},\widetilde{\omega_{1}}}+\mathsf{h}_{\phi_{2},\widetilde{\omega_{2}}})^2=\chi_{\{\mathsf{h}_{\phi_{1},\omega_{1}}+\mathsf{h}_{\phi_{2},\omega_{2}}>0\}}=\mathsf{h}_{\phi_{1},\widetilde{\omega_{1}}}+\mathsf{h}_{\phi_{2},\widetilde{\omega_{2}}}$\quad a.e. $[\mu]$,}
\end{align*}
which implies, by using Lemma \ref{mulsp} (ii), that $\mathsf{M}_{\mathsf{h}_{\phi_{1},\widetilde{\omega_{1}}}+\mathsf{h}_{\phi_{2},\widetilde{\omega_{2}}}}^2=\mathsf{M}_{\mathsf{h}_{\phi_{1},\widetilde{\omega_{1}}}+\mathsf{h}_{\phi_{2},\widetilde{\omega_{2}}}}$, i.e., $\textbf{C}_{\phi,\widetilde{\omega}}^\ast\textbf{C}_{\phi,\widetilde{\omega}}$ is a projection. On the other hand, since
\begin{equation*}
	\mathcal{N}\Bigg(\left(
	\begin{array}{c}
		C_{\phi_{1},\widetilde{\omega_{1}}} \\
		C_{\phi_{2},\widetilde{\omega_{2}}}
	\end{array}
	\right)\Bigg)=\Bigg\{f:\bigcap\limits_{i=1}^2C_{\phi_{i},\widetilde{\omega_{i}}}f=0\Bigg\}=\Bigg\{f:\bigcap\limits_{i=1}^2\frac{\omega_{i}\cdot\chi_{\{\omega_{i}\ne 0\}}}{\sqrt{(\mathsf{h}_{\phi_{1},\omega_{1}}+\mathsf{h}_{\phi_{2},\omega_{2}})\circ\phi_{i}}}\cdot f\circ\phi_{i}=0\Bigg\}
\end{equation*}
and
\begin{align*}
	\mathcal{N}\Bigg(\left(
	\begin{array}{c}
		C_{\phi_{1},\omega_{1}} \\
		C_{\phi_{2},\omega_{2}}
	\end{array}
	\right)\Bigg)&=\Bigg\{f:\bigcap\limits_{i=1}^2C_{\phi_{i},\omega_{i}}f=0\Bigg\}=\Bigg\{f:\bigcap\limits_{i=1}^2\omega_{i}\cdot f\circ\phi_{i}=0\Bigg\}\\&=\Bigg\{f:\bigcap\limits_{i=1}^2\omega_{i}\cdot\chi_{\{\omega_{i}\ne0\}}\cdot f\circ\phi_{i}=0\Bigg\},
\end{align*}
then we can get the desired result from \eqref{mueq0}.
\end{proof}
\begin{rem}
 Suppose	$\textbf{C}_{\phi,\omega}=\left(
\begin{array}{c}
	U_{1} \\
	U_{2}
\end{array}
\right)|\textbf{C}_{\phi,\omega}|$ is the polar decomposition of the pair $\textbf{C}_{\phi,\omega}=(C_{\phi_{1},\omega_{1}},C_{\phi_{2},\omega_{2}})$. For $f\in\EuScript{D}(\textbf{C}_{\phi,\omega})$, we have
\begin{equation*}
	\begin{aligned}
		\textbf{C}_{\phi,\omega}f=\left(
		\begin{array}{c}
			U_{1} \\
			U_{2}
		\end{array}
		\right)|\textbf{C}_{\phi,\omega}|f&\Rightarrow\left(
		\begin{array}{c}
			C_{\phi_{1},\omega_{1}}f \\
			C_{\phi_{2},\omega_{2}}f
		\end{array}
		\right)=\left(
		\begin{array}{c}
			U_{1}|\textbf{C}_{\phi,\omega}|f \\
			U_{2}|\textbf{C}_{\phi,\omega}|f
		\end{array}
		\right)\\&\Rightarrow C_{\phi_{i},\omega_{i}}f=U_{i}|\textbf{C}_{\phi,\omega}|f\quad\text{for}\ \  i=1,2.
	\end{aligned}
\end{equation*}
Hence, \begin{equation}\label{nepolar}
	C_{\phi_{i},\omega_{i}}\supseteq U_{i}|\textbf{C}_{\phi,\omega}|
\end{equation}
since $$\EuScript{D}(U_{i}|\textbf{C}_{\phi,\omega}|)=\EuScript{D}(|\textbf{C}_{\phi,\omega}|)=\EuScript{D}(\textbf{C}_{\phi,\omega})=\bigcap\limits_{i=1}^2\EuScript{D}(C_{\phi_{i},\omega_{i}})\subseteq\EuScript{D}(C_{\phi_{j},\omega_{j}})$$ for $i,j\in\{1,2\}$.

(i) In general, the equality in \eqref{nepolar} does not hold for $i\in\{1,2\}$. Indeed, let $C_{\phi,\omega}$ be a unbounded weighted composition operator with domain $\EuScript{D}(C_{\phi,\omega})\ne L^2(\mu)$. Let $C_{\phi,\omega}=U|C_{\phi,\omega}|$ be the polar decomposition of $C_{\phi,\omega}$, then we immediately have the polar decomposition
\begin{equation*}
	\textbf{C}_{\phi,\omega}=\left(
	\begin{array}{c}
		C_{\phi,\omega} \\
		C_{x,0}
	\end{array}
	\right)=\left(
	\begin{array}{c}
		U \\
		0
	\end{array}
	\right)|C_{\phi,\omega}|.
\end{equation*}
It follows clearly that $C_{x,0}\ne0\cdot|C_{\phi,\omega}|$ since $\EuScript{D}(|C_{\phi,\omega}|)=\EuScript{D}(C_{\phi,\omega})\ne L^2(\mu)=\EuScript{D}(C_{x,0})$.

(ii) In general, $C_{\phi_{i},\omega_{i}}=U_{i}|\textbf{C}_{\phi,\omega}|$ is not the polar decomposition of $C_{\phi_{i},\omega_{i}}$ even if the equality in \eqref{nepolar} holds for $i\in\{1,2\}$. Assume $C_{\phi,\omega}=C_{x,1}=1$, where $C_{x,1}$ is a weighted composition operator $C_{\phi,\omega}$ with $\phi(x)=x$ and $\omega(x)=1$. Consider the polar decomposition
\begin{equation*}
	\textbf{C}_{\phi,\omega}=\left(
	\begin{array}{c}
		C_{x,1} \\
		C_{x,1}
	\end{array}
	\right)=\left(
	\begin{array}{c}
		\frac{\sqrt{2}}{2} \\
		\frac{\sqrt{2}}{2}
	\end{array}
	\right)\sqrt{2}=\left(
	\begin{array}{c}
		U \\
		U
	\end{array}
	\right)|\textbf{C}_{\phi,\omega}|,
\end{equation*}
 then it follows clearly that $U=\frac{\sqrt{2}}{2}$ is not a partial isometry. Thus, $C_{\phi,\omega}=U|\textbf{C}_{\phi,\omega}|$ is not the polar decomposition of $C_{\phi,\omega}$.
\end{rem}
\section{$\lambda$-spherical mean transforms of $\textbf{C}_{\phi,\omega}$}
Let $C_{\phi,\omega}=U|C_{\phi,\omega}|$ be the polar decomposition of a densely defined composition operator $C_{\phi,\omega}$ and $\lambda\in[0,1]$.  we define the {\em $\lambda$-mean transform} of $C_{\phi,\omega}$, denoted $\mathcal{M}_{\lambda}(C_{\phi,\omega})$, as $$\mathcal{M}_{\lambda}(C_{\phi,\omega})=\lambda U|C_{\phi,\omega}|+(1-\lambda)|C_{\phi,\omega}|U.$$
This coincides with the classical $\lambda$-mean transform when $C_{\phi,\omega}$ is bounded \cite{zamani-jmaa-2021}. It should be mentioned that $\mathcal{M}_{0}(C_{\phi,\omega})\ne|C_{\phi,\omega}|U$ and $\mathcal{M}_{1}(C_{\phi,\omega})\ne C_{\phi,\omega}$ since $$\EuScript{D}(\mathcal{M}_{0}(C_{\phi,\omega}))=\EuScript{D}(U|C_{\phi,\omega}|)\bigcap\EuScript{D}(|C_{\phi,\omega}|U)\ne\EuScript{D}(|C_{\phi,\omega}|U)$$
and
$$\EuScript{D}(\mathcal{M}_{1}(C_{\phi,\omega}))=\EuScript{D}(U|C_{\phi,\omega}|)\bigcap\EuScript{D}(|C_{\phi,\omega}|U)\ne\EuScript{D}(U|C_{\phi,\omega}|)$$
in general (see Example \ref{nemean}). Analogous to this definition, we introduce the concept of the $\lambda$-spherical mean transform of the pair $\textbf{C}_{\phi,\omega}$ as follows.

\begin{dfn}\label{smtdfn}
Suppose $C_{\phi_{1},\omega_{1}}$ and $C_{\phi_{2},\omega_{2}}$ are densely defined. Let $\textbf{C}_{\phi,\omega}=\left(
\begin{array}{c}
	U_{1} \\
	U_{2}
\end{array}
\right)|\textbf{C}_{\phi,\omega}|$ be the polar decomposition of the pair $\textbf{C}_{\phi,\omega}=(C_{\phi_{1},\omega_{1}},C_{\phi_{2},\omega_{2}})$. For $\lambda\in[0,1],$ the pair of operators \begin{equation*}
	\mathcal{M}_{\lambda}(\textbf{C}_{\phi,\omega})=(\lambda U_{1}|\textbf{C}_{\phi,\omega}|+(1-\lambda)|\textbf{C}_{\phi,\omega}|U_{1},\lambda U_{2}|\textbf{C}_{\phi,\omega}|+(1-\lambda)|\textbf{C}_{\phi,\omega}|U_{2})
\end{equation*}
is called the {\em $\lambda$-spherical mean transform} of $\textbf{C}_{\phi,\omega}$.  
\end{dfn}
\begin{rem}\label{singlemult}
Let $C_{\phi,\omega}=U|C_{\phi,\omega}|$ be the polar decomposition of $C_{\phi,\omega}$. Since $$\left(
\begin{array}{c}
	U \\
	0
\end{array}
\right)^\ast\left(
\begin{array}{c}
	U \\
	0
\end{array}
\right)=\left(
\begin{array}{cc}
	U^\ast&0
\end{array}
\right)\left(
\begin{array}{c}
	U \\
	0
\end{array}
\right)=U^\ast U$$ is a projection with $\mathcal{N}\Bigg(\left(
\begin{array}{c}
	U \\
	0
\end{array}
\right)\Bigg)=\mathcal{N}\Bigg(\left(
\begin{array}{c}
	C_{\phi,\omega} \\
	0
\end{array}
\right)\Bigg)$, then $$\left(
\begin{array}{c}
	C_{\phi,\omega} \\
	0
\end{array}
\right)=\left(
\begin{array}{c}
	C_{\phi,\omega} \\
	C_{x,0}
\end{array}
\right)=\left(
\begin{array}{c}
	U \\
	0
\end{array}
\right)|C_{\phi,\omega}|$$ is the polar decomposition of $\left(
\begin{array}{c}
	C_{\phi,\omega} \\
	C_{x,0}
\end{array}
\right)$, where the weighted composition operator $C_{x,0}$ defined as in the proof of Theorem \ref{polar} (i). Hence, $\mathcal{M}_{\lambda}(C_{\phi,\omega},0)$ is well-defined. Further, it is not difficult to verify that $\mathcal{M}_{\lambda}((C_{\phi,\omega},0))$ is closed (resp. densely defined) if and only if $\mathcal{M}_{\lambda}(C_{\phi,\omega})$ is closed (resp. densely defined) since
\begin{equation}\label{singlemulequiv}
	\mathcal{M}_{\lambda}((C_{\phi,\omega},0))=(\mathcal{M}_{\lambda}(C_{\phi,\omega}),0)
\end{equation}
for $\lambda\in[0,1]$.
\end{rem}

To give some basic properties of the $\lambda$-spherical mean transform of $\textbf{C}_{\phi,\omega}=(C_{\phi_{1},\omega_{1}},C_{\phi_{2},\omega_{2}})$, we define the $\mathscr{A}$-measurable functions $\mathsf{E}_{i}^\alpha:X\to\mathbb{R}$ for $\alpha\ge0$ and $i=1,2$ by
\begin{equation}\label{Eequation}
	\mathsf{E}_{i}^\alpha(x)=\Bigg(\Big(\mathsf{E}_{\phi_{i},\omega_{i}}\big((\mathsf{h}_{\phi_{1},\omega_{1}}+\mathsf{h}_{\phi_{2},\omega_{2}})^\alpha\big)\circ\phi_{i}^{-1}\Big)\cdot\frac{\mathsf{h}_{\phi_{i},\omega_{i}}\cdot\chi_{\{\mathsf{h}_{\phi_{1},\omega_{1}}+\mathsf{h}_{\phi_{2},\omega_{2}}>0\}}}{(\mathsf{h}_{\phi_{1},\omega_{1}}+\mathsf{h}_{\phi_{2},\omega_{2}})^\alpha}\Bigg)(x).
\end{equation}
 In particular, when $\alpha=1$, we write $\mathsf{E}_{i}$ for $\mathsf{E}_{i}^1$. Then the following lemma is established.

\begin{lem}\label{weighted}
Suppose $C_{\phi_{1},\omega_{1}}$ and $C_{\phi_{2},\omega_{2}}$ are densely defined. Fix $\lambda\in[0,1]$ and set $i\in\{1,2\}$. Let $\omega_{\lambda}^i$ be a complex $\mathscr{A}$-measurable function such that
\begin{align}\label{wfun}
	\text{$\omega_{\lambda}^i=\left(\mathsf{h}_{i}+\lambda(1-\mathsf{h}_{i})\right)\omega_{i}$\quad a.e. $[\mu]$},
\end{align} where $\mathsf{h}_{i}=\chi_{\{\omega_{i}\ne0\}}\cdot\left(\frac{\mathsf{h}_{\phi_{1},\omega_{1}}+\mathsf{h}_{\phi_{2},\omega_{2}}}{(\mathsf{h}_{\phi_{1},\omega_{1}}+\mathsf{h}_{\phi_{2},\omega_{2}})\circ\phi_{i}}\right)^\frac{1}{2}$.
Then the following properties hold$:$ 
\begin{itemize}
	\item[(i)] $\mu_{\omega_{\lambda}^i}\circ\phi_{i}^{-1}\ll\mu$,
	\item[(ii)] 
	$\mathsf{h}_{\phi_{i},\omega_{\lambda}^i}=\lambda^2\mathsf{h}_{\phi_{i},\omega_{i}}+2\lambda(1-\lambda)\mathsf{E}_{i}^{\frac{1}{2}}+(1-\lambda)^2\mathsf{E}_{i}$\quad a.e.\ $[\mu]$,
\item[(iii)]
for $\alpha\ge0$, we have $\mathsf{E}_{i}^\alpha\ge0$\quad a.e.\ $[\mu]$,
\item[(iv)]
$\Big\{\sum\limits_{i=1}^2\mathsf{E}_{i}<\infty\Big\}=\Big\{\sum\limits_{i=1}^2\mathsf{E}_{\phi_{i},\omega_{i}}(\mathsf{h}_{\phi_{1},\omega_{1}}+\mathsf{h}_{\phi_{2},\omega_{2}})\circ\phi_{i}^{-1}<\infty\Big\}$\quad a.e.\ $[\mu]$.
\end{itemize}
\end{lem}
\begin{proof} 
(i) Given $i\in\{1,2\}$. Suppose $\Delta\in\mathcal{A}$ satisfying $\mu(\Delta)=0$. Then, we have
\begin{equation*}
	\mu_{\omega_{\lambda}^i}\circ\phi_{i}^{-1}(\Delta)=\int_{\phi_{i}^{-1}(\Delta)}|\omega_{\lambda}^i|^2\D\mu=\int_{\phi_{i}^{-1}(\Delta)}(\mathsf{h}_{i}+\lambda(1-\mathsf{h}_{i}))|\omega_{i}|^2\D\mu=0,
\end{equation*}
where the last equality holds since $\mu_{\omega_{i}}\circ\phi_{i}^{-1}\ll\mu$. Hence, we obtain $\mu_{\omega_{\lambda}^i}\circ\phi_{i}^{-1}\ll\mu$.

(ii) Fix $\alpha\ge0$ and let $\mathsf{h}_{i}=\chi_{\{\omega_{i}\ne0\}}\cdot\left(\frac{\mathsf{h}_{\phi_{1},\omega_{1}}+\mathsf{h}_{\phi_{2},\omega_{2}}}{(\mathsf{h}_{\phi_{1},\omega_{1}}+\mathsf{h}_{\phi_{2},\omega_{2}})\circ\phi_{i}}\right)^\frac{1}{2}$. For any $\varDelta\in\mathscr{A}$, we have
\begin{align*}
	\int_{\phi_{i}^{-1}(\varDelta)}\mathsf{h}_{i}^\alpha\D\mu_{\omega_{i}}&\stackrel{\eqref{mueq0}}{=}\int_{\phi_{i}^{-1}(\varDelta\bigcap\{\sum\limits_{i=1}^2\mathsf{h}_{\phi_{i},\omega_{i}}>0\})}\chi_{\{\omega_{i}\ne0\}}\cdot\left(\frac{\mathsf{h}_{\phi_{1},\omega_{1}}+\mathsf{h}_{\phi_{2},\omega_{2}}}{(\mathsf{h}_{\phi_{1},\omega_{1}}+\mathsf{h}_{\phi_{2},\omega_{2}})\circ\phi_{i}}\right)^\frac{\alpha}{2}\D\mu_{\omega_{i}}\\&\xlongequal{\eqref{exp},\eqref{invexp}}\int_{\phi_{i}^{-1}(\varDelta\bigcap\{\sum\limits_{i=1}^2\mathsf{h}_{\phi_{i},\omega_{i}}>0\})}\frac{(\mathsf{E}_{\phi_{i},\omega_{i}}((\mathsf{h}_{\phi_{1},\omega_{1}}+\mathsf{h}_{\phi_{2},\omega_{2}})^\frac{\alpha}{2})\circ\phi_{i}^{-1})\circ\phi_{i}}{(\mathsf{h}_{\phi_{1},\omega_{1}}+\mathsf{h}_{\phi_{2},\omega_{2}})^\frac{\alpha}{2}\circ\phi_{i}}\D\mu_{\omega_{i}}\\&\stackrel{\eqref{intcom}}{=}\int_{\varDelta}\big(\mathsf{E}_{\phi_{i},\omega_{i}}((\mathsf{h}_{\phi_{1},\omega_{1}}+\mathsf{h}_{\phi_{2},\omega_{2}})^\frac{\alpha}{2})\circ\phi_{i}^{-1}\big)\cdot\frac{\mathsf{h}_{\phi_{i},\omega_{i}}\cdot\chi_{\{\mathsf{h}_{\phi_{1},\omega_{1}}+\mathsf{h}_{\phi_{2},\omega_{2}}>0\}}}{(\mathsf{h}_{\phi_{1},\omega_{1}}+\mathsf{h}_{\phi_{2},\omega_{2}})^\frac{\alpha}{2}}\D\mu.
\end{align*} 
Further, by (i), $\mathsf{h}_{\phi_{i},\omega_{\lambda}^i}$ can be defined as in \eqref{rd}. Then, by using the above equality, we derive \begin{align*}
	&\int_{\varDelta}\mathsf{h}_{\phi_{i},\omega_{\lambda}^i}\D\mu=\int_{X}\chi_{\varDelta}\mathsf{h}_{\phi_{i},\omega_{\lambda}^i}\D\mu\stackrel{\eqref{intcom}}{=}\int_{X}\chi_{\varDelta}\circ\phi_{i}\D\mu_{\omega_{\lambda}^i}\\&=\int_{\phi_{i}^{-1}(\varDelta)}|\omega_{\lambda}^i|^2\D\mu\stackrel{\eqref{wfun}}{=}\int_{\phi_{i}^{-1}(\varDelta)}(\lambda^2+2\lambda(1-\lambda)\mathsf{h}_{i}+(1-\lambda)^2\mathsf{h}_{i}^2)\D\mu_{\omega_{i}}\\&=\int_{\varDelta}\lambda^2\mathsf{h}_{\phi_{i},\omega_{i}}\D\mu+\int_{\varDelta}2\lambda(1-\lambda)\big(\mathsf{E}_{\phi_{i},\omega_{i}}((\mathsf{h}_{\phi_{1},\omega_{1}}+\mathsf{h}_{\phi_{2},\omega_{2}})^\frac{1}{2})\circ\phi_{i}^{-1}\big)\cdot\frac{\mathsf{h}_{\phi_{i},\omega_{i}}\cdot\chi_{\{\mathsf{h}_{\phi_{1},\omega_{1}}+\mathsf{h}_{\phi_{2},\omega_{2}}>0\}}}{(\mathsf{h}_{\phi_{1},\omega_{1}}+\mathsf{h}_{\phi_{2},\omega_{2}})^\frac{1}{2}}\D\mu\\&+\int_{\varDelta}(1-\lambda)^2\big(\mathsf{E}_{\phi_{i},\omega_{i}}(\mathsf{h}_{\phi_{1},\omega_{1}}+\mathsf{h}_{\phi_{2},\omega_{2}})\circ\phi_{i}^{-1}\big)\cdot\frac{\mathsf{h}_{\phi_{i},\omega_{i}}\cdot\chi_{\{\mathsf{h}_{\phi_{1},\omega_{1}}+\mathsf{h}_{\phi_{2},\omega_{2}}>0\}}}{\mathsf{h}_{\phi_{1},\omega_{1}}+\mathsf{h}_{\phi_{2},\omega_{2}}}\D\mu,
\end{align*}
Since $\varDelta$ is arbitrary, it follows immediately that\begin{align*}
	\mathsf{h}_{\phi_{i},\omega_{\lambda}^i}&=\lambda^2\mathsf{h}_{\phi_{i},\omega_{i}}+2\lambda(1-\lambda)\big(\mathsf{E}_{\phi_{i},\omega_{i}}((\mathsf{h}_{\phi_{1},\omega_{1}}+\mathsf{h}_{\phi_{2},\omega_{2}})^\frac{1}{2})\circ\phi_{i}^{-1}\big)\cdot\frac{\mathsf{h}_{\phi_{i},\omega_{i}}\cdot\chi_{\{\mathsf{h}_{\phi_{1},\omega_{1}}+\mathsf{h}_{\phi_{2},\omega_{2}}>0\}}}{(\mathsf{h}_{\phi_{1},\omega_{1}}+\mathsf{h}_{\phi_{2},\omega_{2}})^\frac{1}{2}}\\&+(1-\lambda)^2\big(\mathsf{E}_{\phi_{i},\omega_{i}}(\mathsf{h}_{\phi_{1},\omega_{1}}+\mathsf{h}_{\phi_{2},\omega_{2}})\circ\phi_{i}^{-1}\big)\cdot\frac{\mathsf{h}_{\phi_{i},\omega_{i}}\cdot\chi_{\{\mathsf{h}_{\phi_{1},\omega_{1}}+\mathsf{h}_{\phi_{2},\omega_{2}}>0\}}}{\mathsf{h}_{\phi_{1},\omega_{1}}+\mathsf{h}_{\phi_{2},\omega_{2}}}\\&=\lambda^2\mathsf{h}_{\phi_{i},\omega_{i}}+2\lambda(1-\lambda)\mathsf{E}_{i}^{\frac{1}{2}}+(1-\lambda)^2\mathsf{E}_{i}\quad a.e.\ [\mu].
\end{align*}
(iii) For any $\alpha\ge0$, we have \begin{align*}
	\int_{\varDelta}\mathsf{E}_{\phi_{i},\omega_{i}}((\mathsf{h}_{\phi_{1},\omega_{1}}+\mathsf{h}_{\phi_{2},\omega_{2}})^\alpha)\circ\phi_{i}^{-1}\mathsf{h}_{\phi_{i},\omega_{i}}\D\mu&\xlongequal{\eqref{intcom},\eqref{invexp}}\int_{\phi_{i}^{-1}(\varDelta)}\mathsf{E}_{\phi_{i},\omega_{i}}((\mathsf{h}_{\phi_{1},\omega_{1}}+\mathsf{h}_{\phi_{2},\omega_{2}})^\alpha)\D\mu_{\omega_{i}}\\&\stackrel{\eqref{exp}}{=}\int_{\phi_{i}^{-1}(\varDelta)}(\mathsf{h}_{\phi_{1},\omega_{1}}+\mathsf{h}_{\phi_{2},\omega_{2}})^\alpha\D\mu_{\omega_{i}}\ge0	
\end{align*}
for any $\varDelta\in\mathscr{A}$, which implies\begin{equation*}
	\mathsf{E}_{\phi_{i},\omega_{i}}((\mathsf{h}_{\phi_{1},\omega_{1}}+\mathsf{h}_{\phi_{2},\omega_{2}})^\alpha)\circ\phi_{i}^{-1}\mathsf{h}_{\phi_{i},\omega_{i}}\ge 0\quad \text{a.e.}\ [\mu].
\end{equation*}
By multiplying both sides of the above inequality by $\frac{\chi_{\{\mathsf{h}_{\phi_{1},\omega_{1}}+\mathsf{h}_{\phi_{2},\omega_{2}}>0\}}}{\mathsf{h}_{\phi_{1},\omega_{1}}+\mathsf{h}_{\phi_{2},\omega_{2}}}$, we obtain $\mathsf{E}_{i}^\alpha\ge0$ a.e. [$\mu$].

(iv) For simplicity, we write\begin{equation}\label{E}
	\textbf{E}=\Big\{\sum\limits_{i=1}^2\mathsf{E}_{\phi_{i},\omega_{i}}(\mathsf{h}_{\phi_{1},\omega_{1}}+\mathsf{h}_{\phi_{2},\omega_{2}})\circ\phi_{i}^{-1}=\infty\Big\}
\end{equation}
and denote its complement by $\textbf{E}^\complement$. It can be easily verified that
\begin{align*}
	\Bigg(\Big\{\sum\limits_{i=1}^2\mathsf{E}_{i}<\infty\Big\}\bigcap\textbf{E}\Bigg)\subseteq\Bigg(\bigcup\limits_{i=1}^2(\{\mathsf{h}_{\phi_{i},\omega_{i}}=0\}\bigcap\{\mathsf{E}_{\phi_{i},\omega_{i}}(\mathsf{h}_{\phi_{1},\omega_{1}}+\mathsf{h}_{\phi_{2},\omega_{2}})\circ\phi_{i}^{-1}=\infty\})\Bigg)
\end{align*}
Using \eqref{Epro}, we know that the set on the right-hand side of the above inclusion has measure zero. Consequently, we have
\begin{equation}\label{Er}
	\mu\Big(\Big\{\sum\limits_{i=1}^2\mathsf{E}_{i}<\infty\Big\}\bigcap\textbf{E}\Big)=0.
\end{equation} 
On the other hand, by the fact that\begin{equation*}
\Bigg(\Big\{\sum\limits_{i=1}^2\mathsf{E}_{i}=\infty\Big\}\bigcap\textbf{E}^\complement\Bigg)\subseteq\Bigg(\bigcup\limits_{i=1}^2\big(\{\mathsf{E}_{i}=\infty\}\bigcap\textbf{E}^\complement\big)\Bigg)=\emptyset,	
\end{equation*}
we obtain
\begin{equation}\label{El}
	\mu\Big(\Big\{\sum\limits_{i=1}^2\mathsf{E}_{i}=\infty\Big\}\bigcap\textbf{E}^\complement\Big)=0.
\end{equation}
Then, combining \eqref{Er} and \eqref{El} yields the desired result.
\end{proof}

\begin{thm}\label{meanbasic}
 Suppose $C_{\phi_{1},\omega_{1}}$ and $C_{\phi_{2},\omega_{2}}$ are densely defined. Fix $\lambda\in[0,1]$. Then the pair $\mathcal{M}_{\lambda}(\textbf{C}_{\phi,\omega})$ satisfies the following properties$:$
\begin{enumerate}
	\item[(i)]
 $\EuScript{D}(\mathcal{M}_{\lambda}(\textbf{C}_{\phi,\omega}))=L^2\Big(\Big(1+\sum\limits_{i=1}^2(\mathsf{h}_{\phi_{i},\omega_{i}}+\mathsf{E}_{i})\Big)\D\mu\Big)$,
	\item[(ii)]
	$\overline{\EuScript{D}(\mathcal{M}_{\lambda}(\textbf{C}_{\phi,\omega}))}=\chi_{\Big\{\sum\limits_{i=1}^2\mathsf{E}_{\phi_{i},\omega_{i}}(\mathsf{h}_{\phi_{1},\omega_{1}}+\mathsf{h}_{\phi_{2},\omega_{2}})\circ\phi_{i}^{-1}<\infty\Big\}}\cdot L^2(\mu)$,
	\item[(iii)]
		$\EuScript{D}(\mathcal{M}_{\lambda}(\textbf{C}_{\phi,\omega}))^\perp=\chi_{\Big\{\sum\limits_{i=1}^2\mathsf{E}_{\phi_{i},\omega_{i}}(\mathsf{h}_{\phi_{1},\omega_{1}}+\mathsf{h}_{\phi_{2},\omega_{2}})\circ\phi_{i}^{-1}=\infty\Big\}}\cdot L^2(\mu)$,
	\item[(iv)]
	$\mathcal{M}_{\lambda}(\textbf{C}_{\phi,\omega})\begin{cases}
	=(C_{\phi_{1},\omega_{\lambda}^1},C_{\phi_{2},\omega_{\lambda}^2})&\lambda\in(0,1)\\
	\subseteq(C_{\phi_{1},\omega_{\lambda}^1},C_{\phi_{2},\omega_{\lambda}^2})&\lambda\in\{0,1\}
	\end{cases}$, where $\omega_{\lambda}^i$ {\em ($i=1,2$)} are defined as in \eqref{wfun}.
\end{enumerate}
\end{thm}
\begin{proof}
	Given $i\in\{1,2\}$. Let $\mathcal{M}_{\lambda}(\textbf{C}_{\phi,\omega})_{i}=\lambda U_{i}|\textbf{C}_{\phi,\omega}|+(1-\lambda)|\textbf{C}_{\phi,\omega}|U_{i}$. Applying Theorem \ref{polar}, we have
	\begin{equation}\label{mpair}
		\begin{aligned}
			\mathcal{M}_{\lambda}(\textbf{C}_{\phi,\omega})_{i}f&=\lambda C_{\phi_{i},\omega_{i}}f+(1-\lambda)\mathsf{M}_{\sqrt{\mathsf{h}_{\phi_{1},\omega_{1}}+\mathsf{h}_{\phi_{2},\omega_{2}}}}C_{\phi_{i},\widetilde{\omega_{i}}}f\\&=\lambda C_{\phi_{i},\omega_{i}}f+(1-\lambda)\sqrt{\mathsf{h}_{\phi_{1},\omega_{1}}+\mathsf{h}_{\phi_{2},\omega_{2}}}\frac{\omega_{i}\cdot\chi_{\{\omega_{i}\ne 0\}}}{\sqrt{(\mathsf{h}_{\phi_{1},\omega_{1}}+\mathsf{h}_{\phi_{2},\omega_{2}})\circ\phi_{i}}}\cdot f\circ\phi_{i}\\&=\lambda\omega_{i}\cdot f\circ\phi_{i}+(1-\lambda)\omega_{i}\cdot\chi_{\{\omega_{i}\ne 0\}}
			\left(\frac{\mathsf{h}_{\phi_{1},\omega_{1}}+\mathsf{h}_{\phi_{2},\omega_{2}}}{(\mathsf{h}_{\phi_{1},\omega_{1}}+\mathsf{h}_{\phi_{2},\omega_{2}})\circ\phi_{i}}\right)^\frac{1}{2}\cdot f\circ\phi_{i}\\&=\left(\mathsf{h}_{i}+\lambda(1-\mathsf{h}_{i})\right)\omega_{i}\cdot f\circ\phi_{i}\\&=\omega_{\lambda}^i\cdot f\circ\phi_{i}
		\end{aligned}
	\end{equation} 
	for $f\in\mathcal{M}_{\lambda}(\textbf{C}_{\phi,\omega})_{i}$, where  $\widetilde{\omega_{i}}$, $\mathsf{h}_{i}$ and $\omega_{\lambda}^i$ are defined as in Theorem \ref{polar} (ii) and Lemma \ref{weighted}, respectively.
	
(i) Since $\mathcal{M}_{\lambda}(\textbf{C}_{\phi,\omega})=(\mathcal{M}_{\lambda}(\textbf{C}_{\phi,\omega})_{1},\mathcal{M}_{\lambda}(\textbf{C}_{\phi,\omega})_{2}),$ then by \eqref{mpair} and Theorem \ref{polar} (ii), we derive
\begin{equation}\label{domen}
	\begin{aligned}
	\EuScript{D}(\mathcal{M}_{\lambda}(\textbf{C}_{\phi,\omega}))&=\bigcap\limits_{i=1}^{2}\EuScript{D}(\mathcal{M}_{\lambda}(\textbf{C}_{\phi,\omega})_{i})\\&=\bigcap\limits_{i=1}^{2}\left(\EuScript{D}(C_{\phi_{i},\omega_{i}})\bigcap\EuScript{D}(C_{\phi_{i},\widetilde{\omega_{i}}})\bigcap\EuScript{D}(\mathsf{M}_{\sqrt{\mathsf{h}_{\phi_{1},\omega_{1}}+\mathsf{h}_{\phi_{2},\omega_{2}}}}C_{\phi_{i},\widetilde{\omega_{i}}})\right)\\&=\bigcap\limits_{i=1}^{2}\left(\EuScript{D}(C_{\phi_{i},\omega_{i}})\bigcap\EuScript{D}(\mathsf{M}_{\sqrt{\mathsf{h}_{\phi_{1},\omega_{1}}+\mathsf{h}_{\phi_{2},\omega_{2}}}}C_{\phi_{i},\widetilde{\omega_{i}}})\right).
	\end{aligned}
\end{equation}
Further, we see that
\begin{small}
	\begin{equation*}\label{quasinormalchar}
		\begin{aligned}
			&\EuScript{D}(|\textbf{C}_{\phi,\omega}|U_{i})=\EuScript{D}(\mathsf{M}_{\sqrt{\mathsf{h}_{\phi_{1},\omega_{1}}+\mathsf{h}_{\phi_{2},\omega_{2}}}}C_{\phi_{i},\widetilde{\omega_{i}}})=\Bigg\{f\in L^2(\mu)\colon\int_{X}\Big|\sqrt{\mathsf{h}_{\phi_{1},\omega_{1}}+\mathsf{h}_{\phi_{2},\omega_{2}}}\cdot \widetilde{\omega_{i}}\cdot f\circ\phi_{i}\Big|^2\D\mu<\infty\Bigg\}\\&=\Bigg\{f\in L^2(\mu)\colon\int_{X}|\mathsf{h}_{i}|^2|\omega_{i}|^2|f\circ\phi_{i}|^2\D\mu<\infty\Bigg\}\\&\stackrel{\eqref{mueq0}}{=}\Bigg\{f\in L^2(\mu)\colon\int_{\{(\sum\limits_{i=1}^2\mathsf{h}_{\phi_{i},\omega_{i}})\circ\phi_{i}>0\}}\frac{\mathsf{h}_{\phi_{1},\omega_{1}}+\mathsf{h}_{\phi_{2},\omega_{2}}}{(\mathsf{h}_{\phi_{1},\omega_{1}}+\mathsf{h}_{\phi_{2},\omega_{2}})\circ\phi_{i}}|f\circ\phi_{i}|^2\D\mu_{\omega_{i}}<\infty\Bigg\}\\&\stackrel{\eqref{exp}}{=}\Bigg\{f\in L^2(\mu)\colon\int_{\phi_{i}^{-1}(\{\sum\limits_{i=1}^2\mathsf{h}_{\phi_{i},\omega_{i}}>0\})}\mathsf{E}_{\phi_{i},\omega_{i}}(\mathsf{h}_{\phi_{1},\omega_{1}}+\mathsf{h}_{\phi_{2},\omega_{2}})\frac{|f\circ\phi_{i}|^2}{(\mathsf{h}_{\phi_{1},\omega_{1}}+\mathsf{h}_{\phi_{2},\omega_{2}})\circ\phi_{i}}\D\mu_{\omega_{i}}<\infty\Bigg\}\\&\stackrel{\eqref{invexp}}{=}\Bigg\{f\in L^2(\mu)\colon\int_{\phi_{i}^{-1}(\{\sum\limits_{i=1}^2\mathsf{h}_{\phi_{i},\omega_{i}}>0\})}((\mathsf{E}_{\phi_{i},\omega_{i}}(\sum\limits_{i=1}^2\mathsf{h}_{\phi_{i},\omega_{i}})\circ\phi_{i}^{-1})\circ\phi_{i})\frac{|f\circ\phi_{i}|^2}{(\mathsf{h}_{\phi_{1},\omega_{1}}+\mathsf{h}_{\phi_{2},\omega_{2}})\circ\phi_{i}}\D\mu_{\omega_{i}}<\infty\Bigg\}
		\end{aligned} 
	\end{equation*}	
\end{small}
\begin{small}
\begin{equation}
	\begin{aligned}
	&\stackrel{\eqref{intcom}}{=}\Bigg\{f\in L^2(\mu)\colon\int_{\{\sum\limits_{i=1}^2\mathsf{h}_{\phi_{i},\omega_{i}}>0\}}(\mathsf{E}_{\phi_{i},\omega_{i}}(\sum\limits_{i=1}^2\mathsf{h}_{\phi_{i},\omega_{i}})\circ\phi_{i}^{-1})\frac{\mathsf{h}_{\phi_{i},\omega_{i}}|f|^2}{\mathsf{h}_{\phi_{1},\omega_{1}}+\mathsf{h}_{\phi_{2},\omega_{2}}}\D\mu<\infty\Bigg\}\\&=L^2\Bigg(\Big(1+\big(\mathsf{E}_{\phi_{i},\omega_{i}}(\mathsf{h}_{\phi_{1},\omega_{1}}+\mathsf{h}_{\phi_{2},\omega_{2}})\circ\phi_{i}^{-1}\big)\cdot\frac{\mathsf{h}_{\phi_{i},\omega_{i}}\cdot\chi_{\{\mathsf{h}_{\phi_{1},\omega_{1}}+\mathsf{h}_{\phi_{2},\omega_{2}}>0\}}}{\mathsf{h}_{\phi_{1},\omega_{1}}+\mathsf{h}_{\phi_{2},\omega_{2}}}\Big)\D\mu\Bigg)\quad\quad\quad\quad\quad\quad\quad\quad\quad\\&=L^2((1+\mathsf{E}_{i})\D\mu).
	\end{aligned} 
\end{equation}	
\end{small}Thus, combining \eqref{domen} and the above equality with the argument used in Proposition \ref{basic} (i), it is not difficult to obtain that\begin{equation*}
\EuScript{D}(\mathcal{M}_{\lambda}(\textbf{C}_{\phi,\omega}))=L^2\Big(\Big(1+\sum\limits_{i=1}^2(\mathsf{h}_{\phi_{i},\omega_{i}}+\mathsf{E}_{i})\Big)\D\mu\Big).
\end{equation*} 

(ii) Let $\textbf{E}$ be defined as in \eqref{E}. In view of \eqref{Epro}, we have $0<\frac{\mathsf{h}_{\phi_{i},\omega_{i}}\cdot\chi_{\{\mathsf{h}_{\phi_{1},\omega_{1}}+\mathsf{h}_{\phi_{2},\omega_{2}}>0\}}}{\mathsf{h}_{\phi_{1},\omega_{1}}+\mathsf{h}_{\phi_{2},\omega_{2}}}<\infty$ a.e. $[\mu]$ on $\textbf{E}$ for $i=1,2$. This implies $\sum\limits_{i=1}^2\mathsf{E}_{i}=\infty$ a.e. $[\mu]$ on $\textbf{E}$. Suppose $f\in\EuScript{D}(\mathcal{M}_{\lambda}(\textbf{C}_{\phi,\omega}))$, then $\mu(\{f\ne0\}\bigcap\textbf{E})=0$. Otherwise, we conclude that there exists $n\in\mathbb{N}$ such that $\mu(\{|f|^2\ge\frac{1}{n}\}\bigcap\textbf{E})\ne0$, which yields that\begin{align*}
\int_{X}|f|^2(1+\sum\limits_{i=1}^2(\mathsf{h}_{\phi_{i},\omega_{i}}+\mathsf{E}_{i}))\D\mu&\ge\int_{\{|f|^2\ge\frac{1}{n}\}\bigcap\textbf{E}}|f|^2\sum\limits_{i=1}^2\mathsf{E}_{i}\D\mu\\&\ge\int_{\{|f|^2\ge\frac{1}{n}\}\bigcap\textbf{E}}\frac{1}{n}\cdot\infty\D\mu\\&\ge\infty\cdot\mu(\{|f|^2\ge\frac{1}{n}\}\bigcap\textbf{E})=\infty.
\end{align*}
This contradicts our assumption. Therefore, we get $\EuScript{D}(\mathcal{M}_{\lambda}(\textbf{C}_{\phi,\omega}))\subseteq\chi_{\textbf{E}^\complement}\cdot L^{2}(\mu)$. Moreover, it can be easily verified that $\overline{\EuScript{D}(\mathcal{M}_{\lambda}(\textbf{C}_{\phi,\omega}))}\subseteq\chi_{\textbf{E}^\complement}\cdot L^{2}(\mu)$. On the other hand, take $f\in\chi_{\textbf{E}^\complement}\cdot L^{2}(\mu)$ and set $X_{n}=\Big\{x\colon \sum\limits_{i=1}^2(\mathsf{h}_{\phi_{i},\omega_{i}}+\mathsf{E}_{i})\le n\Big\}$. By Lemma \ref*{weighted} (iv), we see that $\chi_{X_{n}}$ converges pointwise to $\chi_{\textbf{E}^\complement}$ a.e. $[\mu]$. Hence, $\lim\limits_{n\to\infty}|f\chi_{X_{n}}(x)-f(x)|^2=0$ a.e. $[\mu]$. Applying the Dominated Convergence Theorem, we have $\lim\limits_{n\to\infty}\int_{X}|f\chi_{X_{n}}-f|^2\D\mu=0$, which implies, by using the fact $f_{\chi_{X_{n}}}\in\EuScript{D}(\mathcal{M}_{\lambda}(\textbf{C}_{\phi,\omega}))$, that
$f\in\overline{\EuScript{D}(\mathcal{M}_{\lambda}(\textbf{C}_{\phi,\omega}))}$. Consequently, we obtain $\chi_{\textbf{E}^\complement}\cdot L^{2}(\mu)=\overline{\EuScript{D}(\mathcal{M}_{\lambda}(\textbf{C}_{\phi,\omega}))}$. %Finally, we claim that

(iii) This follows directly from equation $L^2(\mu)=\overline{\EuScript{D}(\mathcal{M}_{\lambda}(\textbf{C}_{\phi,\omega}))}\oplus\EuScript{D}(\mathcal{M}_{\lambda}(\textbf{C}_{\phi,\omega}))^\perp$ and (ii).

(iv) Using Lemma \ref{weighted} (i), we know that $C_{\phi_{i},\omega_{\lambda}^i}$ ($i=1,2$) are well-defined, which yields, by \eqref{mpair}, that \begin{equation}\label{meang}
	\mathcal{G}(\mathcal{M}_{\lambda}(\textbf{C}_{\phi,\omega}))\subseteq\mathcal{G}((C_{\phi_{1},\omega_{\lambda}^1},C_{\phi_{2},\omega_{\lambda}^2}))
\end{equation}
for $\lambda\in[0,1]$. By Theorem \ref{combasic} (i) and the argument used in Proposition \ref{basic} (i), we have
\begin{equation*}
\EuScript{D}((C_{\phi_{1},\omega_{\lambda}^1},C_{\phi_{2},\omega_{\lambda}^2}))=\bigcap\limits_{i=1}^2\EuScript{D}(C_{\phi_{i},\omega_{\lambda}^i})=L^2\Big(\big(1+\sum\limits_{i=1}^2\mathsf{h}_{\phi_{i},\omega_{\lambda}^i}\big)\D\mu\Big),
\end{equation*} 
where $\mathsf{h}_{\phi_{i},\omega_{\lambda}^i}$ ($i=1,2$) are defined as in Lemma \ref{weighted} (ii). Consider (i) of this theorem, then one can easily verify that \begin{equation}\label{meandomain}
	\EuScript{D}((C_{\phi_{1},\omega_{\lambda}^1},C_{\phi_{2},\omega_{\lambda}^2}))\subseteq\EuScript{D}(\mathcal{M}_{\lambda}(\textbf{C}_{\phi,\omega}))
\end{equation} for $\lambda\in(0,1)$ since in view of Lemma \ref{weighted} (iii) we obtain $\mathsf{E}_{i}^\frac{1}{2}\ge0$ and $\mathsf{E}_{i}\ge0$. Therefore, by \eqref{meang} and \eqref{meandomain}, we get the desired equality.
\end{proof}

Suppose $C_{\phi,\omega}$ is densely defind. By setting $C_{\phi_{1},\omega_{1}}=C_{\phi,\omega}$ and $C_{\phi_{2},\omega_{2}}=C_{x,0}=0$ in the above theorem and \eqref{singlemulequiv}, we immediately have the single-operator version of Theorem \ref{meanbasic}.

\begin{cor}\label{singlemeanbasic}
Suppose $C_{\phi,\omega}$ is densely defined. Fix $\lambda\in[0,1]$. Then
\begin{enumerate}
	\item[(i)]
	$\EuScript{D}(\mathcal{M}_{\lambda}(C_{\phi,\omega}))=L^2\Big(\Big(1+\mathsf{h}_{\phi,\omega}+\mathsf{E}_{\phi,\omega}(\mathsf{h}_{\phi,\omega})\circ\phi^{-1}\Big)\D\mu\Big)$,
	\item[(ii)]
	$\overline{\EuScript{D}(\mathcal{M}_{\lambda}(C_{\phi,\omega}))}=\chi_{\Big\{\mathsf{E}_{\phi,\omega}(\mathsf{h}_{\phi,\omega})\circ\phi^{-1}<\infty\Big\}}\cdot L^2(\mu)$,
	\item[(iii)]
	$\EuScript{D}(\mathcal{M}_{\lambda}(C_{\phi,\omega}))^\perp=\chi_{\Big\{\mathsf{E}_{\phi,\omega}(\mathsf{h}_{\phi,\omega})\circ\phi^{-1}=\infty\Big\}}\cdot L^2(\mu)$,
	\item[(iv)]
	$\mathcal{M}_{\lambda}(C_{\phi,\omega})\begin{cases}
		=C_{\phi,\omega_{\lambda}}&\lambda\in(0,1)\\
		\subseteq C_{\phi,\omega_{\lambda}}&\lambda\in\{0,1\}
	\end{cases}$, where 
	$\omega_{\lambda}=((1-\lambda)(\frac{\chi_{\{\omega\ne0\}}\mathsf{h}_{\phi,\omega}}{\mathsf{h}_{\phi,\omega}\circ\phi})^\frac{1}{2}+\lambda)\omega$.
\end{enumerate}
\end{cor}

Given $\lambda\in\{0,1\}$. To show that $\mathcal{M}_{\lambda}(\textbf{C}_{\phi,\omega})\ne(C_{\phi_{1},\omega_{\lambda}^1},C_{\phi_{2},\omega_{\lambda}^2})$ in general, it suffices to provide a weighted composition operator $C_{\phi,\omega}$ such that $\EuScript{D}(\mathcal{M}_{\lambda}(C_{\phi,\omega}))\ne\EuScript{D}(C_{\phi,\omega_{\lambda}})$ by Corollary \ref{singlemeanbasic} (iv).

\begin{exa}\label{nemean}
	Let $X=\mathbb{N}$, $\mathscr{A}=2^X$ and $\mu$ be the counting measure on $X$. Let $m\in\mathbb{N}$. Define the $\mathscr{A}$-measurable transform $\phi:X\to X$ by$$\phi(n)=\begin{cases}
		2m-1&n=2m\\
		2m&n=2m-1
	\end{cases}$$
	and the function $\omega:X\to\mathbb{R}$ by
	$$\omega(n)=\begin{cases}
		2m&n=2m\\
		0&n=2m-1
	\end{cases}.$$
 Since $\mu(\varDelta)=0$ if and only if $\varDelta=\emptyset$, then $\mu_{\omega}\circ\phi^{-1}\ll\mu$. Hence, $C_{\phi,\omega}$ is well-defined. By \eqref{h}, we have
 \begin{equation}\label{rnex1}
 	\mathsf{h}_{\phi,\omega}(n)=\sum\limits_{x\in\phi^{-1}(n)}|\omega(x)|^2=\begin{cases}
 		0&n=2m\\
 		4m^2&n=2m-1
 	\end{cases},
 \end{equation}
which implies $C_{\phi,\omega}$ is densely defined. On the other hand, it follows readily that $\phi^{-1}(2^X)=2^X$ since $\phi$ is invertible, and thus
	\begin{equation}\label{closed3}
		\mathsf{E}_{\phi,\omega}(\mathsf{h}_{\phi,\omega})\circ\phi^{-1}(n)=\mathsf{h}_{\phi,\omega}\circ\phi^{-1}(n)=\begin{cases}
			4m^2&n=2m\\
			0&n=2m-1
		\end{cases}.
	\end{equation} 
By setting $C_{\phi_{1},\omega_{1}}=C_{\phi,\omega}$ and $C_{\phi_{2},\omega_{2}}=C_{x,0}=0$ in Lemma \ref{weighted} (ii), we have
	$$\EuScript{D}(C_{\phi,\omega_{0}})\xlongequal{\text{Theorem \ref{combasic} (i)}}L^2((1+\mathsf{h}_{\phi,\omega_{0}})\D\mu)\xlongequal{\text{Lemma}\ \ref{weighted}\ (ii)}L^2((1+\mathsf{E}_{\phi,\omega}(\mathsf{h}_{\phi,\omega})\circ\phi^{-1})\D\mu).$$
Consider the function $f(n)=\begin{cases}
		\frac{1}{n+1}&n=2m-1\\
		0&n=2m
	\end{cases}.$ By \eqref{rnex1} and \eqref{closed3}, we obtain
	$$\int_{X}|f|^2\mathsf{h}_{\phi,\omega}\D\mu=\sum\limits_{m=1}^\infty (\frac{1}{4m^2}\cdot 4m^2)=\sum\limits_{m=1}^\infty 1=\infty$$
	and
	 $$\int_{X}|f|^2(1+\mathsf{E}_{\phi,\omega}(\mathsf{h}_{\phi,\omega})\circ\phi^{-1})\D\mu=\int_{X}|f|^2\D\mu=\sum\limits_{m=1}^\infty\frac{1}{4m^2}<\infty.$$
	 Hence, $f\in\EuScript{D}(C_{\phi,\omega_{0}})$ and $f\notin L^2((1+\mathsf{h}_{\phi,\omega}+\mathsf{E}_{\phi,\omega}(\mathsf{h}_{\phi,\omega})\circ\phi^{-1})\D\mu)\xlongequal{\text{Corollary}\ \ref{singlemeanbasic}\ (i)}\EuScript{D}(\mathcal{M}_{0}(C_{\phi,\omega})))$. Therefore, we get $\mathcal{M}_{0}(C_{\phi,\omega})\ne C_{\phi,\omega_{0}}$. 
	 
	 Define the $\mathscr{A}$-measurable transform $\widetilde{\phi}:X\to X$ by$$\widetilde{\phi}(n)=\begin{cases}
	 	2m-1&n=2m\\
	 	2&n=2m-1
	 \end{cases}$$
	 and the function $\widetilde{\omega}:X\to\mathbb{R}$ by
	 $$\widetilde{\omega}(n)=\begin{cases}
	 	2m&n=2m\\
	 	\frac{1}{2m}&n=2m-1
	 \end{cases}.$$
	 By a similar argument as above, we obtain$$\mathsf{h}_{\widetilde{\phi},\widetilde{\omega}}(n)=\begin{cases}
	 	4m^2&n=2m-1\\
	 	0&n=2m\backslash\{2\}\\
	 	\sum\limits_{m=1}^\infty\frac{1}{4m^2}&n=2
	\end{cases}$$
	and$$\mathsf{E}_{\widetilde{\phi},\widetilde{\omega}}(\mathsf{h}_{\widetilde{\phi},\widetilde{\omega}})\circ\widetilde{\phi}^{-1}(n)=\begin{cases}
		\mathsf{h}_{\widetilde{\phi},\widetilde{\omega}}(2m)&n=2m-1\\
		0&n=2m\backslash\{2\}\\
		\frac{\sum\limits_{m=1}^\infty\mathsf{h}_{\widetilde{\phi},\widetilde{\omega}}(2m-1)|\widetilde{\omega}(2m-1)|^2}{\sum\limits_{m=1}^\infty|\widetilde{\omega}(2m-1)|^2}=\infty&n=2
	\end{cases}.$$
From Theorem \ref{combasic} (i) and Corollary \ref{singlemeanbasic} (i), it is clear that the characteristic function $\chi_{\{2\}}\in\EuScript{D}(C_{\widetilde{\phi},\widetilde{\omega}})=L^2((1+\mathsf{h}_{\widetilde{\phi},\widetilde{\omega}})\D\mu)$ but $\chi_{\{2\}}\notin L^2((1+\mathsf{h}_{\widetilde{\phi},\widetilde{\omega}}+\mathsf{E}_{\widetilde{\phi},\widetilde{\omega}}(\mathsf{h}_{\widetilde{\phi},\widetilde{\omega}})\circ\widetilde{\phi}^{-1})\D\mu)=\EuScript{D}(\mathcal{M}_{1}(C_{\widetilde{\phi},\widetilde{\omega}}))$, that is, $\mathcal{M}_{1}(C_{\widetilde{\phi},\widetilde{\omega}})\ne C_{\widetilde{\phi},\widetilde{\omega}}$.
\end{exa}

A direct application of \cite[Lemma 12.3]{Bud-AMPP-2014} yields the following necessary and sufficient condition for the equality $\mathcal{M}_{\lambda}(\textbf{C}_{\phi,\omega})=(C_{\phi_{1},\omega_{\lambda}^{1}},C_{\phi_{2},\omega_{\lambda}^{2}})$ to hold when $\lambda \in \{0, 1\}$.

\begin{cor}
Suppose $C_{\phi_{1},\omega_{1}}$ and $C_{\phi_{2},\omega_{2}}$ are densely defined. Fix $\lambda\in\{0,1\}$. Then $\mathcal{M}_{\lambda}(\textbf{C}_{\phi,\omega})=(C_{\phi_{1},\omega_{\lambda}^{1}},C_{\phi_{2},\omega_{\lambda}^{2}})$ if and only if $\frac{1+\sum\limits_{i=1}^2\mathsf{h}_{\phi_{i},\omega_{i}}+\sum\limits_{i=1}^2\mathsf{E}_{i}}{1+\lambda\sum\limits_{i=1}^2\mathsf{h}_{\phi_{i},\omega_{i}}+(1-\lambda)\sum\limits_{i=1}^2\mathsf{E}_{i}}\le c$ a.e. $[\mu]$ for some $c\in(0,\infty)$.
\end{cor}
\begin{proof}
By Theorem \ref{meanbasic} (iv), it suffices to prove that $\EuScript{D}(C_{\phi_{1},\omega_{\lambda}^1},C_{\phi_{2},\omega_{\lambda}^2})\subseteq\EuScript{D}(\mathcal{M}_{\lambda}(\textbf{C}_{\phi,\omega}))$ if and only if $\frac{1+\sum\limits_{i=1}^2\mathsf{h}_{\phi_{i},\omega_{i}}+\sum\limits_{i=1}^2\mathsf{E}_{i}}{1+\lambda\sum\limits_{i=1}^2\mathsf{h}_{\phi_{i},\omega_{i}}+(1-\lambda)\sum\limits_{i=1}^2\mathsf{E}_{i}}\le c$ a.e. $[\mu]$ for some $c\in(0,\infty)$. By Proposition \ref{basic} (i), Lemma \ref{weighted} (ii) and Theorem \ref{meanbasic} (i), we have $$\EuScript{D}(C_{\phi_{1},\omega_{\lambda}^1},C_{\phi_{2},\omega_{\lambda}^2})=L^2((1+\sum\limits_{i=1}^2\mathsf{h}_{\phi_{i},\omega_{\lambda}^i})\D\mu)=L^2((1+\lambda\sum\limits_{i=1}^2\mathsf{h}_{\phi_{i},\omega_{i}}+(1-\lambda)\sum\limits_{i=1}^2\mathsf{E}_{i})\D\mu)$$
and
$$\EuScript{D}(\mathcal{M}_{\lambda}(\textbf{C}_{\phi,\omega}))=L^2((1+\sum\limits_{i=1}^2\mathsf{h}_{\phi_{i},\omega_{i}}+\sum\limits_{i=1}^2\mathsf{E}_{i})\D\mu).$$ By \cite[Lemma 12.3]{Bud-AMPP-2014}, we know that $$L^2((1+\lambda\sum\limits_{i=1}^2\mathsf{h}_{\phi_{i},\omega_{i}}+(1-\lambda)\sum\limits_{i=1}^2\mathsf{E}_{i})\D\mu)\subseteq L^2((1+\sum\limits_{i=1}^2\mathsf{h}_{\phi_{i},\omega_{i}}+\sum\limits_{i=1}^2\mathsf{E}_{i})\D\mu)$$ if and only if $\frac{1+\sum\limits_{i=1}^2\mathsf{h}_{\phi_{i},\omega_{i}}+\sum\limits_{i=1}^2\mathsf{E}_{i}}{1+\lambda\sum\limits_{i=1}^2\mathsf{h}_{\phi_{i},\omega_{i}}+(1-\lambda)\sum\limits_{i=1}^2\mathsf{E}_{i}}\le c$ a.e. $[\mu]$ for some $c\in(0,\infty)$.
\end{proof}
\begin{rem}
If $C_{\phi_{i},\omega_{i}}$ is bounded on $L^2(\mu)$ for $i=1,2$, then $\mathcal{M}_{\lambda}(\textbf{C}_{\phi,\omega})=(C_{\phi_{1},\omega_{\lambda}^{1}},C_{\phi_{2},\omega_{\lambda}^{2}})$ for $\lambda\in\{0,1\}$. Although this result is obvious in the bounded case, we provide a measure-theoretic explanation for the sake of compatibility. Indeed, since $\mathsf{h}_{\phi_{1},\omega_{1}}+\mathsf{h}_{\phi_{2},\omega_{2}}\le\|\mathsf{h}_{\phi_{1},\omega_{1}}\|_{\infty}+\|\mathsf{h}_{\phi_{2},\omega_{2}}\|_{\infty}$, then by \cite[Proposition 15 (i)]{b-j-j-sW} we have $$\mathsf{E}_{\phi_{i},\omega_{i}}(\mathsf{h}_{\phi_{1},\omega_{1}}+\mathsf{h}_{\phi_{2},\omega_{2}})\circ\phi_{i}^{-1}\le\|\mathsf{h}_{\phi_{1},\omega_{1}}\|_{\infty}+\|\mathsf{h}_{\phi_{2},\omega_{2}}\|_{\infty},$$which implies $\mathsf{E}_{i}\le\|\mathsf{h}_{\phi_{1},\omega_{1}}\|_{\infty}+\|\mathsf{h}_{\phi_{2},\omega_{2}}\|_{\infty}$ for $i=1,2$. Thus there exists $c=1+2(\|\mathsf{h}_{\phi_{1},\omega_{1}}\|_{\infty}+\|\mathsf{h}_{\phi_{2},\omega_{2}}\|_{\infty})>0$ such that $\frac{1+\sum\limits_{i=1}^2\mathsf{h}_{\phi_{i},\omega_{i}}+\sum\limits_{i=1}^2\mathsf{E}_{i}}{1+\lambda\sum\limits_{i=1}^2\mathsf{h}_{\phi_{i},\omega_{i}}+(1-\lambda)\sum\limits_{i=1}^2\mathsf{E}_{i}}\le c$ a.e. $[\mu]$.
\end{rem}

We now focus on the connections between the pair $(C_{\phi_{1},\omega_{\lambda}^1},C_{\phi_{2},\omega_{\lambda}^2})$ and $\textbf{C}_{\phi,\omega}$ under the condition that $\mathcal{M}_{\lambda}(\textbf{C}_{\phi,\omega})=(C_{\phi_{1},\omega_{\lambda}^1},C_{\phi_{2},\omega_{\lambda}^2})$ for some $\lambda\in[0,1]$. For simplicity, we will always assume $\lambda\in(0,1)$. Let us start with some notions. Denote the set of all weighted composition operator pairs by $\mathcal{WCOP}$, i.e., $$\mathcal{WCOP}=\{(C_{\phi_{1},\omega_{1}},C_{\phi_{2},\omega_{2}})\in\mathcal{L}(L^2(\mu),L^2(\mu)\oplus L^2(\mu))\}.$$ We write
$$\mathcal{WCOP}_{\text{bounded}}=\{(C_{\phi_{1},\omega_{1}},C_{\phi_{2},\omega_{2}})\in\mathcal{WCOP}\bigcap\mathcal{B}(L^2(\mu),L^2(\mu)\oplus L^2(\mu))\}$$
and
$$\mathcal{WCOP}_{\text{dense}}=\{(C_{\phi_{1},\omega_{1}},C_{\phi_{2},\omega_{2}})\in\mathcal{WCOP}:C_{\phi_{i},\omega_{i}}\ \text{is densely defined in $L^2(\mu)$ for}\ i=1,2\}.$$ By Theorem \ref{meanbasic} (iv), we know that the $\lambda$-spherical mean transform maps a pair of weighted composition operators to another. However, as will be shown in Corollary \ref{denseequ}, the image of a densely defined weighted composition operator pair under this map is not necessarily densely defined. A natural question is whether, for $\lambda\in (0,1)$ the restricted map
$$\mathcal{M}_{\lambda}|_{\mathcal{WCOP}_{\text{bounded}}}:\mathcal{WCOP}_{\text{bounded}}\to\mathcal{WCOP}_{\text{bounded}},\quad\quad\textbf{C}_{\phi,\omega}\mapsto\mathcal{M}_{\lambda}(\textbf{C}_{\phi,\omega})$$ is bijective. By Remark \ref{singlemult} and Corollary \ref{singlemeanbasic}, the following example provides a negative answer. Moreover, this example also shows that $\mathcal{M}_{\lambda}|_{\mathcal{WCOP}_{\text{bounded}}}$ is not norm-continuous.

\begin{exa}\label{17}
Let $X=\mathbb{N}$, $\mathscr{A}=2^{X}$ and $\mu$ be the counting measure on $X$. Suppose $\lambda\in(0,1)$, then we have

(i) $\mathcal{M}_{\lambda}$ is not injective.

(ii) $\mathcal{M}_{\lambda}$ is not surjective.

(iii) $\mathcal{M}_{\lambda}|_{\mathcal{WCOP}_{\text{bounded}}}$ is not norm-continuous.
\begin{proof}
(i) Case 1: $\lambda\in(0,\frac{1}{2}]$. Let $c\in(0,\frac{1}{\lambda})$. Let $k\in\mathbb{N}$. Define the function $^c\omega_{inj}:X\to\mathbb{R}$ as
$$^{c}\omega_{inj}(x)=\begin{cases}
	0&x=1\\
	c\cdot(\frac{\lambda}{1-\lambda})^{2k-2}&x=2k\\
	(\frac{1}{\lambda}-c)\cdot(\frac{\lambda}{1-\lambda})^{2k-1}&x=2k+1
\end{cases}.$$ The $\mathscr{A}$-measurable transform $\phi_{inj}:X\to X$ is defined by
\begin{equation}\label{phiinj}
\phi_{inj}(x)=\begin{cases}
	1&x=1,2\\
	x-1&x\in\mathbb{N}\backslash\{1,2\}\\
\end{cases}.
\end{equation}
By \eqref{h}, we have
$$\mathsf{h}_{\phi_{inj},^c\omega_{inj}}(x)=\begin{cases}
	c^2&x=1\\
	|^c\omega_{inj}(x+1)|^2&x\in\mathbb{N}\backslash\{1\}
\end{cases}\le c^2+\frac{1}{\lambda^2},$$which yields, by Theorem \ref{combasic} (iv), that $C_{\phi_{inj},^c\omega_{inj}}$ is bounded. A direct calculation shows that
\begin{small}
	\begin{equation}\label{inj}
		\begin{aligned}
			&(1-\lambda)^c\omega_{inj}(x+1)+\lambda^c\omega_{inj}(x)\\&=\begin{cases}(1-\lambda)(\frac{1}{\lambda}-c)(\frac{\lambda}{1-\lambda})^{2k-1}+\lambda c(\frac{\lambda}{1-\lambda})^{2k-2}=(\frac{\lambda}{1-\lambda})^{2k-2}=(\frac{\lambda}{1-\lambda})^{x-2}&x=2k\\(1-\lambda)c(\frac{\lambda}{1-\lambda})^{2k}+\lambda(\frac{1}{\lambda}-c)(\frac{\lambda}{1-\lambda})^{2k-1}=(\frac{\lambda}{1-\lambda})^{2k-1}=(\frac{\lambda}{1-\lambda})^{x-2}&x=2k+1
			\end{cases}\\&=(\frac{\lambda}{1-\lambda})^{x-2}\quad\quad x\in\mathbb{N}\backslash\{1\}
		\end{aligned}
	\end{equation}
\end{small}for $c\in(0,\frac{1}{\lambda})$. 

\textbf{Claim 1:} $C_{\phi_{inj},^{\alpha}\omega_{inj}}\ne C_{\phi_{inj},^{\beta}\omega_{inj}}$ and $\mathcal{M}_{\lambda}(C_{\phi_{inj},^{\alpha}\omega_{inj}})=\mathcal{M}_{\lambda}(C_{\phi_{inj},^{\beta}\omega_{inj}})=C_{\phi_{inj},^{1}\omega_{inj}}$, where $\frac{1}{\lambda}>\alpha>\beta>0$, that is, the map $\mathcal{M}_{\lambda}$ is not injective for $\lambda\in(0,\frac{1}{2}]$.

Indeed, since $$C_{\phi_{inj},^{\alpha}\omega_{inj}}\chi_{\{2\}}=^{\alpha}\omega_{inj}\cdot\chi_{\{2\}}\circ\phi_{inj}=\frac{1-\lambda \alpha}{1-\lambda}\cdot\chi_{\{3\}}\ne\frac{1-\lambda\beta}{1-\lambda}\cdot\chi_{\{3\}}=C_{\phi_{inj},^{\beta}\omega_{inj}}\chi_{\{2\}},$$ then $C_{\phi_{inj},^{\alpha}\omega_{inj}}\ne C_{\phi_{inj},^{\beta}\omega_{inj}}$. On the other hand, by \eqref{inj} we obtain \begin{align*}
&(((1-\lambda)\chi_{\{^{c}\omega_{inj}\ne0\}}(\frac{\mathsf{h}_{\phi_{inj},^{c}\omega_{inj}}}{\mathsf{h}_{\phi_{inj},^{c}\omega_{inj}}\circ\phi_{inj}})^\frac{1}{2}+\lambda)\cdot^{c}\omega_{inj})(x)\\&=\begin{cases}
	0&x=1\\
	(1-\lambda)^c\omega_{inj}(x+1)+\lambda^c\omega_{inj}(x)&x\in\mathbb{N}\backslash\{1\}
\end{cases}=^1\omega_{inj},\quad\text{a.e.}\ [\mu],
\end{align*}
which implies, by Corollary \ref{singlemeanbasic} (iv), that $\mathcal{M}_{\lambda}(C_{\phi_{inj},^c\omega_{inj}})=C_{\phi_{inj},^1\omega_{inj}}$ for $c\in(0,\frac{1}{\lambda})$. Hence, $\mathcal{M}_{\lambda}(C_{\phi_{inj},^\alpha\omega_{inj}})=\mathcal{M}_{\lambda}(C_{\phi_{inj},^\beta\omega_{inj}})=C_{\phi_{inj},^1\omega_{inj}}$ for $\frac{1}{\lambda}>\alpha>\beta>0$. Therefore, the claim follows. 

Case 2: $\lambda\in(\frac{1}{2},1)$. Let $c\in(0,\frac{1}{\lambda})$. Let $k\in\mathbb{N}$. Consider the modified counting measure on $2^{\mathbb{N}}$ defined by $\widetilde{\mu}(\{x\})=\begin{cases}
	0&x=1\\
	1&x\ne1
\end{cases}$ and weighted composition operators on the measure space $(X,2^X,\widetilde{\mu})$. Let $k\in\mathbb{N}$. Define the function $\widetilde{^c\omega_{inj}}:X\to\mathbb{R}$ as
$$\widetilde{^{c}\omega_{inj}}(x)=\begin{cases}
	\frac{1-\lambda c}{1-\lambda}\cdot(\frac{1-\lambda}{\lambda})^{2k}&x=2k\\
	c\cdot(\frac{1-\lambda}{\lambda})^{2k-2}&x=2k-1
\end{cases}.$$ The $\mathscr{A}$-measurable transform $\widetilde{\phi_{inj}}:X\to X$ is defined by $$\widetilde{\phi_{inj}}(x)=x+1.$$
By \cite[Proposition 79]{b-j-j-sW}, we have
$$\mathsf{h}_{\widetilde{\phi_{inj}},\widetilde{^c\omega_{inj}}}(x)=\begin{cases}
	0&x=1\\
	|\widetilde{^c\omega_{inj}}(x-1)|^2&x\in\mathbb{N}\backslash\{1\}
\end{cases}\le c^2,$$which, via Theorem \ref{combasic} (iv), shows that $C_{\widetilde{\phi_{inj}},\widetilde{^c\omega_{inj}}}$ is bounded. A direct calculation yields that
\begin{small}
	\begin{equation}
		\begin{aligned}\label{injw}
			&(1-\lambda)^c\widetilde{\omega_{inj}}(x-1)+\lambda^c\widetilde{\omega_{inj}}(x)\\&=\begin{cases}
				(1-\lambda)c\cdot(\frac{1-\lambda}{\lambda})^{2k-2}+\lambda\cdot\frac{1-\lambda c}{1-\lambda}\cdot(\frac{1-\lambda}{\lambda})^{2k}=(\frac{1-\lambda}{\lambda})^{2k-1}&x=2k\\
				(1-\lambda)\cdot\frac{1-\lambda c}{1-\lambda}\cdot(\frac{1-\lambda}{\lambda})^{2k}+\lambda\cdot c\cdot(\frac{1-\lambda}{\lambda})^{2k}=(\frac{1-\lambda}{\lambda})^{2k}&x=2k+1
			\end{cases}\\&=(\frac{1-\lambda}{\lambda})^{x-1}\quad\quad x\in\mathbb{N}
		\end{aligned}
	\end{equation}
\end{small}for $c\in(0,\frac{1}{\lambda})$.

 \textbf{Claim 2:} $C_{\widetilde{\phi_{inj}},\widetilde{^{\alpha}\omega_{inj}}}\ne C_{\widetilde{\phi_{inj}},\widetilde{^{\beta}\omega_{inj}}}$ and $\mathcal{M}_{\lambda}(C_{\widetilde{\phi_{inj}},\widetilde{^{\alpha}\omega_{inj}}})=\mathcal{M}_{\lambda}(C_{\widetilde{\phi_{inj}},\widetilde{^{\beta}\omega_{inj}}})=C_{\widetilde{\phi_{inj}},\widetilde{\omega_{inj}}}$, where $\frac{1}{\lambda}>\alpha>\beta>0$ and $\widetilde{\omega_{inj}}=(\frac{1-\lambda}{\lambda})^{x-1}$, that is, the map $\mathcal{M}_{\lambda}$ is not injective for $\lambda\in(\frac{1}{2},1)$.

Indeed, since $$C_{\widetilde{\phi_{inj}},\widetilde{^{\alpha}\omega_{inj}}}\chi_{\{3\}}=\widetilde{^{\alpha}\omega_{inj}}\cdot\chi_{\{3\}}\circ\widetilde{\phi_{inj}}=\frac{1-\lambda \alpha}{\lambda}\cdot\chi_{\{2\}}\ne\frac{1-\lambda\beta}{\lambda}\cdot\chi_{\{2\}}=C_{\widetilde{\phi_{inj}},\widetilde{^{\beta}\omega_{inj}}}\chi_{\{3\}},$$ then $C_{\widetilde{\phi_{inj}},\widetilde{^{\alpha}\omega_{inj}}}\ne C_{\widetilde{\phi_{inj}},\widetilde{^{\beta}\omega_{inj}}}$. Moreover, from \eqref{injw} we get \begin{align*}
&(((1-\lambda)\chi_{\{\widetilde{^{c}\omega_{inj}}\ne0\}}(\frac{\mathsf{h}_{\widetilde{\phi_{inj}},\widetilde{^{c}\omega_{inj}}}}{\mathsf{h}_{\widetilde{\phi_{inj}},\widetilde{^{c}\omega_{inj}}}\circ\widetilde{\phi_{inj}}})^\frac{1}{2}+\lambda)\cdot\widetilde{^{c}\omega_{inj}})(x)\\&=\begin{cases}
	\lambda\cdot c&x=1\\
	(1-\lambda)\widetilde{^c\omega_{inj}}(x+1)+\lambda\widetilde{^c\omega_{inj}}(x)&x\in\mathbb{N}\backslash\{1\}
\end{cases}=\widetilde{\omega_{inj}},\quad\text{a.e.}\ [\widetilde{\mu}].
\end{align*}
By Corollary \ref{singlemeanbasic} (iv), we conclude that $\mathcal{M}_{\lambda}(C_{\widetilde{\phi_{inj}},\widetilde{^c\omega_{inj}}})=C_{\widetilde{\phi_{inj}},\widetilde{\omega_{inj}}}$ for $c\in(0,\frac{1}{\lambda})$. In particular, for $\frac{1}{\lambda}>\alpha>\beta>0$, $$\mathcal{M}_{\lambda}(C_{\widetilde{\phi_{inj}},\widetilde{^\alpha\omega_{inj}}})=\mathcal{M}_{\lambda}(C_{\widetilde{\phi_{inj}},\widetilde{^\beta\omega_{inj}}})=C_{\widetilde{\phi_{inj}},\widetilde{^1\omega_{inj}}},$$ and the claim follows. 

(ii) Case 1: $\lambda\in(0,1)\backslash\{\frac{1}{2}\}$. Define the $\mathscr{A}$-measurable transform $\phi:X\to X$ by
\begin{equation}\label{phiequ}
	\phi_{sur}(x)=\begin{cases}
		2&x=1\\
		1&x=2\\
		x&x\in\mathbb{N}\backslash\{1,2\}
	\end{cases}
\end{equation}
and the function $\omega_{sur}:X\to\mathbb{R}$ by
$\omega_{sur}(x)=\begin{cases}
	1-\lambda&x=1\\
	\lambda&x=2\\
	1&x\in\mathbb{N}\backslash\{1,2\}
\end{cases}.$

\textbf{Claim 3:} There dose not exist a densely defined weighted composition operator $C_{\psi,\omega}$ such that $\mathcal{M}_{\lambda}(C_{\psi,\omega})=C_{\phi_{sur},\omega_{sur}}$ for $\lambda\in(0,1)\backslash\{\frac{1}{2}\}$, that is, the map $\mathcal{M}_{\lambda}$ is not surjective.

Suppose Claim 3 fails, then $\mathcal{M}_{\lambda}(C_{\psi,\omega})=C_{\phi_{sur},\omega_{sur}}$ for some densely defined weighted composition operator $C_{\psi,\omega}$. By Remark \ref{singlemult} and Theorem \ref{meanbasic} (iv), we have $$\mathcal{M}_{\lambda}(C_{\psi,\omega})=C_{\psi,((1-\lambda)\chi_{\{\omega\ne0\}}(\frac{\mathsf{h}_{\psi,\omega}}{\mathsf{h}_{\psi,\omega}\circ\psi})^\frac{1}{2}+\lambda)\omega}=C_{\phi_{sur},\omega_{sur}},$$
which implies
\begin{equation}\label{comid}
	\omega_{\lambda}(x)\cdot\chi_{\{\phi_{sur}(x)\}}\circ\psi(x)=\omega_{sur}(x)\cdot\chi_{\{\phi_{sur}(x)\}}\circ\phi_{sur}(x)=\omega_{sur}(x)
\end{equation}
for any $x\in X$, where
\begin{equation}\label{exmlambda}
	\omega_{\lambda}=((1-\lambda)\chi_{\{\omega\ne0\}}(\frac{\mathsf{h}_{\psi,\omega}}{\mathsf{h}_{\psi,\omega}\circ\psi})^\frac{1}{2}+\lambda)\omega.
\end{equation}
Since $\omega_{sur}>0$ and \eqref{comid}, we have
\begin{equation}\label{weighteq1}
	\omega_{\lambda}=\omega_{sur}
\end{equation}
and $\psi=\phi_{sur}$. Using the fact $\omega_{\lambda}=\omega_{sur}>0$, we obtain $\omega>0$. Hence, by \eqref{h}, we have
$$\mathsf{h}_{\psi,\omega}(x)=\mathsf{h}_{\phi_{sur},\omega}(x)=|\omega(\phi_{sur}^{-1}(x))|^2=\begin{cases}
	(\omega(2))^2&x=1\\
	(\omega(1))^2&x=2\\
	(\omega(x))^2&x\in\mathbb{N}\backslash\{1,2\}
\end{cases},$$which yields, by \eqref{exmlambda} and \eqref{weighteq1}, that$$\omega_{\lambda}(x)=\begin{cases}
	((1-\lambda)\frac{\omega(2)}{\omega(1)}+\lambda)\omega(1)&x=1\\
	((1-\lambda)\frac{\omega(1)}{\omega(2)}+\lambda)\omega(2)&x=2\\
	\omega(x)&x\in\mathbb{N}\backslash\{1,2\}
\end{cases}=\begin{cases}
	1-\lambda&x=1\\
	\lambda&x=2\\
	1&x\in\mathbb{N}\backslash\{1,2\}
\end{cases}.$$
A simple calculation shows that there do not exist positive real numbers $\omega(1)$ and $\omega(2)$ satisfying the above system of equations. This contradicts \eqref{weighteq1}. Therefore, the claim follows.

Case 2: $\lambda=\frac{1}{2}$. Set $\widehat{\phi_{sur}}$ as in \eqref{phiequ}. The function $\widehat{\omega_{sur}}:X\to\mathbb{R}$ is defined as
$\widehat{\omega_{sur}}(x)=\begin{cases}
	2&x=1\\
	1&x\in\mathbb{N}\backslash\{1\}
\end{cases}.$ Using a proof similar to that of Claim 3, one can verify that there dose not exist a densely defined weighted composition operator $C_{\psi,\omega}$ such that $\mathcal{M}_{\frac{1}{2}}(C_{\psi,\omega})=C_{\widehat{\phi_{sur}},\widehat{\omega_{sur}}}$, that is, the map $\mathcal{M}_{\frac{1}{2}}$ is not surjective.

(iii) By Remark \ref{singlemult}, it suffices to prove Claim 4. Define the $\mathscr{A}$-measurable transform $\phi:X\to X$ by
$\phi(x)=\begin{cases}
	2&x=1\\
	1&x=2\\
	x&x\in\mathbb{N}\backslash\{1,2\}
\end{cases}.$ The functions $\omega_{n},\omega:X\to\mathbb{R}$ are defined as $\omega_{n}(x)=\begin{cases}
	1&x=1\\
	\frac{1}{n}&x=2\\
	0&x\in\mathbb{N}\backslash\{1,2\}
\end{cases}$ for $n\in\mathbb{N}$ and 
$\omega(x)=\begin{cases}
	1&x=1\\
	0&x\in\mathbb{N}\backslash\{1\}
\end{cases},$ respectively.

\textbf{Claim 4:} $C_{\phi,\omega_{n}}$ and $C_{\phi,\omega}$ are bounded operators such that $\lim\limits_{n\to\infty}C_{\phi,\omega_{n}}=C_{\phi,\omega}$ in the norm topology, but $\mathcal{M}_{\lambda}(C_{\phi,\omega_{n}})\not\to\mathcal{M}_{\lambda}(C_{\phi,\omega})$ in norm for $\lambda\in (0,1)$.

Indeed, by \eqref{h} we have
$\mathsf{h}_{\phi,\omega_{n}}=\begin{cases}
	\frac{1}{n^2}&x=1\\
	1&x=2\\
	0&x\in\mathbb{N}\backslash\{1,2\}
\end{cases}$ and 
\begin{equation}\label{couduggal}
\mathsf{h}_{\phi,\omega}=\begin{cases}
	1&x=2\\
	0&x\in\mathbb{N}\backslash\{2\}
\end{cases},
\end{equation}
which yields, by Theorem \ref{combasic} (iv), that $C_{\phi,\omega_{n}}$ and $C_{\phi,\omega}$ are bounded operators. Again from \eqref{h} and Theorem \ref{combasic} (iv), it follows that $\mathsf{h}_{\phi,\omega_{n}-\omega}(x)=\begin{cases}
	\frac{1}{n^2}&x=1\\
	0&x\in\mathbb{N}\backslash\{1\}
\end{cases}$
and $\|C_{\phi,\omega_{n}}-C_{\phi,\omega}\|=\|C_{\phi,\omega_{n}-\omega}\|=\|\mathsf{h}_{\phi,\omega_{n}-\omega}\|_{\infty}^{\frac{1}{2}}=\frac{1}{n}\to 0$ as $n\to\infty$. That is, $\lim\limits_{n\to\infty}C_{\phi,\omega_{n}}=C_{\phi,\omega}$ in the norm topology. Let $$\omega_{\lambda}^n\triangleq(((1-\lambda)\chi_{\{\omega_{n}\ne0\}}(\frac{\mathsf{h}_{\phi,\omega_{n}}}{\mathsf{h}_{\phi,\omega_{n}}\circ\phi})^{\frac{1}{2}}+\lambda)\cdot\omega_{n})(x)=\begin{cases}
	(1-\lambda)\frac{1}{n}+\lambda&x=1\\
	(1-\lambda)+\lambda\cdot\frac{1}{n}&x=2\\
	0&x\in\mathbb{N}\backslash\{1,2\}
\end{cases}$$
and
\begin{equation}\label{meanduggal}
\omega_{\lambda}\triangleq(((1-\lambda)\chi_{\{\omega\ne0\}}(\frac{\mathsf{h}_{\phi,\omega}}{\mathsf{h}_{\phi,\omega}\circ\phi})^{\frac{1}{2}}+\lambda)\cdot\omega)(x)=\begin{cases}
	\lambda&x=1\\
	0&x\in\mathbb{N}\backslash\{1\}
\end{cases}.
\end{equation}
Using \eqref{h}, we have $$\mathsf{h}_{\phi,\omega_{\lambda}^n-\omega_{\lambda}}=\begin{cases}
	((1-\lambda)+\lambda\cdot\frac{1}{n})^2&x=1\\
	((1-\lambda)\frac{1}{n})^2&x=2\\
	0&x\in\mathbb{N}\backslash\{1,2\}
\end{cases},$$ which implies, by by Corollary \ref{singlemeanbasic} (iv), that
\begin{align*}
	\|\mathcal{M}_{\lambda}(C_{\phi,\omega_{n}})-\mathcal{M}_{\lambda}(C_{\phi,\omega})\|&=\|C_{\phi,\omega_{\lambda}^n}-C_{\phi,\omega_{\lambda}}\|=\|C_{\phi,\omega_{\lambda}^n-\omega_{\lambda}}\|\\&=\|\mathsf{h}_{\phi,\omega_{\lambda}^n-\omega_{\lambda}}\|_{\infty}^{\frac{1}{2}}\to1-\lambda\ne0
\end{align*}
as $n\to\infty$. Therefore, we get the desired result.
\end{proof}
\end{exa}
\begin{rem}
Recently, in \cite[Section 4]{sjj-mn-2023}, Djordjević et al. proved that for $t\in(0,1]$, the generalized Aluthge map $T\mapsto|T|^t U|T|^{1-t}$ is norm-continuous on $\mathcal{B}(\mathcal{H})$, where $U$ is the partial isometry in the polar decomposition of $T$. In fact, this result holds for $t\in(0,1)$ (see also \cite[Theorem 2.5]{zhou-glma-2023}), but fails for $t=1$. Indeed, the interested reader may verify that Claim 4 provides a counterexample.
\end{rem}

Inspired by \cite[Lemma 4.13]{zhou-glma-2023}, we investigate the unitary equivalence of the $\lambda$-spherical mean transform of $\textbf{C}_{\phi,\omega}$. Let $U\in\mathcal{B}(L^2(\mu))$ be a unitary operator. For $\textbf{C}_{\phi,\omega}=(C_{\phi_{1},\omega_{1}},C_{\phi_{2},\omega_{2}})\in\mathcal{WCOP}$, the pair $U\textbf{C}_{\phi,\omega}U^\ast$ is given by
 $$U\textbf{C}_{\phi,\omega}U^\ast=(UC_{\phi_{1},\omega_{1}}U^\ast,UC_{\phi_{2},\omega_{2}}U^\ast).$$
Since
\begin{equation}\label{unidomain}
	\EuScript{D}(U\textbf{C}_{\phi,\omega}U^\ast)=\bigcap\limits_{i=1}^2(U\EuScript{D}(C_{\phi_{i},\omega_{i}}))=U(\bigcap\limits_{i=1}^2\EuScript{D}(C_{\phi_{i},\omega_{i}})),
\end{equation}it follows that $U\textbf{C}_{\phi,\omega}U^\ast$ is densely defined if and only if so is $(C_{\phi_{1},\omega_{1}},C_{\phi_{2},\omega_{2}})$. Moreover, by a similar argument as in Proposition \ref{basic} (iii), we also conclude that $U\textbf{C}_{\phi,\omega}U^\ast$ is closed. Therefore, if $\textbf{C}_{\phi,\omega}\in\mathcal{WCOP}_{\text{dense}}\subset\mathcal{WCOP}$, then the densely defined closed operator pair $U\textbf{C}_{\phi,\omega}U^\ast$ admits polar decompositions. 

Suppose $\textbf{C}_{\phi,\omega}=\left(
\begin{array}{c}
	U_{1} \\
	U_{2}
\end{array}
\right)|\textbf{C}_{\phi,\omega}|$ is the polar decomposition of the pair $\textbf{C}_{\phi,\omega}\in\mathcal{WCOP}_{\text{dense}}$. We  claim that 
$U\textbf{C}_{\phi,\omega}U^\ast=\left(
\begin{array}{c}
	UU_{1}U^\ast \\
	UU_{2}U^\ast
\end{array}
\right)(U|\textbf{C}_{\phi,\omega}|U^\ast)$ is the polar decomposition of $U\textbf{C}_{\phi,\omega}U^\ast$. Since $U_{1}^\ast U_{1}+U_{2}^\ast U_{2}$ is a partial isometry, it is easy to verify that $$\left(
\begin{array}{c}
	UU_{1}U^\ast \\
	UU_{2}U^\ast
\end{array}
\right)^\ast\left(
\begin{array}{c}
	UU_{1}U^\ast \\
	UU_{2}U^\ast
\end{array}
\right)=UU_{1}^\ast U^\ast UU_{1}U^\ast+UU_{2}^\ast U^\ast UU_{2}U^\ast=U(U_{1}^\ast U_{1}+U_{2}^\ast U_{2})U^\ast$$ is also a partial isometry. Note that $$\mathcal{N}\Bigg(\left(
\begin{array}{c}
	UU_{1}U^\ast \\
	UU_{2}U^\ast
\end{array}
\right)\Bigg)=U\mathcal{N}\Bigg(\left(
\begin{array}{c}
	U_{1} \\
	U_{2}
\end{array}
\right)\Bigg)=U\mathcal{N}\Bigg(\left(
\begin{array}{c}
	C_{\phi_{1},\omega_{1}} \\
	C_{\phi_{2},\omega_{2}}
\end{array}
\right)\Bigg)=\mathcal{N}\Bigg(\left(
\begin{array}{c}
	UC_{\phi_{1},\omega_{1}}U^\ast \\
	UC_{\phi_{2},\omega_{2}}U^\ast
\end{array}
\right)\Bigg)$$ and from \eqref{nepolar} we have
\begin{equation}\label{unitaryeqpolar}
	UC_{\phi_{i},\omega_{i}}U^\ast\supseteq UU_{i}|\textbf{C}_{\phi,\omega}|U^\ast=UU_{i}U^\ast U|\textbf{C}_{\phi,\omega}|U^\ast
\end{equation}
for $i=1,2$. Thus, if we can prove that
\begin{equation}\label{unitaryfactor}
	|U\textbf{C}_{\phi,\omega}U^\ast|=U|\textbf{C}_{\phi,\omega}|U^\ast,
\end{equation}then by \eqref{unitaryeqpolar} and the uniqueness of the polar decomposition we get the desired result. Indeed, by \eqref{unidomain}, we have
\begin{align*}
\EuScript{D}(|U\textbf{C}_{\phi,\omega}U^\ast|)&=\EuScript{D}(U\textbf{C}_{\phi,\omega}U^\ast)=U(\bigcap\limits_{i=1}^2\EuScript{D}(C_{\phi_{i},\omega_{i}}))\\&=U\EuScript{D}(\textbf{C}_{\phi,\omega})=U\EuScript{D}(|\textbf{C}_{\phi,\omega}|)=\EuScript{D}(U|\textbf{C}_{\phi,\omega}|U^\ast).
\end{align*}
 Using \eqref{mulcomeq} and Theorem \ref{polar} (i), we obtain
\begin{align*}
\||U\textbf{C}_{\phi,\omega}U^\ast |f\|^2&=\|U\textbf{C}_{\phi,\omega}U^\ast f\|^2=\|UC_{\phi_{1},\omega_{1}}U^\ast f\|^2+\|UC_{\phi_{2},\omega_{2}}U^\ast f\|^2\\&=\|C_{\phi_{1},\omega_{1}}U^\ast f\|^2+\|C_{\phi_{2},\omega_{2}}U^\ast f\|^2\\&=\|\mathsf{M}_{\sqrt{\mathsf{h}_{\phi_{1},\omega_{1}}+\mathsf{h}_{\phi_{2},\omega_{2}}}}U^\ast f\|^2\\&=\||\textbf{C}_{\phi,\omega}|U^\ast f\|^2=	\|U|\textbf{C}_{\phi,\omega}|U^\ast f\|^2
\end{align*}
for $f\in \EuScript{D}(|U\textbf{C}_{\phi,\omega}U^\ast|)=U\EuScript{D}(\textbf{C}_ {\phi,\omega})$. It is well-known that if a densely defined operator $T$ is selfadjoint on $\mathcal{H}$, then so is $UTU^\ast$ for any unitary operator $U\in\mathcal{B}(\mathcal{H})$. Hence, $U|\textbf{C}_{\phi,\omega}|U^\ast$ is selfadjoint. Thus, $|U\textbf{C}_{\phi,\omega}U^\ast|=U|\textbf{C}_{\phi,\omega}|U^\ast$ follows from \cite[Lemma 4]{b-j-j-sW}. Therefore, the claim follows, which implies Definition \ref{smtdfn} can be extended naturally to the operator pair $U\textbf{C}_{\phi,\omega}U^\ast$. Consequently, the operator pair $$(\lambda UU_{1}U^\ast|U\textbf{C}_{\phi,\omega}U^\ast|+(1-\lambda)|U\textbf{C}_{\phi,\omega}U^\ast|UU_{1}U^\ast,\lambda UU_{2}U^\ast|U\textbf{C}_{\phi,\omega}U^\ast|+(1-\lambda)|U\textbf{C}_{\phi,\omega}U^\ast|UU_{2}U^\ast)$$
may be denoted as $\mathcal{M}_{\lambda}(U\textbf{C}_{\phi,\omega}U^\ast)$. Since\begin{align*}
U(\lambda U_{i}|\textbf{C}_{\phi,\omega}|+(1-\lambda)|\textbf{C}_{\phi,\omega}|U_{i})U^\ast&=\lambda UU_{i}|\textbf{C}_{\phi,\omega}|U^\ast+(1-\lambda)U|\textbf{C}_{\phi,\omega}|U_{i}U^\ast\\&=\lambda UU_{i}U^\ast U|\textbf{C}_{\phi,\omega}|U^\ast+(1-\lambda)U|\textbf{C}_{\phi,\omega}|U^\ast UU_{i}U^\ast\\&\stackrel{\eqref{unitaryfactor}}{=}\lambda UU_{i}U^\ast|U\textbf{C}_{\phi,\omega}U^\ast|+(1-\lambda)|U\textbf{C}_{\phi,\omega}U^\ast|UU_{i}U^\ast
\end{align*}
for $i=1,2$, we conclude that $\mathcal{M}_{\lambda}(U\textbf{C}_{\phi,\omega}U^\ast)=U\mathcal{M}_{\lambda}(\textbf{C}_{\phi,\omega})U^\ast$ for $\lambda\in(0,1)$.

Summarizing, we have the following.
 \begin{pro}
 	Suppose $C_{\phi_{1},\omega_{1}}$ and $C_{\phi_{2},\omega_{2}}$ are densely defined. Fix $\lambda\in(0,1)$. Let $U\in\mathcal{B}(L^2(\mu))$ be a unitary operator and $k\in\mathbb{C}$. Then 
 	\begin{enumerate}
 		\item[(i)]
 		$\mathcal{M}_{\lambda}(U\textbf{C}_{\phi,\omega}U^\ast)=U\mathcal{M}_{\lambda}(\textbf{C}_{\phi,\omega})U^\ast$,
 		\item[(ii)]
 	$\mathcal{N}(\mathcal{M}_{\lambda}(\textbf{C}_{\phi,\omega}))=\mathcal{N}(\textbf{C}_{\phi,\omega})$. In particular, $\textbf{C}_{\phi,\omega}=\textbf{0}$ if and only if $\mathcal{M}_{\lambda}(\textbf{C}_{\phi,\omega})=\textbf{0}$, where $\textbf{0}$ is the operator pair $(0,0)$.
 	\end{enumerate}
 \end{pro}
 \begin{proof}
 	(ii) Using Theorem \ref{meanbasic} (iv) and Theorem \ref{combasic} (vi), we have $$\mathcal{N}(\mathcal{M}_{\lambda}(\textbf{C}_{\phi,\omega}))=\bigcap\limits_{i=1}^2\mathcal{N}(C_{\phi_{i},\omega_{\lambda}^i})=\bigcap\limits_{i=1}^2\chi_{\{\mathsf{h}_{\phi_{i},\omega_{\lambda}^i=0}\}}L^2(\mu)=\chi_{\bigcap\limits_{i=1}^2\{\mathsf{h}_{\phi_{i},\omega_{\lambda}^i}=0\}}L^2(\mu)$$
 	and $$\mathcal{N}(\textbf{C}_{\phi,\omega})=\bigcap\limits_{i=1}^2\mathcal{N}(C_{\phi_{i},\omega_{i}})=\bigcap\limits_{i=1}^2\chi_{\{\mathsf{h}_{\phi_{i},\omega_{i}=0}\}}L^2(\mu)=\chi_{\bigcap\limits_{i=1}^2\{\mathsf{h}_{\phi_{i},\omega_{i}}=0\}}L^2(\mu).$$ Since
 	 $\bigcap\limits_{i=1}^2\{\mathsf{h}_{\phi_{i},\omega_{\lambda}^i}=0\}=\bigcap\limits_{i=1}^2\{\mathsf{h}_{\phi_{i},\omega_{i}}=0\}$ a.e. $[\mu]$ for $\lambda\in(0,1)$ follows immediately from Lemma \ref{weighted},  then by the above equalities we get $\mathcal{N}(\mathcal{M}_{\lambda}(\textbf{C}_{\phi,\omega}))=\mathcal{N}(\textbf{C}_{\phi,\omega})$. The latter is  a direct corollary.
 \end{proof}
  
 It should be mentioned that the $\lambda$-Aluthge transform of a weighted composition operator is not necessarily closed \cite[Theorem 3.2]{benhida-mn-2020}, whereas its $\lambda$-mean transform is always closed (Corollary \ref{duggal}). For $\lambda=\frac{1}{2}$, Benhida et al. constructed a composition operator whose Aluthge transform is not closed basing on linear transformations of $\mathbb{R}^n$ \cite[Example 3.7]{benhida-mn-2020}. Here, we use Example \ref{nemean} to illustrate the difference in closedness between $\Delta_{\lambda}(C_{\phi,\omega})$ and $\mathcal{M}_{\lambda}(C_{\phi,\omega})$ of $C_{\phi,\omega}$.

\begin{exa}
	Let $X=\mathbb{N}$, $\mathscr{A}=2^X$ and $\mu$ be the counting measure on $X$. Set $C_{\phi,\omega}$ as in Example \ref{nemean}. Given $\lambda\in(0,1)$. It follows readily that $\phi^{-1}(2^X)=2^X$ since $\phi$ is invertible, and thus
	\begin{equation}\label{closed1}
		1+\mathsf{E}_{\phi,\omega}(\mathsf{h}_{\phi,\omega}^\lambda)\circ\phi^{-1}(n)\cdot\mathsf{h}_{\phi,\omega}^{1-\lambda}(n)=1+\mathsf{h}_{\phi,\omega}^{\lambda}\circ\phi^{-1}(n)\cdot\mathsf{h}_{\phi,\omega}^{1-\lambda}(n)=1,\quad n\in\mathbb{N}.
	\end{equation}
	On the other hand, for every $c>0$ there exists $n\in\mathbb{N}$ such that
	\begin{equation}\label{closed2}
		\mathsf{h}_{\phi,\omega}^{1-\lambda}(2n-1)=(2n)^{2(1-\lambda)}>c.
	\end{equation}
	Combining \eqref{closed1} and \eqref{closed2} with Theorem 3.2 in \cite{benhida-mn-2020}, we conclude that $\Delta_{\lambda}(C_{\phi,\omega})$ fails to be closed. However, the following corollary shows that $\mathcal{M}_{\lambda}(C_{\phi,\omega})$ is closed.
\end{exa} 
 
\begin{cor}\label{duggal}
Suppose $C_{\phi,\omega}$ and $C_{\phi_{i},\omega_{i}}$ $(i=1,2)$ are densely defined. Fix $\lambda\in(0,1)$. Then $\mathcal{M}_{\lambda}(C_{\phi_{1},\omega_{1}},C_{\phi_{2},\omega_{2}})$ is closed. In particular, $\mathcal{M}_{\lambda}(C_{\phi,\omega})$ is closed.
\end{cor}
\begin{proof}
Using Theorem \ref{meanbasic} (iv) and Proposition \ref{basic} (iii), we immediately know that $\lambda$-spherical mean transform $\mathcal{M}_{\lambda}(C_{\phi_{1},\omega_{1}},C_{\phi_{2},\omega_{2}})$ is closed. The last statement holds by setting $C_{\phi_{1},\omega_{1}}=C_{\phi,\omega}$, $C_{\phi_{2},\omega_{2}}=0$ and Remark \ref{singlemult}.
\end{proof}
\begin{rem}
By Corollary \ref{singlemeanbasic} (iv), we have $|C_{\phi,\omega}|U=C_{\phi,\omega_{0}}$. Corollary \ref{duggal} implies the convex combinations of the closed operators $C_{\phi,\omega}$ and $C_{\phi,\omega_{0}}$ remain closed. This is nontrivial since closedness is not generally preserved under convex combinations (for example, the convex combination $\lambda\cdot\frac{A}{\lambda}+(1-\lambda)\cdot\frac{-A}{1-\lambda}$ of $\frac{A}{\lambda}$ and $\frac{-A}{1-\lambda}$ is not closed, where $\lambda\in(0,1)$ and $A$ is a closed operator with a non-closed domain).
\end{rem}

In what follows, we turn to investigate the dense definiteness of $\mathcal{M}_{\lambda}(\textbf{C}_{\phi,\omega})$.

\begin{cor}\label{denseequ}
	Suppose $C_{\phi_{1},\omega_{1}}$ and $C_{\phi_{2},\omega_{2}}$ are densely defined. Fix $\lambda\in(0,1)$. Then the following conditions are equivalent$:$
	\begin{enumerate}
		\item[(i)]
		the pair $\mathcal{M}_{\lambda}(\textbf{C}_{\phi,\omega})=(C_{\phi_{1},\omega_{\lambda}^1},C_{\phi_{2},\omega_{\lambda}^2})$ is densely defined,
		\item[(ii)]
		$C_{\phi_{i},\omega_{\lambda}^{i}}$ is densely defined for $i=1,2$,
		\item[(iii)]
		$\mathsf{E}_{\phi_{i},\omega_{i}}(\mathsf{h}_{\phi_{1},\omega_{1}}+\mathsf{h}_{\phi_{2},\omega_{2}})\circ\phi_{i}^{-1}<\infty$ a.e. $[\mu]$ for $i=1,2$,
		\item[(iv)]
		$\mathsf{E}_{\phi_{i},\omega_{i}}(\mathsf{h}_{\phi_{j},\omega_{j}})\circ\phi_{i}^{-1}<\infty$ a.e. $[\mu]$ for $i,j=1,2$,
		\item[(v)]
		$\mathsf{E}_{\phi_{i},\omega_{i}}(\mathsf{h}_{\phi_{j},\omega_{j}})<\infty$ a.e. $[\mu]$ on $\{\omega_{i}\ne0\}$ for $i,j=1,2$,
		\item[(vi)]
		$\mathsf{E}_{\phi_{i},\omega_{i}}(\mathsf{h}_{\phi_{1},\omega_{1}}+\mathsf{h}_{\phi_{2},\omega_{2}})<\infty$ a.e. $[\mu]$ on $\{\omega_{i}\ne0\}$ for $i=1,2$,
		\item[(vii)]
		$\mathsf{E}_{\phi_{i},\omega_{i}}(\mathsf{h}_{\phi_{j},\omega_{j}})<\infty$ a.e. $[\mu_{\omega_{i}}]$ for $i,j=1,2$,
		\item[(viii)]
	$\mathsf{E}_{\phi_{i},\omega_{i}}(\mathsf{h}_{\phi_{1},\omega_{1}}+\mathsf{h}_{\phi_{2},\omega_{2}})<\infty$ a.e. $[\mu_{\omega_{i}}]$ for $i=1,2$,
		\item[(ix)]
			$\mu_{\sqrt{\mathsf{h}_{\phi_{1},\omega_{1}}+\mathsf{h}_{\phi_{2},\omega_{2}}}\cdot\omega_{i}}\circ\phi_{i}^{-1}$ is $\sigma$-finite for $i=1,2$,
		\item[(x)]
		$\mu_{\sqrt{\mathsf{h}_{\phi_{1},\omega_{1}}+\mathsf{h}_{\phi_{2},\omega_{2}}}\cdot\omega_{i}}|_{\phi_{i}^{-1}(\mathscr{A})}$ is $\sigma$-finite for $i=1,2$,
		\item[(xi)]
		$\mu_{\sqrt{\mathsf{h}_{\phi_{j},\omega_{j}}}\cdot\omega_{i}}\circ\phi_{i}^{-1}$ is $\sigma$-finite for $i,j=1,2$,
		\item[(xii)]
		$\mu_{\sqrt{\mathsf{h}_{\phi_{j},\omega_{j}}}\cdot\omega_{i}}|_{\phi_{i}^{-1}(\mathscr{A})}$ is $\sigma$-finite for $i,j=1,2$.
	\end{enumerate}
\end{cor}
\begin{proof}
(i) $\Rightarrow$ (ii): Theorem \ref{meanbasic} (iv).

(ii) $\Rightarrow$ (i): Proposition \ref{basic} (ii).

(i) $\Rightarrow$ (iii): Suppose (iii) fails, then
\begin{equation}\label{dense3}
\mu\Big(\Big\{\mathsf{E}_{\phi_{i},\omega_{i}}(\mathsf{h}_{\phi_{1},\omega_{1}}+\mathsf{h}_{\phi_{2},\omega_{2}})\circ\phi_{i}^{-1}=\infty\Big\}\Big)\ne0
\end{equation} 
for $i=1,2$. It follows from the proof of Lemma \ref{weighted} (iii) that
\begin{equation}\label{dense4}
	\mathsf{E}_{\phi_{i},\omega_{i}}(\mathsf{h}_{\phi_{1},\omega_{1}}+\mathsf{h}_{\phi_{2},\omega_{2}})\circ\phi_{i}^{-1}\ge 0\quad \text{a.e.}\ [\mu].
\end{equation}
By \eqref{dense3} and \eqref{dense4}, we have 
\begin{equation}\label{denseeq}
\mu\Big(\Big\{\sum\limits_{i=1}^2\mathsf{E}_{\phi_{i},\omega_{i}}(\mathsf{h}_{\phi_{1},\omega_{1}}+\mathsf{h}_{\phi_{2},\omega_{2}})\circ\phi_{i}^{-1}=\infty\Big\}\Big)\ne0.
\end{equation} Since $\mu$ is $\sigma$-finite, then there exists an increasing sequence $\{X_{n}\}\subset\mathscr{A}$ such that $\mu(X_{n})<\infty$ and $\bigcup\limits_{n=1}^\infty X_{n}=X$. We claim that there exists some $n_{0}\in\mathbb{N}$ such that
\begin{equation}\label{denseeq2}
	\mu(\Big\{\sum\limits_{i=1}^2\mathsf{E}_{\phi_{i},\omega_{i}}(\mathsf{h}_{\phi_{1},\omega_{1}}+\mathsf{h}_{\phi_{2},\omega_{2}})\circ\phi_{i}^{-1}=\infty\Big\}\bigcap X_{n_{0}})\ne0.
\end{equation}
 Otherwise, we have
\begin{align*}
&\mu\Big(\Big\{\sum\limits_{i=1}^2\mathsf{E}_{\phi_{i},\omega_{i}}(\mathsf{h}_{\phi_{1},\omega_{1}}+\mathsf{h}_{\phi_{2},\omega_{2}})\circ\phi_{i}^{-1}=\infty\Big\}\Big)\\&=\mu\Big(\Big\{\sum\limits_{i=1}^2\mathsf{E}_{\phi_{i},\omega_{i}}(\mathsf{h}_{\phi_{1},\omega_{1}}+\mathsf{h}_{\phi_{2},\omega_{2}})\circ\phi_{i}^{-1}=\infty\Big\}\bigcap\Big(\bigcup\limits_{n=1}^\infty X_{n}\Big)\Big)\\&\le\sum\limits_{n=1}^\infty\mu\Big(\Big\{\sum\limits_{i=1}^2\mathsf{E}_{\phi_{i},\omega_{i}}(\mathsf{h}_{\phi_{1},\omega_{1}}+\mathsf{h}_{\phi_{2},\omega_{2}})\circ\phi_{i}^{-1}=\infty\Big\}\bigcap X_{n}\Big)=0.
\end{align*}
This contradicts \eqref{denseeq}. Thus, by Theorem \ref{meanbasic} (iii) and \eqref{denseeq2}, we have $$\chi_{\Big\{\sum\limits_{i=1}^2\mathsf{E}_{\phi_{i},\omega_{i}}(\mathsf{h}_{\phi_{1},\omega_{1}}+\mathsf{h}_{\phi_{2},\omega_{2}})\circ\phi_{i}^{-1}=\infty\Big\}}\cdot\chi_{X_{n_{0}}}\in\EuScript{D}(\mathcal{M}_{\lambda}(\textbf{C}_{\phi,\omega}))^{\perp}\ne\{0\}.$$
It follows  from equation $L^2(\mu)=\overline{\EuScript{D}(\mathcal{M}_{\lambda}(\textbf{C}_{\phi,\omega}))}\oplus\EuScript{D}(\mathcal{M}_{\lambda}(\textbf{C}_{\phi,\omega}))^\perp$ that $\mathcal{M}_{\lambda}(\textbf{C}_{\phi,\omega})$ is not densely defined. Contradiction.

(iii) $\Rightarrow$ (i): Theorem \ref{meanbasic} (iii).

(iii) $\Leftrightarrow$ (iv) follows from the equality
$$\mathsf{E}_{\phi_{j},\omega_{j}}(\mathsf{h}_{\phi_{1},\omega_{1}}+\mathsf{h}_{\phi_{2},\omega_{2}})\circ\phi_{j}^{-1}=\mathsf{E}_{\phi_{j},\omega_{j}}(\mathsf{h}_{\phi_{1},\omega_{1}})\circ\phi_{j}^{-1}+\mathsf{E}_{\phi_{j},\omega_{j}}(\mathsf{h}_{\phi_{2},\omega_{2}})\circ\phi_{j}^{-1}\quad\text{a.e.}\ [\mu]$$ for $j=1,2$. Indeed, we have
\begin{align*}
	&(\mathsf{E}_{\phi_{j},\omega_{j}}(\mathsf{h}_{\phi_{1},\omega_{1}}+\mathsf{h}_{\phi_{2},\omega_{2}})\circ\phi_{j}^{-1})\circ\phi_{j}\\&\stackrel{\eqref{invexp}}{=}\mathsf{E}_{\phi_{j},\omega_{j}}(\mathsf{h}_{\phi_{1},\omega_{1}}+\mathsf{h}_{\phi_{2},\omega_{2}})\quad\text{a.e.}\ [\mu_{\omega_{j}}]\\&=\mathsf{E}_{\phi_{j},\omega_{j}}(\mathsf{h}_{\phi_{1},\omega_{1}})+\mathsf{E}_{\phi_{j},\omega_{j}}(\mathsf{h}_{\phi_{2},\omega_{2}})\quad\text{a.e.}\ [\mu_{\omega_{j}}]\\&\stackrel{\eqref{invexp}}{=}(\mathsf{E}_{\phi_{j},\omega_{j}}(\mathsf{h}_{\phi_{1},\omega_{1}})\circ\phi_{j}^{-1})\circ\phi_{j}+(\mathsf{E}_{\phi_{j},\omega_{j}}(\mathsf{h}_{\phi_{2},\omega_{2}})\circ\phi_{j}^{-1})\circ\phi_{j}\quad\text{a.e.}\ [\mu_{\omega_{j}}]\\&=(\mathsf{E}_{\phi_{j},\omega_{j}}(\mathsf{h}_{\phi_{1},\omega_{1}})\circ\phi_{j}^{-1}+\mathsf{E}_{\phi_{j},\omega_{j}}(\mathsf{h}_{\phi_{2},\omega_{2}})\circ\phi_{j}^{-1})\circ\phi_{j}\quad\text{a.e.}\ [\mu_{\omega_{j}}]
\end{align*}
where the second equality follows from the linearity of the conditional expectation $\mathsf{E}_{\phi_{j},\omega_{j}}(\cdot)$. Then by \eqref{Epro}, we obtain
$$\mathsf{E}_{\phi_{j},\omega_{j}}(\mathsf{h}_{\phi_{1},\omega_{1}}+\mathsf{h}_{\phi_{2},\omega_{2}})\circ\phi_{j}^{-1}=\mathsf{E}_{\phi_{j},\omega_{j}}(\mathsf{h}_{\phi_{1},\omega_{1}})\circ\phi_{j}^{-1}+\mathsf{E}_{\phi_{j},\omega_{j}}(\mathsf{h}_{\phi_{2},\omega_{2}})\circ\phi_{j}^{-1}=0\quad\text{a.e.}\ [\mu]$$
on $\{\mathsf{h}_{\phi_{j},\omega_{j}}=0\}$. Therefore, by the uniqueness part of \cite[Proposition 14]{b-j-j-sW}, we can get the desired result.

(iv) $\Rightarrow$ (v): Since
\begin{equation}\label{recon}
	\begin{aligned}
	\int_{X}\chi_{\{\mathsf{E}_{\phi_{i},\omega_{i}}(\mathsf{h}_{\phi_{j},\omega_{j}})\circ\phi_{i}^{-1}=\infty\}}&\cdot\mathsf{h}_{\phi_{i},\omega_{i}}\D\mu\\&\stackrel{\eqref{intcom}}{=}\int_{X}\chi_{\{\mathsf{E}_{\phi_{i},\omega_{i}}(\mathsf{h}_{\phi_{j},\omega_{j}})\circ\phi_{i}^{-1}=\infty\}}\circ\phi_{i}\cdot(\chi_{\{\omega_{i}\ne0\}}+\chi_{\{\omega_{i}=0\}})\D\mu_{\omega_{i}}\\&=\int_{X}\chi_{\{\mathsf{E}_{\phi_{i},\omega_{i}}(\mathsf{h}_{\phi_{j},\omega_{j}})\circ\phi_{i}^{-1}=\infty\}}\circ\phi_{i}\cdot\chi_{\{\omega_{i}\ne0\}}\D\mu_{\omega_{i}}\\&=\int_{\{(\mathsf{E}_{\phi_{i},\omega_{i}}(\mathsf{h}_{\phi_{j},\omega_{j}})\circ\phi_{i}^{-1})\circ\phi_{i}=\infty\}}\chi_{\{\omega_{i}\ne0\}}\D\mu_{\omega_{i}}\\&\stackrel{\eqref{invexp}}{=}\int_{\{\mathsf{E}_{\phi_{i},\omega_{i}}(\mathsf{h}_{\phi_{j},\omega_{j}})=\infty\}}\chi_{\{\omega_{i}\ne0\}}\D\mu_{\omega_{i}}\\&\stackrel{\eqref{muomega}}{=}\int_{\{\mathsf{E}_{\phi_{i},\omega_{i}}(\mathsf{h}_{\phi_{j},\omega_{j}})=\infty\}}\chi_{\{\omega_{i}\ne0\}}\cdot|\omega_{i}|^2\D\mu\\&=\int_{X}\chi_{\{\mathsf{E}_{\phi_{i},\omega_{i}}(\mathsf{h}_{\phi_{j},\omega_{j}})=\infty\}}\cdot\chi_{\{\omega_{i}\ne0\}}\cdot|\omega_{i}|^2\D\mu
	\end{aligned}
\end{equation}
for $i,j=1,2$, then by (iv) we have $\chi_{\{\mathsf{E}_{\phi_{i},\omega_{i}}(\mathsf{h}_{\phi_{j},\omega_{j}})=\infty\}}\cdot\chi_{\{\omega_{i}\ne0\}}\cdot|\omega_{i}|^2=0$ a.e. $[\mu]$, which implies $$\chi_{\{\mathsf{E}_{\phi_{i},\omega_{i}}(\mathsf{h}_{\phi_{j},\omega_{j}})=\infty\}}\cdot\chi_{\{\omega_{i}\}\ne0}=0\quad\text{a.e.}\ [\mu],$$
that is, $\mu(\{\mathsf{E}_{\phi_{i},\omega_{i}}(\mathsf{h}_{\phi_{j},\omega_{j}})=\infty\}\bigcap\{\omega_{i}\ne0\})=0$. Hence, (v) holds.

(v) $\Rightarrow$ (iv): By \eqref{recon} and (v), we have $$\chi_{\{\mathsf{E}_{\phi_{i},\omega_{i}}(\mathsf{h}_{\phi_{j},\omega_{j}})\circ\phi_{i}^{-1}=\infty\}}\cdot\mathsf{h}_{\phi_{i},\omega_{i}}=0\quad\text{a.e.}\ [\mu]$$
for $i,j=1,2$, then  
$$\mu(\{\mathsf{E}_{\phi_{i},\omega_{i}}(\mathsf{h}_{\phi_{j},\omega_{j}})\circ\phi_{i}^{-1}=\infty\}\bigcap\{\mathsf{h}_{\phi_{i},\omega_{i}}\ne0\})=0$$
follows. Using the above equality and \eqref{Epro}, we obtain (iv).

(v) $\Leftrightarrow$ (vi) follows from the linearity of the condition expectation $\mathsf{E}_{\phi_{i},\omega_{i}}(\cdot)$.

(v) $\Leftrightarrow$ (vii) follows from the equality$$\mu_{\omega_{i}}(\{\mathsf{E}_{\phi_{i},\omega_{i}}(\mathsf{h}_{\phi_{j},\omega_{j}})=\infty\})=\int_{\{\mathsf{E}_{\phi_{i},\omega_{i}}(\mathsf{h}_{\phi_{j},\omega_{j}})=\infty\}}|\omega_{i}|^2\D\mu=\int_{\{\mathsf{E}_{\phi_{i},\omega_{i}}(\mathsf{h}_{\phi_{j},\omega_{j}})=\infty\}\bigcap\{\omega_{i}\ne0\}}|\omega_{i}|^2\D\mu$$ for $i,j=1,2$.

(vi) $\Leftrightarrow$ (viii) follows from the equality
\begin{align*}
	\mu_{\omega_{i}}(\{\mathsf{E}_{\phi_{i},\omega_{i}}(\mathsf{h}_{\phi_{1},\omega_{1}}+\mathsf{h}_{\phi_{2},\omega_{2}})=\infty\})&=\int_{\{\mathsf{E}_{\phi_{i},\omega_{i}}(\mathsf{h}_{\phi_{1},\omega_{1}}+\mathsf{h}_{\phi_{2},\omega_{2}})=\infty\}}|\omega_{i}|^2\D\mu\\&=\int_{\{\mathsf{E}_{\phi_{i},\omega_{i}}(\mathsf{h}_{\phi_{1},\omega_{1}}+\mathsf{h}_{\phi_{2},\omega_{2}})=\infty\}\bigcap\{\omega_{i}\ne0\}}|\omega_{i}|^2\D\mu
\end{align*}
for $i=1,2$.

(iii) $\Rightarrow$ (ix): By (iii) and Theorem \ref{combasic} (iii), we have $$\mathsf{E}_{\phi_{i},\omega_{i}}(\mathsf{h}_{\phi_{1},\omega_{1}}+\mathsf{h}_{\phi_{2},\omega_{2}})\circ\phi_{i}^{-1}\cdot\mathsf{h}_{\phi_{i},\omega_{i}}<\infty\quad\text{a.e.}\ [\mu].$$ By the $\sigma$-finiteness of $\mu$ and the above equality, it is not difficult to verify that there exists an increasing sequence $\{X_{n}\}\subset\mathscr{A}$ such that $\mu(X_{n})<\infty$, $X=\bigcup\limits_{n=1}^\infty X_{n}$ and $\mathsf{E}_{\phi_{i},\omega_{i}}(\mathsf{h}_{\phi_{1},\omega_{1}}+\mathsf{h}_{\phi_{2},\omega_{2}})\circ\phi_{i}^{-1}\cdot\mathsf{h}_{\phi_{i},\omega_{i}}\le n$ a.e. $[\mu]$ on $X_{n}$ for every $n\in\mathbb{N}$. Thus, we have
\begin{equation}\label{finite2}
	\begin{aligned}
		\mu_{\sqrt{\mathsf{h}_{\phi_{1},\omega_{1}}+\mathsf{h}_{\phi_{2},\omega_{2}}}\cdot\omega_{i}}\circ\phi_{i}^{-1}(X_{n})&\stackrel{\eqref{rd}}{=}\int_{\phi_{i}^{-1}(X_{n})}\Big|\sqrt{\mathsf{h}_{\phi_{1},\omega_{1}}+\mathsf{h}_{\phi_{2},\omega_{2}}}\cdot\omega_{i}\Big|^2\D\mu\\&\stackrel{\eqref{muomega}}{=}\int_{\phi_{i}^{-1}(X_{n})}(\mathsf{h}_{\phi_{1},\omega_{1}}+\mathsf{h}_{\phi_{2},\omega_{2}})\D\mu_{\omega_{i}}\\&\stackrel{\eqref{exp}}{=}\int_{\phi_{i}^{-1}(X_{n})}\mathsf{E}_{\phi_{i},\omega_{i}}(\mathsf{h}_{\phi_{1},\omega_{1}}+\mathsf{h}_{\phi_{2},\omega_{2}})\D\mu_{\omega_{i}}\\&\stackrel{\eqref{invexp}}{=}\int_{\phi_{i}^{-1}(X_{n})}(\mathsf{E}_{\phi_{i},\omega_{i}}(\mathsf{h}_{\phi_{1},\omega_{1}}+\mathsf{h}_{\phi_{2},\omega_{2}})\circ\phi_{i}^{-1})\circ\phi_{i}\D\mu_{\omega_{i}}\\&\stackrel{\eqref{intcom}}{=}\int_{X_{n}}\mathsf{E}_{\phi_{i},\omega_{i}}(\mathsf{h}_{\phi_{1},\omega_{1}}+\mathsf{h}_{\phi_{2},\omega_{2}})\circ\phi_{i}^{-1}\mathsf{h}_{\phi_{i},\omega_{i}}\D\mu\le n\mu(X_{n})<\infty.
	\end{aligned}
\end{equation}
This yields (ix).

(ix) $\Rightarrow$ (x) is clear.

(x) $\Rightarrow$ (iii): 
For any two sets $A$ and $B$, we write $A-B$ for the set $\{x\colon x\in A,x\notin B\}$. By (viii), we know that there exists a sequence $\{X_{n}\}\subset\mathscr{A}$ be such that
$\phi_{i}^{-1}(X_{n})\subset\phi_{i}^{-1}(X_{n+1})$, $\mu_{\sqrt{\mathsf{h}_{\phi_{1},\omega_{1}}+\mathsf{h}_{\phi_{2},\omega_{2}}}\cdot\omega_{i}}(\phi_{i}^{-1}(X_{n}))<\infty$ for every $n\in\mathbb{N}$ and $X=\bigcup\limits_{n=1}^\infty\phi_{i}^{-1}(X_{n})$. Through set operations, we obtain \begin{align*}
&\mu_{\sqrt{\mathsf{h}_{\phi_{1},\omega_{1}}+\mathsf{h}_{\phi_{2},\omega_{2}}}\cdot\omega_{i}}\Big(\phi_{i}^{-1}\Big(X-\bigcup\limits_{n=1}^\infty X_{n}\Big)\Big)\\&=\mu_{\sqrt{\mathsf{h}_{\phi_{1},\omega_{1}}+\mathsf{h}_{\phi_{2},\omega_{2}}}\cdot\omega_{i}}\Big(\phi_{i}^{-1}(X)-\phi_{i}^{-1}\Big(\bigcup\limits_{n=1}^{\infty}X_{n}\Big)\Big)\\&=\mu_{\sqrt{\mathsf{h}_{\phi_{1},\omega_{1}}+\mathsf{h}_{\phi_{2},\omega_{2}}}\cdot\omega_{i}}(\phi_{i}^{-1}(X))-\mu_{\sqrt{\mathsf{h}_{\phi_{1},\omega_{1}}+\mathsf{h}_{\phi_{2},\omega_{2}}}\cdot\omega_{i}}\Big(\phi_{i}^{-1}(X)\bigcap\Big(\phi_{i}^{-1}\Big(\bigcup\limits_{n=1}^{\infty}X_{n}\Big)\Big)\Big)
\end{align*}
\begin{align*}
	&=\mu_{\sqrt{\mathsf{h}_{\phi_{1},\omega_{1}}+\mathsf{h}_{\phi_{2},\omega_{2}}}\cdot\omega_{i}}(X)-\mu_{\sqrt{\mathsf{h}_{\phi_{1},\omega_{1}}+\mathsf{h}_{\phi_{2},\omega_{2}}}\cdot\omega_{i}}\Big(\bigcup\limits_{n=1}^\infty\phi_{i}^{-1}\Big(X_{n}\Big)\Big)\\&=\mu_{\sqrt{\mathsf{h}_{\phi_{1},\omega_{1}}+\mathsf{h}_{\phi_{2},\omega_{2}}}\cdot\omega_{i}}(X)-\mu_{\sqrt{\mathsf{h}_{\phi_{1},\omega_{1}}+\mathsf{h}_{\phi_{2},\omega_{2}}}\cdot\omega_{i}}(X)=0.\quad\quad\quad\quad\quad\quad\quad\quad\quad\quad\quad\quad
\end{align*}
It follows that
\begin{align*}
	\mu_{\sqrt{\mathsf{h}_{\phi_{1},\omega_{1}}+\mathsf{h}_{\phi_{2},\omega_{2}}}\cdot\omega_{i}}&(\phi_{i}^{-1}(\widetilde{X_{n}}))=\mu_{\sqrt{\mathsf{h}_{\phi_{1},\omega_{1}}+\mathsf{h}_{\phi_{2},\omega_{2}}}\cdot\omega_{i}}(\phi_{i}^{-1}(X_{n}\bigcup(X-\bigcup\limits_{n=1}^\infty X_{n})))\\&=\mu_{\sqrt{\mathsf{h}_{\phi_{1},\omega_{1}}+\mathsf{h}_{\phi_{2},\omega_{2}}}\cdot\omega_{i}}(\phi_{i}^{-1}(X_{n})\bigcup\phi_{i}^{-1}((X-\bigcup\limits_{n=1}^\infty X_{n})))\\&\le\mu_{\sqrt{\mathsf{h}_{\phi_{1},\omega_{1}}+\mathsf{h}_{\phi_{2},\omega_{2}}}\cdot\omega_{i}}(\phi_{i}^{-1}(X_{n}))+\mu_{\sqrt{\mathsf{h}_{\phi_{1},\omega_{1}}+\mathsf{h}_{\phi_{2},\omega_{2}}}\cdot\omega_{i}}(\phi_{i}^{-1}((X-\bigcup\limits_{n=1}^\infty X_{n})))\\&<\infty
\end{align*}
 where $\widetilde{X_{n}}=X_{n}\bigcup(X-\bigcup\limits_{n=1}^\infty X_{n})$. Then, by the argument used in \eqref{finite2}, we have
\begin{align*}
\int_{\widetilde{X_{n}}}\mathsf{E}_{\phi_{i},\omega_{i}}(\mathsf{h}_{\phi_{1},\omega_{1}}+\mathsf{h}_{\phi_{2},\omega_{2}})\circ\phi_{i}^{-1}\mathsf{h}_{\phi_{i},\omega_{i}}\D\mu&=\mu_{\sqrt{\mathsf{h}_{\phi_{1},\omega_{1}}+\mathsf{h}_{\phi_{2},\omega_{2}}}\cdot\omega_{i}}\circ\phi_{i}^{-1}(\widetilde{X_{n}})\\&=\mu_{\sqrt{\mathsf{h}_{\phi_{1},\omega_{1}}+\mathsf{h}_{\phi_{2},\omega_{2}}}\cdot\omega_{i}}(\phi_{i}^{-1}(\widetilde{X_{n}}))<\infty,
\end{align*}
which implies $$\mu(\{\mathsf{E}_{\phi_{i},\omega_{i}}(\mathsf{h}_{\phi_{1},\omega_{1}}+\mathsf{h}_{\phi_{2},\omega_{2}})\circ\phi_{i}^{-1}=\infty\}\bigcap\{\mathsf{h}_{\phi_{i},\omega_{i}}\ne0\}\bigcap \widetilde{X_{n}})=0$$
for every $n\in\mathbb{N}$. Thus,
\begin{equation}
\mu(\{\mathsf{E}_{\phi_{i},\omega_{i}}(\mathsf{h}_{\phi_{1},\omega_{1}}+\mathsf{h}_{\phi_{2},\omega_{2}})\circ\phi_{i}^{-1}=\infty\}\bigcap\{\mathsf{h}_{\phi_{i},\omega_{i}}\ne0\})=0
\end{equation}
follows from the equality $X=\bigcup\limits_{n=1}^\infty\widetilde{X_{n}}$. Therefore, we obtain (iii) by the above equality and \eqref{Epro}.

The equivalence (v) $\Leftrightarrow$ (xi) $\Leftrightarrow$ (xii) follows from (iii) $\Leftrightarrow$ (ix) $\Leftrightarrow$ (x) by setting either $\mathsf{h}_{\phi_{1},\omega_{1}}$ or $\mathsf{h}_{\phi_{2},\omega_{2}}$ to $0$.

The proof is complete.
\end{proof}

Using Corollary \ref{denseequ}, we present a $p$-hyponormal weighted composition operator $C_{\phi,\omega}$ whose $\lambda$-Aluthge transform $\Delta_{\lambda}(C_{\phi,\omega})$is densely defined, while its $\lambda$-mean transform $\mathcal{M}_{\lambda}(C_{\phi,\omega})$ has trival domain.

Let $A$ be a positive selfadjoint operator with the spectral measure $E$. By the Stone-von Neumann operator calculus, we may define 
\begin{equation}\label{selfadjoint}
A^p=\int_{[0,\infty)}x^p\D E
\end{equation}
for $p>0$ (see \cite{birman-book-1987}). Recall that in \cite{benhida-mn-2020}, an unbounded operator $T$ is said to be {\em $p$-hyponormal} ($p>0$) if $\EuScript{D}(|T|^p)\subset\EuScript{D}(|T^\ast|^p)$ and $\||T^\ast|^px\|\le\||T|^px\|$ for any $x\in\EuScript{D}(|T|^p)$. A $1-$hyponormal operator is called a {\em hyponormal} operator. We remark that one of Trepkowski's main results in \cite{trep-jmaa-2015} is the construction of a hyponormal weighted composition operator with trival Aluthge transform on a directed tree.

\begin{exa}\label{trep}
	Let $X=\{(x_{1},x_{2},\cdots,x_{m},0,0,0,\cdots)\colon x_{i}\in\mathbb{N}\ \text{for}\ i\in\{1,2,\cdots,m\},m\in\mathbb{N}\}$, $\mathscr{A}=2^X$ and $\mu$ be the counting measure on $X$ (Note that $X$ is a countable set, which  allows us to apply \eqref{h}, \eqref{Ef} and \eqref{Efphi}). The $\mathscr{A}$-measurable transform $\phi$ is defined by
	\begin{align*}
		\phi((x_{1},x_{2},\cdots,x_{m},0,0,\cdots))=\begin{cases}
			(x_{1},0,0,\cdots)&m=1\\
			(x_{1},x_{2},\cdots,x_{m-1},0,0,\cdots)&m\ge2
		\end{cases}.
	\end{align*}
	Let the $\mathscr{A}$-measurable function $\omega:X\to\mathbb{R}$ be given by $$\omega((x_{1},x_{2},\cdots,x_{m},0,0,0,\cdots))=\begin{cases}
		0&m=1\\\frac{x_{1}+\cdots+x_{m-1}}{x_{m}^\frac{3}{2}}&m\ge2
	\end{cases}.$$ Since $\mu(\varDelta)=0$ if and only if $\varDelta=\emptyset$, then $\mu_{\omega}\circ\varphi^{-1}\ll\mu$. Hence, $C_{\phi,\omega}$ is well-defined. By \eqref{h}, we have
	\begin{align*}
		\mathsf{h}_{\phi,\omega}((x_{1},x_{2},\cdots,x_{m},0,0,\cdots))&=\sum\limits_{n\ge1}|\omega((x_{1},x_{2},\cdots,x_{m},n,0,\cdots))|^2\\&=\sum\limits_{n\ge1}\frac{(x_{1}+\cdots+x_{m})^2}{n^3}\\&=(x_{1}+\cdots+x_{m})^2\sum\limits_{n\ge1}\frac{1}{n^3}<\infty
	\end{align*} for $m\in\mathbb{N}$, which implies $C_{\phi,\omega}$ is densely defined.
	Further, we have
	\begin{equation*}
	\begin{aligned}
		0<&\frac{\mathsf{h}_{\phi,\omega}^p\circ\phi}{\mathsf{h}_{\phi,\omega}^p}((x_{1},x_{2},\cdots,x_{m},0,0,\cdots))\\&=\begin{cases}
			1&m=1\\
			\frac{(x_{1}+\cdots+x_{m-1})^{2p}(\sum\limits_{n\ge1}\frac{1}{n^3})^p}{(x_{1}+\cdots+x_{m})^{2p}(\sum\limits_{n\ge1}\frac{1}{n^3})^p}=(1-\frac{x_{m}}{x_{1}+x_{2}+\cdots+x_{m}})^{2p}&m\ge2
		\end{cases}\\&\le1
	\end{aligned}
	\end{equation*}
	for $p\in(0,\infty)$. Thus, by \cite[Appendix A (A.5)]{b-j-j-sW} we derive $\mathsf{E}_{\phi,\omega}(\frac{\mathsf{h}_{\phi,\omega}^p\circ\phi}{\mathsf{h}_{\phi,\omega}^p})\le1$ a.e. $[\mu_{\omega}]$, which yields, by \cite[Theorem 4.5]{benhida-mn-2020}, that $C_{\phi,\omega}$ is $p$-hyponormal. Now, we prove that $\Delta_{\lambda}(C_{\phi,\omega})$ is densely defined and $\mathcal{M}_{\lambda}(C_{\phi,\omega})$ has trival domain for $\lambda\in(0,1)$. Indeed, by \eqref{Efphi}, we obtain
	\begin{align*}
		&\mathsf{E}_{\phi,\omega}(\mathsf{h}_{\phi,\omega}^\lambda)\circ\phi^{-1}((x_{1},x_{2},\cdots,x_{m},0,0,\cdots))\\&=\frac{\sum\limits_{n\ge1}\Big(\sum\limits_{l
				\ge1}|\omega((x_{1},x_{2},\cdots,x_{m},n,l,0,0,\cdots))|^2\Big)^\lambda|\omega((x_{1},x_{2},\cdots,x_{m},n,0,0,\cdots))|^2}{\mathsf{h}_{\phi,\omega}((x_{1},x_{2},\cdots,x_{m},0,0,\cdots))}\\&=\frac{\sum\limits_{n\ge1}\Bigg(\sum\limits_{l\ge1}\frac{(x_{1}+\cdots+x_{m}+n)^2}{l^3}\Bigg)^\lambda\frac{(x_{1}+\cdots+x_{m})^2}{n^3}}{(x_{1}+\cdots+x_{m})^2\sum\limits_{n\ge1}\frac{1}{n^3}}=\Bigg(\sum\limits_{l\ge1}\frac{1}{l^3}\Bigg)^\lambda\cdot\frac{\sum\limits_{n\ge1}\frac{(x_{1}+\cdots+x_{m}+n)^{2\lambda}}{n^3}}{\sum\limits_{n\ge1}\frac{1}{n^3}}
	\end{align*}
	for $m\in\mathbb{N}$. Thus, the convergence of $\sum\limits_{n\ge1}\frac{(x_{1}+\cdots+x_{m}+n)^{2\lambda}}{n^3}$ (since $\lim\limits_{n\to\infty}\frac{\frac{(x_{1}+\cdots+x_{m}+n)^{2\lambda}}{n^3}}{\frac{n^{2\lambda}}{n^3}}=1$) easily yields that $$\mathsf{E}_{\phi,\omega}(\mathsf{h}_{\phi,\omega}^{\lambda})\circ\phi^{-1}((x_{1},x_{2},\cdots,x_{m},0,0,\cdots))<\infty$$
	and 
	$$\mathsf{E}_{\phi,\omega}(\mathsf{h}_{\phi,\omega})\circ\phi^{-1}((x_{1},x_{2},\cdots,x_{m},0,0,\cdots))=\infty$$
	for any $(x_{1},x_{2},\cdots,x_{m},0,0,\cdots)\in X$. By \cite[Corollary 3.9]{benhida-mn-2020}, Remark \ref{singlemult} and Corollary \eqref{denseequ} (iv), we get the desired result.
\end{exa} 

We continue to characterize the dense definiteness of $M_{\lambda}(\textbf{C}_{\phi,\omega})$, based on some new notions introduced below. Analogous to the toral Aluthge transform of operator tuples \cite{curto-ieot-2018}, we call the pair $\widehat{\mathcal{M}_{\lambda}(\textbf{C}_{\phi,\omega})}=(\mathcal{M}_{\lambda}(C_{\phi_{1},\omega_{1}}),\mathcal{M}_{\lambda}(C_{\phi_{2},\omega_{2}}))$ the {\em toral $\lambda$-spherical mean transform} of $(C_{\phi_{1},\omega_{1}},C_{\phi_{2},\omega_{2}})$, where $\mathcal{M}_{\lambda}(C_{\phi_{i},\omega_{i}})$ is the $\lambda$-mean transform of $C_{\phi_{i},\omega_{i}}$ for $i=1,2$.

\begin{pro}\label{pointwisemean}
Suppose $C_{\phi_{i},\omega_{i}}$ $(i=1,2)$ are densely defined. Fix $\lambda\in(0,1)$. If $\mathcal{M}_{\lambda}(\textbf{C}_{\phi,\omega})$ is densely defined, then $\widehat{\mathcal{M}_{\lambda}(\textbf{C}_{\phi,\omega})}$ is densely defined.
\end{pro}
\begin{proof}
By Corollary \ref{denseequ} (iv), we have that $\mathsf{E}_{\phi_{i},\omega_{i}}(\mathsf{h}_{\phi_{i},\omega_{i}})\circ\phi_{i}^{-1}<\infty$ a.e. $[\mu]$ for $i=1,2$, which implies $\mathcal{M}_{\lambda}(C_{\phi_{i},\omega_{i}},0)$ is densely defined for $i=1,2$. Thus, it follows from Remark \ref{singlemult} and Theorem \ref{meanbasic} (iv) that $\mathcal{M}_{\lambda}(C_{\phi_{i},\omega_{i}})=C_{\phi_{i},((1-\lambda)\chi_{\{\omega_{i}\ne0\}}\cdot(\frac{\mathsf{h}_{\phi_{i},\omega_{i}}}{\mathsf{h}_{\phi_{i},\omega_{i}}\circ\phi_{i}})^\frac{1}{2}+\lambda)\omega_{i}}$ is densely defined for $i=1,2$. Using Proposition \ref{basic} (ii), we know that the pair $\widehat{\mathcal{M}_{\lambda}(\textbf{C}_{\phi,\omega})}=(\mathcal{M}_{\lambda}(C_{\phi_{1},\omega_{1}}),\mathcal{M}_{\lambda}(C_{\phi_{2},\omega_{2}}))$ is densely defined.
\end{proof}

The converse of the above proposition does not hold in general (Example \ref{counterex}). However, the dense definiteness of $\mathcal{M}_{\lambda}(\textbf{C}_{\phi,\omega})$ can be characterized via the powers of $\widehat{\mathcal{M}_{\lambda}(\textbf{C}_{\phi,\omega})}$ in Proposition \ref{powers}. Before that, we introduce the definition of power of operator tuples. Put $n,m\in\mathbb{N}$. Denote the set of all $n$-tuples of operators by $\mathcal{L}(\mathcal{H})^n$, i.e., $\mathcal{L}(\mathcal{H})^n=\{(T_{1},T_{2},\cdots,T_{n}):T_{i}\in\mathcal{L}(\mathcal{H}),1\le i\le n\}$. Let $\textbf{T}=(T_{1},T_{2},\cdots,T_{n})\in\mathcal{L}(\mathcal{H})^n$ and $\textbf{S}=(S_{1},S_{2},\cdots,S_{m})\in\mathcal{L}(\mathcal{H})^m$. The product $\textbf{T}\textbf{S}\in\mathcal{L}(\mathcal{H})^{nm}$ is defined as
$$\textbf{T}\textbf{S}=(T_{1}S_{1},T_{1}S_{2},\cdots,T_{1}S_{m},T_{2}S_{1},T_{2}S_{2},\cdots,T_{2}S_{m},\cdots,T_{n}S_{1},T_{n}S_{2},\cdots,T_{n}S_{m}).$$
Further, the power $\textbf{T}^{k+1}$ ($k\in\mathbb{N}$) of $\textbf{T}$ is inductively given by
$\textbf{T}^{k+1}=\textbf{T}\textbf{T}^{k}$. In particular, if $\textbf{T}\in\mathcal{B}(\mathcal{H})$, then $\textbf{T}^k$ coincides with the usual power for single operators. The notion of powers of bounded operator tuples was first introduced in \cite{muller-studia-1992} by M{\"u}ller and Soltysiak to derive an analogue of the classical spectral radius formula in a multivariable setting. On the other hand, in the sense of Stochel and Szafraniec \cite{js-jmsj-2003}, an $n$-tuple $\textbf{T}=(T_{1},T_{2},\cdots,T_{n})$ of unbounded operators acting on $\mathcal{H}$ is said to be densely defined if the joint domain $\EuScript{D}(\textbf{T})=\bigcap\limits_{i=1}^n\EuScript{D}(T_{i})$ is dense in $\mathcal{H}$. Suppose $\textbf{T}^{k+1}$ is densely defined for some $k\in\mathbb{N}$, then for each $i\in\{1,2,\cdots,k+1\}$, $\textbf{T}^{i}$ is also densely defined. To see this, observe that
\begin{equation}\label{densepro}
\EuScript{D}(\textbf{T}^{k+1})=\bigcap\limits_{1\le i_{j}\le n,1\le j\le k+1}\EuScript{D}(T_{i_{1}}T_{i_{2}}\cdots T_{i_{k+1}})=\bigcap\limits_{1\le l\le k+1,1\le j\le l,1\le i_{j}\le n}\EuScript{D}(T_{i_{1}}\cdots T_{i_{l}}),
\end{equation}
which implies $\EuScript{D}(\textbf{T}^{k+1})\subset\EuScript{D}(\textbf{T}^i)$ for $i=1,2,\cdots,k+1$. Consequently, since $\EuScript{D}(\textbf{T}^{k+1})$ is dense, so is $\EuScript{D}(\textbf{T}^i)$. This fact will be used in the proof of Lemma \ref{productdense}. 

Let $l\in\mathbb{N}$ and $\sigma_{l}(\{1,2\})=\{(i_{1},i_{2},\cdots,i_{l}):i_{j}\in\{1,2\},1\le j\le l,j\in\mathbb{N}\}$. Suppose $\phi_{i}$ is an $\mathscr{A}$-measurable transform of $X$ and $\omega_{i}$ is a complex $\mathscr{A}$-measurable function for $i=1,2$. For any $(i_{1},i_{2},\cdots,i_{l})\in\sigma_{l}(\{1,2\})$, we write $$\phi_{i_{j}}^{i_{l}}=\phi_{i_{j}}\circ\phi_{i_{j+1}}\circ\phi_{i_{j+2}}\circ\cdots\circ\phi_{i_{l}}$$ and $$\omega_{i_{j}}^{i_{l}}=\omega_{i_{j}}\circ\phi_{i_{j+1}}\circ\phi_{i_{j+2}}\circ\cdots\circ\phi_{i_{l}}$$
for $1\le j\le l$.

\begin{lem}\label{productdense}
Suppose $C_{\phi_{i},\omega_{i}}$ $(i=1,2)$ are densely defined. Fix $k\in\mathbb{N}$.  Then $(C_{\phi_{1},\omega_{1}},C_{\phi_{2},\omega_{2}})^k$ is densely defined if and only if $C_{\phi_{i_{1}}^{i_{l}},\prod_{j=1}^{l}\omega_{i_{j}}^{i_{l}}}$ is densely defined for any $(i_{1},i_{2},\cdots,i_{l})\in\sigma_{l}(\{1,2\})$, $1\le l\le k$.
\end{lem}
\begin{proof}
Let $(i_{1},i_{2},\cdots,i_{l})\in\sigma_{l}(\{1,2\})$. We inductively defined a sequence of $\mathscr{A}$-measurable functions $\{\mathsf{h}_{n}\}_{n\ge0}$ by setting $\mathsf{h}_{0}=\mathsf{h}_{\phi_{i_{l}},\omega_{i_{l}}}$ and
\begin{equation}\label{induct}
	\mathsf{h}_{n}=\mathsf{E}_{\phi_{i_{l-n}},\omega_{i_{l-n}}}(\mathsf{h}_{n-1})\circ\phi_{i_{l-n}^{-1}}\cdot\mathsf{h}_{\phi_{i_{l-n}},\omega_{i_{l-n}}}
\end{equation}
for $n\in\mathbb{N}$ (By similar arguments as in Lemma \ref{weighted} (iii) we have $\mathsf{h}_{n-1}\ge 0$ a.e. $[\mu]$, which implies $\mathsf{E}_{\phi_{i_{l-n}},\omega_{i_{l-n}}}(\mathsf{h}_{n-1})$ is well-defined for $n\in\mathbb{N}$). 

Note that $C_{\phi_{i_{1}}^{i_{l}},\prod_{j=1}^{l}\omega_{i_{j}}^{i_{l}}}$ equals either $C_{\phi_{1},\omega_{1}}$ or $C_{\phi_{2},\omega_{2}}$ if $l=1$. For $l\ge2$ we first prove that the weighted composition operator $C_{\phi_{i_{1}}^{i_{l}},\prod_{j=1}^{l}\omega_{i_{j}}^{i_{l}}}$ is well-defined, that is, $\mu_{\prod_{j=1}^{l}\omega_{i_{j}}^{i_{l}}}\circ(\phi_{i_{1}}^{i_{l}})^{-1}\ll\mu$. Suppose $\Delta\in\mathcal{A}$ satisfying $\mu(\Delta)=0$. Then, we inductively have
\begin{equation}\label{welldefined}
	\begin{aligned}
		&\mu_{\prod_{j=1}^{l}\omega_{i_{j}}^{i_{l}}}\circ(\phi_{i_{1}}^{i_{l}})^{-1}(\Delta)\stackrel{\eqref{rd}}{=}\int_{(\phi_{i_{1}}^{i_{l}})^{-1}(\Delta)}\Bigg|\prod_{j=1}^{l}\omega_{i_{j}}^{i_{l}}\Bigg|^2\D\mu\\&=\int_{X}\chi_{(\phi_{i_{1}}^{i_{l-1}})^{-1}(\Delta)}\circ\phi_{i_{l}}\cdot\Bigg|\prod_{j=1}^{l-1}\omega_{i_{j}}^{i_{l-1}}\Bigg|^2\circ\phi_{i_{l}}\D\mu_{\omega_{i_{l}}}\\&\stackrel{\eqref{intcom}}{=}\int_{X}\chi_{(\phi_{i_{1}}^{i_{l-1}})^{-1}(\Delta)}\cdot\Bigg|\prod_{j=1}^{l-1}\omega_{i_{j}}^{i_{l-1}}\Bigg|^2\mathsf{h}_{\phi_{i_{l},\omega_{i_{l}}}}\D\mu\\&\stackrel{\eqref{induct}}{=}\int_{X}\chi_{(\phi_{i_{1}}^{i_{l-1}})^{-1}(\Delta)}\cdot\Bigg|\prod_{j=1}^{l-1}\omega_{i_{j}}^{i_{l-1}}\Bigg|^2\mathsf{h}_{0}\D\mu\\&\xlongequal{\eqref{exp},\eqref{invexp}}\int_{X}\chi_{(\phi_{i_{1}}^{i_{l-2}})^{-1}(\Delta)}\circ\phi_{i_{l-1}}\cdot\Bigg|\prod_{j=1}^{l-2}\omega_{i_{j}}^{i_{l-2}}\Bigg|^2\circ\phi_{i_{l-1}}\cdot\mathsf{E}_{\phi_{i_{l-1}},\omega_{i_{l-1}}}(\mathsf{h}_{0})\circ\phi_{i_{l-1}}^{-1}\circ\phi_{i_{l-1}}\D\mu_{\omega_{i_{l-1}}}\\&\stackrel{\eqref{intcom}}{=}\int_{X}\chi_{(\phi_{i_{1}}^{i_{l-2}})^{-1}(\Delta)}\cdot\Bigg|\prod_{j=1}^{l-2}\omega_{i_{j}}^{i_{l-2}}\Bigg|^2\cdot\mathsf{E}_{\phi_{i_{l-1}},\omega_{i_{l-1}}}(\mathsf{h}_{0})\circ\phi_{i_{l-1}}^{-1}\cdot\mathsf{h}_{\phi_{i_{l-1}},\omega_{i_{l-1}}}\D\mu\\&\stackrel{\eqref{induct}}{=}\int_{X}\chi_{(\phi_{i_{1}}^{i_{l-2}})^{-1}(\Delta)}\cdot\Bigg|\prod_{j=1}^{l-2}\omega_{i_{j}}^{i_{l-2}}\Bigg|^2\cdot\mathsf{h}_{1}\D\mu\\&\xlongequal{\eqref{exp},\eqref{invexp}}\int_{X}\chi_{(\phi_{i_{1}}^{i_{l-3}})^{-1}(\Delta)}\circ\phi_{i_{l-2}}\cdot\Bigg|\prod_{j=1}^{l-3}\omega_{i_{j}}^{i_{l-3}}\Bigg|^2\circ\phi_{i_{l-2}}\cdot\mathsf{E}_{\phi_{i_{l-2}},\omega_{i_{l-2}}}(\mathsf{h}_{1})\circ\phi_{i_{l-2}}^{-1}\circ\phi_{i_{l-2}}\D\mu_{\omega_{i_{l-2}}}\\&=\cdots\\&=\int_{X}\chi_{\phi_{i_{1}}^{-1}(\Delta)}\cdot|\omega_{i_{1}}|^2\cdot\mathsf{h}_{l-2}\D\mu\\&=\int_{\phi_{i_{1}}^{-1}(\Delta)}\mathsf{h}_{l-2}\D\mu_{\omega_{i_{1}}}=0
	\end{aligned} 	
\end{equation}
for $l\ge2$, where the last equality holds since $\mu_{\omega_{i_{1}}}\circ\phi_{i_{1}}^{-1}\ll\mu$ for $i_{1}\in\{1,2\}$. Hence, $C_{\phi_{i_{1}}^{i_{l}},\prod_{j=1}^{l}\omega_{i_{j}}^{i_{l}}}$  is well-defined. 

Now, suppose $(C_{\phi_{1},\omega_{1}},C_{\phi_{2},\omega_{2}})^k$ is densely defined. It follows from the fact mentioned above that $(C_{\phi_{1},\omega_{1}},C_{\phi_{2},\omega_{2}})^l$ is densely defined for $1\le l\le k$. Then by the definition of powers of unbounded operator tuples, we have that $C_{\phi_{i_{l}},\omega_{i_{l}}}C_{\phi_{i_{l-1}},\omega_{i_{l-1}}}\cdots C_{\phi_{i_{1}},\omega_{i_{1}}}$ is densely defined for $(i_{1},i_{2},\cdots,i_{l})\in\sigma_{l}(\{1,2\})$. On the other hand, since
\begin{align*}
	C_{\phi_{i_{l}},\omega_{i_{l}}}C_{\phi_{i_{l-1}},\omega_{i_{l-1}}}\cdots C_{\phi_{i_{1}},\omega_{i_{1}}}f&=C_{\phi_{i_{l}},\omega_{i_{l}}}C_{\phi_{i_{l-1}},\omega_{i_{l-1}}}\cdots C_{\phi_{i_{2}},\omega_{i_{2}}}\omega_{i_{1}}\cdot f\circ\phi_{i_{1}}\\&=C_{\phi_{i_{l}},\omega_{i_{l}}}C_{\phi_{i_{l-1}},\omega_{i_{l-1}}}\cdots C_{\phi_{i_{3}},\omega_{i_{3}}}\omega_{i_{2}}\cdot\omega_{i_{1}}\circ\phi_{i_{2}}\cdot f\circ\phi_{i_{1}}\circ\phi_{i_{2}}\\&=\cdots\\&=\prod_{j=1}^{l}\omega_{i_{j}}^{i_{l}}\cdot f\circ\phi_{i_{1}}^{i_{l}}=C_{\phi_{i_{1}}^{i_{l}},\prod_{j=1}^{l}\omega_{i_{j}}^{i_{l}}}f
\end{align*}	
for any $f\in\EuScript{D}(C_{\phi_{i_{l}},\omega_{i_{l}}}C_{\phi_{i_{l-1}},\omega_{i_{l-1}}}\cdots C_{\phi_{i_{1}},\omega_{i_{1}}})$, then 
\begin{equation}\label{densepro1}
	\begin{aligned}
		&\EuScript{D}(C_{\phi_{i_{l}},\omega_{i_{l}}}C_{\phi_{i_{l-1}},\omega_{i_{l-1}}}\cdots C_{\phi_{i_{1}},\omega_{i_{1}}})\\&=\EuScript{D}(C_{\phi_{i_{1}}^{i_{l}},\prod_{j=1}^{l}\omega_{i_{j}}^{i_{l}}})\bigcap\EuScript{D}(	C_{\phi_{i_{l-1}},\omega_{i_{l-1}}}C_{\phi_{i_{l-2}},\omega_{i_{l-2}}}\cdots C_{\phi_{i_{1}},\omega_{i_{1}}})\\&=\EuScript{D}(C_{\phi_{i_{1}}^{i_{l}},\prod_{j=1}^{l}\omega_{i_{j}}^{i_{l}}})\bigcap\EuScript{D}(C_{\phi_{i_{1}}^{i_{l-1}},\prod_{j=1}^{l-1}\omega_{i_{j}}^{i_{l-1}}})\bigcap\EuScript{D}(	C_{\phi_{i_{l-2}},\omega_{i_{l-2}}}C_{\phi_{i_{l-3}},\omega_{i_{l-3}}}\cdots C_{\phi_{i_{1}},\omega_{i_{1}}})\\&=\cdots=\bigcap\limits_{1\le t\le l}\EuScript{D}(C_{\phi_{i_{1}}^{i_{t}},\prod_{j=1}^{t}\omega_{i_{j}}^{i_{t}}})
	\end{aligned}
\end{equation}
which implies, by the dense definiteness of $C_{\phi_{i_{l}},\omega_{i_{l}}}C_{\phi_{i_{l-1}},\omega_{i_{l-1}}}\cdots C_{\phi_{i_{1}},\omega_{i_{1}}}$, that $C_{\phi_{i_{1}}^{i_{l}},\prod_{j=1}^{l}\omega_{i_{j}}^{i_{l}}}$ is densely defined.

Conversely, by \eqref{densepro1} and \eqref{densepro}, we have
\begin{equation}\label{densepro2}
\begin{aligned}
	\EuScript{D}((C_{\phi_{1},\omega_{1}},C_{\phi_{2},\omega_{2}})^k)&=\bigcap\limits_{1\le l\le k,1\le j\le l,i_{j}\in\{1,2\}}\EuScript{D}(C_{\phi_{i_{l}},\omega_{i_{l}}}\cdots C_{\phi_{i_{1}},\omega_{i_{1}}})\\&=\bigcap\limits_{1\le l\le k,1\le j\le l,i_{j}\in\{1,2\}}\bigcap\limits_{1\le t\le l}\EuScript{D}(C_{\phi_{i_{1}}^{i_{t}},\prod_{j=1}^{t}\omega_{i_{j}}^{i_{t}}}).
\end{aligned}
\end{equation}
Since $C_{\phi_{i_{1}}^{i_{l}},\prod_{j=1}^{l}\omega_{i_{j}}^{i_{l}}}$ is densely defined for any $(i_{1},i_{2},\cdots,i_{l})\in\sigma_{l}(\{1,2\})$, $1\le l\le k$, then by \eqref{densepro2} and the argument used in Proposition \ref{basic} (ii), it is not difficult to get the desired rusult.
\end{proof}

\begin{pro}\label{powers}
Suppose $C_{\phi_{i},\omega_{i}}$ $(i=1,2)$ are densely defined. Fix $\lambda\in(0,1)$. Consider the following conditions$:$
\begin{enumerate}
	\item[(i)]
	$\widehat{\mathcal{M}_{\lambda}(\textbf{C}_{\phi,\omega})}^k$ is densely defined for some $k\in\mathbb{N}\backslash\{1\}$,
	\item[(ii)]
	$(C_{\phi_{1},\omega_{1}},C_{\phi_{2},\omega_{2}})^k$ is densely defined for some $k\in\mathbb{N}\backslash\{1\}$,
	\item[(iii)]
$\mathcal{M}_{\lambda}(\textbf{C}_{\phi,\omega})$ is densely defined.
\end{enumerate}
Then (i) $\Rightarrow$ (ii) $\Rightarrow$ (iii) holds. Moreover, if $k=2$, then (iii) $\Rightarrow$ (ii) holds.
\end{pro}
\begin{proof}
(i) $\Rightarrow$ (ii):	By Remark \ref{singlemult} and Theorem \ref{meanbasic}, we have $\widehat{\mathcal{M}_{\lambda}(\textbf{C}_{\phi,\omega})}^k=(C_{\phi_{1},\omega_{\lambda}^1},C_{\phi_{2},\omega_{\lambda}^2})^k,$
where 
$\omega_{\lambda}^i=((1-\lambda)\chi_{\{\omega_{i}\ne0\}}\cdot(\frac{\mathsf{h}_{\phi_{i},\omega_{i}}}{\mathsf{h}_{\phi_{i},\omega_{i}}\circ\phi_{i}})^\frac{1}{2}+\lambda)\omega_{i}$ for $i=1,2$. By (i) and the definitions of powers and dense definiteness for unbounded operator tuples, we have that $C_{\phi_{i},\omega_{\lambda}^i}$ is densely defined for $i=1,2$. Then using Lemma \ref{productdense}, we obtain $C_{\phi_{i_{1}}^{i_{l}},\prod_{j=1}^{l}(\mathsf{h}_{i_{j}}^{i_{l}}\cdot\omega_{i_{j}}^{i_{l}})}$ is densely defined for any $(i_{1},i_{2},\cdots,i_{l})\in\sigma_{l}(\{1,2\})$, $1\le l\le k$, where $\mathsf{h}_{i_{j}}^{i_{l}}=((1-\lambda)\chi_{\{\omega_{i_{j}}\ne0\}}\cdot(\frac{\mathsf{h}_{\phi_{i_{j}},\omega_{i_{j}}}}{\mathsf{h}_{\phi_{i_{j}},\omega_{i_{j}}}\circ\phi_{i_{j}}})^\frac{1}{2}+\lambda)\circ\phi_{i_{j+1}}\circ\phi_{i_{j+2}}\circ\cdots\circ\phi_{i_{l}}$. Since$$(1-\lambda)\chi_{\{\omega_{i}\ne0\}}\cdot(\frac{\mathsf{h}_{\phi_{i},\omega_{i}}}{\mathsf{h}_{\phi_{i},\omega_{i}}\circ\phi_{i}})^\frac{1}{2}+\lambda\ge\lambda$$ for $i\in\{1,2\}$, then we have
$$\frac{1}{(\prod_{j=1}^{l}\mathsf{h}_{i_{j}}^{i_{l}})^2}\le\frac{1}{\lambda^{2l}},$$ which implies\begin{align*}
	\int_{X}|C_{\phi_{i_{1}}^{i_{l}},\prod_{j=1}^{l}\omega_{i_{j}}^{i_{l}}}f|^2\D\mu&=\int_{X}\frac{1}{(\prod_{j=1}^{l}\mathsf{h}_{i_{j}}^{i_{l}})^2}|C_{\phi_{i_{1}}^{i_{l}},\prod_{j=1}^{l}(\mathsf{h}_{i_{j}}^{i_{l}}\cdot\omega_{i_{j}}^{i_{l}})}f|^2\D\mu\\&\le\frac{1}{\lambda^{2l}}\int_{X}|C_{\phi_{i_{1}}^{i_{l}},\prod_{j=1}^{l}(\mathsf{h}_{i_{j}}^{i_{l}}\cdot\omega_{i_{j}}^{i_{l}})}f|^2\D\mu<\infty
\end{align*}
for any $f\in\EuScript{D}\Big(C_{\phi_{i_{1}}^{i_{l}},\prod_{j=1}^{l}(\mathsf{h}_{i_{j}}^{i_{l}}\cdot\omega_{i_{j}}^{i_{l}})}\Big)$, that is, $\EuScript{D}\Big(C_{\phi_{i_{1}}^{i_{l}},\prod_{j=1}^{l}(\mathsf{h}_{i_{j}}^{i_{l}}\cdot\omega_{i_{j}}^{i_{l}})}\Big)\subset\EuScript{D}(C_{\phi_{i_{1}}^{i_{l}},\prod_{j=1}^{l}\omega_{i_{j}}^{i_{l}}})$. Hence, $C_{\phi_{i_{1}}^{i_{l}},\prod_{j=1}^{l}\omega_{i_{j}}^{i_{l}}}$ is densely defined for any $(i_{1},i_{2},\cdots,i_{l})\in\sigma_{l}(\{1,2\})$, $1\le l\le k$. By Lemma \ref{productdense}, we obtain (ii).

(ii) $\Rightarrow$ (iii): Suppose $(C_{\phi_{1},\omega_{1}},C_{\phi_{2},\omega_{2}})^k$ is densely defined for some $k\in\mathbb{N}\backslash\{1\}$. By Lemma \ref{productdense}, we have that $C_{\phi_{j}\circ\phi_{i},\omega_{i}\cdot\omega_{j}\circ\phi_{i}}$ is densely defined for $i,j\in\{1,2\}$. Thus, by Theorem \ref{combasic} (iii), we have
\begin{equation}\label{finite}
\mathsf{h}_{\phi_{j}\circ\phi_{i},\omega_{i}\cdot\omega_{j}\circ\phi_{i}}<\infty\quad\text{a.e.}\ [\mu].
\end{equation} On the other hand, for any $\varDelta\in\mathscr{A}$, we derive
\begin{align*}
	\int_{\varDelta}\mathsf{h}_{\phi_{j}\circ\phi_{i},\omega_{i}\cdot\omega_{j}\circ\phi_{i}}\D\mu&\stackrel{\eqref{rd}}{=}\int_{(\phi_{j}\circ\phi_{i})^{-1}(\Delta)}|\omega_{i}\cdot\omega_{j}\circ\phi_{i}|^2\D\mu\stackrel{\eqref{welldefined}}{=}\int_{\phi_{j}^{-1}(\Delta)}\mathsf{h}_{\phi_{i},\omega_{i}}\D\mu_{\omega_{j}}\\&=\int_{\phi_{j}^{-1}(\Delta)}\mathsf{E}_{\phi_{j},\omega_{j}}(\mathsf{h}_{\phi_{i},\omega_{i}})\D\mu_{\omega_{j}}\\&\xlongequal{\eqref{intcom},\eqref{invexp}}\int_{\Delta}\mathsf{E}_{\phi_{j},\omega_{j}}(\mathsf{h}_{\phi_{i},\omega_{i}})\circ\phi_{j}^{-1}\mathsf{h}_{\phi_{j},\omega_{j}}\D\mu.
\end{align*} 
Since $\Delta$ is arbitrary, then
\begin{equation}\label{powercom}
\mathsf{h}_{\phi_{j}\circ\phi_{i},\omega_{i}\cdot\omega_{j}\circ\phi_{i}}=\mathsf{E}_{\phi_{j},\omega_{j}}(\mathsf{h}_{\phi_{i},\omega_{i}})\circ\phi_{j}^{-1}\mathsf{h}_{\phi_{j},\omega_{j}}\quad\text{a.e.}\ [\mu]
\end{equation}
 for $i,j=1,2$. Therefore, by \eqref{finite}, \eqref{Epro}, \eqref{powercom} and Theorem \ref{combasic} (iii), we have $$\mathsf{E}_{\phi_{j},\omega_{j}}(\mathsf{h}_{\phi_{i},\omega_{i}})\circ\phi_{j}^{-1}<\infty\quad\text{a.e.}\ [\mu],$$
 which yields, by using Corollary \ref{denseequ} (iv), that $\mathcal{M}_{\lambda}(\textbf{C}_{\phi,\omega})$ is densely defined.
 
 Next, we prove (iii) $\Rightarrow$ (ii) if $k=2$. Recalling the definitions of powers and dense definiteness for unbounded operator tuples, we have
 \begin{align*}
 \EuScript{D}((C_{\phi_{1},\omega_{1}},C_{\phi_{2},\omega_{2}})^2)&=\EuScript{D}(C_{\phi_{1},\omega_{1}}^2,C_{\phi_{1},\omega_{1}}C_{\phi_{2},\omega_{2}},C_{\phi_{2},\omega_{2}}C_{\phi_{1},\omega_{1}},C_{\phi_{2},\omega_{2}}^2)\\&=\bigcap\limits_{i,j\in\{1,2\}}\EuScript{D}(C_{\phi_{i},\omega_{i}}C_{\phi_{j},\omega_{j}})\\&\stackrel{\eqref{densepro1}}{=}\EuScript{D}(C_{\phi_{1},\omega_{1}})\bigcap\EuScript{D}(C_{\phi_{2},\omega_{2}})\bigcap\limits_{i,j\in\{1,2\}}\EuScript{D}(C_{\phi_{j}\circ\phi_{i},\omega_{i}\cdot\omega_{j}\circ\phi_{i}}).
 \end{align*}
 If we can prove $C_{\phi_{j}\circ\phi_{i},\omega_{i}\cdot\omega_{j}\circ\phi_{i}}$ is densely defined for $i,j=1,2$, then combining the above equality with the argument used in Proposition \ref{basic} (ii), it is not difficult to verify that$$\overline{\EuScript{D}((C_{\phi_{1},\omega_{1}},C_{\phi_{2},\omega_{2}})^2)}=\overline{L^2\Big(\big(1+\sum\limits_{i=1}^2\mathsf{h}_{\phi_{i},\omega_{i}}+\sum\limits_{i,j\in\{1,2\}}\mathsf{h}_{\phi_{j}\circ\phi_{i},\omega_{i}\cdot\omega_{j}\circ\phi_{i}})\D\mu\Big)}=L^2(\mu),$$
 that is, $(C_{\phi_{1},\omega_{1}},C_{\phi_{2},\omega_{2}})^2$ is densely defined. Indeed, since $\mathcal{M}_{\lambda}(\textbf{C}_{\phi,\omega})$ is densely defined, then by Corollary \ref{denseequ} (iv) we have $\mathsf{E}_{\phi_{j},\omega_{j}}(\mathsf{h}_{\phi_{i},\omega_{i}})\circ\phi_{j}^{-1}<\infty\quad\text{a.e.}\ [\mu]$ for $i,j=1,2$, which implies, by \eqref{powercom} and Theorem \ref{combasic} (iii), that $\mathsf{h}_{\phi_{j}\circ\phi_{i},\omega_{i}\cdot\omega_{j}\circ\phi_{i}}<\infty\quad\text{a.e.}\ [\mu]$ for $i,j\in\{1,2\}$. Thus, by Theorem \ref{combasic} (iii), we obtain $C_{\phi_{j}\circ\phi_{i},\omega_{i}\cdot\omega_{j}\circ\phi_{i}}$ is densely defined, as desired. This completes the proof.
\end{proof}

The following example shows that: \textbf{(a)} condition (iii) of the above proposition does not imply (ii) if $k\ne2$; \textbf{(b)} (iii) does not imply (i); and \textbf{(c)} the converse of Proposition \ref{pointwisemean} is false.

\begin{exa}\label{counterex}
Let $X=\mathbb{Z}_{+}\bigcup\mathbb{N}\times\mathbb{N}$, $\mathscr{A}=2^X$ and $\mu$ be the counting measure on $X$. Suppose $\lambda\in(0,1)$, then we have that

\textbf{(a)} $\mathcal{M}_{\lambda}(C_{\phi_{1},\omega_{1}},0)$ is densely defined, but $(C_{\phi_{1},\omega_{1}},0)^3$ is not, where the $\mathscr{A}$-measurable transform $\phi_{1}:X\to X$ is defined by $\phi_{1}(x)=\begin{cases}
	k&x=k-1\in\mathbb{Z}_{+}\\
	0&x=(k,1)\in\mathbb{N}\times\{1\}\\
	(m,n-1)&x=(m,n)\in\mathbb{N}\times(\mathbb{N}\backslash\{1\})
\end{cases}$ and the function $\omega_{1}:X\to\mathbb{R}$ is defined by $\omega_{1}(x)=\begin{cases}
	\frac{1}{k}&x=(k,1)\in\mathbb{N}\times\{1\}\\
		k&x=(k,3)\in\mathbb{N}\times\{3\}\\
		1&x\in X\backslash\Big(\bigcup\limits_{i=1,3}\mathbb{N}\times\{i\}\Big)
\end{cases}$.

\textbf{(b)} $\mathcal{M}_{\lambda}(C_{\phi_{2},\omega_{2}},0)$ is densely defined, but $\widehat{\mathcal{M}_{\lambda}(C_{\phi_{2},\omega_{2}},0)}^2$ is not, where the $\mathscr{A}$-measurable transform $\phi_{2}:X\to X$ is defined by $\phi_{2}(x)=\begin{cases}
	1&x=3k\\
	3k+2&x=3k+1\\
	3k&x=3k+2\\
	x&x\in \{0,1,2\}\bigcup\mathbb{N}\times\mathbb{N}
\end{cases}$, the function $\omega_{2}:X\to\mathbb{R}$ is defined by $\omega_{2}(x)=\begin{cases}
	\frac{1}{k}&x=3k\\
	\frac{k-\lambda^2}{\lambda(1-\lambda)}&x=3k+1\\
	1&x=3k+2\\
	0&x\in \{0,1,2\}\bigcup\mathbb{N}\times\mathbb{N}
\end{cases}$
and $k\in\mathbb{N}$.

\textbf{(c)} $\reallywidehat{\mathcal{M}_{\lambda}(C_{\phi_{3},\omega_{3}},C_{\phi_{4},\omega_{4}})}$ is densely defined, but $\mathcal{M}_{\lambda}(C_{\phi_{3},\omega_{3}},C_{\phi_{4},\omega_{4}})$ is not, where $\phi_{3}=\phi_{4}=\phi_{1}$ ($\phi_{1}$ is defined in \textbf{(a)}) 
and the functions $\omega_{3},\omega_{4}:X\to\mathbb{R}$ are defined as $$\omega_{3}(x)=\begin{cases}
	\frac{1}{k}&x=(k,l)\in\mathbb{N}\times\{1\}\bigcup\mathbb{N}\times\{3\}\\
	1&x\in\mathbb{N}\times\{2\}\\
	0&x\notin\bigcup\limits_{i=1}^3\mathbb{N}\times\{i\}
\end{cases}$$ and $$\omega_{4}(x)=\begin{cases}
	\sqrt{k}&x=(k,2)\in\mathbb{N}\times\{2\}\\
	0&x\notin\mathbb{N}\times\{2\}
\end{cases},$$ respectively. 
\begin{proof}
\textbf{(a)} It suffices by Remark \ref{singlemult} to show that $\mathcal{M}_{\lambda}(C_{\phi_{1},\omega_{1}})$ is densely defined, but $C_{\phi_{1},\omega_{1}}^3$ is not. Indeed, by \eqref{h}, we have
$$\mathsf{h}_{\phi_{1},\omega_{1}}(x)=\sum\limits_{y\in\phi_{1}^{-1}(x)}|\omega_{1}(y)|^2=\begin{cases}
	1&x=k\in\mathbb{N}\\
	\sum\limits_{k=1}^\infty\frac{1}{k^2}&x=0\\
	|\omega_{1}(m,n)|^2&x=(m,n-1)\in\mathbb{N}\times\mathbb{N}
\end{cases}.$$
Hence, by Theorem \ref{combasic} (iii), $C_{\phi_{1},\omega_{1}}$ is densely defined. By \eqref{Efphi} and the above equality, we obtain
\begin{equation}\label{power33}
	\mathsf{E}_{\phi_{1},\omega_{1}}(\mathsf{h}_{\phi_{1},\omega_{1}})\circ\phi_{1}^{-1}(x)=\frac{\sum\limits_{y\in\phi_{1}^{-1}(x)}\mathsf{h}_{\phi_{1},\omega_{1}}(y)|\omega_{1}(y)|^2}{\mathsf{h}_{\phi_{1},\omega_{1}}(x)}=\begin{cases}
		\sum\limits_{k=1}^{\infty}\frac{1}{k^2}&x=1\\
		1&x\in\mathbb{Z}_{+}\backslash\{1\}\\
		|\omega_{1}(m,n+1)|^2&x=(m,n-1)\in\mathbb{N}\times\mathbb{N}
	\end{cases},
\end{equation}
which implies, by Remark \ref{singlemult} and Corollary \ref{denseequ} (iv), that $\mathcal{M}_{\lambda}(C_{\phi,\omega})$ is densely defined. On the other hand, since \begin{align*}
	C_{\phi_{1},\omega_{1}}^3f&=C_{\phi_{1},\omega_{1}}C_{\phi_{1},\omega_{1}}\omega_{1}\cdot f\circ\phi_{1}=C_{\phi_{1},\omega_{1}}\omega_{1}\cdot\omega_{1}\circ\phi_{1}\cdot f\circ\phi_{1}\circ\phi_{1}\\&=\omega_{1}\cdot\omega_{1}\circ\phi_{1}\cdot\omega_{1}\circ\phi_{1}\circ\phi_{1}\cdot f\circ\phi_{1}\circ\phi_{1}\circ\phi_{1}\\&=C_{\phi_{1}\circ\phi_{1}\circ\phi_{1},\omega_{1}\cdot\omega_{1}\circ\phi_{1}\cdot\omega_{1}\circ\phi_{1}\circ\phi_{1}}f
\end{align*}
for any $f\in\EuScript{D}(C_{\phi_{1},\omega_{1}}^3)$, we have $\EuScript{D}(C_{\phi_{1},\omega_{1}}^3)\subseteq\EuScript{D}(C_{\phi_{1}\circ\phi_{1}\circ\phi_{1},\omega_{1}\cdot\omega_{1}\circ\phi_{1}\cdot\omega_{1}\circ\phi_{1}\circ\phi_{1}})$. If we can prove $$\mathsf{h}_{\phi_{1}\circ\phi_{1}\circ\phi_{1},\omega_{1}\cdot\omega_{1}\circ\phi_{1}\cdot\omega_{1}\circ\phi_{1}\circ\phi_{1}}(0)=\infty,$$ then by Theorem \ref{combasic} (iii) we have $\overline{\EuScript{D}(C_{\phi_{1},\omega_{1}}^3)}\subseteq\overline{\EuScript{D}(C_{\phi_{1}\circ\phi_{1}\circ\phi_{1},\omega_{1}\cdot\omega_{1}\circ\phi_{1}\cdot\omega_{1}\circ\phi_{1}\circ\phi_{1}})}\ne L^2(\mu)$, that is, $C_{\phi_{1},\omega_{1}}^3$ is not densely defined. We now compute $\mathsf{h}_{\phi_{1}\circ\phi_{1}\circ\phi_{1},\omega_{1}\cdot\omega_{1}\circ\phi_{1}\cdot\omega_{1}\circ\phi_{1}\circ\phi_{1}}(0)$. In fact, for any $\Delta\in 2^X$, we have\begin{align*}
	\int_{\Delta}\mathsf{h}_{\phi_{1}\circ\phi_{1}\circ\phi_{1},\omega_{1}\cdot\omega_{1}\circ\phi_{1}\cdot\omega_{1}\circ\phi_{1}\circ\phi_{1}}\D\mu&\stackrel{\eqref{rd}}{=}\int_{(\phi_{1}\circ\phi_{1}\circ\phi_{1})^{-1}(\Delta)}|\omega_{1}\cdot\omega_{1}\circ\phi_{1}\cdot\omega_{1}\circ\phi_{1}\circ\phi_{1}|^2\D\mu\\&\stackrel{\eqref{muomega}}{=}\int_{X}\chi_{(\phi_{1}\circ\phi_{1})^{-1}(\Delta)}\circ\phi_{1}\cdot|\omega_{1}\cdot\omega_{1}\circ\phi_{1}|^2\circ\phi_{1}\D\mu_{\omega_{1}}\\&\stackrel{\eqref{intcom}}{=}\int_{X}\chi_{(\phi_{1}\circ\phi_{1})^{-1}(\Delta)}\cdot|\omega_{1}\cdot\omega_{1}\circ\phi_{1}|^2\mathsf{h}_{\phi_{1},\omega_{1}}\D\mu\\&\stackrel{\eqref{muomega}}{=}\int_{X}\chi_{\phi_{1}^{-1}(\Delta)}\circ\phi_{1}\cdot|\omega_{1}|^2\circ\phi_{1}\mathsf{h}_{\phi_{1},\omega_{1}}\D\mu_{\omega_{1}}\\&\stackrel{\eqref{exp}}{=}\int_{X}\chi_{\phi_{1}^{-1}(\Delta)}\circ\phi_{1}\cdot|\omega_{1}|^2\circ\phi_{1}\mathsf{E}_{\phi_{1},\omega_{1}}(\mathsf{h}_{\phi_{1},\omega_{1}})\D\mu_{\omega_{1}}\\&\xlongequal{\eqref{invexp},\eqref{intcom}}\int_{\phi_{1}^{-1}(\Delta)}\mathsf{E}_{\phi_{1},\omega_{1}}(\mathsf{h}_{\phi_{1},\omega_{1}})\circ\phi_{1}^{-1}\mathsf{h}_{\phi_{1},\omega_{1}}\D\mu_{\omega_{1}}\\&\xlongequal{\eqref{exp},\eqref{invexp},\eqref{intcom}}\int_{\Delta}\mathsf{E}_{\phi_{1},\omega_{1}}(\mathsf{E}_{\phi_{1},\omega_{1}}(\mathsf{h}_{\phi_{1},\omega_{1}})\circ\phi_{1}^{-1}\mathsf{h}_{\phi_{1},\omega_{1}})\circ\phi_{1}^{-1}\mathsf{h}_{\phi_{1},\omega_{1}}\D\mu.
\end{align*}
Since $\Delta$ is arbitrary, then $\mathsf{h}_{\phi_{1}\circ\phi_{1}\circ\phi_{1},\omega_{1}\cdot\omega_{1}\circ\phi_{1}\cdot\omega_{1}\circ\phi_{1}\circ\phi_{1}}=\mathsf{E}_{\phi_{1},\omega_{1}}(\mathsf{E}_{\phi_{1},\omega_{1}}(\mathsf{h}_{\phi_{1},\omega_{1}})\circ\phi_{1}^{-1}\mathsf{h}_{\phi_{1},\omega_{1}})\circ\phi_{1}^{-1}\mathsf{h}_{\phi_{1},\omega_{1}}$ a.e. $[\mu]$. Thus, by \eqref{power33} and \eqref{Efphi}, we obtain\begin{align*}
	\mathsf{h}_{\phi_{1}\circ\phi_{1}\circ\phi_{1},\omega_{1}\cdot\omega_{1}\circ\phi_{1}\cdot\omega_{1}\circ\phi_{1}\circ\phi_{1}}(0)&=(\mathsf{E}_{\phi_{1},\omega_{1}}(\mathsf{E}_{\phi_{1},\omega_{1}}(\mathsf{h}_{\phi_{1},\omega_{1}})\circ\phi_{1}^{-1}\mathsf{h}_{\phi_{1},\omega_{1}})\circ\phi_{1}^{-1}\mathsf{h}_{\phi_{1},\omega_{1}})(0)\\&=\frac{\sum\limits_{y\in\phi_{1}^{-1}(0)}(\mathsf{E}_{\phi_{1},\omega_{1}}(\mathsf{h}_{\phi_{1},\omega_{1}})\circ\phi_{1}^{-1}\mathsf{h}_{\phi_{1},\omega_{1}})(y)|\omega_{1}(y)|^2}{\mathsf{h}_{\phi_{1},\omega_{1}}(0)}\cdot\mathsf{h}_{\phi_{1},\omega_{1}}(0)\\&=\sum\limits_{y\in\phi_{1}^{-1}(0)}(\mathsf{E}_{\phi_{1},\omega_{1}}(\mathsf{h}_{\phi_{1},\omega_{1}})\circ\phi_{1}^{-1}\mathsf{h}_{\phi_{1},\omega_{1}})(y)|\omega_{1}(y)|^2\\&=\sum\limits_{k=1}^\infty(\mathsf{E}_{\phi_{1},\omega_{1}}(\mathsf{h}_{\phi_{1},\omega_{1}})\circ\phi_{1}^{-1})(k,1)|\omega_{1}(k,2)|^2|\omega_{1}(k,1)|^2\\&=\sum\limits_{k=1}^\infty|\omega_{1}(k,3)|^2|\omega_{1}(k,2)|^2|\omega_{1}(k,1)|^2=\sum\limits_{k=1}^\infty 1=\infty.
\end{align*}
Therefore, we get the desired result.

\textbf{(b)} It suffices by Remark \ref{singlemult} to show that $\mathcal{M}_{\lambda}(C_{\phi_{2},\omega_{2}})$ is densely defined, but $(\mathcal{M}_{\lambda}(C_{\phi_{2},\omega_{2}}))^2$ is not. Indeed, by \eqref{Ef}, we have
$$\mathsf{h}_{\phi_{2},\omega_{2}}(x)=\sum\limits_{y\in\phi_{2}^{-1}(x)}|\omega_{2}(y)|^2=\begin{cases}
	\sum\limits_{k=1}^\infty\frac{1}{k^2}&x=1\\
	1&x=3k\\
	(\frac{k-\lambda^2}{\lambda(1-\lambda)})^2&x=3k+2\\
	0&x\in \mathbb{N}\times\mathbb{N}\bigcup(3\mathbb{N}+1)\bigcup\{0,2\}
\end{cases},$$
Hence, by Theorem \ref{combasic} (iii), $C_{\phi_{2},\omega_{2}}$ is densely defined. By \eqref{Efphi} and the above equality, we obtain
\begin{equation}
	\mathsf{E}_{\phi_{2},\omega_{2}}(\mathsf{h}_{\phi_{2},\omega_{2}})\circ\phi_{2}^{-1}(x)=\frac{\sum\limits_{y\in\phi_{2}^{-1}(x)}\mathsf{h}_{\phi_{2},\omega_{2}}(y)|\omega_{2}(y)|^2}{\mathsf{h}_{\phi_{2},\omega_{2}}(x)}=\begin{cases}
		1&x=1\\
		(\frac{k-\lambda^2}{\lambda(1-\lambda)})^2&x=3k\\
		0&x\in X\backslash(\{1\}\bigcup3\mathbb{N})
	\end{cases},
\end{equation}
which implies, by Remark \ref{singlemult} and Corollary \ref{denseequ} (iv), that $\mathcal{M}_{\lambda}(C_{\phi_{2},\omega_{2}})$ is densely defined. Let $\omega_{2}^\lambda=((1-\lambda)(\frac{\chi_{\{\omega_{2}\ne0\}}\mathsf{h}_{\phi_{2},\omega_{2}}}{\mathsf{h}_{\phi_{2},\omega_{2}}\circ\phi_{2}})^\frac{1}{2}+\lambda)\omega_{2}$, then $$\omega_{2}^\lambda(x)=\begin{cases}
	((1-\lambda)\frac{1}{\sqrt{\sum\limits_{k=1}^\infty\frac{1}{k^2}}}+\lambda)\frac{1}{k}&x=3k\\
	\frac{k-\lambda^2}{1-\lambda}&x=3k+1\\
\frac{k}{\lambda}&x=3k+2\\
0&x\in \{0,1,2\}\bigcup\mathbb{N}\times\mathbb{N}
\end{cases}.$$
On the other hand, since \begin{align*}
	(\mathcal{M}_{\lambda}(C_{\phi_{2},\omega_{2}})^2)f\xlongequal{\text{Corollary \ref{singlemeanbasic} (iv)}}C_{\phi_{2},\omega_{2}^\lambda}^2f&=C_{\phi_{2},\omega_{2}^\lambda}\omega_{2}^\lambda\cdot f\circ\phi_{2}=\omega_{2}^\lambda\cdot\omega_{2}^\lambda\circ\phi_{2}\cdot f\circ\phi_{2}\circ\phi_{2}\\&=C_{\phi_{2}\circ\phi_{2},\omega_{2}^\lambda\cdot\omega_{2}^\lambda\circ\phi_{2}}f
\end{align*}
for any $f\in\EuScript{D}(\mathcal{M}_{\lambda}(C_{\phi_{2},\omega_{2}})^2)$, we have $\EuScript{D}(\mathcal{M}_{\lambda}(C_{\phi_{2},\omega_{2}})^2)\subseteq\EuScript{D}(C_{\phi_{2}\circ\phi_{2},\omega_{2}^\lambda\cdot\omega_{2}^\lambda\circ\phi_{2}})$. If we can prove $$\mathsf{h}_{\phi_{2}\circ\phi_{2},\omega_{2}^\lambda\cdot\omega_{2}^\lambda\circ\phi_{2}}(1)=\infty,$$ then by Theorem \ref{combasic} (iii) we have $$\overline{\EuScript{D}(\mathcal{M}_{\lambda}(C_{\phi_{2},\omega_{2}})^2)}\subseteq\overline{\EuScript{D}(C_{\phi_{2}\circ\phi_{2},\omega_{2}^\lambda\cdot\omega_{2}^\lambda\circ\phi_{2}})}\ne L^2(\mu),$$ that is, $\mathcal{M}_{\lambda}(C_{\phi_{2},\omega_{2}})^2$ is not densely defined. By \eqref{h}, we obtain\begin{align*}
\mathsf{h}_{\phi_{2}\circ\phi_{2},\omega_{2}^\lambda\cdot\omega_{2}^\lambda\circ\phi_{2}}(1)&=\sum\limits_{y\in(\phi_{2}\circ\phi_{2})^{-1}(1)}|\omega_{2}^\lambda\cdot\omega_{2}^\lambda\circ\phi_{2}|^2(y)=\sum\limits_{k=1}^\infty|\omega_{2}^\lambda(3k+2)\cdot\omega_{2}^\lambda\circ\phi_{2}(3k+2)|^2\\&=\sum\limits_{k=1}^\infty|\frac{k}{\lambda}((1-\lambda)\frac{1}{\sqrt{\sum\limits_{k=1}^\infty\frac{1}{k^2}}}+\lambda)\frac{1}{k}|^2=\sum\limits_{k=1}^\infty((1-\lambda)\frac{1}{\sqrt{\sum\limits_{k=1}^\infty\frac{1}{k^2}}}+\lambda)=\infty.
\end{align*}
Therefore, we get the desired result.

\textbf{(c)} By \eqref{h}, we have$$\mathsf{h}_{\phi_{3},\omega_{3}}(x)=\sum\limits_{y\in\phi_{3}^{-1}(x)}|\omega_{3}(y)|^2=\begin{cases}
	\sum\limits_{k=1}^\infty\frac{1}{k^2}&x=0\\
	1&x\in\mathbb{N}\times\{1\}\\
	\frac{1}{n^2}&x=(n,2)\in\mathbb{N}\times\{2\}\\
0&x\in X\backslash\Big(\mathbb{N}\times\{1\}\bigcup\mathbb{N}\times\{2\}\bigcup\{0\}\Big)
\end{cases}$$
and$$\mathsf{h}_{\phi_{4},\omega_{4}}(x)=\sum\limits_{y\in\phi_{4}^{-1}(x)}|\omega_{4}(y)|^2=\begin{cases}
	n&x=(n,1)\in\mathbb{N}\times\{1\}\\
	0&x\notin\mathbb{N}\times\{1\}
\end{cases}.$$
Combining \eqref{Efphi} and the above equalities, we obtain
\begin{equation*}
	\mathsf{E}_{\phi_{3},\omega_{3}}(\mathsf{h}_{\phi_{3},\omega_{3}})\circ\phi_{3}^{-1}(x)=\frac{\sum\limits_{y\in\phi_{3}^{-1}(x)}\mathsf{h}_{\phi_{3},\omega_{3}}(y)|\omega_{3}(y)|^2}{\mathsf{h}_{\phi_{3},\omega_{3}}(x)}=\begin{cases}
		0&x=n\in\mathbb{N}\\
		1&x=0\\
		|\omega_{3}(m,n+1)|^2&x=(m,n-1)\in\mathbb{N}\times\mathbb{N}
	\end{cases},
\end{equation*}
\begin{equation}\label{countermean}
	\mathsf{E}_{\phi_{3},\omega_{3}}(\mathsf{h}_{\phi_{4},\omega_{4}})\circ\phi_{3}^{-1}(0)=\frac{\sum\limits_{k}^\infty\mathsf{h}_{\phi_{4},\omega_{4}}(k,1)|\omega_{3}(k,1)|^2}{\mathsf{h}_{\phi_{3},\omega_{3}}(0)}=\frac{\sum\limits_{k=1}^\infty\frac{1}{k}}{\sum\limits_{k=1}^\infty\frac{1}{k^2}}=\infty,
\end{equation}
and
\begin{equation*}
	\mathsf{E}_{\phi_{4},\omega_{4}}(\mathsf{h}_{\phi_{4},\omega_{4}})\circ\phi_{4}^{-1}(x)=\frac{\sum\limits_{y\in\phi_{4}^{-1}(x)}\mathsf{h}_{\phi_{4},\omega_{4}}(y)|\omega_{4}(y)|^2}{\mathsf{h}_{\phi_{4},\omega_{4}}(x)}=0.
\end{equation*}
It follows from \eqref{countermean} and Corollary \eqref{denseequ} (iv) that $\mathcal{M}_{\lambda}(C_{\phi_{3},\omega_{3}},C_{\phi_{4},\omega_{4}})$ is not densely defined. From the proof of Proposition \ref{pointwisemean}, we know that $\reallywidehat{\mathcal{M}_{\lambda}(C_{\phi_{3},\omega_{3}},C_{\phi_{4},\omega_{4}})}$ is densely defined since $\mathsf{E}_{\phi_{3},\omega_{3}}(\mathsf{h}_{\phi_{3},\omega_{3}})\circ\phi_{3}^{-1}<\infty$ and $\mathsf{E}_{\phi_{4},\omega_{4}}(\mathsf{h}_{\phi_{4},\omega_{4}})\circ\phi_{4}^{-1}<\infty$.

The proof is complete.
\end{proof}
\end{exa}

\section{generalized normality of $\textbf{C}_{\phi,\omega}$}
In the final section, we investigate the generalized normality of $\textbf{C}_{\phi,\omega}$, specifically its spherical quasinormality and spherical $p$-hyponormality. We start with the following observation.

Suppose $C_{\phi_{1},\omega_{1}}$ and $C_{\phi_{2},\omega_{2}}$ are densely defined, that is, $\mathsf{h}_{\phi_{i},\omega_{i}}<\infty$ a.e. $[\mu]$ for $i=1,2$. Then \begin{align}\label{sumfinite}
	\mathsf{h}_{\phi_{1},\omega_{1}}+\mathsf{h}_{\phi_{2},\omega_{2}}<\infty\quad \text{a.e.}\ [\mu_{\omega_{i}}]
\end{align}
since $\{\mathsf{h}_{\phi_{1},\omega_{1}}+\mathsf{h}_{\phi_{2},\omega_{2}}=\infty\}\subset\bigcup\limits_{i=1}^2\{\mathsf{h}_{\phi_{i},\omega_{i}}=\infty\}$ and $\mu_{\omega_{i}}$ is absolutely continuous with respect to $\mu$ for $i=1,2$. Corollary \ref{denseequ} implies that the dense definiteness is not necessarily preserved under the $\lambda$-spherical mean transform of $(C_{\phi_{1},\omega_{1}},C_{\phi_{2},\omega_{2}})$. However, inspecting Corollary \ref{denseequ} (viii) in view of \eqref{sumfinite}, we find that the natural condition
\begin{align}\label{fix}
	\mathsf{h}_{\phi_{1},\omega_{1}}+\mathsf{h}_{\phi_{2},\omega_{2}}=(\mathsf{h}_{\phi_{1},\omega_{1}}+\mathsf{h}_{\phi_{2},\omega_{2}})\circ\phi_{i}\quad \text{a.e.}\ [\mu_{\omega_{i}}]
\end{align}
immediately yields that
\begin{align*}
	\mathsf{E}_{\phi_{i},\omega_{i}}(\mathsf{h}_{\phi_{1},\omega_{1}}+\mathsf{h}_{\phi_{2},\omega_{2}})=\mathsf{h}_{\phi_{1},\omega_{1}}+\mathsf{h}_{\phi_{2},\omega_{2}}<\infty\quad \text{a.e.}\ [\mu_{\omega_{i}}]
\end{align*}
for $i=1,2$, that is, $\mathcal{M}_{\lambda}(\textbf{C}_{\phi,\omega})$ is densely defined for $\lambda\in(0,1)$. Moreover, recalling \eqref{muomega1}, one can easily check that Equality \eqref{fix} holds if and only if
\begin{equation}\label{equivquasinormal}
	\mathsf{h}_{\phi_{1},\omega_{1}}+\mathsf{h}_{\phi_{2},\omega_{2}}=(\mathsf{h}_{\phi_{1},\omega_{1}}+\mathsf{h}_{\phi_{2},\omega_{2}})\circ\phi_{i}\quad \text{a.e.}\ [\mu]\quad\text{on}\ \{\omega_{i}\ne0\}
\end{equation}
since the measure $\mu^{\omega_{i}}$ and $\mu_{\omega_{i}}$ are mutuallly absolutely continuous. Thus, by \eqref{wfun} we have
$$\omega_{\lambda}^i=((1-\lambda)\chi_{\{\omega_{i}\ne0\}}+\lambda)\omega_{i}=\omega_{i}\quad\text{a.e.}\ [\mu]$$
for $i=1,2$ if \eqref{fix} holds, which implies, by Theorem \ref{meanbasic} (iv), that $$\mathcal{M}_{\lambda}(\textbf{C}_{\phi,\omega})=(C_{\phi_{1},\omega_{1}},C_{\phi_{2},\omega_{2}})=\textbf{C}_{\phi,\omega}.$$
In fact, we have the following.

\begin{thm}\label{quasinormal}
	Suppose $C_{\phi_{1},\omega_{1}}$ and $C_{\phi_{2},\omega_{2}}$ are densely defined. Fix $\lambda\in(0,1)$. Then $\mathcal{M}_{\lambda}(\textbf{C}_{\phi,\omega})=\textbf{C}_{\phi,\omega}$ if and only if $\mathsf{h}_{\phi_{1},\omega_{1}}+\mathsf{h}_{\phi_{2},\omega_{2}}=(\mathsf{h}_{\phi_{1},\omega_{1}}+\mathsf{h}_{\phi_{2},\omega_{2}})\circ\phi_{i}$ a.e. $[\mu_{\omega_{i}}]$ for $i=1,2$.
\end{thm}
\begin{proof}
	{\em Necessity}. This follows from the above discussion.
	
	{\em Sufficiency}. Assume $\mathcal{M}_{\lambda}(\textbf{C}_{\phi,\omega})=\textbf{C}_{\phi,\omega}$. Using Proposition \ref{basic} (ii), we have $\mathcal{M}_{\lambda}(\textbf{C}_{\phi,\omega})$ is densely defined. Thus by Corollary \ref{denseequ} (iii) we obtain $$\mathsf{E}_{\phi_{i},\omega_{i}}(\mathsf{h}_{\phi_{1},\omega_{1}}+\mathsf{h}_{\phi_{2},\omega_{2}})\circ\phi_{i}^{-1}<\infty\quad \text{a.e.}\ [\mu]$$ for $i=1,2$. Recall \eqref{Eequation} and $\mathsf{h}_{\phi_{i},\omega_{i}}<\infty$ a.e. $[\mu]$ for $i=1,2$, then by the above inequality we have $$1+\sum\limits_{i=1}^2(\mathsf{h}_{\phi_{i},\omega_{i}}+\mathsf{E}_{i})<\infty\quad\text{a.e.}\ [\mu].$$ By \cite[Lemma 9]{b-j-j-sW}, there exists an increasing sequence $\{X_{n}\}\subset\mathscr{A}$ such that $\mu(X_{n})<\infty$, $1+\sum\limits_{i=1}^2(\mathsf{h}_{\phi_{i},\omega_{i}}+\mathsf{E}_{i})\le n$ a.e. $[\mu]$ on $X_{n}$ for $n\in\mathbb{N}$ and $\bigcup\limits_{n=1}^\infty X_{n}=X$. Hence, $$\chi_{X_{n}}\in L^2((1+\sum\limits_{i=1}^2(\mathsf{h}_{\phi_{i},\omega_{i}}+\mathsf{E}_{i}))\D\mu)\xlongequal{\text{Theorem}\ \ref{meanbasic}\ (i)}\EuScript{D}(\mathcal{M}_{\lambda}(\textbf{C}_{\phi,\omega})).$$
	By assumption, we have that $$C_{\phi_{i},\omega_{\lambda}^i}\chi_{X_{n}}=C_{\phi_{i},\omega_{i}}\chi_{X_{n}}\Rightarrow\omega_{\lambda}^i\cdot \chi_{X_{n}}\circ\phi_{i}=\omega_{i}\cdot \chi_{X_{n}}\circ\phi_{i}\quad\text{a.e.}\ [\mu]$$ for $n\in\mathbb{N}$ and $i\in\{1,2\}$. Since $\phi_{i}^{-1}(X_{n})\to X$ as $n\to\infty$, then by the above inequality we have
	\begin{align*}
		&\omega_{\lambda}^i=\omega_{i}\quad\text{a.e.}\ [\mu]\xlongequal{\eqref{wfun}}((1-\lambda)\chi_{\{\omega_{i}\ne0\}}\left(\frac{\mathsf{h}_{\phi_{1},\omega_{1}}+\mathsf{h}_{\phi_{2},\omega_{2}}}{(\mathsf{h}_{\phi_{1},\omega_{1}}+\mathsf{h}_{\phi_{2},\omega_{2}})\circ\phi_{i}}\right)^\frac{1}{2}+\lambda)\omega_{i}=\omega_{i}\quad\text{a.e.}\ [\mu]\\&=(1-\lambda)\chi_{\{\omega_{i}\ne0\}}\left(\frac{\mathsf{h}_{\phi_{1},\omega_{1}}+\mathsf{h}_{\phi_{2},\omega_{2}}}{(\mathsf{h}_{\phi_{1},\omega_{1}}+\mathsf{h}_{\phi_{2},\omega_{2}})\circ\phi_{i}}\right)^\frac{1}{2}+\lambda=1\quad\text{a.e.}\ [\mu]\quad\text{on}\ \{\omega_{i}\ne0\}\\&=\mathsf{h}_{\phi_{1},\omega_{1}}+\mathsf{h}_{\phi_{2},\omega_{2}}=(\mathsf{h}_{\phi_{1},\omega_{1}}+\mathsf{h}_{\phi_{2},\omega_{2}})\circ\phi_{i}\quad\text{a.e.}\ [\mu]\quad\text{on} \{\omega_{i}\ne0\}\\&\stackrel{\eqref{equivquasinormal}}{=}\mathsf{h}_{\phi_{1},\omega_{1}}+\mathsf{h}_{\phi_{2},\omega_{2}}=(\mathsf{h}_{\phi_{1},\omega_{1}}+\mathsf{h}_{\phi_{2},\omega_{2}})\circ\phi_{i}\quad\text{a.e.}\ [\mu_{\omega_{i}}]
	\end{align*}
	for $i=1,2$. Therefore, the proof is complete.
\end{proof}

Suppose $C_{\phi_{1},\omega_{1}}$ and $C_{\phi_{2},\omega_{2}}$ are densely defined. Given $\lambda\in(0,1)$ and $n\in\mathbb{N}$. Since the $\lambda$-spherical mean transform maps a pair of weighted composition operators to another, we can define the {\em iterated $\lambda$-spherical mean transform} of the pair $\textbf{C}_{\phi,\omega}=(C_{\phi_{1},\omega_{1}},C_{\phi_{2},\omega_{2}})$ by$$\mathcal{M}_{\lambda}^0(\textbf{C}_{\phi,\omega})=\textbf{C}_{\phi,\omega}\quad\text{and}\quad\mathcal{M}_{\lambda}^n(\textbf{C}_{\phi,\omega})=\mathcal{M}_{\lambda}(\mathcal{M}_{\lambda}^{n-1}(\textbf{C}_{\phi,\omega}))$$
if $\mathcal{M}_{\lambda}^{n-1}(\textbf{C}_{\phi,\omega})$ is densely defined. 

This assumption in the  above definition cannot be dropped. Indeed, let $n=2$. As mentioned above, the $\lambda$-spherical mean transform $\mathcal{M}_{\lambda}(\textbf{C}_{\phi,\omega})$ of the pair $\textbf{C}_{\phi,\omega}=(C_{\phi_{1},\omega_{1}},C_{\phi_{2},\omega_{2}})$ is not necessarily densely defined. However, the equality $\mathcal{M}_{\lambda}^2(\textbf{C}_{\phi,\omega})=\mathcal{M}_{\lambda}(\mathcal{M}_{\lambda}(\textbf{C}_{\phi,\omega}))$ requires the polar decomposition of $\mathcal{M}_{\lambda}(\textbf{C}_{\phi,\omega})$, for which the dense definiteness of $\mathcal{M}_{\lambda}(\textbf{C}_{\phi,\omega})$ is a necessary condition. Therefore, an arbitrary weighted composition operator pair need not admit the $n$-th $\lambda$-spherical mean transform for some $n\in\mathbb{N}$. Nevertheless, the above theorem provides a class of weighted composition operator pairs whose $n$-th $\lambda$-spherical mean transform exists for all $n\in\mathbb{N}$. More precisely, if $\mathcal{M}_{\lambda}(\textbf{C}_{\phi,\omega})=\textbf{C}_{\phi,\omega}$, then
$$\textbf{C}_{\phi,\omega}=\mathcal{M}_{\lambda}(\textbf{C}_{\phi,\omega})=\mathcal{M}_{\lambda}(\mathcal{M}_{\lambda}(\textbf{C}_{\phi,\omega}))=\mathcal{M}_{\lambda}(\mathcal{M}_{\lambda}(\mathcal{M}_{\lambda}(\textbf{C}_{\phi,\omega})))=\cdots=\mathcal{M}_{\lambda}^n(\textbf{C}_{\phi,\omega})$$ for $n\in\mathbb{N}$. Furthermore, using Theorem \ref{quasinormal}, we can readily give operator-theoretic characterizations of the equality $\mathcal{M}_{\lambda}(\textbf{C}_{\phi,\omega})=\textbf{C}_{\phi,\omega}$.
\begin{cor}\label{unboundedquasinormal}
	Suppose  $C_{\phi_{1},\omega_{1}}$ and $C_{\phi_{2},\omega_{2}}$ are densely defined. Let $\textbf{C}_{\phi,\omega}=\left(
	\begin{array}{c}
		U_{1} \\
		U_{2}
	\end{array}
	\right)|\textbf{C}_{\phi,\omega}|$ be the polar decomposition of the pair $\textbf{C}_{\phi,\omega}=(C_{\phi_{1},\omega_{1}},C_{\phi_{2},\omega_{2}})$. Fix $\lambda\in (0,1)$. Then the following assertions are equivalent$:$
	\begin{enumerate}
		\item[(i)]
		$\mathcal{M}_{\lambda}(\textbf{C}_{\phi,\omega})=\textbf{C}_{\phi,\omega}$,
		\item[(ii)]
		$U_{i}|\textbf{C}_{\phi,\omega}|\subseteq|\textbf{C}_{\phi,\omega}|U_{i}$ for $i=1,2$,
		\item[(iii)]
		$C_{\phi_{i},\omega_{i}}|\textbf{C}_{\phi,\omega}|^2\subseteq|\textbf{C}_{\phi,\omega}|^2C_{\phi_{i},\omega_{i}}$ for $i=1,2$.
		\item[(iv)]
		$C_{\phi_{i},\omega_{i}}(C_{\phi_{1},\omega_{1}}^\ast C_{\phi_{1},\omega_{1}}+C_{\phi_{2},\omega_{2}}^\ast C_{\phi_{2},\omega_{2}})\subseteq (C_{\phi_{1},\omega_{1}}^\ast C_{\phi_{1},\omega_{1}}+C_{\phi_{2},\omega_{2}}^\ast C_{\phi_{2},\omega_{2}})C_{\phi_{i},\omega_{i}}$ for $i=1,2$.
		
	\end{enumerate}
\end{cor}
\begin{proof}
	(i) $\Rightarrow$ (ii): Recalling \eqref{Eequation}, we have
	\begin{align*}
		\mathsf{E}_{i}&=\Big(\mathsf{E}_{\phi_{i},\omega_{i}}\big(\mathsf{h}_{\phi_{1},\omega_{1}}+\mathsf{h}_{\phi_{2},\omega_{2}}\big)\circ\phi_{i}^{-1}\Big)\cdot\frac{\mathsf{h}_{\phi_{i},\omega_{i}}\cdot\chi_{\{\mathsf{h}_{\phi_{1},\omega_{1}}+\mathsf{h}_{\phi_{2},\omega_{2}}>0\}}}{\mathsf{h}_{\phi_{1},\omega_{1}}+\mathsf{h}_{\phi_{2},\omega_{2}}}\\&\xlongequal{\text{Theorem}\ \ref{quasinormal}}\Big(\mathsf{E}_{\phi_{i},\omega_{i}}\big((\mathsf{h}_{\phi_{1},\omega_{1}}+\mathsf{h}_{\phi_{2},\omega_{2}})\circ\phi_{i}\big)\circ\phi_{i}^{-1}\Big)\cdot\frac{\mathsf{h}_{\phi_{i},\omega_{i}}\cdot\chi_{\{\mathsf{h}_{\phi_{1},\omega_{1}}+\mathsf{h}_{\phi_{2},\omega_{2}}>0\}}}{\mathsf{h}_{\phi_{1},\omega_{1}}+\mathsf{h}_{\phi_{2},\omega_{2}}}\\&=(\mathsf{h}_{\phi_{1},\omega_{1}}+\mathsf{h}_{\phi_{2},\omega_{2}})\cdot\frac{\mathsf{h}_{\phi_{i},\omega_{i}}\cdot\chi_{\{\mathsf{h}_{\phi_{1},\omega_{1}}+\mathsf{h}_{\phi_{2},\omega_{2}}>0\}}}{\mathsf{h}_{\phi_{1},\omega_{1}}+\mathsf{h}_{\phi_{2},\omega_{2}}}=\mathsf{h}_{\phi_{i},\omega_{i}}\quad\text{a.e.}\ [\mu].
	\end{align*}
	Hence, $$\EuScript{D}(|\textbf{C}_{\phi,\omega}|U_{i})\stackrel{\eqref{quasinormalchar}}{=}L^2((1+\mathsf{E}_{i})\D\mu)=L^2((1+\mathsf{h}_{\phi_{i},\omega_{i}})\D\mu),$$ which yields 
	\begin{align*}
		\EuScript{D}(U_{i}|\textbf{C}_{\phi,\omega}|)=\EuScript{D}(|\textbf{C}_{\phi,\omega}|)=\EuScript{D}(\textbf{C}_{\phi,\omega})&\xlongequal{\text{Proposition}\ \ref{basic}\ (i)}L^2((1+\mathsf{h}_{\phi_{1},\omega_{1}}+\mathsf{h}_{\phi_{2},\omega_{2}})\D\mu)\\&\subseteq L^2((1+\mathsf{h}_{\phi_{i},\omega_{i}})\D\mu)=\EuScript{D}(|\textbf{C}_{\phi,\omega}|U_{i}).
	\end{align*}
	On the other hand, a direct calculation shows that\begin{align*}
		U_{i}|\textbf{C}_{\phi,\omega}|f&\xlongequal{\text{Theorem}\ \ref{polar}}C_{\phi_{i},\widetilde{\omega_{i}}}\mathsf{M}_{\sqrt{\mathsf{h}_{\phi_{1},\omega_{1}}+\mathsf{h}_{\phi_{2},\omega_{2}}}}f=\widetilde{\omega_{i}}\cdot \sqrt{\mathsf{h}_{\phi_{1},\omega_{1}}+\mathsf{h}_{\phi_{2},\omega_{2}}}\circ\phi_{i}\cdot f\circ\phi_{i}\\&\xlongequal{\text{Theorem} \ref{quasinormal}}\sqrt{\mathsf{h}_{\phi_{1},\omega_{1}}+\mathsf{h}_{\phi_{2},\omega_{2}}}\cdot\widetilde{\omega_{i}}\cdot f\circ\phi_{i}=\mathsf{M}_{\sqrt{\mathsf{h}_{\phi_{1},\omega_{1}}+\mathsf{h}_{\phi_{2},\omega_{2}}}}C_{\widetilde{\omega_{i}},\phi_{i}}f\\&\xlongequal{\text{Theorem}\ \ref{polar}}|\textbf{C}_{\phi,\omega}|U_{i}f
	\end{align*}	
	for $f\in\EuScript{D}(U_{i}|\textbf{C}_{\phi,\omega}|)$. Therefore, we obtain $U_{i}|\textbf{C}_{\phi,\omega}|\subseteq|\textbf{C}_{\phi,\omega}|U_{i}$ for $i=1,2$.
	
	(ii) $\Rightarrow$ (i):	The fact $\mathsf{h}_{\phi_{1},\omega_{1}},\mathsf{h}_{\phi_{2},\omega_{2}}<\infty$ a.e. $[\mu]$ implies $1+\sum\limits_{i=1}^2\mathsf{h}_{\phi_{i},\omega_{i}}<\infty$ a.e. By \cite[Lemma 9]{b-j-j-sW}, there exists an increasing sequence $\{X_{n}\}\subset\mathscr{A}$ such that $\mu(X_{n})<\infty$, $1+\sum\limits_{i=1}^2\mathsf{h}_{\phi_{i},\omega_{i}}\le n$ a.e. $[\mu]$ on $X_{n}$ for $n\in\mathbb{N}$ and $\bigcup\limits_{n=1}^\infty X_{n}=X$. Hence, $$\chi_{X_{n}}\in L^2((1+\sum\limits_{i=1}^2\mathsf{h}_{\phi_{i},\omega_{i}})\D\mu)\xlongequal{\text{Proposition}\ \ref{basic}\ (i)}\EuScript{D}(\textbf{C}_{\phi,\omega})=\EuScript{D}(U_{i}|\textbf{C}_{\phi,\omega}|).$$ By assumption, we have 
	\begin{align*}
		&U_{i}|\textbf{C}_{\phi,\omega}|\chi_{X_{n}}=|\textbf{C}_{\phi,\omega}|U_{i}\chi_{X_{n}}\\&\xLongrightarrow{\text{Theorem}\ \ref{polar}}\widetilde{\omega_{i}}\cdot \sqrt{\mathsf{h}_{\phi_{1},\omega_{1}}+\mathsf{h}_{\phi_{2},\omega_{2}}}\circ\phi_{i}\cdot \chi_{X_{n}}\circ\phi_{i}=\sqrt{\mathsf{h}_{\phi_{1},\omega_{1}}+\mathsf{h}_{\phi_{2},\omega_{2}}}\cdot\widetilde{\omega_{i}}\cdot \chi_{X_{n}}\circ\phi_{i}\quad \text{a.e.}\ [\mu]\\&\Rightarrow\omega_{i}\cdot\chi_{\{\omega_{i}\ne0\}}\cdot\chi_{X_{n}}\circ\phi_{i}=\sqrt{\frac{\mathsf{h}_{\phi_{1},\omega_{1}}+\mathsf{h}_{\phi_{2},\omega_{2}}}{(\mathsf{h}_{\phi_{1},\omega_{1}}+\mathsf{h}_{\phi_{2},\omega_{2}})\circ\phi_{i}}}\cdot\omega_{i}\cdot\chi_{\{\omega_{i}\ne0\}}\cdot\chi_{X_{n}}\circ\phi_{i}\quad\text{a.e.}\ [\mu].
	\end{align*}
	Since $\phi_{i}^{-1}(X_{n})\to X$ as $n\to\infty$, then by the above inequality we have
	\begin{align*}
		&\sqrt{\frac{\mathsf{h}_{\phi_{1},\omega_{1}}+\mathsf{h}_{\phi_{2},\omega_{2}}}{(\mathsf{h}_{\phi_{1},\omega_{1}}+\mathsf{h}_{\phi_{2},\omega_{2}})\circ\phi_{i}}}=1\quad\text{a.e.}\ [\mu]\quad\text{on}\ \{\omega_{i}\ne0\}\\&\stackrel{\eqref{equivquasinormal}}{=}\mathsf{h}_{\phi_{1},\omega_{1}}+\mathsf{h}_{\phi_{2},\omega_{2}}=(\mathsf{h}_{\phi_{1},\omega_{1}}+\mathsf{h}_{\phi_{2},\omega_{2}})\circ\phi_{i}\quad\text{a.e.}\ [\mu_{\omega_{i}}]
	\end{align*}
	for $i=1,2$. Applying Theorem \ref{quasinormal}, we conclude the proof.
	
	(i) $\Rightarrow$ (iii): By Theorem \ref{polar} (i) and Lemma \ref{mulsp} (ii) we have
	\begin{equation}\label{quasimultiplication}
		|\textbf{C}_{\phi,\omega}|^2=\mathsf{M}_{\mathsf{h}_{\phi_{1},\omega_{1}}+\mathsf{h}_{\phi_{2},\omega_{2}}}.
	\end{equation}
	Then, for $f\in\EuScript{D}(C_{\phi_{i},\omega_{i}}|\textbf{C}_{\phi,\omega}|^2)\bigcap\EuScript{D}(|\textbf{C}_{\phi,\omega}|^2C_{\phi_{i},\omega_{i}})$, we obtain\begin{align*}
		C_{\phi_{i},\omega_{i}}|\textbf{C}_{\phi,\omega}|^2f&=C_{\phi_{i},\omega_{i}}\mathsf{M}_{\mathsf{h}_{\phi_{1},\omega_{1}}+\mathsf{h}_{\phi_{2},\omega_{2}}}f=\omega_{i}\cdot (\mathsf{h}_{\phi_{1},\omega_{1}}+\mathsf{h}_{\phi_{2},\omega_{2}})\circ\phi\cdot f\circ\phi_{i}\quad\text{a.e.}\ [\mu]\\&\xlongequal{\text{Theorem}\ \ref{quasinormal}\ \text{and}\ \eqref{equivquasinormal}}(\mathsf{h}_{\phi_{1},\omega_{1}}+\mathsf{h}_{\phi_{2},\omega_{2}})\cdot\omega_{i}\cdot f\circ\phi_{i}\quad\text{a.e.}\ [\mu]\\&=\mathsf{M}_{\mathsf{h}_{\phi_{1},\omega_{1}}+\mathsf{h}_{\phi_{2},\omega_{2}}}C_{\phi_{i},\omega_{i}}f=|\textbf{C}_{\phi,\omega}|^2C_{\phi_{i},\omega_{i}}f.
	\end{align*}
	Thus, it suffices to prove that $\EuScript{D}(C_{\phi_{i},\omega_{i}}|\textbf{C}_{\phi,\omega}|^2)\subseteq\EuScript{D}(|\textbf{C}_{\phi,\omega}|^2C_{\phi_{i},\omega_{i}})$. We first claim that
	\begin{equation}
		\EuScript{D}(C_{\phi_{i},\omega_{i}}|\textbf{C}_{\phi,\omega}|^2)=L^2((1+(1+\mathsf{h}_{\phi_{i},\omega_{i}})(\mathsf{h}_{\phi_{1},\omega_{1}}+\mathsf{h}_{\phi_{2},\omega_{2}})^2)\D\mu).
	\end{equation}
	Since
	\begin{equation}\label{quasirnde}
		\begin{aligned}
			\mu_{\omega_{i}\cdot(\mathsf{h}_{\phi_{1},\omega_{1}}+\mathsf{h}_{\phi_{2},\omega_{2}})\circ\phi_{i}}\circ\phi_{i}^{-1}(\Delta)&\stackrel{\eqref{rd}}{=}\int_{\phi_{i}^{-1}(\Delta)}|\omega_{i}\cdot(\mathsf{h}_{\phi_{1},\omega_{1}}+\mathsf{h}_{\phi_{2},\omega_{2}})\circ\phi_{i}|^2\D\mu\\&=\int_{X}\chi_{\Delta}\circ\phi_{i}\cdot|\omega_{i}\cdot(\mathsf{h}_{\phi_{1},\omega_{1}}+\mathsf{h}_{\phi_{2},\omega_{2}})\circ\phi_{i}|^2\D\mu\\&=\int_{X}(\chi_{\Delta}\cdot (\mathsf{h}_{\phi_{1},\omega_{1}}+\mathsf{h}_{\phi_{2},\omega_{2}})^2)\circ\phi_{i}\D\mu_{\omega_{i}}\\&\stackrel{\eqref{intcom}}{=}\int_{\Delta}(\mathsf{h}_{\phi_{1},\omega_{1}}+\mathsf{h}_{\phi_{2},\omega_{2}})^2\mathsf{h}_{\phi_{i},\omega_{i}}\D\mu
		\end{aligned}
	\end{equation}
	for any $\Delta\in\mathcal{A}$, then we have
	\begin{equation}\label{quasirn}
		\mu_{\omega_{i}\cdot(\mathsf{h}_{\phi_{1},\omega_{1}}+\mathsf{h}_{\phi_{2},\omega_{2}})\circ\phi_{i}}\circ\phi_{i}^{-1}\ll\mu
	\end{equation}
	which implies that the weighted composition operator $C_{\phi_{i},\omega_{i}\cdot(\mathsf{h}_{\phi_{1},\omega_{1}}+\mathsf{h}_{\phi_{2},\omega_{2}})\circ\phi_{i}}$ is well-defined. Combining \eqref{quasirn} and \eqref{quasirnde} with \eqref{rd}, we obtain
	\begin{equation}\label{quasirnderivate}
		\mathsf{h}_{\phi_{i},\omega_{i}\cdot(\mathsf{h}_{\phi_{1},\omega_{1}}+\mathsf{h}_{\phi_{2},\omega_{2}})\circ\phi_{i}}=(\mathsf{h}_{\phi_{1},\omega_{1}}+\mathsf{h}_{\phi_{2},\omega_{2}})^2\mathsf{h}_{\phi_{i},\omega_{i}}.
	\end{equation}
	By \eqref{quasimultiplication}, we derive $$C_{\phi_{i},\omega_{i}}|\textbf{C}_{\phi,\omega}|^2f=C_{\phi_{i},\omega_{i}}\mathsf{M}_{\mathsf{h}_{\phi_{1},\omega_{1}}+\mathsf{h}_{\phi_{2},\omega_{2}}}f=\omega_{i}\cdot(\mathsf{h}_{\phi_{1},\omega_{1}}+\mathsf{h}_{\phi_{2},\omega_{2}})\circ\phi_{i}\cdot f\circ\phi=C_{\phi_{i},\omega_{i}\cdot(\mathsf{h}_{\phi_{1},\omega_{1}}+\mathsf{h}_{\phi_{2},\omega_{2}})\circ\phi_{i}}f$$
	for $f\in\EuScript{D}(C_{\phi_{i},\omega_{i}}|\textbf{C}_{\phi,\omega}|^2)$, which yields that
	\begin{equation}\label{domainquasi}
		\begin{aligned}
			\EuScript{D}(C_{\phi_{i},\omega_{i}}|\textbf{C}_{\phi,\omega}|^2)&=\EuScript{D}(C_{\phi_{i},\omega_{i}}\mathsf{M}_{\mathsf{h}_{\phi_{1},\omega_{1}}+\mathsf{h}_{\phi_{2},\omega_{2}}})=\EuScript{D}(C_{\phi_{i},\omega_{i}\cdot(\mathsf{h}_{\phi_{1},\omega_{1}}+\mathsf{h}_{\phi_{2},\omega_{2}})\circ\phi_{i}})\bigcap\EuScript{D}(\mathsf{M}_{\mathsf{h}_{\phi_{1},\omega_{1}}+\mathsf{h}_{\phi_{2},\omega_{2}}})\\&\xlongequal{\text{Proposition}\ \ref{basic}\ (i)}L^2((1+\mathsf{h}_{\phi_{i},\omega_{i}\cdot(\mathsf{h}_{\phi_{1},\omega_{1}}+\mathsf{h}_{\phi_{2},\omega_{2}})\circ\phi_{i}}+(\mathsf{h}_{\phi_{1},\omega_{1}}+\mathsf{h}_{\phi_{2},\omega_{2}})^2)\D\mu)\\&\stackrel{\eqref{quasirnderivate}}{=}L^2((1+(1+\mathsf{h}_{\phi_{i},\omega_{i}})(\mathsf{h}_{\phi_{1},\omega_{1}}+\mathsf{h}_{\phi_{2},\omega_{2}})^2)\D\mu),
		\end{aligned}
	\end{equation}
	as desired. On the other hand, since
	\begin{equation}\label{quasirn2}
		\begin{aligned}
			&\mu_{\omega_{i}\cdot(\mathsf{h}_{\phi_{1},\omega_{1}}+\mathsf{h}_{\phi_{2},\omega_{2}})}\circ\phi_{i}^{-1}(\Delta)\stackrel{\eqref{rd}}{=}\int_{\phi_{i}^{-1}(\Delta)}|\omega_{i}\cdot(\mathsf{h}_{\phi_{1},\omega_{1}}+\mathsf{h}_{\phi_{2},\omega_{2}})|^2\D\mu\\&\stackrel{\eqref{equivquasinormal}}{=}\int_{\phi_{i}^{-1}(\Delta)}|\omega_{i}\cdot(\mathsf{h}_{\phi_{1},\omega_{1}}+\mathsf{h}_{\phi_{2},\omega_{2}})\circ\phi_{i}|^2\D\mu\stackrel{\eqref{quasirnde}}{=}\int_{\Delta}(\mathsf{h}_{\phi_{1},\omega_{1}}+\mathsf{h}_{\phi_{2},\omega_{2}})^2\mathsf{h}_{\phi_{i},\omega_{i}}\D\mu
		\end{aligned}
	\end{equation}
	for any $\Delta\in\mathscr{A}$, then we have
	\begin{equation}\label{quasirn3}
		\mu_{\omega_{i}\cdot(\mathsf{h}_{\phi_{1},\omega_{1}}+\mathsf{h}_{\phi_{2},\omega_{2}})}\circ\phi_{i}^{-1}\ll\mu
	\end{equation}
	which implies that the weighted composition operator $C_{\phi_{i},\omega_{i}\cdot(\mathsf{h}_{\phi_{1},\omega_{1}}+\mathsf{h}_{\phi_{2},\omega_{2}})}$ is well-defined. Combining \eqref{quasirn2} and \eqref{quasirn3} with \eqref{rd}, we obtain
	\begin{equation}\label{quasinormalrn4}
		\mathsf{h}_{\phi_{i},\omega_{i}\cdot(\mathsf{h}_{\phi_{1},\omega_{1}}+\mathsf{h}_{\phi_{2},\omega_{2}})}=(\mathsf{h}_{\phi_{1},\omega_{1}}+\mathsf{h}_{\phi_{2},\omega_{2}})^2\mathsf{h}_{\phi_{i},\omega_{i}}.
	\end{equation}
	Thus, we have 
	\begin{equation}\label{domainequasinormal2}
		\begin{aligned}
			&\EuScript{D}(|\textbf{C}_{\phi,\omega}|^2C_{\phi_{i},\omega_{i}})\stackrel{\eqref{quasimultiplication}}{=}\EuScript{D}(\mathsf{M}_{\mathsf{h}_{\phi_{1},\omega_{1}}+\mathsf{h}_{\phi_{2},\omega_{2}}}C_{\phi_{i},\omega_{i}})\\&=\EuScript{D}(\mathsf{M}_{\mathsf{h}_{\phi_{1},\omega_{1}}+\mathsf{h}_{\phi_{2},\omega_{2}}}C_{\phi_{i},\omega_{i}})\bigcap\EuScript{D}(C_{\phi_{i},\omega_{i}})\\&=\EuScript{D}(C_{\phi_{i},\omega_{i}\cdot(\mathsf{h}_{\phi_{1},\omega_{1}}+\mathsf{h}_{\phi_{2},\omega_{2}})})\bigcap\EuScript{D}(C_{\phi_{i},\omega_{i}})\\&\xlongequal{\text{Proposition \ref{basic} (i)}}L^2((1+\mathsf{h}_{\phi_{i},\omega_{i}\cdot(\mathsf{h}_{\phi_{1},\omega_{1}}+\mathsf{h}_{\phi_{2},\omega_{2}})}+\mathsf{h}_{\phi_{i},\omega_{i}})\D\mu)\\&\stackrel{\eqref{quasinormalrn4}}{=}L^2((1+(1+(\mathsf{h}_{\phi_{1},\omega_{1}}+\mathsf{h}_{\phi_{2},\omega_{2}})^2)\mathsf{h}_{\phi_{i},\omega_{i}})\D\mu).
		\end{aligned}
	\end{equation}
	Finally, let 
	\begin{equation}\label{gsubspace}
		g\in L^2((1+(1+\mathsf{h}_{\phi_{i},\omega_{i}})(\mathsf{h}_{\phi_{1},\omega_{1}}+\mathsf{h}_{\phi_{2},\omega_{2}})^2)\D\mu)\stackrel{\eqref{domainquasi}}{=}\EuScript{D}(C_{\phi_{i},\omega_{i}}|\textbf{C}_{\phi,\omega}|^2),
	\end{equation}which immediately yields that\begin{equation}\label{subspace}
		\begin{aligned}
			&\int_{X}|g|^2(1+(\mathsf{h}_{\phi_{1},\omega_{1}}+\mathsf{h}_{\phi_{2},\omega_{2}})^2)\D\mu<\infty\\&\int_{X}|g|^2(\mathsf{h}_{\phi_{1},\omega_{1}}+\mathsf{h}_{\phi_{2},\omega_{2}})^2\mathsf{h}_{\phi_{i},\omega_{i}}\D\mu<\infty
		\end{aligned}.
	\end{equation}
	Since the function $h(t)=\frac{1+t}{1+t^2}$ has a maximum value $c$ for $t>0$, we have
	\begin{align*}
		\int_{X}|g|^2(1+\mathsf{h}_{\phi_{i},\omega_{i}})\D\mu&=\int_{X}|g|^2(1+(\mathsf{h}_{\phi_{1},\omega_{1}}+\mathsf{h}_{\phi_{2},\omega_{2}})^2)\cdot\frac{1+\mathsf{h}_{\phi_{i},\omega_{i}}}{1+\mathsf{h}_{\phi_{i},\omega_{i}}^2}\cdot\frac{1+\mathsf{h}_{\phi_{i},\omega_{i}}^2}{1+(\mathsf{h}_{\phi_{1},\omega_{1}}+\mathsf{h}_{\phi_{2},\omega_{2}})^2}\D\mu\\&\stackrel{\eqref{subspace}}{\le}\int_{X}|g|^2(1+(\mathsf{h}_{\phi_{1},\omega_{1}}+\mathsf{h}_{\phi_{2},\omega_{2}})^2)\cdot c\D\mu<\infty.
	\end{align*}
	Thus, combining the above inequality and \eqref{subspace}, we obtain$$g\in L^2((1+(1+(\mathsf{h}_{\phi_{1},\omega_{1}}+\mathsf{h}_{\phi_{2},\omega_{2}})^2)\mathsf{h}_{\phi_{i},\omega_{i}})\D\mu)\stackrel{\eqref{domainequasinormal2}}{=}\EuScript{D}(|\textbf{C}_{\phi,\omega}|^2C_{\phi_{i},\omega_{i}}),$$
	which implies, by $\eqref{gsubspace}$, that $\EuScript{D}(C_{\phi_{i},\omega_{i}}|\textbf{C}_{\phi,\omega}|^2)\subseteq\EuScript{D}(|\textbf{C}_{\phi,\omega}|^2C_{\phi_{i},\omega_{i}})$, as desired. The proof is complete.
	
	(iii) $\Rightarrow$ (i): The fact $\mathsf{h}_{\phi_{1},\omega_{1}},\mathsf{h}_{\phi_{2},\omega_{2}}<\infty$ a.e. $[\mu]$ implies that $1+(1+\mathsf{h}_{\phi_{i},\omega_{i}})(\mathsf{h}_{\phi_{1},\omega_{1}}+\mathsf{h}_{\phi_{2},\omega_{2}})^2<\infty$ a.e. $[\mu]$. By \cite[Lemma 9]{b-j-j-sW}, there exists an increasing sequence $\{X_{n}\}\subset\mathscr{A}$ such that $\mu(X_{n})<\infty$, $1+(1+\mathsf{h}_{\phi_{i},\omega_{i}})(\mathsf{h}_{\phi_{1},\omega_{1}}+\mathsf{h}_{\phi_{2},\omega_{2}})^2\le n$ a.e. $[\mu]$ on $X_{n}$ for $n\in\mathbb{N}$ and $\bigcup\limits_{n=1}^\infty X_{n}=X$. Therefore, it follows from \eqref{quasimultiplication}, \eqref{domainquasi} and $C_{\phi_{i},\omega_{i}}|\textbf{C}_{\phi,\omega}|^2\subseteq|\textbf{C}_{\phi,\omega}|^2C_{\phi_{i},\omega_{i}}$ that
	\begin{align*}
		&C_{\phi_{i},\omega_{i}}|\textbf{C}_{\phi,\omega}|^2\chi_{X_{n}}=|\textbf{C}_{\phi,\omega}|^2C_{\phi_{i},\omega_{i}}\chi_{X_{n}}\\&\Rightarrow\omega_{i}\cdot(\mathsf{h}_{\phi_{1},\omega_{1}}+\mathsf{h}_{\phi_{2},\omega_{2}})\circ\phi_{i}\cdot \chi_{X_{n}}\circ\phi=(\mathsf{h}_{\phi_{1},\omega_{1}}+\mathsf{h}_{\phi_{2},\omega_{2}})\cdot\omega_{i}\cdot \chi_{X_{n}}\circ\phi_{i}\quad \text{a.e}\ [\mu]
	\end{align*}
	Since $\phi_{i}^{-1}(X_{n})\to X$ as $n\to\infty$, then by the above inequality we have
	\begin{align*}
		&\mathsf{h}_{\phi_{1},\omega_{1}}+\mathsf{h}_{\phi_{2},\omega_{2}}=(\mathsf{h}_{\phi_{1},\omega_{1}}+\mathsf{h}_{\phi_{2},\omega_{2}})\circ\phi_{i}\quad\text{a.e.}\ [\mu]\quad\text{on}\ \{\omega_{i}\ne0\}\\&\stackrel{\eqref{equivquasinormal}}{=}\mathsf{h}_{\phi_{1},\omega_{1}}+\mathsf{h}_{\phi_{2},\omega_{2}}=(\mathsf{h}_{\phi_{1},\omega_{1}}+\mathsf{h}_{\phi_{2},\omega_{2}})\circ\phi_{i}\quad\text{a.e.}\ [\mu_{\omega_{i}}]
	\end{align*}
	for $i=1,2$. By Theorem \ref{quasinormal}, we get the desired result.
	
	(iii) $\Leftrightarrow$ (iv): It suffices to prove that $|\textbf{C}_{\phi,\omega}|^2=C_{\phi_{1},\omega_{1}}^\ast C_{\phi_{1},\omega_{1}}+C_{\phi_{2},\omega_{2}}^\ast C_{\phi_{2},\omega_{2}}$. Indeed, by Theorem \ref{polar} (i) and Lemma \ref{mulsp} (ii) we have $C_{\phi_{i},\omega_{i}}^\ast C_{\phi_{i},\omega_{i}}=|C_{\phi_{i},\omega_{i}}|^2=\mathsf{M}_{\sqrt{\mathsf{h}_{\phi_{i},\omega_{i}}}}^2=\mathsf{M}_{\mathsf{h}_{\phi_{i},\omega_{i}}}$ for $i=1,2$. Thus, we obtain
	\begin{equation}\label{selfadjointmut}
		C_{\phi_{1},\omega_{1}}^\ast C_{\phi_{1},\omega_{1}}+C_{\phi_{2},\omega_{2}}^\ast C_{\phi_{2},\omega_{2}}=\mathsf{M}_{\mathsf{h}_{\phi_{1},\omega_{1}}}+\mathsf{M}_{\mathsf{h}_{\phi_{2},\omega_{2}}}\xlongequal{\text{Lemma \ref{mulsp} (i)}}\mathsf{M}_{\mathsf{h}_{\phi_{1},\omega_{1}}+\mathsf{h}_{\phi_{2},\omega_{2}}}\stackrel{\eqref{quasimultiplication}}{=}|\textbf{C}_{\phi,\omega}|^2.
	\end{equation}
\end{proof}

It is not difficult to verify that if $C_{\phi_{1},\omega_{1}}$ and $C_{\phi_{2},\omega_{2}}$ are bounded on $L^2(\mu)$, then the inclusions in (ii)-(iv) of the corollary above become equalities. In this case, the pair $\textbf{C}_{\phi,\omega}$ is said to be spherically quasinormal, as defined in \cite{curto-AOF-2019}, which generalizes the notion of a single bounded quasinormal operator to the multivariable setting. On the other hand, suppose $T$ is a closed densely defined operator with the polar decomposition $T=U|T|$. Recall from \cite{b-j-j-sW} that $T$ is said to be quasinormal if $U|T|\subseteq |T|U$. Therefore, by Corollary \ref{unboundedquasinormal} (ii), it is justified that the unbounded pair $\textbf{C}_{\phi,\omega}$ satisfying $\mathcal{M}_{\lambda}(\textbf{C}_{\phi,\omega})=\textbf{C}_{\phi,\omega}$ is said to be {\em spherically quasinormal}. We remark that Jab{\l}o\'nski et al. in \cite[Theorem 3.1]{jab-ieot-2014} showed that $T$ is quasinormal if and only if $U|T|=|T|U$ if and only if $T|T|^2=|T|^2T$ if and only if $T|T|^2\subseteq |T|^2T$. However, the equalities in (ii)-(iv) of Corollary \ref{unboundedquasinormal} do not hold in general for the spherically quasinormal pair $\textbf{C}_{\phi,\omega}$, as shown in the following example.

\begin{exa}
	Let $X=\mathbb{N}$, $\mathscr{A}=2^X$ and $\mu$ be the counting measure on $X$. Define the $\mathscr{A}$-measurable transforms $\phi_{1}(n)=\phi_{2}(n)=n$. The functions $\omega_{1}$ and $\omega_{2}$ are defined by $\omega_{1}(n)=\frac{1}{n}$ and $\omega_{2}(n)=\sqrt{n}$, respectively.  By \eqref{h}, we have $\mathsf{h}_{\phi_{1},\omega_{1}}(n)=|\omega_{1}(n)|^2=\frac{1}{n^2}$ and $\mathsf{h}_{\phi_{2},\omega_{2}}(n)=|\omega_{2}(n)|^2=n$. It is clear that $\mathsf{h}_{\phi_{1},\omega_{1}}+\mathsf{h}_{\phi_{2},\omega_{2}}=(\mathsf{h}_{\phi_{1},\omega_{1}}+\mathsf{h}_{\phi_{2},\omega_{2}})\circ\phi_{i}$ a.e. $[\mu_{\omega_{i}}]$ for $i=1,2$ since $\phi_{1}$ and $\phi_{2}$ are identity maps, which implies, by Theorem \ref{quasinormal}, that $\mathcal{M}_{\lambda}(\textbf{C}_{\phi,\omega})=\textbf{C}_{\phi,\omega}$. Now, consider the function $f(n)=\frac{1}{n}\in L^2(\mu)$. We first prove that $U_{1}|\textbf{C}_{\phi,\omega}|\ne|\textbf{C}_{\phi,\omega}|U_{1}$. Since $$(|\textbf{C}_{\phi,\omega}|f)(n)\xlongequal{\text{Theorem \ref{polar} (i)}}(\mathsf{M}_{\sqrt{\mathsf{h}_{\phi_{1},\omega_{1}}+\mathsf{h}_{\phi_{2},\omega_{2}}}}f)(n)=(\frac{1}{n^2}+n)^\frac{1}{2}\cdot\frac{1}{n},$$then
	\begin{equation}\label{ex27 (1)}
		\int_{X}||\textbf{C}_{\phi,\omega}|f|^2\D\mu=\int_{X}(\frac{1}{n^2}+n)\cdot\frac{1}{n^2}\D\mu=\int_{X}\frac{1}{n^4}+\frac{1}{n}\D\mu>\infty\Rightarrow f\notin\EuScript{D}(|\textbf{C}_{\phi,\omega}|).
	\end{equation}
	Since$$(|\textbf{C}_{\phi,\omega}|U_{1}f)(n)\xlongequal{\text{Theorem \ref{polar} (i)}}(\sqrt{\mathsf{h}_{\phi_{1},\omega_{1}}+\mathsf{h}_{\phi_{2},\omega_{2}}}\cdot C_{\phi_{1},\frac{\omega_{1}}{\sqrt{\mathsf{h}_{\phi_{1},\omega_{1}}+\mathsf{h}_{\phi_{2},\omega_{2}}}}}f)(n)=(\omega_{1}\cdot f)(n)=\frac{1}{n^2},$$
	then
	\begin{equation}\label{ex27 (2)}
		\int_{X}||\textbf{C}_{\phi,\omega}|U_{1}f|^2\D\mu=\int_{X}\frac{1}{n^4}\D\mu<\infty\Rightarrow f\in\EuScript{D}(|\textbf{C}_{\phi,\omega}|U_{1}).
	\end{equation}
	Combining \eqref{ex27 (1)} and \eqref{ex27 (2)}, we have $U_{1}|\textbf{C}_{\phi,\omega}|\ne|\textbf{C}_{\phi,\omega}|U_{1}$. Then, we prove that $C_{\phi_{1},\omega_{1}}|\textbf{C}_{\phi,\omega}|^2\ne|\textbf{C}_{\phi,\omega}|^2C_{\phi_{1},\omega_{1}}$. Since $$(|\textbf{C}_{\phi,\omega}|^2f)(n)\xlongequal{\text{Theorem \ref{polar} (i)}}(\mathsf{M}_{\mathsf{h}_{\phi_{1},\omega_{1}}+\mathsf{h}_{\phi_{2},\omega_{2}}}f)(n)=(\frac{1}{n^2}+n)\cdot\frac{1}{n}=\frac{1}{n^3}+1,$$then
	\begin{equation}\label{ex27 (3)}
		\int_{X}||\textbf{C}_{\phi,\omega}|f|^2\D\mu=\int_{X}(\frac{1}{n^3}+1)^2\D\mu>\infty\Rightarrow f\notin\EuScript{D}(|\textbf{C}_{\phi,\omega}|^2).
	\end{equation}
	Since
	$$\int_{X}|C_{\phi_{1},\omega_{1}}f|^2\D\mu=\int_{X}|\omega_{1}\cdot f|^2\D\mu=\int_{X}\frac{1}{n^4}\D\mu<\infty\Rightarrow f\in\EuScript{D}(C_{\phi_{1},\omega_{1}})$$
	and
	$$(|\textbf{C}_{\phi,\omega}|^2C_{\phi_{1},\omega_{1}}f)(n)\xlongequal{\text{Theorem \ref{polar} (i)}}(\mathsf{M}_{\mathsf{h}_{\phi_{1},\omega_{1}}+\mathsf{h}_{\phi_{2},\omega_{2}}}\omega_{1}\cdot f)(n)=(\frac{1}{n^2}+n)\cdot \frac{1}{n^2}=\frac{1}{n^4}+\frac{1}{n},$$then
	\begin{equation}\label{ex27 (4)}
		\int_{X}||\textbf{C}_{\phi,\omega}|^2C_{\phi_{1},\omega_{1}}f|^2\D\mu=\int_{X}(\frac{1}{n^4}+\frac{1}{n})^2\D\mu<\infty\Rightarrow f\in\EuScript{D}(|\textbf{C}_{\phi,\omega}|^2C_{\phi_{1},\omega_{1}}).
	\end{equation}	
	Combining \eqref{ex27 (3)} and \eqref{ex27 (4)}, we have
	$C_{\phi_{1},\omega_{1}}|\textbf{C}_{\phi,\omega}|^2\ne|\textbf{C}_{\phi,\omega}|^2C_{\phi_{1},\omega_{1}}$.
\end{exa}

Recently, Kim et al. \cite{Yoon-glma-2022} introduced the spherical $p$-hyponormality for operator pairs. A pair $\textbf{T}=(T_{1},T_{2})$ is said to be spherically $p$-hyponormal (0$<$$p$$\le$1) if $T_{1}$, $T_{2}\in\mathcal{B}(\mathcal{H})$, $T_{1}T_{2}=T_{2}T_{1}$ and
$$(T_{1}T_{1}^\ast+T_{2}T_{2}^{\ast})^p\le(T_{1}^\ast T_{1}+T_{2}^\ast T_{2})^p.$$
The equivalence,$$(T_{1}T_{1}^\ast+T_{2}T_{2}^{\ast})^p\le(T_{1}^\ast T_{1}+T_{2}^\ast T_{2})^p\Leftrightarrow\|(T_{1}T_{1}^\ast+T_{2}T_{2}^{\ast})^\frac{p}{2}x\|\le\|((T_{1}^\ast T_{1}+T_{2}^\ast T_{2}))^\frac{p}{2}x\|$$ for any $x\in\mathcal{H}$, together with \eqref{selfadjoint}, allows us to extend this definition to unbounded weighted composition operator pairs without assuming commutativity. Note that our definition, being based on the functional calculus for self-adjoint operators, requires operators $C_{\phi_{1},\omega_{1}}^\ast C_{\phi_{1},\omega_{1}}+C_{\phi_{2},\omega_{2}}^\ast C_{\phi_{2},\omega_{2}}$ and $C_{\phi_{1},\omega_{1}}C_{\phi_{1},\omega_{1}}^\ast+C_{\phi_{2},\omega_{2}}C_{\phi_{2},\omega_{2}}^\ast$ to be self-adjoint. By \eqref{selfadjointmut}, the operator $$C_{\phi_{1},\omega_{1}}^\ast C_{\phi_{1},\omega_{1}}+C_{\phi_{2},\omega_{2}}^\ast C_{\phi_{2},\omega_{2}}\ (=\mathsf{M}_{\mathsf{h}_{\phi_{1},\omega_{1}}+\mathsf{h}_{\phi_{2},\omega_{2}}})$$is always self-adjoint. However, $C_{\phi_{1},\omega_{1}}C_{\phi_{1},\omega_{1}}^\ast+C_{\phi_{2},\omega_{2}}C_{\phi_{2},\omega_{2}}^\ast$ is not necessarily self-adjoint. More precisely, for $\textbf{C}_{\phi,\omega}=(C_{\phi_{1},\omega_{1}},C_{\phi_{2},\omega_{2}})\in\mathcal{WCOP}_{\text{dense}}$, we say that the pair $\textbf{C}_{\phi,\omega}$ is {\em spherically $p$-hyponormal} (0$<$$p$$\le$1) if $C_{\phi_{1},\omega_{1}}C_{\phi_{1},\omega_{1}}^\ast+C_{\phi_{2},\omega_{2}}C_{\phi_{2},\omega_{2}}^\ast$ is self-adjoint,
$$\EuScript{D}((C_{\phi_{1},\omega_{1}}^\ast C_{\phi_{1},\omega_{1}}+C_{\phi_{2},\omega_{2}}^\ast C_{\phi_{2},\omega_{2}})^\frac{p}{2})\subseteq\EuScript{D}((C_{\phi_{1},\omega_{1}}C_{\phi_{1},\omega_{1}}^\ast+C_{\phi_{2},\omega_{2}}C_{\phi_{2},\omega_{2}}^\ast)^\frac{p}{2})$$
and$$\|(C_{\phi_{1},\omega_{1}}C_{\phi_{1},\omega_{1}}^\ast+C_{\phi_{2},\omega_{2}}C_{\phi_{2},\omega_{2}}^\ast)^\frac{p}{2}f\|\le\|(C_{\phi_{1},\omega_{1}}^\ast C_{\phi_{1},\omega_{1}}+C_{\phi_{2},\omega_{2}}^\ast C_{\phi_{2},\omega_{2}})^\frac{p}{2}f\|$$
for any $f\in\EuScript{D}((C_{\phi_{1},\omega_{1}}^\ast C_{\phi_{1},\omega_{1}}+C_{\phi_{2},\omega_{2}}^\ast C_{\phi_{2},\omega_{2}})^\frac{p}{2})$.

We now present an example showing that operator $C_{\phi_{1},\omega_{1}}C_{\phi_{1},\omega_{1}}^\ast+C_{\phi_{2},\omega_{2}}C_{\phi_{2},\omega_{2}}^\ast$ is not necessarily self-adjoint, but first, the following result guarantees the existence of its adjoint.

\begin{pro}
	Suppose $C_{\phi_{1},\omega_{1}}$ and $C_{\phi_{2},\omega_{2}}$ are densely defined. Then the sum $C_{\phi_{1},\omega_{1}}C_{\phi_{1},\omega_{1}}^\ast+C_{\phi_{2},\omega_{2}}C_{\phi_{2},\omega_{2}}^\ast$ is densely defined.
\end{pro}
\begin{proof}
By \cite[Theorem 4.3]{b-j-j-sW}, we have
\begin{equation}\label{98}
C_{\phi_{i},\omega_{i}}C_{\phi_{i},\omega_{i}}^\ast f=|C_{\phi_{i},\omega_{i}}^\ast|^2f=\omega_{i}\cdot\mathsf{h}_{\phi_{i},\omega_{i}}\circ\phi_{i}\cdot\mathsf{E}_{\phi_{i},\omega_{i}}(f_{\omega_{i}}),\quad\quad f\in\EuScript{D}(C_{\phi_{i},\omega_{i}}C_{\phi_{i},\omega_{i}}^\ast)
\end{equation}
and
\begin{equation}\label{99}
	\EuScript{D}(C_{\phi_{i},\omega_{i}}C_{\phi_{i},\omega_{i}}^\ast)=\EuScript{D}(|C_{\phi_{i},\omega_{i}}^\ast|^2)=\{f\in L^2(\mu)\colon\omega_{i}\cdot\mathsf{h}_{\phi_{i},\omega_{i}}\circ\phi_{i}\cdot\mathsf{E}_{\phi_{i},\omega_{i}}(f_{\omega_{i}})\in L^2(\mu)\}
\end{equation}for $i=1,2$. Thus, we derive that \begin{align*}
		&\int_{X}|C_{\phi_{i},\omega_{i}}C_{\phi_{i},\omega_{i}}^\ast f|^2\D\mu=\int_{X}|\omega_{i}\cdot\mathsf{h}_{\phi_{i},\omega_{i}}\circ\phi_{i}\cdot\mathsf{E}_{\phi_{i},\omega_{i}}(f_{\omega_{i}})|^2\D\mu\\&\stackrel{\eqref{muomega}}{=}\int_{X}|\mathsf{h}_{\phi_{i},\omega_{i}}\circ\phi_{i}\cdot\mathsf{E}_{\phi_{i},\omega_{i}}(f_{\omega_{i}})|^2\D\mu_{\omega_{i}}\stackrel{\text{\cite[Theorem A.4]{b-j-j-sW}}}{\le}\int_{X}\mathsf{h}_{\phi_{i},\omega_{i}}^2\circ\phi_{i}\cdot\mathsf{E}_{\phi_{i},\omega_{i}}(|f_{\omega_{i}}|^2)d\mu_{\omega_{i}}\\&\stackrel{\eqref{exp+}}{=}\int_{X}\mathsf{E}_{\phi_{i},\omega_{i}}(\mathsf{h}_{\phi_{i},\omega_{i}}^2\circ\phi_{i}\cdot|f_{\omega_{i}}|^2)d\mu_{\omega_{i}}=\int_{X}\mathsf{h}_{\phi_{i},\omega_{i}}^2\circ\phi_{i}\cdot|f_{\omega_{i}}|^2\D\mu_{\omega_{i}}\stackrel{\eqref{muomega}}{=}\int_{X}\mathsf{h}_{\phi_{i},\omega_{i}}^2\circ\phi_{i}\cdot|f|^2\D\mu
	\end{align*}implies $L^2(\mathsf{h}_{\phi_{i},\omega_{i}}^2\circ\phi_{i}\D\mu)\subseteq\EuScript{D}(C_{\phi_{i},\omega_{i}}C_{\phi_{i},\omega_{i}}^\ast)$ for $i=1,2$, then
\begin{align*}
&\bigcap\limits_{i=1}^2L^2(\mathsf{h}_{\phi_{i},\omega_{i}}^2\circ\phi_{i}\D\mu)=L^2(\sum\limits_{i=1}^2\mathsf{h}_{\phi_{i},\omega_{i}}^2\circ\phi_{i}\D\mu)\\&\subseteq\bigcap\limits_{i=1}^2\EuScript{D}(C_{\phi_{i},\omega_{i}}C_{\phi_{i},\omega_{i}}^\ast)=\EuScript{D}(C_{\phi_{1},\omega_{1}}C_{\phi_{1},\omega_{1}}^\ast+C_{\phi_{2},\omega_{2}}C_{\phi_{2},\omega_{2}}^\ast).
\end{align*}
 If we can prove
 \begin{equation}\label{sumdense}
 \overline{L^2(\sum\limits_{i=1}^2\mathsf{h}_{\phi_{i},\omega_{i}}^2\circ\phi_{i}\D\mu)}=L^2(\mu),
 \end{equation}then $$\overline{\EuScript{D}(C_{\phi_{1},\omega_{1}}C_{\phi_{1},\omega_{1}}^\ast+C_{\phi_{2},\omega_{2}}C_{\phi_{2},\omega_{2}}^\ast)}=L^2(\mu),$$as desired. Indeed, since $\mathsf{h}_{\phi_{i},\omega_{i}}<\infty$ a.e. $[\mu]$ implies that $\mathsf{h}_{\phi_{i},\omega_{i}}^2\circ\phi_{i}<\infty$ for $i=1,2$, an argument similar to the proof of Proposition \ref{basic} (ii) readily yields Equation \eqref{sumdense}.
\end{proof}

\begin{exa}\label{notselfadjoint}
Let $X=\mathbb{N}$ and $\mu$ be the counting measure on $X$. Set $k\in \mathbb{N}$. Define the $\mathscr{A}$-measurable transform $\phi:X\to X$ by $\phi(n)=\begin{cases}
2k-1&n=2k-1\\
2k-1&n=2k
\end{cases}.$ The functions $\omega_{1},\omega_{2}:X\to\mathbb{R}$ are defined as$$\omega_{1}(n)=\begin{cases}
1&n=2k-1\\
2k-1&n=2k
\end{cases}\quad\text{and}\quad\omega_{2}(n)=\begin{cases}
1&n=2k-1\\
1-2k&n=2k
\end{cases},$$
respectively. By \eqref{h}, we have $\mathsf{h}_{\phi,\omega_{1}}(n)=\sum\limits_{y\in\phi^{-1}(n)}|\omega_{1}(y)|^2=\begin{cases}
(2k-1)^2+1&n=2k-1\\
0&n=2k
\end{cases}$ and$$\mathsf{h}_{\phi,\omega_{2}}(n)=\sum\limits_{y\in\phi^{-1}(n)}|\omega_{2}(y)|^2=\begin{cases}
(1-2k)^2+1&n=2k-1\\
0&n=2k
\end{cases}.$$
Combining \eqref{Ef} and the above inequalities with Theorem \ref{adjoint}, we have
\begin{equation}\label{C_{1}}
\begin{aligned}
	(C_{\phi,\omega_{1}}C_{\phi,\omega_{1}}^\ast r)(n)&=(\omega_{1}\cdot\mathsf{h}_{\phi,\omega_{1}}\circ\phi\cdot\mathsf{E}_{\phi,\omega_{1}}(r_{\omega_{1}}))(n)\\&=\omega_{1}(n)\cdot\sum\limits_{y\in\phi^{-1}(\phi(n))}r_{\omega_{1}}(y)|\omega_{1}(y)|^2\\&=\begin{cases}
		r(2k)(2k-1)+r(2k-1)&n=2k-1\\
		(2k-1)\cdot(r(2k)(2k-1)+f(2k-1))&n=2k\\
	\end{cases}
\end{aligned}
\end{equation}
for $r\in\EuScript{D}(C_{\phi,\omega_{1}}C_{\phi,\omega_{1}}^\ast)$ and
\begin{equation}\label{C_{2}}
	\begin{aligned}
		&(C_{\phi,\omega_{2}}C_{\phi,\omega_{2}}^\ast s)(n)=(\omega_{2}\cdot\mathsf{h}_{\phi,\omega_{2}}\circ\phi\cdot\mathsf{E}_{\phi,\omega_{2}}(s_{\omega_{2}}))(n)=\omega_{2}(n)\cdot\sum\limits_{y\in\phi^{-1}(\phi(n))}s_{\omega_{2}}(y)|\omega_{2}(y)|^2\\&=\begin{cases}
		s(2k)(1-2k)+s(2k-1)&n=2k-1\\
		(1-2k)(s(2k)(1-2k)+s(2k-1))&n=2k\\
		\end{cases}
	\end{aligned}
\end{equation}
for $s\in\EuScript{D}(C_{\phi,\omega_{2}}C_{\phi,\omega_{2}}^\ast)$. Let $p(n)=\begin{cases}
	1&n=2k-1\\
	2\cdot(2k-1)^2&n=2k
\end{cases},$ then by \eqref{C_{1}} and \eqref{C_{2}} we obtain
$$((C_{\phi,\omega_{1}}C_{\phi,\omega_{1}}^\ast+C_{\phi,\omega_{2}}C_{\phi,\omega_{2}}^\ast)q)(n)=\begin{cases}
		q(2k-1)&n=2k-1\\
		2\cdot(2k-1)^2\cdot q(2k)&n=2k
	\end{cases}=p(n)\cdot q(n)$$
for $q\in\EuScript{D}(C_{\phi,\omega_{1}}C_{\phi,\omega_{1}}^\ast)\bigcap\EuScript{D}(C_{\phi,\omega_{2}}C_{\phi,\omega_{2}}^\ast)$, which immediately yields $$\EuScript{D}((C_{\phi,\omega_{1}}C_{\phi,\omega_{1}}^\ast+C_{\phi,\omega_{2}}C_{\phi,\omega_{2}}^\ast)^\ast)=\{f\in L^2(\mu)\colon p\cdot f\in L^2(\mu)\}.$$
Considering the function $f(n)=\begin{cases}
	\frac{1}{2k-1}&n=2k-1\\
	0&n=2k
\end{cases}\in L^{2}(\mu)$, by \eqref{C_{1}} we derive
   $$(C_{\phi,\omega_{1}}C_{\phi,\omega_{1}}^\ast f)(n)=\begin{cases}
   	\frac{1}{2k-1}&n=2k-1\\
   	1&n=2k
   \end{cases}.$$
A direct calculation gives
 $$(p\cdot f)(n)=\begin{cases}
 	\frac{1}{2k-1}&n=2k-1\\
 	0&n=2k
 \end{cases}.$$
 Therefore, we conclude that $$f\in\EuScript{D}((C_{\phi,\omega_{1}}C_{\phi,\omega_{1}}^\ast+C_{\phi,\omega_{2}}C_{\phi,\omega_{2}}^\ast)^\ast)$$ and $$f\notin\EuScript{D}(C_{\phi,\omega_{1}}C_{\phi,\omega_{1}}^\ast)\Rightarrow f\notin\EuScript{D}(C_{\phi,\omega_{1}}C_{\phi,\omega_{1}}^\ast+C_{\phi,\omega_{2}}C_{\phi,\omega_{2}}^\ast),$$ that is, $C_{\phi,\omega_{1}}C_{\phi,\omega_{1}}^\ast+C_{\phi,\omega_{2}}C_{\phi,\omega_{2}}^\ast$ is not selfadjoint.  
\end{exa}

 As shown in the previous sections, we have presented numerous examples based on the measurable space $(X,2^X,\mu)$, where $X$ is a countable set and $\mu$ is the counting measure on $X$. Indeed, $(X,2^X,\mu)$ is a discrete measure space. A measure space $(Y,\mathscr{A},\nu)$ is said to be {\em discrete} if $\mathscr{A}=2^Y$, card($\mathsf{At}(\nu)$)$\le\aleph_{0}$, $\nu(\mathsf{At}^\complement(\nu))=0$ and $\nu(y)<\infty$ for all $y\in Y$, where $\mathsf{At}(\nu)=\{y\in Y\colon\nu(y)>0\}$ and $\mathsf{At}^\complement(\nu)=Y\backslash\mathsf{At}(\nu)$. Let $\omega:Y\to\mathbb{C}$ be a function, and we have
\begin{equation}\label{discount1}
\nu_{\omega}(\mathsf{At}^\complement(\nu))=\int_{\mathsf{At}^\complement(\nu)}|\omega|^2\D\nu=0,
\end{equation} which implies
\begin{equation}\label{discount}
\{y\colon\nu_{\omega}(\{y\})=0\}=\mathsf{At}^\complement(\nu)\bigcup\widetilde{\mathsf{At}(\nu)},
\end{equation}
 where $\widetilde{\mathsf{At}({\nu})}\ (\subseteq\mathsf{At}(\nu))$ is a countable set since card($\mathsf{At}(\nu)$)$\le\aleph_{0}$. By \eqref{discount1} and \eqref{discount}, we obtain\begin{equation}\label{discount2}
 	\nu_{\omega}(\{y\colon\nu_{\omega}(\{y\})=0\})=\nu_{\omega}(\widetilde{\mathsf{At}(\nu)})\le\sum\limits_{y\in\widetilde{\mathsf{At}(\nu)}}\nu_{\omega}(\{y\})=0.
 \end{equation} 
 Let $C_{\phi,\omega}$ be a densely defined weighted composition operator on $(Y,2^Y,\nu)$. Then
 \begin{equation}\label{dish}
\mathsf{h}_{\phi,\omega}(y)=\begin{cases}
	\frac{\nu_{\omega}(\phi^{-1}(\{y\}))}{\nu(\{y\})}&y\in\mathsf{At}(\nu)\\
	0&y\in\mathsf{At}^\complement(\nu)
\end{cases}
 \end{equation}
 and
 \begin{equation}\label{disdomain}
 (\mathsf{E}_{\phi,\omega}(f))(y)=\begin{cases}
 	\frac{\int_{\phi^{-1}(\phi(y))}f\D\nu_{\omega}}{\nu_{\omega}(\phi^{-1}(\phi(y)))}&y\in\mathsf{At}(\nu_{\omega}\circ\phi^{-1}\circ\phi)\\
 	0&y\in\mathsf{At}^\complement(\nu_{\omega}\circ\phi^{-1}\circ\phi)
 \end{cases}
 \end{equation}
for any $f: Y\to[0,+\infty]$ (For proofs and more facts about discrete measure spaces we refer to \cite[Chapter 6]{b-j-j-sW}). Now, consider the discrete measure space $(\mathbb{Z}_{+}\times\mathbb{Z}_{+},2^{\mathbb{Z}_{+}\times\mathbb{Z}_{+}},\upsilon)$ equipped with the counting measure. 
Define the transforms $\phi_{1}$ and $\phi_{2}$ on $\mathbb{Z}_{+}\times\mathbb{Z}_{+}$ by
$$\phi_{1}(n,m)=\begin{cases}
 	(n-1,m)&(n,m)\in(\mathbb{Z}_{+}\backslash\{0\})\times\mathbb{Z}_{+}\\
 	(0,m)&(n,m)\in\{0\}\times\mathbb{Z}_{+}
 \end{cases}$$
and$$\phi_{2}(n,m)=\begin{cases}
	(n,m-1)&(n,m)\in\mathbb{Z}_{+}\times(\mathbb{Z}_{+}\backslash\{0\})\\
	(n,0)&(n,m)\in\mathbb{Z}_{+}\times\{0\}
\end{cases},$$
respectively. The weight functions $\omega_{1}$ and $\omega_{2}$ on $\mathbb{Z}_{+}\times\mathbb{Z}_{+}$ are respectively given by the positive sequences $\{\alpha_{(n,m)}\}$ and $\{\beta_{(n,m)}\}$ as follows:
$$\omega_{1}(n,m)=\begin{cases}
	\alpha_{(n-1,m)}&(n,m)\in(\mathbb{Z}_{+}\backslash\{0\})\times\mathbb{Z}_{+}\\
	0&(n,m)\in\{0\}\times\mathbb{Z}_{+}
\end{cases}$$
and$$\omega_{2}(n,m)=\begin{cases}
	\beta_{(n,m-1)}&(n,m)\in\mathbb{Z}_{+}\times(\mathbb{Z}_{+}\backslash\{0\})\\
	0&(n,m)\in\mathbb{Z}_{+}\times\{0\}
\end{cases},$$
respectively. Then, we have $$C_{\phi_{1},\omega_{1}}e_{(n,m)}=\omega_{1}\cdot e_{(n,m)}\circ\phi_{1}=\omega_{1}(n+1,m)e_{(n+1,m)}=\alpha_{(n,m)}e_{(n+1,m)}$$ and $$C_{\phi_{2},\omega_{2}}e_{(n,m)}=\omega_{2}\cdot e_{(n,m)}\circ\phi_{2}=\omega_{2}(n,m+1)e_{(n,m+1)}=\beta_{(n,m)}e_{(n,m+1)},$$ where the function $e_{(i,j)}$ on $\mathbb{Z}_{+}\times\mathbb{Z}_{+}$ is defined as $e_{(i,j)}(x)=\begin{cases}
	1&x=(i,j)\\
	0&x\ne(i,j)
\end{cases}$. By \eqref{dish}, we have
\begin{equation}\label{2vws1}
\mathsf{h}_{\phi_{1},\omega_{1}}(n,m)=\omega_{1}^2(\phi_{1}^{-1}(n,m))=\omega_{1}^2(n+1,m)<\infty
\end{equation}
 and
 \begin{equation}\label{2vws2}
 \mathsf{h}_{\phi_{2},\omega_{2}}(n,m)=\omega_{2}^2(\phi_{2}^{-1}(n,m))=\omega_{2}^2(n,m+1)<\infty,
 \end{equation} which implies $C_{\phi_{1},\omega_{1}}$ and $C_{\phi_{2},\omega_{2}}$ are densely defined. The operator pair $\textbf{C}_{\phi,\omega}=(C_{\phi_{1},\omega_{1}},C_{\phi_{2},\omega_{2}})$ is known as the {\em $2$-variable weighted shift}. Some related topics can be found in \cite{curto-ieot-2018, curto-CRA-2016, Yoon-jmaa-2011}. In \cite[Theorem 2.3]{Yoon-glma-2022}, Kim et al. established a characterization for the class of spherically $p$-hyponormal $2$–variable weighted shift under the condition that $\omega_{1}$ and $\omega_{2}$ are bounded. The next main result provides a characterization of a class of unbounded spherically $p$-hyponormal weighted composition operator pairs on a discrete measurable space. As a corollary, we present the unbounded analogue of the result by Kim et al. Let us start with the following lemma. 

\begin{lem}\label{mulselfadjoint}
Suppose the measure $(X,2^X,\mu)$ is discrete and $C_{\phi,\omega}$ is densely defined. Then the map $f\mapsto\mathsf{E}_{\phi,\omega}(\frac{\chi_{\{\omega\ne0\}}}{\omega}\cdot f)$ is a multiplication operator if and only if $\phi$ is injective a.e. $[\mu_{\omega}]$. In this case, $\mathsf{E}_{\phi,\omega}(\frac{\chi_{\{\omega\ne0\}}}{\omega}\cdot f)=\frac{\chi_{\{\omega\ne0\}}}{\omega}\cdot f\ \text{a.e.}\ [\mu_{\omega}]$ for $f\in L^2(\mu)$ and $C_{\phi,\omega}C_{\phi,\omega}^\ast=\mathsf{M}_{\mathsf{h}_{\phi,\omega}\circ\phi\cdot\chi_{\{\omega\ne0\}}}$.
\end{lem}
\begin{proof}
For each $x\in X$, define the subset $\mathsf{\Phi}_{x}$ of $X$ as follows:$$\mathsf{\Phi}_{x}=\{y:\phi(y)=\phi(x),y\ne x\}.$$ We say that $\phi$ is injective a.e. $[\mu_{\omega}]$ if $\mu_{\omega}(\mathsf{\Phi}_{z})=0$ for every $z\in\mathsf{At}(\mu_{\omega})$.

{\em Sufficiency}. Suppose there exists a function $g$ such that $\mathsf{E}_{\phi,\omega}(\frac{\chi_{\{\omega\ne0\}}}{\omega}\cdot f)=\mathsf{M}_{g}f$ a.e. $[\mu_{\omega}]$ for every $f\in L^2(\mu)$. Then for $z\in\mathsf{At}(\mu_{\omega})$, we have
\begin{equation}\label{dis1}
	\begin{aligned}
	(\mathsf{E}_{\phi,\omega}(\frac{\chi_{\{\omega\ne0\}}}{\omega}\cdot f))(z)&=\frac{\int_{\phi^{-1}(\phi(z))}\frac{\chi_{\{\omega\ne0\}}}{\omega}\cdot f\D\mu_{\omega}}{\mu_{\omega}(\phi^{-1}(\phi(z)))}\\&=\frac{\int_{\phi^{-1}(\phi(z))\bigcap\mathsf{At}(\mu_{\omega})}\frac{\chi_{\{\omega\ne0\}}}{\omega}\cdot f\D\mu_{\omega}}{\mu_{\omega}(\phi^{-1}(\phi(z)))}\\&=\frac{\frac{1}{\omega(z)}\cdot f(z)\cdot \mu_{\omega}(\{z\})+\int_{\mathsf{\Phi}_{z}\bigcap\mathsf{At}(\mu_{\omega})}\frac{\chi_{\{\omega\ne0\}}}{\omega}\cdot f\D\mu_{\omega}}{\mu_{\omega}(\{z\})+\mu_{\omega}(\mathsf{\Phi}_{z})}\\&=g(z)\cdot f(z)
	\end{aligned}
\end{equation}
Taking $f(x)=\chi_{\{z\}}\in L^{2}(\mu)$ in the above equality, we get
\begin{equation}\label{dis2}
g(z)=\frac{\frac{1}{\omega(z)}\cdot \mu_{\omega}(\{z\})}{\mu_{\omega}(\{z\})+\mu_{\omega}(\mathsf{\Phi}_{z})}.
\end{equation}
Substituting \eqref{dis2} into the last equality of \eqref{dis1} yields
\begin{equation}\label{dis4}
	\int_{\mathsf{\Phi}_{z}\bigcap\mathsf{At}(\mu_{\omega})}\frac{\chi_{\{\omega\ne0\}}}{\omega}\cdot f\D\mu_{\omega}=0
\end{equation}
 for $z\in\mathsf{At}(\mu_{\omega})$ and $f\in L^2(\mu)$. Assume that $\phi$ is not injective a.e. $[\mu_{\omega}]$ (that is, there exists $z\in\mathsf{At}(\mu_{\omega})$ such that $\mu_{\omega}(\mathsf{\Phi}_{z})>0$), then we assert
 \begin{equation}\label{dis3}
 	\mathsf{\Phi}_{z}\bigcap\mathsf{At}(\mu_{\omega})\ne\emptyset;
 \end{equation}
 otherwise, $\mu_{\omega}(\mathsf{\Phi}_{z})=\mu_{\omega}(\mathsf{\Phi}_{z}\bigcap\mathsf{At}(\mu_{\omega}))+\mu_{\omega}(\mathsf{\Phi}_{z}\bigcap\mathsf{At}^\complement(\mu_{\omega}))=0$. By \eqref{dis3}, we can take $p\in\mathsf{\Phi}_{z}\bigcap\mathsf{At}(\mu_{\omega})$. Considering the function $\chi_{\{z,p\}}$, we obtain \begin{align*}
 	\int_{\mathsf{\Phi}_{z}\bigcap\mathsf{At}(\mu_{\omega})}\frac{\chi_{\{\omega\ne0\}}}{\omega}\cdot \chi_{\{z,p\}}\D\mu_{\omega}\ge\int_{\{p\}}\frac{\chi_{\{\omega\ne0\}}}{\omega}\cdot\chi_{\{z,p\}}\D\mu_{\omega}=\frac{1}{\omega(p)}\cdot\mu_{\omega}(p)>0,
 \end{align*}
and this contradicts \eqref{dis4}. Therefore, we conclude that $\phi$ is injective a.e. $[\mu_{\omega}]$.

{\em Necessity}. Suppose $\phi$ is injective a.e. $[\mu_{\omega}]$, then for $z\in\mathsf{At}(\mu_{\omega})$, we have $\mu_{\omega}(\mathsf{\Phi}_{z})=0$, which implies 
\begin{align*}
&(\mathsf{E}_{\phi,\omega}(\frac{\chi_{\{\omega\ne0\}}}{\omega}\cdot f))(z)=\frac{\int_{\phi^{-1}(\phi(z))}\frac{\chi_{\{\omega\ne0\}}}{\omega}\cdot f\D\mu_{\omega}}{\mu_{\omega}(\phi^{-1}(\phi(z)))}\\&=\frac{\frac{1}{\omega(z)}\cdot f(z)\cdot\mu_{\omega}(\{z\})+\int_{\mathsf{\Phi}_{z}}\frac{\chi_{\{\omega\ne0\}}}{\omega}\cdot f\D\mu_{\omega}}{\mu_{\omega}(\{z\})+\mu_{\omega}(\mathsf{\Phi}_{z})}=\frac{1}{\omega(z)}f(z)\quad\quad z\in\mathsf{At}(\mu_{\omega})
\end{align*}
for any $f\in L^2(\mu)$. By \eqref{discount2}, we have $\mu_{\omega}(\mathsf{At}^\complement(\mu_{\omega}))=0$. Thus, we obtain
\begin{equation}\label{mul}
(\mathsf{E}_{\phi,\omega}(\frac{\chi_{\{\omega\ne0\}}}{\omega}\cdot f))(z)=\frac{\chi_{\{\omega\ne0\}}}{\omega(z)}\cdot f(z)\quad\text{a.e.}\ [\mu_{\omega}]
\end{equation}
for any $f\in L^2(\mu)$, that is,  $f\mapsto\mathsf{E}_{\phi,\omega}(\frac{\chi_{\{\omega\ne0\}}}{\omega}\cdot f)$ is a multiplication operator.

In this case, we claim that $C_{\phi,\omega}C_{\phi,\omega}^\ast=\mathsf{M}_{\mathsf{h}_{\phi,\omega}\circ\phi\cdot\chi_{\{\omega\ne0\}}}$. Since 
\begin{align*}
\|C_{\phi,\omega}C_{\phi,\omega}^\ast f\|^2&\stackrel{\eqref{98}}{=}\int_{X}|\omega\cdot\mathsf{h}_{\phi,\omega}\circ\phi\cdot\mathsf{E}_{\phi,\omega}(f_{\omega})|^2\D\mu=\int_{X}|\mathsf{h}_{\phi,\omega}\circ\phi\cdot\mathsf{E}_{\phi,\omega}(f_{\omega})|^2\D\mu_{\omega}\\&\stackrel{\eqref{mul}}{=}\int_{X}|\mathsf{h}_{\phi,\omega}\circ\phi|^2|f_{\omega}|^2\D\mu_{\omega}=\int_{X}\mathsf{h}_{\phi,\omega}^2\circ\phi\cdot\chi_{\{\omega\ne0\}}\cdot|f|^2\D\mu=\|\mathsf{M}_{\mathsf{h}_{\phi,\omega}\circ\phi\cdot\chi_{\{\omega\ne0\}}}f\|^2,
\end{align*}then by \eqref{99} we have $\EuScript{D}(C_{\phi,\omega}C_{\phi,\omega}^\ast)=\EuScript{D}(\mathsf{M}_{\mathsf{h}_{\phi,\omega}\circ\phi\cdot\chi_{\{\omega\ne0\}}})$, which implies that it suffices to prove that
\begin{equation}\label{equdis}
C_{\phi,\omega}C_{\phi,\omega}^\ast f=\mathsf{M}_{\mathsf{h}_{\phi,\omega}\circ\phi\cdot\chi_{\{\omega\ne0\}}}f\quad\text{a.e.}\ [\mu]
\end{equation}
 for $f\in\EuScript{D}(C_{\phi,\omega}C_{\phi,\omega}^\ast)$. Indeed, by \eqref{98} and \eqref{mul}, we have $$C_{\phi,\omega}C_{\phi,\omega}^\ast f=\omega\cdot\mathsf{h}_{\phi,\omega}\circ\phi\cdot\mathsf{E}_{\phi,\omega}(f_{\omega})=\omega\cdot\mathsf{h}_{\phi,\omega}\circ\phi\cdot \frac{\chi_{\{\omega\ne0\}}}{\omega}\cdot f=\mathsf{h}_{\phi,\omega}\circ\phi\cdot\chi_{\{\omega\ne0\}}\cdot f\quad\text{a.e.}\ [\mu_{\omega}]$$
for $f\in\EuScript{D}(C_{\phi,\omega}C_{\phi,\omega}^\ast)$. Using the above equality and the fact that $\mu^\omega$ and $\mu_{\omega}$ are mutuallly absolutely continuous, we obtain
\begin{equation*}
\mu(\{C_{\phi,\omega}C_{\phi,\omega}^\ast f\ne\mathsf{M}_{\mathsf{h}_{\phi,\omega}\circ\phi\cdot\chi_{\{\omega\ne0\}}}f\}\bigcap\{\omega\ne0\})=\mu^\omega(\{C_{\phi,\omega}C_{\phi,\omega}^\ast f\ne\mathsf{M}_{\mathsf{h}_{\phi,\omega}\circ\phi\cdot\chi_{\{\omega\ne0\}}}f\})=0
\end{equation*}
for $f\in\EuScript{D}(C_{\phi,\omega}C_{\phi,\omega}^\ast)$.  Thus, we conclude that
\begin{align*}
&\mu(\{C_{\phi,\omega}C_{\phi,\omega}^\ast f\ne\mathsf{M}_{\mathsf{h}_{\phi,\omega}\circ\phi\cdot\chi_{\{\omega\ne0\}}}f\})\\&=\mu(\{C_{\phi,\omega}C_{\phi,\omega}^\ast f\ne\mathsf{M}_{\mathsf{h}_{\phi,\omega}\circ\phi\cdot\chi_{\{\omega\ne0\}}}f\}\bigcap\{\omega\ne0\})+\mu(\{C_{\phi,\omega}C_{\phi,\omega}^\ast f\ne\mathsf{M}_{\mathsf{h}_{\phi,\omega}\circ\phi\cdot\chi_{\{\omega\ne0\}}}f\}\bigcap\{\omega=0\})\\&=0+\mu(\emptyset)=0.
\end{align*}
Therefore, \eqref{equdis} holds. The proof is complete.
\end{proof}

\begin{thm}\label{thm35}
Suppose the measure space $(X,\mathscr{A},\mu)$ is discrete. Let $p\in(0,1]$ and $C_{\phi_{i},\omega_{i}}$ be densely defined in $L^2(\mu)$ such that $\phi_{i}$ is injective a.e. $[\mu_{\omega_{i}}]$ for $i=1,2$. Then the following assertions are equivalent$:$
\begin{enumerate}
	\item[(i)]
	$\textbf{C}_{\phi,\omega}=(C_{\phi_{1},\omega_{1}},C_{\phi_{2},\omega_{2}})$ is spherically $p$-hyponormal,
	\item[(ii)]
	for every $f\in L^2(\mu)$,
$$\int_{X}(\mathsf{h}_{\phi_{1},\omega_{1}}\circ\phi_{1}\cdot\chi_{\{\omega_{1}\ne0\}}+\mathsf{h}_{\phi_{2},\omega_{2}}\circ\phi_{2}\cdot\chi_{\{\omega_{2}\ne0\}})^p|f|^2\D\mu\le\int_{X}(\mathsf{h}_{\phi_{1},\omega_{1}}+\mathsf{h}_{\phi_{2},\omega_{2}})^p|f|^2\D\mu,$$
	\item[(iii)]
	$\mathsf{h}_{\phi_{1},\omega_{1}}\circ\phi_{1}\cdot\chi_{\{\omega_{1}\ne0\}}+\mathsf{h}_{\phi_{2},\omega_{2}}\circ\phi_{2}\cdot\chi_{\{\omega_{2}\ne0\}}\le\mathsf{h}_{\phi_{1},\omega_{1}}+\mathsf{h}_{\phi_{2},\omega_{2}}$ a.e. $[\mu]$.
\end{enumerate}
\end{thm}
\begin{proof}
	By \eqref{selfadjointmut}, we have 
	\begin{equation}\label{charspp}
	(C_{\phi_{1},\omega_{1}}^\ast C_{\phi_{1},\omega_{1}}+C_{\phi_{2},\omega_{2}}^\ast C_{\phi_{2},\omega_{2}})^\frac{p}{2}=(\mathsf{M}_{\mathsf{h}_{\phi_{1},\omega_{1}}+\mathsf{h}_{\phi_{2},\omega_{2}}}))^\frac{p}{2}=\mathsf{M}_{(\mathsf{h}_{\phi_{1},\omega_{1}}+\mathsf{h}_{\phi_{2},\omega_{2}})^\frac{p}{2}}.
	\end{equation}
	By Lemma \ref{mulselfadjoint}, we obtain that
	\begin{align*}
		C_{\phi_{1},\omega_{1}}C_{\phi_{1},\omega_{1}}^\ast+C_{\phi_{2},\omega_{2}}C_{\phi_{2},\omega_{2}}^\ast&=\mathsf{M}_{\mathsf{h}_{\phi_{1},\omega_{1}}\circ\phi_{1}\cdot\chi_{\{\omega_{1}\ne0\}}}+\mathsf{M}_{\mathsf{h}_{\phi_{2},\omega_{2}}\circ\phi_{2}\cdot\chi_{\{\omega_{2}\ne0\}}}\\&=\mathsf{M}_{\mathsf{h}_{\phi_{1},\omega_{1}}\circ\phi_{1}\cdot\chi_{\{\omega_{1}\ne0\}}+\mathsf{h}_{\phi_{2},\omega_{2}}\circ\phi_{2}\cdot\chi_{\{\omega_{2}\ne0\}}}
	\end{align*}
is self-adjoint, which yields $$(C_{\phi_{1},\omega_{1}}C_{\phi_{1},\omega_{1}}^\ast+C_{\phi_{2},\omega_{2}}C_{\phi_{2},\omega_{2}}^\ast)^\frac{p}{2}=\mathsf{M}_{(\mathsf{h}_{\phi_{1},\omega_{1}}\circ\phi_{1}\cdot\chi_{\{\omega_{1}\ne0\}}+\mathsf{h}_{\phi_{2},\omega_{2}}\circ\phi_{2}\cdot\chi_{\{\omega_{2}\ne0\}})^\frac{p}{2}}.$$
	
(i) $\Rightarrow$ (ii): Since $$\|(C_{\phi_{1},\omega_{1}}C_{\phi_{1},\omega_{1}}^\ast+C_{\phi_{2},\omega_{2}}C_{\phi_{2},\omega_{2}}^\ast)^\frac{p}{2}f\|\le\|(C_{\phi_{1},\omega_{1}}^\ast C_{\phi_{1},\omega_{1}}+C_{\phi_{2},\omega_{2}}^\ast C_{\phi_{2},\omega_{2}})^\frac{p}{2}f\|$$
for $f\in\EuScript{D}((C_{\phi_{1},\omega_{1}}^\ast C_{\phi_{1},\omega_{1}}+C_{\phi_{2},\omega_{2}}^\ast C_{\phi_{2},\omega_{2}})^\frac{p}{2})$, then
\begin{align*}
	&\int_{X}|\mathsf{M}_{(\mathsf{h}_{\phi_{1},\omega_{1}}\circ\phi_{1}\cdot\chi_{\{\omega_{1}\ne0\}}+\mathsf{h}_{\phi_{2},\omega_{2}}\circ\phi_{2}\cdot\chi_{\{\omega_{2}\ne0\}})^\frac{p}{2}}f|^2\D\mu\\&=\int_{X}(\mathsf{h}_{\phi_{1},\omega_{1}}\circ\phi_{1}\cdot\chi_{\{\omega_{1}\ne0\}}+\mathsf{h}_{\phi_{2},\omega_{2}}\circ\phi_{2}\cdot\chi_{\{\omega_{2}\ne0\}})^p|f|^2\D\mu\\&\le\int_{X}|\mathsf{M}_{(\mathsf{h}_{\phi_{1},\omega_{1}}+\mathsf{h}_{\phi_{2},\omega_{2}})^\frac{p}{2}}f|^2\D\mu=\int_{X}(\mathsf{h}_{\phi_{1},\omega_{1}}+\mathsf{h}_{\phi_{2},\omega_{2}})^p|f|^2\D\mu
\end{align*}
for $f\in\EuScript{D}((C_{\phi_{1},\omega_{1}}^\ast C_{\phi_{1},\omega_{1}}+C_{\phi_{2},\omega_{2}}^\ast C_{\phi_{2},\omega_{2}})^\frac{p}{2})$. If $f\notin\EuScript{D}((C_{\phi_{1},\omega_{1}}^\ast C_{\phi_{1},\omega_{1}}+C_{\phi_{2},\omega_{2}}^\ast C_{\phi_{2},\omega_{2}})^\frac{p}{2})$, then by \eqref{charspp} we derive $\int_{X}(\mathsf{h}_{\phi_{1},\omega_{1}}+\mathsf{h}_{\phi_{2},\omega_{2}})^p|f|^2\D\mu=\infty$, which clearly implies that (ii) holds. 

(ii) $\Rightarrow$ (iii): Suppose (iii) does not hold. Since $X$ is $\sigma$-finite, there exists a measurable set $\Delta$ such that $0<\mu(\Delta)<\infty$ and $\mathsf{h}_{\phi_{1},\omega_{1}}\circ\phi_{1}\cdot\chi_{\{\omega_{1}\ne0\}}+\mathsf{h}_{\phi_{2},\omega_{2}}\circ\phi_{2}\cdot\chi_{\{\omega_{2}\ne0\}}>\mathsf{h}_{\phi_{1},\omega_{1}}+\mathsf{h}_{\phi_{2},\omega_{2}}$ a.e. $[\mu]$ on $\Delta$, which yields that$$\int_{X}(\mathsf{h}_{\phi_{1},\omega_{1}}\circ\phi_{1}\cdot\chi_{\{\omega_{1}\ne0\}}+\mathsf{h}_{\phi_{2},\omega_{2}}\circ\phi_{2}\cdot\chi_{\{\omega_{2}\ne0\}})^p|\chi_{\Delta}|^2\D\mu>\int_{X}(\mathsf{h}_{\phi_{1},\omega_{1}}+\mathsf{h}_{\phi_{2},\omega_{2}})^p|\chi_{\Delta}|^2\D\mu.$$ This contradicts (ii).

(iii) $\Rightarrow$ (i): By (iii), we have
\begin{equation}\label{114}
\int_{X}(\mathsf{h}_{\phi_{1},\omega_{1}}\circ\phi_{1}\cdot\chi_{\{\omega_{1}\ne0\}}+\mathsf{h}_{\phi_{2},\omega_{2}}\circ\phi_{2}\cdot\chi_{\{\omega_{2}\ne0\}})^p|f|^2\D\mu\le\int_{X}(\mathsf{h}_{\phi_{1},\omega_{1}}+\mathsf{h}_{\phi_{2},\omega_{2}})^p|f|^2\D\mu
\end{equation}
	for every $f\in L^2(\mu)$. Combining the proof of (i) $\Rightarrow$ (ii), we obtain$$\|(C_{\phi_{1},\omega_{1}}C_{\phi_{1},\omega_{1}}^\ast+C_{\phi_{2},\omega_{2}}C_{\phi_{2},\omega_{2}}^\ast)^\frac{p}{2}f\|\le\|(C_{\phi_{1},\omega_{1}}^\ast C_{\phi_{1},\omega_{1}}+C_{\phi_{2},\omega_{2}}^\ast C_{\phi_{2},\omega_{2}})^\frac{p}{2}f\|.$$ Furthermore, if $f\in\EuScript{D}((C_{\phi_{1},\omega_{1}}^\ast C_{\phi_{1},\omega_{1}}+C_{\phi_{2},\omega_{2}}^\ast C_{\phi_{2},\omega_{2}})^\frac{p}{2})$, it follows from \eqref{114} that $$\int_{X}(\mathsf{h}_{\phi_{1},\omega_{1}}\circ\phi_{1}\cdot\chi_{\{\omega_{1}\ne0\}}+\mathsf{h}_{\phi_{2},\omega_{2}}\circ\phi_{2}\cdot\chi_{\{\omega_{2}\ne0\}})^p|f|^2\D\mu<\infty,$$ and hence $f\in\EuScript{D}((C_{\phi_{1},\omega_{1}}C_{\phi_{1},\omega_{1}}^\ast+C_{\phi_{2},\omega_{2}}C_{\phi_{2},\omega_{2}}^\ast)^\frac{p}{2})$. Therefore, $\textbf{C}_{\phi,\omega}=(C_{\phi_{1},\omega_{1}},C_{\phi_{2},\omega_{2}})$ is spherically $p$-hyponormal.
\end{proof}
\begin{cor}\label{phypres}
Let $X$ be  a countable set and $\mu$ the counting measure on $X$. Suppose $\phi_{i}:X\to X$ is bijective and $\omega_{i}$ is a function mapping $X$ to $\mathbb{R}\backslash\{0\}$ for $i=1,2$. Let $\lambda\in(0,1)$. Then
\begin{enumerate}
	\item[(i)]
	$\textbf{C}_{\phi,\omega}=(C_{\phi_{1},\omega_{1}},C_{\phi_{2},\omega_{2}})$ is spherically $p$-hyponormal if and only if $$\omega_{1}^2+\omega_{2}^2\le\omega_{1}^2\circ\phi_{1}^{-1}+\omega_{2}^2\circ\phi_{2}^{-1}$$
	\item[(ii)]
$\mathcal{M}_{\lambda}(\textbf{C}_{\phi,\omega})=(C_{\phi_{1},\omega_{\lambda}^1},C_{\phi_{2},\omega_{\lambda}^2})$ is spherically $p$-hyponormal if and only if $$(\omega_{\lambda}^1)^2+(\omega_{\lambda}^2)^2\le(\omega_{\lambda}^1)^2\circ\phi_{1}^{-1}+(\omega_{\lambda}^2)^2\circ\phi_{2}^{-1},$$ where $\omega_{\lambda}^i=((1-\lambda)\sqrt{\frac{|\omega_{1}\circ\phi_{1}^{-1}|^2+|\omega_{2}\circ\phi_{2}^{-1}|^2}{|\omega_{1}\circ\phi_{1}^{-1}\circ\phi_{i}|^2+|\omega_{2}\circ\phi_{2}^{-1}\circ\phi_{i}|^2}}+\lambda)\cdot\omega_{i}$ for $i=1,2$.
\end{enumerate}
\end{cor}
\begin{proof}
(i)	Observe that $(X,2^X,\mu)$ is a discrete measure space and that $\phi_{i}$ is injective a.e. $[\mu_{\omega_{i}}]$ for $i=1,2$. From Equation \eqref{dish}, we have
\begin{equation}\label{dish2}
\mathsf{h}_{\phi_{i},\omega_{i}}(x)=|\omega_{i}(\phi_{i}^{-1}(x))|^2,
\end{equation}
which implies $C_{\phi_{i},\omega_{i}}$ is densely defined for $i=1,2$. Moreover, we derive the equivalence
\begin{align*}
&\mathsf{h}_{\phi_{1},\omega_{1}}\circ\phi_{1}\cdot\chi_{\{\omega_{1}\ne0\}}+\mathsf{h}_{\phi_{2},\omega_{2}}\circ\phi_{2}\cdot\chi_{\{\omega_{2}\ne0\}}\le\mathsf{h}_{\phi_{1},\omega_{1}}+\mathsf{h}_{\phi_{2},\omega_{2}}\\&\Leftrightarrow |\omega_{1}(x)|^2+|\omega_{2}(x)|^2\le|\omega_{1}(\phi_{1}^{-1}(x))|^2+|\omega_{2}(\phi_{2}^{-1}(x))|^2.
\end{align*}
Therefore, by applying Theorem \ref{thm35}, we obtain the desired result.

(ii) By \eqref{dish2} and recalling \eqref{wfun}, we have
\begin{align*}
\omega_{\lambda}^i&=((1-\lambda)\cdot\left(\frac{\mathsf{h}_{\phi_{1},\omega_{1}}+\mathsf{h}_{\phi_{2},\omega_{2}}}{(\mathsf{h}_{\phi_{1},\omega_{1}}+\mathsf{h}_{\phi_{2},\omega_{2}})\circ\phi_{i}}\right)^\frac{1}{2}+\lambda)\cdot\omega_{i}\\&=((1-\lambda)\sqrt{\frac{|\omega_{1}\circ\phi_{1}^{-1}|^2+|\omega_{2}\circ\phi_{2}^{-1}|^2}{|\omega_{1}\circ\phi_{1}^{-1}\circ\phi_{i}|^2+|\omega_{2}\circ\phi_{2}^{-1}\circ\phi_{i}|^2}}+\lambda)\cdot\omega_{i}.
\end{align*}
Since $\omega_i: X \to \mathbb{R}\backslash\{0\}$ for $i = 1,2$, it follows that $\omega_{\lambda}^i$ also maps $X$ to $\mathbb{R}\backslash\{0\}$. Replacing $(C_{\phi_{1},\omega_{1}},C_{\phi_{2},\omega_{2}})$ in (i) by $(C_{\phi_{1},\omega_{\lambda}^1},C_{\phi_{2},\omega_{\lambda}^2})$, we obtain the desired result.
\end{proof}
Let $\lambda\in(0,1)$. In general, the $\lambda$-spherical mean transform does not preserve the spherical $p$-hyponormality of weighted composite operator pair $(C_{\phi_{1},\omega_{1}},C_{\phi_{2},\omega_{2}})$, even when $X$ is a finite set and $\phi_{i}$ is a bijective for $i=1,2$. For the reader's convenience, we give such an example below. 
\begin{exa}
Let $X=\{1,2,3\}$ and $\mu$ be the counting measure on $X$. Define the transforms $\phi_{1}$ and $\phi_{2}$ on $X$ by $$\phi_{1}(x)=\begin{cases}
	2&x=1\\
	3&x=2\\
	1&x=3
\end{cases}\quad\text{and}\quad\phi_{2}(x)=\begin{cases}
3&x=1\\
1&x=2\\
2&x=3
\end{cases},$$respectively. Define the functions $\omega_{1}$ and $\omega_{2}$ on $X$ by $$\omega_{1}(x)=\begin{cases}
\sqrt{10}&x=1\\
1&x=2\\
1&x=3
\end{cases}\quad\text{and}\quad\omega_{2}(x)=\begin{cases}
\sqrt{10}&x=1\\
\sqrt{19}&x=2\\
\sqrt{10}&x=3
\end{cases},$$ respectively. Since
$$(\omega_{1}^2+\omega_{2}^2)(x)=\begin{cases}
	10+10=20&x=1\\
	1+19=20&x=2\\
	1+10=11&x=3
\end{cases}\quad\le\quad(\omega_{1}^2\circ\phi_{1}^{-1}+\omega_{2}^2\circ\phi_{2}^{-1})(x)=\begin{cases}
1+19=20&x=1\\
10+10=20&x=2\\
1+10=11&x=3
\end{cases},$$ then by Corollary \ref{phypres} (i) we know that $\textbf{C}_{\phi,\omega}=(C_{\phi_{1},\omega_{1}},C_{\phi_{2},\omega_{2}})$ is spherically $p$-hyponormal. Let $\lambda\in(0,1)$. A direct calculation gives
\begin{align*}
((\omega_{\lambda}^1)^2+(\omega_{\lambda}^2)^2)(1)&=(((1-\lambda)\sqrt{\frac{1+19}{10+10}}+\lambda)\cdot\sqrt{10})^2+(((1-\lambda)\sqrt{\frac{1+19}{1+10}}+\lambda)\cdot\sqrt{10})^2\\&=10+((1-\lambda)\sqrt{\frac{20}{11}}+\lambda)^2\cdot10\\&>20\\&>((1-\lambda)\sqrt{\frac{11}{20}}+\lambda)^2+19
\end{align*}
\begin{align*}
\quad\quad\quad\quad\quad\quad\quad&=(((1-\lambda)\sqrt{\frac{1+10}{1+19}}+\lambda)\cdot1)^2+(((1-\lambda)\sqrt{\frac{10+10}{1+19}}+\lambda)\cdot\sqrt{19})^2\\&=(\omega_{\lambda}^1)^2(3)+(\omega_{\lambda}^2)^2(2)\\&=(\omega_{\lambda}^1)^2\circ\phi_{1}^{-1}(1)+(\omega_{\lambda}^2)^2\circ\phi_{2}^{-1}(1),
\end{align*}which implies, by Corollary \ref{phypres} (ii), that $\mathcal{M}_{\lambda}(\textbf{C}_{\phi,\omega})=(C_{\phi_{1},\omega_{\lambda}^1},C_{\phi_{2},\omega_{\lambda}^2})$ is not spherically $p$-hyponormal.
\end{exa}

We close this paper with the unbounded counterparts of Theorem 3.1 and 3.6 established recently by Stanković \cite{svic-jmaa-2024}. Our results reduce to those in \cite{svic-jmaa-2024} when $\lambda=\frac{1}{2}$, $p=1$ and the sequences $\{\alpha_{(n,m)}\}$ and $\{\beta_{(n,m)}\}$ are bounded. For the sequel, recall that a real sequence $\{s_{k}\}_{k\in\mathbb{Z}_{+}}$ is called a {\em Stieltjes moment sequence} if there exists
a positive Borel measure $\mu$ on $[0,\infty)$ such that$$s_{k}=\int_{0}^\infty t^k\D\mu(t),\quad k\in\mathbb{Z}_{+}.$$Using the notation introduced above for the $2$-variable weighted shift, we have the following result.
\begin{cor}\label{2vari}
Let $\textbf{C}_{\phi,\omega}=(C_{\phi_{1},\omega_{1}},C_{\phi_{2},\omega_{2}})$ be a $2$-variable weighted shift. Let $\lambda\in(0,1)$ and $p\in(0,1]$. Then
\begin{enumerate}
	\item[(i)]
	$\textbf{C}_{\phi,\omega}=(C_{\phi_{1},\omega_{1}},C_{\phi_{2},\omega_{2}})$ is spherically $p$-hyponormal if and only if $$\alpha_{(n-1,m)}^2+\beta_{(n,m-1)}^2\le\alpha_{(n,m)}^2+\beta_{(n,m)}^2$$ for $(n,m)\in\mathbb{Z}_{+}\times\mathbb{Z}_{+}$, where $\alpha_{(-1,m)}=\beta_{(n,-1)}=0$ for $n,m\in\mathbb{Z}_{+}$.
	\item[(ii)]
	$\mathcal{M}_{\lambda}(\textbf{C}_{\phi,\omega})=(C_{\phi_{1},\omega_{\lambda}^1},C_{\phi_{2},\omega_{\lambda}^2})$ is a $2-$variable weighted shift with domain$$\EuScript{D}(\mathcal{M}_{\lambda}(\textbf{C}_{\phi,\omega}))=\{f:\sum\limits_{(n,m)}|f(n,m)|^2\cdot g(n,m)<\infty\},$$
	and its weighted functions are
\begin{align*}
		\omega_{\lambda}^1(n,m)=\begin{cases}
			0&(n,m)\in\{0\}\times\mathbb{Z}_{+}\\
			\widetilde{\alpha_{(n-1,m)}}&(n,m)\in(\mathbb{Z}_{+}\backslash\{0\})\times\mathbb{Z}_{+}
		\end{cases}
	\end{align*}
	and
	\begin{align*}
		\omega_{\lambda}^2(n,m)=\begin{cases}
			0&(n,m)\in\mathbb{Z}_{+}\times\{0\}\\
			\widetilde{\beta_{(n,m-1)}}&(n,m)\in\mathbb{Z}_{+}\times(\mathbb{Z}_{+}\backslash\{0\})
		\end{cases},
	\end{align*}
	where 
	$$g(n,m)=1+\alpha_{(n,m)}^2+\beta_{(n,m)}^2+\frac{\alpha_{(n,m)}^2(\alpha_{(n+1,m)}^2+\beta_{(n+1,m)}^2)+\beta_{(n,m)}^2(\alpha_{(n,m+1)}^2+\beta_{(n,m+1)}^2)}{\alpha_{(n,m)}^2+\beta_{(n,m)}^2}.$$	
The positive sequences $\{\widetilde{\alpha_{(n,m)}}\}$ and $\{\widetilde{\beta_{(n,m)}}\}$ are defined as $$\{\widetilde{\alpha_{(n,m)}}\}=((1-\lambda)\sqrt{\frac{\alpha_{(n+1,m)}^2+\beta_{(n+1,m)}^2}{\alpha_{(n,m)}^2+\beta_{(n,m)}^2}}+\lambda)\cdot\alpha_{(n,m)}$$ and $$\{\widetilde{\beta_{(n,m)}}\}=((1-\lambda)\sqrt{\frac{\alpha_{(n,m+1)}^2+\beta_{(n,m+1)}^2}{\alpha_{(n,m)}^2+\beta_{(n,m)}^2}}+\lambda)\cdot\beta_{(n,m)},$$respectively.
	\item[(iii)]
	$\mathcal{M}_{\lambda}(\textbf{C}_{\phi,\omega})$ is spherically $p$-hyponormal if and only if 
	\begin{small}
	\begin{align*}
		&((1-\lambda)\sqrt{\frac{\alpha_{(n,m)}^2+\beta_{(n,m)}^2}{\alpha_{(n-1,m)}^2+\beta_{(n-1,m)}^2}}+\lambda)^2\cdot\alpha_{(n-1,m)}^2+((1-\lambda)\sqrt{\frac{\alpha_{(n,m)}^2+\beta_{(n,m)}^2}{\alpha_{(n,m-1)}^2+\beta_{(n,m-1)}^2}}+\lambda)^2\cdot\beta_{(n,m-1)}^2\\&\le((1-\lambda)\sqrt{\frac{\alpha_{(n+1,m)}^2+\beta_{(n+1,m)}^2}{\alpha_{(n,m)}^2+\beta_{(n,m)}^2}}+\lambda)^2\cdot\alpha_{(n,m)}^2+((1-\lambda)\sqrt{\frac{\alpha_{(n,m+1)}^2+\beta_{(n,m+1)}^2}{\alpha_{(n,m)}^2+\beta_{(n,m)}^2}}+\lambda)^2\cdot\beta_{(n,m)}^2
	\end{align*}
	\end{small}
	\item[(iv)]
	the $\lambda$-spherical mean transform preserves the spherical $p$-hyponormality of $\textbf{C}_{\phi,\omega}$ if the positive sequences $\{\alpha_{(n,m)}\}$ and $\{\beta_{(n,m)}\}$ defining the weight functions $\omega_{1}$ and $\omega_{2}$ satisfy the following conditions:
	 \begin{enumerate}
	 	\item $\alpha_{(n+1,m)} = \alpha_{(n,m+1)}$ and $\beta_{(n+1,m)} = \beta_{(n,m+1)}$ for all $n, m \ge 0$;
	 	\item For every fixed $m \ge 0$, the sequence $\{\alpha_{(n,m)}^2 + \beta_{(n,m)}^2\}_{n}$ is a Stieltjes moment sequence in $n$.
	 \end{enumerate}
\end{enumerate}
\end{cor}
\begin{proof}
(i) Since the measure space $(\mathbb{Z}_{+}\times\mathbb{Z}_{+},2^{\mathbb{Z}_{+}\times\mathbb{Z}_{+}},\mu)$ equipped with the counting measure is discrete, it suffices to prove that $\phi_{i}$ is injective a.e. $[\mu_{\omega_{i}}]$ for $i=1,2$. The desired result then follows from Theorem \ref{thm35}. A simple calculation shows that $$\mu_{\omega_{1}}(\mathsf{\Phi_{1}}_{(n,m)})=\begin{cases}
\mu_{\omega_{1}}(\emptyset)=0&(n,m)\in(\mathbb{Z}_{+}\backslash\{0,1\})\times\mathbb{Z}_{+}\\
\mu_{\omega_{1}}(\{0\}\times\mathbb{Z}_{+})=0&(1,m)\in\{1\}\times\mathbb{Z}_{+}
\end{cases}$$
and
$$\mu_{\omega_{2}}(\mathsf{\Phi_{2}}_{(n,m)})=\begin{cases}
	\mu_{\omega_{2}}(\emptyset)=0&(n,m)\in\mathbb{Z}_{+}\times(\mathbb{Z}_{+}\backslash\{0,1\})\\
	\mu_{\omega_{2}}(\mathbb{Z}_{+}\times\{0\})=0&(n,1)\in\mathbb{Z}_{+}\times\{1\}
\end{cases}.$$ Thus, $\phi_{i}$ is injective a.e. $[\mu_{\omega_{i}}]$ for $i=1,2$. Moreover, we have
\begin{align*}
&\mathsf{h}_{\phi_{1},\omega_{1}}\circ\phi_{1}\cdot\chi_{\{\omega_{1}\ne0\}}+\mathsf{h}_{\phi_{2},\omega_{2}}\circ\phi_{2}\cdot\chi_{\{\omega_{2}\ne0\}}\le\mathsf{h}_{\phi_{1},\omega_{1}}+\mathsf{h}_{\phi_{2},\omega_{2}}\\&\Leftrightarrow\begin{cases}
0\le\alpha_{(0,0)}^2+\beta_{(0,0)}^2&(n,m)=(0,0)\\
\beta_{(n,m-1)}^2\le\alpha_{(n,m)}^2+\beta_{(n,m)}^2&(n,m)\in\{0\}\times(\mathbb{Z}_{+}\backslash\{0\})\\
\alpha_{(n-1,m)}^2\le\alpha_{(n,m)}^2+\beta_{(n,m)}^2&(n,m)\in(\mathbb{Z}_{+}\backslash\{0\})\times\{0\}\\
\alpha_{(n-1,m)}^2+\beta_{(n,m-1)}^2\le\alpha_{(n,m)}^2+\beta_{(n,m)}^2&(n,m)\in(\mathbb{Z}_{+}\backslash\{0\})\times(\mathbb{Z}_{+}\backslash\{0\}).
\end{cases}
\end{align*}
Therefore, we obtain the desired result.
	
(ii) By \eqref{2vws1}, \eqref{2vws2} and recalling \eqref{wfun}, we have
\begin{align*}
	\omega_{\lambda}^1(n,m)&=((1-\lambda)\cdot\chi_{\{\omega_{1}\ne0\}}\cdot\left(\frac{\mathsf{h}_{\phi_{1},\omega_{1}}+\mathsf{h}_{\phi_{2},\omega_{2}}}{(\mathsf{h}_{\phi_{1},\omega_{1}}+\mathsf{h}_{\phi_{2},\omega_{2}})\circ\phi_{1}}\right)^\frac{1}{2}+\lambda)(n,m)\cdot\omega_{1}(n,m)\\&=((1-\lambda)\cdot\chi_{\{\omega_{1}\ne0\}}\cdot\sqrt{\frac{\omega_{1}^2(n+1,m)+\omega_{2}^2(n,m+1)}{\omega_{1}^2(n,m)+\omega_{2}^2(n-1,m+1)}}+\lambda)\cdot\omega_{1}(n,m)\\&=\begin{cases}
		0&(n,m)\in\{0\}\times\mathbb{Z}_{+}\\
		((1-\lambda)\sqrt{\frac{\alpha_{(n,m)}^2+\beta_{(n,m)}^2}{\alpha_{(n-1,m)}^2+\beta_{(n-1,m)}^2}}+\lambda)\cdot\alpha_{(n-1,m)}&(n,m)\in(\mathbb{Z}_{+}\backslash\{0\})\times\mathbb{Z}_{+}
	\end{cases}
\end{align*}
and
\begin{align*}
	\omega_{\lambda}^2(n,m)&=((1-\lambda)\cdot\chi_{\{\omega_{2}\ne0\}}\cdot\left(\frac{\mathsf{h}_{\phi_{1},\omega_{1}}+\mathsf{h}_{\phi_{2},\omega_{2}}}{(\mathsf{h}_{\phi_{1},\omega_{1}}+\mathsf{h}_{\phi_{2},\omega_{2}})\circ\phi_{2}}\right)^\frac{1}{2}+\lambda)(n,m)\cdot\omega_{2}(n,m)\\&=((1-\lambda)\cdot\chi_{\{\omega_{2}\ne0\}}\cdot\sqrt{\frac{\omega_{1}^2(n+1,m)+\omega_{2}^2(n,m+1)}{\omega_{1}^2(n+1,m-1)+\omega_{2}^2(n,m)}}+\lambda)\cdot\omega_{2}(n,m)
\end{align*}
\begin{align*}
	\quad\quad\quad\quad=\begin{cases}
		0&(n,m)\in\mathbb{Z}_{+}\times\{0\}\\
		((1-\lambda)\sqrt{\frac{\alpha_{(n,m)}^2+\beta_{(n,m)}^2}{\alpha_{(n,m-1)}^2+\beta_{(n,m-1)}^2}}+\lambda)\cdot\beta_{(n,m-1)}&(n,m)\in\mathbb{Z}_{+}\times(\mathbb{Z}_{+}\backslash\{0\})
	\end{cases},
\end{align*}

which yields that 
\begin{align*}
C_{\phi_{1},\omega_{\lambda}^1}e_{(n,m)}&=\omega_{\lambda}^1\cdot e_{(n,m)}\circ\phi_{1}=\omega_{\lambda}^1(n+1,m)\cdot e_{(n+1,m)}\\&=((1-\lambda)\sqrt{\frac{\alpha_{(n+1,m)}^2+\beta_{(n+1,m)}^2}{\alpha_{(n,m)}^2+\beta_{(n,m)}^2}}+\lambda)\cdot\alpha_{(n,m)}\cdot e_{(n+1,m)}
\end{align*}
and
\begin{align*}
	C_{\phi_{2},\omega_{\lambda}^2}e_{(n,m)}&=\omega_{\lambda}^2\cdot e_{(n,m)}\circ\phi_{2}=\omega_{\lambda}^2(n,m+1)\cdot e_{(n,m+1)}\\&=((1-\lambda)\sqrt{\frac{\alpha_{(n,m+1)}^2+\beta_{(n,m+1)}^2}{\alpha_{(n,m)}^2+\beta_{(n,m)}^2}}+\lambda)\cdot\beta_{(n,m)}\cdot e_{(n,m+1)}.
\end{align*}
By \eqref{disdomain}, we have
\begin{align*}
(\mathsf{E}_{\phi_{1},\omega_{1}}(\mathsf{h}_{\phi_{1},\omega_{1}}+\mathsf{h}_{\phi_{2},\omega_{2}})\circ\phi_{1}^{-1})(n,m)&=\frac{\sum\limits_{y\in\phi_{1}^{-1}(\{(n,m)\})}(\mathsf{h}_{\phi_{1},\omega_{1}}+\mathsf{h}_{\phi_{2},\omega_{2}})(y)|\omega_{1}(y)|^2}{\sum\limits_{y\in\phi_{1}^{-1}(\{(n,m)\})}|\omega_{1}(y)|^2}\\&=\alpha_{(n+1,m)}^2+\beta_{(n+1,m)}^2
\end{align*}
and
\begin{align*}
	(\mathsf{E}_{\phi_{2},\omega_{2}}(\mathsf{h}_{\phi_{1},\omega_{1}}+\mathsf{h}_{\phi_{2},\omega_{2}})\circ\phi_{2}^{-1})(n,m)&=\frac{\sum\limits_{y\in\phi_{2}^{-1}(\{(n,m)\})}(\mathsf{h}_{\phi_{1},\omega_{1}}+\mathsf{h}_{\phi_{2},\omega_{2}})(y)|\omega_{2}(y)|^2}{\sum\limits_{y\in\phi_{2}^{-1}(\{(n,m)\})}|\omega_{2}(y)|^2}\\&=\alpha_{(n,m+1)}^2+\beta_{(n,m+1)}^2,
\end{align*}
which yields that $$\Big(\big(\mathsf{E}_{\phi_{1},\omega_{i}}(\mathsf{h}_{\phi_{1},\omega_{1}}+\mathsf{h}_{\phi_{2},\omega_{2}})\circ\phi_{1}^{-1}\big)\cdot\frac{\mathsf{h}_{\phi_{1},\omega_{1}}\cdot\chi_{\{\mathsf{h}_{\phi_{1},\omega_{1}}+\mathsf{h}_{\phi_{2},\omega_{2}}>0\}}}{\mathsf{h}_{\phi_{1},\omega_{1}}+\mathsf{h}_{\phi_{2},\omega_{2}}}\Big)(n,m)=\frac{\alpha_{(n,m)}^2(\alpha_{(n+1,m)}^2+\beta_{(n+1,m)}^2)}{\alpha_{(n,m)}^2+\beta_{(n,m)}^2}$$
and
$$\Big(\big(\mathsf{E}_{\phi_{2},\omega_{2}}(\mathsf{h}_{\phi_{1},\omega_{1}}+\mathsf{h}_{\phi_{2},\omega_{2}})\circ\phi_{2}^{-1}\big)\cdot\frac{\mathsf{h}_{\phi_{2},\omega_{2}}\cdot\chi_{\{\mathsf{h}_{\phi_{1},\omega_{1}}+\mathsf{h}_{\phi_{2},\omega_{2}}>0\}}}{\mathsf{h}_{\phi_{1},\omega_{1}}+\mathsf{h}_{\phi_{2},\omega_{2}}}\Big)(n,m)=\frac{\beta_{(n,m)}^2(\alpha_{(n,m+1)}^2+\beta_{(n,m+1)}^2)}{\alpha_{(n,m)}^2+\beta_{(n,m)}^2}.$$
Thus, by Theorem \ref{meanbasic} (i) we conclude that
$$\EuScript{D}(\mathcal{M}_{\lambda}(\textbf{C}_{\phi,\omega}))=\{f:\sum\limits_{(n,m)}|f(n,m)|^2\cdot g(n,m)<\infty\},$$ where $g(n,m)=1+\alpha_{(n,m)}^2+\beta_{(n,m)}^2+\frac{\alpha_{(n,m)}^2(\alpha_{(n+1,m)}^2+\beta_{(n+1,m)}^2)+\beta_{(n,m)}^2(\alpha_{(n,m+1)}^2+\beta_{(n,m+1)}^2)}{\alpha_{(n,m)}^2+\beta_{(n,m)}^2}.$

(iii) Since $\mathcal{M}_{\lambda}(\textbf{C}_{\phi,\omega})$ is a $2$-variable weighted shift, the desired result is obtained by replacing $(C_{\phi_{1},\omega_{1}},C_{\phi_{2},\omega_{2}})$ in (i) with $(C_{\phi_{1},\omega_{\lambda}^1},C_{\phi_{2},\omega_{\lambda}^2})$.

(iv) Suppose $\textbf{C}_{\phi,\omega}=(C_{\phi_{1},\omega_{1}},C_{\phi_{2},\omega_{2}})$ is spherically $p$-hyponormal. By condition (b), for every $m\ge0$, there exists a positive Borel measure $\mu_{m}$ on $[0,\infty)$ such that 
$$\alpha_{(n,m)}^2 + \beta_{(n,m)}^2=\int_{0}^\infty t^n\D\mu_{m}.$$ It follows that
\begin{align*}
	(\alpha_{(n,m)}^2 + \beta_{(n,m)}^2)^2&=(\int_{0}^\infty t^n\D\mu_{m})^2=(\int_{0}^\infty t^\frac{n-1}{2}t^\frac{n+1}{2}\D\mu_{m})^2\\&\stackrel{\text{Cauchy's inequality}}{\le}\int_{0}^\infty t^{n-1}\D\mu_{m}\int_{0}^\infty t^{n+1}\D\mu_{m}\\&=(\alpha_{(n-1,m)}^2 + \beta_{(n-1,m)}^2)\cdot(\alpha_{(n+1,m)}^2 + \beta_{(n+1,m)}^2).
\end{align*}
Hence, we have\begin{equation}\label{stie}
	\frac{\alpha_{(n+1,m)}^2+\beta_{(n+1,m)}^2}{\alpha_{(n,m)}^2+\beta_{(n,m)}^2}\ge\frac{\alpha_{(n,m)}^2+\beta_{(n,m)}^2}{\alpha_{(n-1,m)}^2+\beta_{(n-1,m)}^2}.
\end{equation}
Thus, we obtain
\begin{align*}
&((1-\lambda)\sqrt{\frac{\alpha_{(n+1,m)}^2+\beta_{(n+1,m)}^2}{\alpha_{(n,m)}^2+\beta_{(n,m)}^2}}+\lambda)^2\cdot\alpha_{(n,m)}^2+((1-\lambda)\sqrt{\frac{\alpha_{(n,m+1)}^2+\beta_{(n,m+1)}^2}{\alpha_{(n,m)}^2+\beta_{(n,m)}^2}}+\lambda)^2\cdot\beta_{(n,m)}^2\\&\xlongequal{\text{condition (a)}}((1-\lambda)\sqrt{\frac{\alpha_{(n+1,m)}^2+\beta_{(n+1,m)}^2}{\alpha_{(n,m)}^2+\beta_{(n,m)}^2}}+\lambda)^2\cdot(\alpha_{(n,m)}^2+\beta_{(n,m)}^2)\\&\stackrel{\text{Corollary \ref{2vari} (i)}}{\ge}((1-\lambda)\sqrt{\frac{\alpha_{(n+1,m)}^2+\beta_{(n+1,m)}^2}{\alpha_{(n,m)}^2+\beta_{(n,m)}^2}}+\lambda)^2\cdot(\alpha_{(n-1,m)}^2+\beta_{(n,m-1)}^2)\\&\stackrel{\eqref{stie}}{\ge}((1-\lambda)\sqrt{\frac{\alpha_{(n,m)}^2+\beta_{(n,m)}^2}{\alpha_{(n-1,m)}^2+\beta_{(n-1,m)}^2}}+\lambda)^2\cdot(\alpha_{(n-1,m)}^2+\beta_{(n,m-1)}^2)\\&\xlongequal{\text{condition (a)}}\\&((1-\lambda)\sqrt{\frac{\alpha_{(n,m)}^2+\beta_{(n,m)}^2}{\alpha_{(n-1,m)}^2+\beta_{(n-1,m)}^2}}+\lambda)^2\cdot\alpha_{(n-1,m)}^2+((1-\lambda)\sqrt{\frac{\alpha_{(n,m)}^2+\beta_{(n,m)}^2}{\alpha_{(n,m-1)}^2+\beta_{(n,m-1)}^2}}+\lambda)^2\cdot\beta_{(n,m-1)}^2,
\end{align*}
which implies, by part (iii) of the corollary, that $\mathcal{M}_{\lambda}(\textbf{C}_{\phi,\omega})$ is spherically $p$-hyponormal.
\end{proof}

	\bibliographystyle{amsalpha}

\begin{thebibliography}{99}
		\bibitem{alu-ieot-1990} A. Aluthge, On $p$-hyponormal operators for $0 < p < 1$, Integr. Equ. Oper. Theory {\bf 13} (1990), 307-315.
		\bibitem{b-j-j-sW} P. Budzy\'{n}ski, Z. J. Jab{\l}o\'nski, I. B. Jung, J. Stochel, Unbounded weighted composition operators in $L^2$-spaces, Lectures Notes in Mathematics {\bf 2209} (2018), Springer.
     	\bibitem{lee-jmaa-2014} S. H. Lee, W. Y. Lee, J. Yoon, The mean transform of bounded linear operators, J. Math. Anal. Appl. {\bf 410} (2014), 70-81.
     	\bibitem{cho-smj-2002}
     	M. Ch\={o}, K. Tanahashi, Spectral relations for Aluthge transform, Sci. Math. Jpn. {\bf 55} (2002), 77-83.
     	\bibitem{curto-crm-2019}
     	C. Benhida, R. E. Curto, S. H. Lee, J. Yoon, Joint spectra of spherical Aluthge transforms of commuting $n$-tuples of Hilbert space operators, C. R. Math. {\bf 357} (2019), 799-802.
     	\bibitem{js-jmsj-2003}
     	 J. Stochel, F. H. Szafraniec, Domination of unbounded operators and commutativity, J. Math. Soc. Japan {\bf 55} (2003), 405–437.
     	 \bibitem{M-pams-2008}
     	 M. Möller, F. H. Szafraniec, Adjoints and formal adjoints of matrices of unbounded operators, Proc. Am. Math. Soc. {\bf 136} (2008), 2165–2176.
     	 \bibitem{makar-spr-2013}
     	 B. Makarov, A. Podkorytov, Real analysis: measures, integrals and applications. Universitext (2013), Springer.
     	 \bibitem{sjj-mn-2023}
     	 S. Djordjević, J. Kim, J. Yoon, Generalized spherical Aluthge transforms and binormality for commuting pairs of operators, Math. Nachr. {\bf 296} (2023), 2734–2757.
     	 \bibitem{svic-jmaa-2024}
     	 H. Stanković, Spherical mean transform of operator pairs, J. Math. Anal. Appl. {\bf 530}(2) (2024), 127743.
     	 \bibitem{benhida-mn-2020}
     	 C. Benhida, P. Budzyński, J. Trepkowski, Aluthge transforms of unbounded weighted composition operators in $L^2$‐spaces, Math. Nachr. {\bf 293}(10) (2020), 1888-1910.
     	 \bibitem{muller-studia-1992}
     	 V. M{\"u}ller, A. Soltysiak, Spectral radius formula for commuting Hilbert space operators, Studia Math. {\bf 103}(3) (1992), 329-333.
     	 \bibitem{trep-jmaa-2015}
     	 J. Trepkowski, Aluthge transforms of weighted shifts on directed trees, J. Math. Anal. Appl. {\bf 425} (2015), 886–899.
     	 \bibitem{curto-ieot-2018}
     	 R. Curto, J. Yoon, Aluthge transforms of $2$-variable weighted shifts, Integr. Equ. Oper. Theory {\bf 90}(5) (2018), 52.
     	 \bibitem{zhou-glma-2023}
     	 J. B. Zhou, Q. Guo, A new generalization of Aluthge transform for tuples of operators, Linear Multilinear A. {\bf 72}(15) (2023), 2466-2488.
     	 %\bibitem{Jabb-rmjm-2020}
     	 %M. R. Jabbarzadeh, S. Haghighatjoo, Equivalent metrics on normal composition %operators, Rocky Mt. J. Math. {\bf 50} (2020), 989-999.
     	 \bibitem{curto-AOF-2019}
     	 R. Curto, J. Yoon, Spherically quasinormal pairs of commuting operators. In: Analysis of Operators on Function Spaces (The Serguei Shimorin Memorial Volume), Trends in Mathmatics, Birkh{\"a}user (2019), pp. 213–237.
     	 \bibitem{Yoon-glma-2022}
     	 H. W. Kim, J. Kim, J. Yoon, Spherical Aluthge transform, spherical $p$ and log–hyponormality of commuting pairs of operators, Linear Multilinear A. {\bf 70}(11) (2022), 2047–2064.
     	 \bibitem{birman-book-1987}
     	 M. Sh. Birman, M. Z. Solomjak, Spectral Theory of Selfadjoint Operators in Hilbert Space, D. Reidel Publishing Co., Dordrecht, 1987.
     	 \bibitem{Bud-AMPP-2014}
     	 P. Budzy\'{n}ski, Z. J. Jab{\l}o\'nski, I. B. Jung, J. Stochel, On unbounded composition operators in $L^2$-spaces, Ann. Mat. Pura Appl. {\bf 193}(4) (2014), 663–688.
     	 \bibitem{jab-ieot-2014}
     	 Z. J. Jab{\l}o\'nski, I. B. Jung, J. Stochel, Unbounded quasinormal operators revisited. Integr. Equ. Oper. Theory, {\bf 79}(1) (2014), 135-149.
     	 %\bibitem{faris-lec-1975}
     	 %W. Faris, Self-adjoint operators, Lectures Notes in Mathematics {\bf 433} (1975), Springer.
     	 \bibitem{curto-CRA-2016}
     	 R. Curto, J. Yoon, Toral and spherical Aluthge transforms for $2$-variable
     	 weighted shifts. C. R. Acad. Sci. Paris. {\bf 354} (2016), 1200–1204.
     	 \bibitem{Yoon-jmaa-2011}
     	  J. Yoon, When does the $k$-hyponormality for a $2$-variable weighted shift
     	 become subnormality?, J. Math. Anal. Appl. {\bf 379} (2011), 487–498.
     	 \bibitem{zamani-jmaa-2021}
     	  A. Zamani, On an extension of operator transforms, J. Math. Anal. Appl. {\bf 493} (2021), 124546.
     	 \bibitem{jung-ieot-2000}
     	  I.B. Jung, E. Ko, C. Pearcy, Aluthge transform of operators, Integr. Equ. Oper. Theory {\bf 13} (2000), 437–448.
     	  \bibitem{jung-ieot-2001}
     	  I.B. Jung, E. Ko, C. Pearcy, Spectral pictures of Aluthge transforms of operators, Integr. Equ. Oper. Theory {\bf 40} (2001), 52–60.
     	  \bibitem{Kim-GM-2005}
     	  M. K. Kim, E. Ko, Some connections between an operator and its
     	  Aluthge transform. Glasg. Math. J. {\bf 47} (2005), 167–175.
     	  \bibitem{Zid-Fil-2022}
     	  S. Zid, S. Menkad, The $\lambda$-Aluthge transform and its applications to some classes of operators. Filomat {\bf 36} (2022), 289–301.
     	  \bibitem{jabb-fil-2017}
     	  M. R. Jabbarzadeh and M. J. Bakhshkandi, Centered operators via Moore-Penrose inverse and Aluthge transformations, Filomat {\bf 31}(20) (2017), 6441-6448.
     	  \bibitem{ben-bjma-2020}
     	  C. Benhida, M. Ch\={o}, E. Ko, J. E. Lee, On the generalized mean transforms of complex symmetric operators. Banach J. Math. Anal. {\bf 14} (2020), 842-855.
     	  \bibitem{cha-pro-2019}
     	  F. Chabbabi, R. Curto, M. Mbekhta, The mean transform and the mean limit of an operator, Proc. Am. Math. Soc. {\bf 147}(3) (2019), 1119–1133.
     	  \bibitem{mortad-2022-book}
     	  M. H. Mortad, Counterexamples in Operator Theory. Birkhäuser/Springer, Cham (2022).
     	  
     	 
	\end{thebibliography}
	
\end{document}